\documentclass[11pt]{amsart}
\usepackage{amsmath,amsthm,amsfonts}
\usepackage{latexsym,txfonts}
\usepackage{amssymb}
\usepackage{slashed}
\usepackage{pifont}
\usepackage{enumitem}
\usepackage{cite}
\usepackage{verbatim}

\newcommand{\R}{\mathbb{R}}

\newcommand{\C}{\mathbb{C}}
\newcommand{\N}{\mathbb{N}}
\newcommand{\Z}{\mathbb{Z}}
\newcommand{\F}{\mathcal{F}}

\renewcommand{\Re}{\mathop{\mathrm{Re}}}
\renewcommand{\Im}{\mathop{\mathrm{Im}}}

\renewcommand{\bar}{\overline}
\renewcommand{\hat}{\widehat}

\numberwithin{equation}{section}

\newtheorem{thm}{Theorem}[section]
\newtheorem{cor}[thm]{Corollary}
\newtheorem{lem}[thm]{Lemma}
\newtheorem{prop}[thm]{Proposition}

\newtheorem{rem}[thm]{Remark}
\newtheorem{defn}[thm]{Definition}

\newcommand{\Del}[1]{}

\newcommand{\supp}{\operatorname{supp}}

\newcommand{\cA}{ {\mathcal A} }
\newcommand{\cF}{ {\mathcal F} }
\newcommand{\cD}{ {\mathcal D} }
\newcommand{\tA}{ \tilde{A} }
\newcommand{\tp}{ \tilde{\phi} }

\begin{document}

\title[Concentration Compactness for the critical MKG equation]{Concentration Compactness for the Critical Maxwell-Klein-Gordon Equation}

\begin{abstract}
 We prove global regularity, scattering and a priori bounds for the energy critical Maxwell-Klein-Gordon equation relative to the Coulomb gauge on $(1+4)$-dimensional Minkowski space. The proof is based upon a modified Bahouri-G\'erard profile decomposition \cite{Bahouri-Gerard} and a concentration compactness/rigidity argument by Kenig-Merle \cite{KM}, following the method developed by the first author and Schlag \cite{KS} in the context of critical wave maps. 
\end{abstract}

\author[]{Joachim Krieger}
\address{B\^atiment des Math\'ematiques \\ EPFL \\ Station 8 \\ 1015 Lausanne \\ Switzerland}
\email{joachim.krieger@epfl.ch}

\author[]{Jonas L\"uhrmann}
\address{Departement Mathematik \\ ETH Z\"urich \\ 8092 Z\"urich \\ Switzerland}
\email{jonas.luehrmann@math.ethz.ch}

\thanks{The first author was supported in part by the Swiss National Science Foundation under Consolidator Grant BSCGI0\_157694. The second author was supported in part by the Swiss National Science Foundation under grant SNF 200020-159925.}

\maketitle

\setcounter{tocdepth}{1}
\tableofcontents

\section{Introduction}

The Maxwell-Klein-Gordon system on Minkowski space-time $\R^{1+n}$, $n\geq 1$, is a classical field theory for a complex scalar field $\phi: \R^{1+n} \to \C$ and a connection $1$-form $A_\alpha: \R^{1+n} \to \R$ for $\alpha = 0, 1, \ldots, n$. Defining the covariant derivative
\[
 D_\alpha = \partial_\alpha + i A_\alpha
\]
and the curvature $2$-form
\[
 F_{\alpha \beta} = \partial_\alpha A_\beta - \partial_\beta A_\alpha,
\]
the formal Lagrangian action functional of the Maxwell-Klein-Gordon system is given by
\[
 \int_{\R^{1+n}} \Big( \frac{1}{4} F_{\alpha\beta} F^{\alpha\beta} + \frac{1}{2} D_{\alpha} \phi \overline{D^{\alpha}\phi}  \Big) \, dx \, dt,
\]
where the Einstein summation convention is in force and Minkowski space $\R^{1+n}$ is endowed with the standard metric $\text{diag}(-1, +1, \ldots, +1)$. Then the Maxwell-Klein-Gordon equations are the associated Euler-Lagrange equations
\begin{equation} \label{equ:MKG}
 \left\{ \begin{aligned}
  \partial^{\beta} F_{\alpha\beta} &= \Im \bigl( \phi \overline{D_{\alpha} \phi} \bigr), \\ 
  \Box_A \phi &= 0,
 \end{aligned} \right.
\end{equation}
where $\Box_A = D_{\alpha} D^{\alpha}$ is the covariant d'Alembertian. The system has two important features. First, it enjoys the {\it gauge invariance} 
\[
 A_\alpha \mapsto A_\alpha - \partial_\alpha \gamma, \quad \phi \mapsto e^{i \gamma} \phi
\]
for any suitably regular function $\gamma: \R^{1+n} \to \R$. Second, it is {\it Lorentz invariant}. Moreover, the system admits a conserved energy 
\begin{equation} \label{equ:energy}
 E(A, \phi) := \int_{\R^n} \Big( \frac{1}{4}\sum_{\alpha, \beta} F_{\alpha\beta}^2 + \frac{1}{2} \sum_{\alpha} \big|D_{\alpha}\phi\big|^2 \Big) \, dx.
\end{equation}
Given that the system of equations \eqref{equ:MKG} is invariant under the scaling transformation
\[
 A_{\alpha}(t,x) \rightarrow \lambda A_{\alpha}(\lambda t, \lambda x), \quad \phi(t,x) \rightarrow \lambda \phi(\lambda t, \lambda x) \quad \text{for } \lambda > 0,
\]
one distinguishes between the {\it energy sub-critical case} corresponding to $n \leq 3$, the {\it energy critical case} for $n = 4$, and the {\it energy super-critical case} for $n \geq 5$. To the best of the authors' knowledge, at this point no methods are available to prove global regularity for large data for super-critical nonlinear dispersive equations. The most advanced results for large data can be achieved for critical equations. 

\medskip

Imposing the {\it Coulomb condition} $\sum_{j=1}^n \partial_j A_j = 0$ for the spatial components of the connection form $A$, the Maxwell-Klein-Gordon system decouples into a system of wave equations for the {\it dynamical variables} $(A_j, \phi)$, $j = 1, \ldots, n,$ coupled to an elliptic equation for the {\it temporal component}~$A_0$,
\begin{equation} \tag{MKG-CG}
 \left\{ \begin{aligned}
  \Box A_j &= - \mathcal{P}_j \Im \bigl( \phi \overline{D_x \phi} \bigr), \\
  \Box_A \phi &= 0, \\
  \Delta A_0 &= - \Im \bigl( \phi \overline{D_0 \phi} \bigr),
 \end{aligned} \right.
\end{equation}
where $\mathcal{P}$ is the standard projection onto divergence free vector fields. 

\medskip

We observe that in the formulation (MKG-CG), the components $(A_j, \phi)$, $j = 1, \ldots, n,$ implicitly completely describe $(A_{\alpha}, \phi)$, since the missing component $A_0$ is uniquely determined by the elliptic equation 
\begin{equation} \label{equ:elliptic}
 \Delta A_0 = - \Im \bigl( \phi \overline{\partial_t \phi} \bigr) + |\phi|^2 A_0.
\end{equation}
For this reason, we will mostly work in terms of the dynamical variables $(A_x, \phi)$, it being understood that required bounds on $A_0$ can be directly inferred from \eqref{equ:elliptic}. In particular, to describe initial data for (MKG-CG), we will use the notation $A_j[0] := (A_j, \partial_t A_j)(0, \cdot)$ and $\phi[0]: = (\phi, \partial_t \phi)(0, \cdot)$. Often, we will simply denote these by $(A_x, \phi)[0]$.

\medskip

The present work will give a complete analysis of the energy critical case $n = 4$. More precisely, we implement an analysis closely analogous to the one by the first author and Schlag \cite{KS} in the context of critical wave maps in order to prove existence, scattering and a priori bounds for large global solutions to (MKG-CG). Moreover, we establish a {\it concentration compactness} phenomenon, which describes a kind of ``atomic decomposition'' of sequences of solutions of bounded energy.

\medskip

To formulate our main result, we introduce the following notion of admissible data for the evolution problem (MKG-CG) on $\R^{1+4}$.
\begin{defn} \label{def:admissible_data} 
We call $C^\infty$-smooth data $(A_x, \phi)[0]$ admissible, provided $A_x[0]$ satisfy the Coulomb condition and $\phi[0]$ as well as all spatial curvature components $F_{jk}[0]$ are Schwartz class. Moreover, we require that for $j = 1, \ldots, 4$,
\[
 \big| A_j[0](x) \big| \lesssim \langle x \rangle^{-3} \text{ as } |x| \to \infty. 
\]
\end{defn}
In particular, admissible data are of class $H^s_x(\R^4) \times H^{s-1}_x(\R^4)$ for any $s \geq 1$ and thus, the local existence and uniqueness theory developed by Selberg \cite{Selberg} applies to them. One can easily verify that as long as the solution exists in the sense of \cite{Selberg}, and hence in the smooth sense, it will be admissible on fixed time slices. The above notion of admissible data therefore leads to a natural concept of solution to work with, and we call such solutions admissible. Our main theorem can then be stated as follows.

\begin{thm} \label{thm:TheMainTheorem} 
Consider the evolution problem (MKG-CG) on $\R^{1+4}$. There exists a function 
\[
 K\colon (0,\infty)\longrightarrow (0, \infty)
\]
with the following property. Let $(A_x, \phi)[0]$ be an admissible Coulomb class data set such that the corresponding full set of components $(A_{\alpha}, \phi)$ has energy $E$. Then there exists a unique global in time admissible solution $(A, \phi)$ to (MKG-CG) with initial data $(A_x, \phi)[0]$ that satisfies for any $\frac{1}{q} + \frac{3}{2r} \leq \frac{3}{4}$ with $2 \leq q \leq \infty$, $2 \leq r < \infty$, $\gamma = 2 - \frac{1}{q} - \frac{4}{r}$ the following a priori bound
\begin{equation} \label{equ:a_priori_bound}
 \bigl\| \bigl( (-\Delta)^{-\frac{\gamma}{2}} \nabla_{t,x} A_x, (-\Delta)^{-\frac{\gamma}{2}} \nabla_{t,x} \phi \bigr) \big\|_{L_t^q L_x^r(\R\times\R^4)} \leq C_r K(E).
\end{equation}
The solution scatters in the sense that there exist finite energy free waves $f_j$ and $g$, $\Box f_j = \Box g = 0$, such that for $j = 1,\ldots,4$,
\[
\lim_{t\rightarrow +\infty}\big\|\nabla_{t,x}A_j(t, \cdot) - \nabla_{t,x}f_j(t, \cdot)\big\|_{L^2_x(\R^4)} = 0, \quad \lim_{t\rightarrow +\infty}\big\|\nabla_{t,x}\phi(t, \cdot) - \nabla_{t,x}g(t, \cdot)\big\|_{L^2_x(\R^4)} = 0,
\]
and analogously with different free waves for $t\rightarrow -\infty$. 
\end{thm}

In fact, we will prove the significantly stronger a priori bound 
\[
 \bigl\| (A_x, \phi) \bigr\|_{S^1(\R\times\R^4)} \leq K(E),
\]
where the precise definition of the $S^1$ norm will be introduced in Section~\ref{sec:preliminaries}. The purpose of this norm is to control the regularity of the solutions.

\medskip

Recently, a proof of the global regularity and scattering affirmations in the preceding theorem was obtained by Oh-Tataru \cite{OhTat_2, OhTat_1, OhTat_3}, following the method developed by Sterbenz-Tataru \cite{Sterbenz-Tataru_1, Sterbenz-Tataru_2} in the context of critical wave maps. Our conclusions were reached before the appearance of their work and our methods are completely independent.

\subsection{A history of the problem}

 In this subsection we first consider this work in the broader context of the study of the local and global in time behavior of nonlinear wave equations and highlight some of the important developments over the last decades that crucially enter the proof of our main theorem. Afterwards we give an overview of previous results on the Maxwell-Klein-Gordon equation.

 \medskip

 \noindent {\it Null structure.} In many geometric wave equations like the wave map equation, the Maxwell-Klein-Gordon equation, and the Yang-Mills equation, the nonlinearities exhibit so-called null structures. Heuristically speaking, such null structures are amenable to better estimates, because they damp the interactions of parallel waves. The key role that these special nonlinear structures play in the global regularity theory of nonlinear waves was first highlighted by Klainerman \cite{Klainerman}. At that point the theory of nonlinear wave equations relied for the most part on vector field methods and parametrices in physical space. However, in more recent times the key role that null structures play also within the more technical harmonic analysis approach cannot be overstated. In fact, in \cite{Klainerman2} a whole program with precise conjectures pertaining to the sharp well-posedness of a number of nonlinear wave equations with null structure was outlined. The present work may be seen as a further vindication of the program outlined by Klainerman. Null structures also play a pivotal role in the much more complex system of Einstein equations, as evidenced for example in the recent deep sequence of works by Klainerman-Rodnianski-Szeftel~\cite{KlRodSz} and Szeftel~\cite{Sz1, Sz2, Sz3, Sz4, Sz5} on the bounded $L^2$ curvature conjecture.

 \medskip

 \noindent {\it Function spaces.} The development of \emph{$X^{s,b}$ spaces} by Klainerman-Machedon in the seminal works \cite{KlMa3, KlMa4, KlMa_MKG, KlMa_YM} in the low regularity study of nonlinear wave equations provided a powerful tool to take advantage of the null structures in geometric wave equations. The fact that the Maxwell-Klein-Gordon and Yang-Mills equations actually display a null structure in the Coulomb gauge was a key observation in these works. Moreover, the observation by Klainerman-Machedon that these null structures are beautifully compatible with the $X^{s,b}$ functional framework has been highly influential ever since. Different variants of the $X^{s,b}$ spaces were independently introduced by Bourgain \cite{Bourgain} in the context of the nonlinear Schr\"odinger equation and the Korteweg-de Vries equation.

 In the quest to prove global regularity for critical wave maps for small initial data it turned out that not even the strongest versions of the critical $X^{s,b}$ type spaces yielded good algebra type estimates. This problem was resolved through the development of the \emph{null frame spaces} in the breakthrough work of Tataru \cite{Tataru_1}. We will introduce these spaces later on, see also \cite{Tao_Wave_MapII}, \cite{Tataru_2}, and \cite{KS} for more discussion.

 \medskip
 
 \noindent {\it Renormalization.} The key difficulty for the (MKG-CG) equation is the equation $\Box_A \phi = 0$, which in expanded form is given by
 \[
  \Box \phi = - 2i A_\alpha \partial^\alpha \phi + i (\partial_t A_0) \phi + A_\alpha A^\alpha \phi.
 \]
 The contribution of the low-high frequency interactions in the \emph{magnetic interaction term}
 \begin{equation} \label{equ:intro_difficult_magnetic_interaction_term}
  - 2 i A_j \partial^j \phi
 \end{equation}
 turns out to be non-perturbative in the case when the spatial components of the connection form are just free waves. This problem already occurs for small data and is not only a large data issue. One encounters a similar situation in the wave map equation. In the breakthrough works \cite{Tao_Wave_MapI, Tao_Wave_MapII} Tao exploited the intrinsic gauge freedom for the wave maps problem to recast the nonlinearity into a perturbative form. However, for the (MKG-CG) equation the gauge invariance is already spent. Rodnianski and Tao \cite{RT} found a way out of this impasse by incorporating the non-perturbative term into the linear operator and by deriving Strichartz estimates for the resulting wave operator via a parametrix construction. This enabled them to prove global regularity for (MKG-CG) for small critical Sobolev data in $n \geq 6$ space dimensions. A key novelty in the small data energy critical global well-posedness result for (MKG-CG) in $n=4$ space dimensions by the first author, Sterbenz, and Tataru \cite{KST} was the realization that the parametrix from \cite{RT} is also compatible with the complicated $X^{s,b}$ type and null frame spaces. The functional calculus from \cite{KST} for the paradifferential magnetic wave operator
 \[
  \Box_A^p = \Box + 2i \sum_{k\in\Z} P_{<k-C} A^{free}_j P_k \partial^j,
 \]
 where $P_k$ denotes the standard Littlewood-Paley projections and $A_j^{free}$, $j = 1, \ldots, 4$, are free waves, plays an important role in this work and has to be adapted to the large data setting. 

 \medskip

 \noindent {\it Bahouri-G\'erard concentration compactness decomposition.} Bahouri and G\'erard \cite{Bahouri-Gerard} proved the following description of sequences of solutions to the free wave equation with uniformly bounded energy. 

 \medskip

 \noindent {\it Let $\{ (\varphi_n, \psi_n) \}_{n \geq 1} \subset \dot{H}^1_x(\R^3) \times L^2_x(\R^3)$ be a bounded sequence and let $v_n$ be the solution to the free wave equation $\Box v_n = 0$ on $\R \times \R^3$ with initial data $(v_n, \partial_t v_n)|_{t=0} = (\varphi_n, \psi_n)$. Then there exists a subsequence $\{ v_n' \}_{n \geq 1}$ of $\{ v_n \}_{n \geq 1}$ and finite energy free waves $V^{(j)}$ as well as sequences $\big\{ (\lambda_n^{(j)}, t_n^{(j)}, x_n^{(j)}) \big\}_{n \geq 1} \subset \R_+ \times \R \times \R^3$, $j \geq 1$, such that for every $l \geq 1$,
 \begin{equation} \label{equ:intro_linear_profile_decomposition}
   v_n'(t,x) = \sum_{j=1}^l \frac{1}{\sqrt{\lambda_n^{(j)}}} V^{(j)} \biggl( \frac{t - t_n^{(j)}}{\lambda_n^{(j)}}, \frac{x - x_n^{(j)}}{\lambda_n^{(j)}} \biggr) + w_n^{(l)}(t,x)
 \end{equation}
 and
 \[
  \lim_{l\to\infty} \limsup_{n\to\infty} \big\| w_n^{(l)} \big\|_{L^5_t L^{10}_x(\R\times\R^3)} = 0.
 \]
 Moreover, there is asymptotic decoupling of the free energy
 \[
  \big\| \nabla_{t,x} v_n' \big\|_{L^2_x}^2 = \sum_{j=1}^l \big\| \nabla_{t,x} V^{(j)} \big\|^2_{L^2_x} + \big\| \nabla_{t,x} w_n^{(l)} \big\|^2_{L^2_x} + o(1) \quad \text{ as } n \to \infty,
 \] 
 and for each $j \neq k$, we have the asymptotic orthogonality property 
 \begin{equation} \label{equ:intro_asymptotic_orthogonality_property}
  \lim_{n\to\infty} \frac{\lambda_n^{(j)}}{\lambda_n^{(k)}} + \frac{\lambda_n^{(k)}}{\lambda_n^{(j)}} + \frac{|t_n^{(j)} - t_n^{(k)}|}{\lambda_n^{(j)}} + \frac{|x_n^{(j)} - x_n^{(k)}|}{\lambda_n^{(j)}} = \infty.
 \end{equation}
 }The free waves $V^{(j)}$ are referred to as concentration profiles and the importance of the \emph{linear profile decomposition} \eqref{equ:intro_linear_profile_decomposition} is that it captures the failure of compactness of the sequence of bounded solutions $\{ v_n \}_{n \geq 1}$ to the free wave equation in terms of the non-compact symmetries of the equation and the superposition of profiles. Simultaneously to Bahouri and G\'erard, Merle and Vega \cite{Merle-Vega} obtained similar concentration compactness decompositions in the context of the mass-critical nonlinear Schr\"odinger equation. This is very analogous to the concentration compactness method originally developed in the context of elliptic equations, see e.g. \cite{Lions1,Lions2,Lions3,Lions4} and \cite{Struwe_CalcVar_Book} for a discussion of the original works.

 Bahouri and G\'erard also established an analogous \emph{nonlinear profile decomposition} for Shatah-Struwe solutions $\{ u_n \}_{n \geq 1}$ to the energy critical defocusing nonlinear wave equation $\Box u_n = u_n^5$ on $\R\times\R^3$, see \cite{Shatah-Struwe}, with the same initial data $(u_n, \partial_t u_n)|_{t=0} = (\varphi_n, \psi_n)$. Their main application of this nonlinear profile decomposition was to prove the existence of a function $A: (0,\infty) \to (0,\infty)$ with the property that for any Shatah-Struwe solution $u$ to $\Box u = u^5$ it holds that
 \[
  \|u\|_{L^5_t L^{10}_x(\R\times\R^3)} \leq A \big( E(u) \big),
 \]
 where $E(u)$ denotes the energy functional associated with the quintic nonlinear wave equation.

 The Bahouri-G\'erard profile decomposition is of fundamental importance for the Kenig-Merle method that we will describe in the next paragraph. In the proof of our main theorem we will have to study sequences of solutions to the (MKG-CG) equation with uniformly bounded energy. A key step will be to obtain an analogous Bahouri-G\'erard profile decomposition for such sequences. This poses significant problems, which can be heuristically understood as follows. Very roughly speaking, the reason why the Bahouri-G\'erard profile decomposition ``works'' for the energy critical nonlinear wave equation $\Box u = u^5$ is that in the quintic nonlinearity the interaction of two nonlinear concentration profiles living at asymptotically orthogonal frequency scales vanishes. This reduces to consider diagonal frequency interactions of the concentration profiles in the nonlinearity, but then these profiles must essentially be supported in different regions of space-time due to the asymptotic orthogonality property \eqref{equ:intro_asymptotic_orthogonality_property} and therefore do not interact strongly. In contrast, for the (MKG-CG) equation frequency diagonalization appears to fail in the difficult magnetic interaction term \eqref{equ:intro_difficult_magnetic_interaction_term} for low-high interactions. A similar situation occurs for critical wave maps. In the latter context the first author and Schlag \cite{KS} devised a novel profile decomposition to take into account the corresponding low-high frequency interactions. Our approach is strongly influenced by \cite{KS}, but we will have to use a slightly different ``covariant'' wave operator to extract the concentration profiles, see the discussion in the next subsection.

 \medskip

 \noindent {\it The Kenig-Merle method.} Kenig and Merle \cite{KM2, KM} introduced a very general method to prove global well-posedness and scattering for critical nonlinear dispersive and wave equations in both defocusing and focusing cases, in the latter case only for energies strictly less than the ground state energy. Their approach has found a vast amount of applications over the last years. We illustrate the method in the context of the energy critical defocusing nonlinear wave equation $\Box u = u^5$ on $\R\times\R^3$. One can use the $L^8_t L^8_x(\R\times\R^3)$ norm to control the regularity of solutions to this equation and easily prove local well-posedness and small data global well-posedness based on this norm. In the first step of the Kenig-Merle method one assumes that global well-posedness and scattering fails for some finite energy level. Then let $E_{crit}$ be the critical energy below which all solutions exist globally in time with finite $L^8_t L^8_x(\R\times\R^3)$ bounds, in particular it must hold that $E_{crit} > 0$. Thus, we find a sequence of solutions $\{ u_n \}_{n \geq 1}$ such that $E(u_n) \to E_{crit}$ and $\|u_n\|_{L^8_t L^8_x} \to \infty$ as $n \to \infty$. Applying the Bahouri-G\'erard profile decomposition to $\{ u_n(0) \}_{n \geq 1}$, we may conclude by the \emph{minimality of $E_{crit}$} that there exists exactly one profile in the decomposition. This enables us to extract a \emph{minimal blowup solution} $u_C$ of lifespan $I$ with $E(u_C) = E_{crit}$ and $\|u_C\|_{L^8_t L^8_x(I\times\R^3)} = \infty$. Moreover, we can infer a crucial \emph{compactness property} of $u_C$, namely that there exist continuous functions $\overline{x}: I \to \R^3$ and $\lambda: I \to \R_+$ such that the family of functions
 \[
  \Bigg\{ \bigg( \frac{1}{\lambda(t)^{\frac{1}{2}}} u_C \bigg( t, \frac{\cdot - \overline{x}(t)}{\lambda(t)} \bigg), \frac{1}{\lambda(t)^{\frac{3}{2}}} \partial_t u_C \bigg( t, \frac{\cdot - \overline{x}(t)}{\lambda(t)} \bigg) \bigg) : t \in I \Biggr\}
 \]
 is pre-compact in $\dot{H}^1_x(\R^3) \times L^2_x(\R^3)$. The second step of the Kenig-Merle method is a rigidity argument to rule out the existence of such a minimal blowup solution $u_C$ by combining the compactness property with conservation laws and other identities of virial or Morawetz type for the energy critical nonlinear wave equation. We will adapt the Kenig-Merle method to the Maxwell-Klein-Gordon equation.

 \medskip
 
 We now review previous results on the Maxwell-Klein-Gordon equation. The existence of global smooth solutions to the Maxwell-Klein-Gordon equation in $n = 3$ space dimensions follows from seminal work of Eardley-Moncrief~\cite{EM_1, EM_2} and the finite energy global well-posedness result by Klainerman-Machedon~\cite{KlMa_MKG}. The latter work introduced many new techniques into the study of (MKG-CG) such as exploitation of null structures to obtain improved local well-posedness. The regularity threshold for local well-posedness of (MKG-CG) was further lowered by Cuccagna \cite{Cuccagna} for small initial data. An essentially optimal local well-posedness result for (MKG-CG) in $n=3$ space dimensions was achieved by Machedon-Sterbenz~\cite{MaSte} by uncovering a deep trilinear null structure in the system that is also of crucial importance in this work. Global well-posedness for (MKG-CG) below the energy norm was proved by Keel-Roy-Tao \cite{KRT} using the I-method. See also Selberg-Tesfahun \cite{Selberg-Tesfahun} for a finite energy global well-posedness result for the Maxwell-Klein-Gordon equation in $n = 3$ space dimensions relative to the Lorenz gauge.

 In $n = 4$ space dimensions almost optimal local well-posedness for a model prolem related to the Maxwell-Klein-Gordon and Yang-Mills equations was established by Klainerman-Tataru~\cite{KlTa}. The corresponding almost optimal local well-posedness result for (MKG-CG) was proved by Selberg~\cite{Selberg}, which we rely on in this paper. A global regularity result in $n=4$ space dimensions for equations of Maxwell-Klein-Gordon type with small initial data in a weighted but scaling invariant Besov space was obtained by Sterbenz \cite{Sterbenz}.

 Global regularity of (MKG-CG) for small critical Sobolev norm data in space dimensions $n \geq 6$ was established by Rodnianski-Tao~\cite{RT}. As already described earlier, an important innovation in \cite{RT} was to incorporate the difficult magnetic interaction term in (MKG-CG) into the linear wave operator and to derive Strichartz estimates for the resulting wave operator via a parametrix construction. 

 Small energy global well-posedness of the energy critical Maxwell-Klein-Gordon equation in $n = 4$ space dimensions was established in joint work of the first author with Sterbenz and Tataru~\cite{KST}. This work also provides the functional analytic framework that we will draw on and constantly refer to in this paper. The key novelty in~\cite{KST} was the realization that the parametrix construction from \cite{RT} is in fact compatible with the complicated spaces introduced for critical wave maps in \cite{Tataru_1, Tataru_2}, \cite{Tao_Wave_MapI, Tao_Wave_MapII}, \cite{Krieger} and that these spaces play a pivotal role in the energy critical small data global existence theory for (MKG-CG) in light of the deep null structure revealed in~\cite{MaSte}. 

 \medskip

 We also mention that for the closely related Yang-Mills equation finite energy global well-posedness in energy sub-critical $n=3$ space dimensions was proved by Klainerman-Machedon~\cite{KlMa_YM}. A different proof was later obtained by Oh~\cite{Oh_2, Oh_1}, using the Yang-Mills heat flow. Global regularity for the Yang-Mills system for small critical Sobolev data for $n \geq 6$ was established by the first author and Sterbenz~\cite{Krieger-Sterbenz}.

 \medskip

 For energy critical problems, it is a standard strategy to attempt to prove global regularity for large energies by reducing to a small energy global existence result via finite speed of propagation and exclusion of an energy concentration scenario. Such a method worked well, for example, for the critical defocusing nonlinear wave equation $\Box u = u^5$ on $\R^{1+3}$ and for radial critical wave maps on $\R^{1+2}$. At this point, with the exception of some special problems, it appears that a general large data result cannot be inferred by using the small data result as a black box, but instead requires a more or less complete re-working of the small data theory. See, for instance, the works on critical large data wave maps \cite{Sterbenz-Tataru_1, Sterbenz-Tataru_2}, \cite{KS}, \cite{Tao_Wave_Map_large}. Our approach here is to implement a similar strategy as the one by the first author and Schlag~\cite{KS} for critical wave maps, which consists of essentially two steps. First, a novel ``covariant'' Bahouri-G\'erard procedure to take into account the non-negligible influence of low on high frequencies in the magnetic interaction term. Second, an implementation of a variant of a concentration compactness/rigidity argument by Kenig-Merle, following more or less the sequence of steps in~\cite{KM}. As the latter was introduced in the context of a scalar wave equation, and we are considering a complex nonlinearly linked system, we believe that the implementation of this step for the energy critical Maxwell-Klein-Gordon equation is also of interest in its own right. 

 \medskip

 We expect that our methods extend to prove global regularity, scattering and a priori bounds for the energy critical Yang-Mills equations in $n = 4$ space dimensions for initial data in Coulomb gauge and with energy below the ground state energy.

\subsection{Overview of the proof}

In this subsection we give a detailed overview of the proof of Theorem~\ref{thm:TheMainTheorem}. In fact, the purpose of this paper is to prove a significantly stronger version of Theorem~\ref{thm:TheMainTheorem}, namely the existence of a function $K\colon (0, \infty) \rightarrow (0, \infty)$ with
\[
 \bigl\| (A_x, \phi) \bigr\|_{S^1(\R\times\R^4)} \leq K(E), \quad E = E(A,\phi)
\]
for any admissible solution $(A, \phi)$ to (MKG-CG). Once this a priori bound is known, one also obtains the scattering assertion in Theorem~\ref{thm:TheMainTheorem}. The fact that the dynamical variables of a global admissible solution scatter to finite energy free waves, and not to solutions to a suitable linear magnetic wave equation, crucially relies on our strong spatial decay assumptions about the data, see the proof of scattering at the end of Section~\ref{sec:rigidity_argument}. The precise definition of the $S^1$ space and its time localized version will be given in Section~\ref{sec:preliminaries} and Definition~\ref{defn:S_norm_breakdown}. 

\medskip

Beginning the argument at this point, we assume that the existence of such a function $K$ fails for some finite energy level. Thus, the set of energies 
\[
 \mathcal{E} := \Bigl\{ E>0 : \sup_{\substack{(A, \phi) \, \text{admissible} \\ E(A, \phi) \leq E}} \bigl\| (A_x, \phi) \bigr\|_{S^1} = + \infty \Bigr\}
\]
is non-empty. In view of the small energy global well-posedness result \cite{KST}, it therefore has a positive infimum, which we denote by $E_{crit}$,
\[
 E_{crit} := \inf \mathcal{E}.
\]
By definition we can then find a sequence of admissible solutions $\{(A^n, \phi^n)\}_{n \geq 1}$ to (MKG-CG) such that
\[
 E(A^n, \phi^n) \rightarrow E_{crit}, \quad \lim_{n \to \infty} \big\|(A^n_x, \phi^n)\big\|_{S^1} = +\infty. 
\]
As in \cite{KS}, we  call such a sequence an {\it essentially singular sequence}. The goal of this paper is to rule out the existence of such an object. This will be accomplished in broad strokes by the following two steps.
\begin{itemize}
\item[(1)] Extracting an energy class minimal blowup solution $\big( \cA^\infty, \Phi^\infty \big)$ to (MKG-CG) with the compactness property via a modified Bahouri-G\'erard procedure, which consists of an inductive sequence of low-frequency approximations and a profile extraction process taking into account the effect of the magnetic potential interaction. Here we closely follow the procedure introduced by the first author and Schlag \cite{KS}, but we have to subtly diverge from the profile extraction process there to correctly capture the asymptotic evolution of the atomic components. We note that the heart of the modified Bahouri-G\'erard procedure resides in Section~\ref{sec:concentration_compactness_step}.
\item[(2)] Ruling out the existence of the minimal blowup solution $\big( \cA^\infty, \Phi^\infty \big)$ by essentially following the method of Kenig and Merle \cite{KM}. This step is carried out in Section~\ref{sec:rigidity_argument}.
\end{itemize}
We now describe these steps in more detail. 

\medskip

\noindent {\it A concept of weak evolution for energy class data.} In order to extract a minimal blowup solution at the end of the modified Bahouri-G\'erard procedure, we first need to introduce the notion of a solution to (MKG-CG) that is merely of energy class. A natural idea here is to approximate a given Coulomb energy class datum by a sequence of admissible data and to define the energy class solution to (MKG-CG) as a suitable limit of the admissible solutions. One then needs a good perturbation theory to show that this limit is well-defined and independent of the approximating sequence. Unfortunately, there is not such a strong perturbation theory for (MKG-CG) as for instance for critical wave maps in \cite{KS} due to a low frequency divergence. However, the problem with evolving irregular data is really a ``high frequency issue'' and in Proposition~\ref{prop:perturbation} we show that there is a good perturbation theory for perturbing frequency localized data by adding high frequency perturbations. We can then define the evolution of a Coulomb energy class datum as a suitable limit of the evolutions of low frequency approximations of the energy class datum, provided these low frequency approximations exist on some joint time interval and satisfy uniform $S^1$ norm bounds there. This is achieved in Proposition~\ref{prop:joint_time_interval} via a suitable method of localizing the data and exploiting a version of Huygens' principle together with the gauge invariance of the Maxwell-Klein-Gordon system. This step is additionally complicated by the fact that the (MKG-CG) equation does not have the finite speed of propagation property due to non-local terms in the equation for the spatial components of the connection form $A$. 

\medskip

\noindent {\it Bahouri-G\'erard I: Filtering out frequency atoms and evolving the ``non-atomic'' lowest frequency approximation.} The extraction of the energy class minimal blowup solution $\big( \cA^\infty, \Phi^\infty \big)$ consists of a two step Bahouri-G\'erard type procedure. This is carried out in Section~\ref{sec:concentration_compactness_step} and forms the core of our argument. In the first step we consider the initial data $(A^n_x, \phi^n)[0]$ at time $t=0$ of the essentially singular sequence $(A^n, \phi^n)$  and use a procedure due to M\'etivier-Schochet \cite{Metivier-Schochet} to extract frequency scales. In what follows we will slightly abuse notation and write $(A^n, \phi^n)[0]$ instead of $(A^n_x, \phi^n)[0]$.  This yields the decompositions
\begin{align*}
 A^n[0] &= \sum_{a=1}^{\Lambda} A^{na}[0] + A^n_{\Lambda}[0], \\
 \phi^n[0] &= \sum_{a=1}^{\Lambda} \phi^{na}[0] + \phi^n_{\Lambda}[0],
\end{align*}
where the ``frequency atoms'' $(A^{na}, \phi^{na})[0]$ are essentially frequency localized to scales $(\lambda^{na})^{-1}$ that tend apart as $n \to \infty$, more precisely
\[
 \lim_{n\to\infty} \frac{\lambda^{na}}{\lambda^{na'}} + \frac{\lambda^{na'}}{\lambda^{na}} = \infty
\]
for $a \neq a'$, while the error $(A^n_\Lambda, \phi^n_\Lambda)[0]$ satisfies
\[
 \limsup_{n\to\infty} \big\| A^n_{\Lambda} [0] \big\|_{\dot{B}^1_{2,\infty} \times \dot{B}^0_{2,\infty}} + \big\| \phi^n_{\Lambda}[0] \big\|_{\dot{B}^1_{2,\infty} \times \dot{B}^0_{2,\infty}} < \delta
\]
for given $\delta > 0$ and $\Lambda = \Lambda(\delta)$ sufficiently large. Moreover, we prepare these frequency atoms such that their frequency supports are sharply separated as $n \to \infty$ and so that the errors $\big\{ (A^n_{\Lambda}, \phi^n_\Lambda)[0] \big\}_{n \geq 1}$ are supported away from the frequency scales $(\lambda^{na})^{-1}$ in frequency space. Then we select a number of ``large'' frequency atoms $(A^{na}, \phi^{na})[0]$, $a = 1, \ldots, \Lambda_0$, whose energy $E(A^{na}, \phi^{na})$ is above a certain small threshold $\varepsilon_0$ depending only on $E_{crit}$. We order these frequency atoms by the scale around which they are essentially supported starting with the lowest one. 

Eventually, we want to conclude that the essentially singular sequence of data $\big\{ (A^n, \phi^n)[0] \big\}_{n \geq 1}$ consists of exactly one non-trivial frequency atom that is composed of exactly one non-trivial physical concentration profile of asymptotic energy $E_{crit}$. We argue by contradiction and assume that this is not the case. Then the idea is to approximate the essentially singular sequence of initial data $(A^n, \phi^n)[0]$ by low frequency truncations, obtained by removing all or some of the atoms $(A^{na}, \phi^{na})[0]$, $a = 1, \ldots, \Lambda_0$, and to inductively derive bounds on the $S^1$ norms of the (MKG-CG) evolutions of the truncated data. As this induction stops after $\Lambda_0$ many steps, we obtain an a priori bound on the evolutions
\begin{equation} \label{equ:intro_a_priori_bound}
 \liminf_{n\to\infty} \|(A^n, \phi^n)\|_{S^1} < \infty,
\end{equation}
contradicting the assumption that $\{ (A^n, \phi^n) \}_{n \geq 1}$ is an essentially singular sequence of solutions to (MKG-CG).

We observe that by construction the ``non-atomic'' errors $(A^n_{\Lambda_0}, \phi^n_{\Lambda_0})[0]$ are split into $\Lambda_0 + 1$ ``shells'' by the frequency supports of the atoms $(A^{na}, \phi^{na})[0]$, i.e. we can write
\[
 A_{\Lambda_0}^n[0] = \sum_{j=1}^{\Lambda_0+1} A^{nj}_{\Lambda_0}[0], \quad \phi^n_{\Lambda_0}[0] = \sum_{j=1}^{\Lambda_0 + 1} \phi^{nj}_{\Lambda_0}[0],
\]
where the first pieces $(A^{n1}_{\Lambda_0}, \phi^{n1}_{\Lambda_0})[0]$ have Fourier support in the region closest to the origin. In Subsection~\ref{subsec:evolving_non_atomic} we then derive a priori $S^1$ norm bounds on the evolutions of the lowest frequency approximations $(A^{n1}_{\Lambda_0}, \phi^{n1}_{\Lambda_0})[0]$. The problem here is that the pieces $(A^{n1}_{\Lambda_0}, \phi^{n1}_{\Lambda_0})[0]$ might still have large energy, which forces us to use a finite number of further delicately chosen low frequency approximations $\big( P_{J_L} A^{n1}_{\Lambda_0}, P_{J_L} \phi^{n1}_{\Lambda_0} \big)[0]$ of these pieces. Importantly, this number only depends on the size of $E_{crit}$. We then inductively obtain bounds on the $S^1$ norms of the (MKG-CG) evolutions of the low frequency approximations $\big( P_{J_L} A^{n1}_{\Lambda_0}, P_{J_L} \phi^{n1}_{\Lambda_0} \big)[0]$ by bootstrap. This step is tied together in Proposition~\ref{prop:bootstrap}. In particular, Step~3 of the proof of Proposition~\ref{prop:bootstrap} is the core perturbative result of this paper and is used in variations at other instances later on.

\medskip

\noindent {\it Bahouri-G\'erard II: Selecting concentration profiles and adding the first large frequency atom.} Having established control over the evolution of the lowest frequency ``non-atomic'' part $(A^{n1}_{\Lambda_0}, \phi^{n1}_{\Lambda_0})[0]$, we then add the first frequency atom $(A^{n1}, \phi^{n1})[0]$ and consider the evolution of the data
\[
 \big( A^{n1}_{\Lambda_0} + A^{n1}, \phi^{n1}_{\Lambda_0} + \phi^{n1} \big)[0].
\]
Here we first have to understand the lack of compactness of the functions $(A^{n1}, \phi^{n1})[0]$. It is at this point that we deviate most significantly from the standard Bahouri-G\'erard profile extraction procedure \cite{Bahouri-Gerard} and also the modified profile extraction procedure developed by the first author and Schlag \cite{KS} in the context of critical wave maps. We still extract the concentration profiles for the data $A^{n1}[0]$ using the standard Bahouri-G\'erard extraction procedure. However, we evolve the data $\phi^{n1}[0]$ with respect to the following ``covariant'' wave operator
\[
 \widetilde{\Box}_{A^{n1}} := \Box + 2i \bigl( A^{n1}_{\Lambda_0, \nu} + A^{n1, free}_\nu \bigr) \partial^\nu
\]
and extract the profiles as weak limits of these evolutions to take into account the strong low-high interactions for (MKG-CG). Here, $A^{n1}_{\Lambda_0}(t,x)$ is the (MKG-CG) evolution of the low frequency data $(A^{n1}_{\Lambda_0}, \phi^{n1}_{\Lambda_0})[0]$, while $A^{n1, free}_j$ is the free wave evolution of the data $A^{n1}_j[0]$ for $j = 1, \ldots, 4$, and we simply put $A^{n1,free}_0 = 0$. In comparison with \cite{Bahouri-Gerard} and \cite{KS}, a key difficulty in this step is that solutions to the covariant linear wave equation $\widetilde{\Box}_{A^{n1}} u = 0$ only conserve the free energy in a mild asymptotic sense, see Lemma~\ref{lem:asymptotic_energy_conservation}. 
Importantly, after passing to subsequences, we may use the same space-time shifts $(t_n^{ab}, x_n^{ab})$ for extracting the concentration profiles both for $A^{n1}[0]$ and for $\phi^{n1}[0]$. Once the profiles have been picked, we use them to construct approximate, but highly accurate, nonlinear profiles in Theorem~\ref{thm:nonlinear_profiles_S1_bound}. To this end we solve the (MKG-CG) system in very large but finite space-time boxes centered around $(t_n^{ab}, x_n^{ab})$, using the concentration profiles as data, while outside of these boxes, we use the free wave propagation for $A$ and the ``full'' covariant wave operator (involving the influence of all other profiles) for $\phi$. This is the same strategy as the one pursued for wave maps in \cite{KS}. Provided that all concentration profiles have energy strictly less than $E_{crit}$ with respect to the Maxwell-Klein-Gordon energy functional, we can then use our perturbation theory to construct the global (MKG-CG) evolution of the data $\big( A^{n1}_{\Lambda_0} + A^{n1}, \phi^{n1}_{\Lambda_0} + \phi^{n1} \big)[0]$ and to obtain a priori $S^1$ norm bounds.

\medskip

\noindent {\it Conclusion of the induction on frequency process.} We may then repeat the preceding steps and ``add in'' all remaining frequency atoms to conclude a priori global $S^1$ norm bounds on the evolution of the full data $(A^n, \phi^n)[0]$. The conclusion of this induction on frequency process is that we arrive at a contradiction, unless the essentially singular sequence of data $(A^n, \phi^n)[0]$ consists of exactly one frequency atom that is composed of precisely one concentration profile of asymptotic energy $E_{crit}$. Due to our relatively poor perturbation theory for (MKG-CG), it then still requires a fair amount of work to extract an energy class minimal blowup solution from this essentially singular sequence $(A^n, \phi^n)$, see Section~\ref{sec:how_to_arrive_at} and Subsection~\ref{subsec:conclusion}. Finally, in Theorem~\ref{thm:compact_orbit} we obtain an energy class minimal blowup solution $(\cA^\infty, \Phi^\infty)$ to (MKG-CG) with lifespan $I$ and with the crucial compactness property that there exist continuous functions $\overline{x}: I \to \R^4$ and $\lambda: I \to (0,\infty)$ so that each of the family of functions 
  \[
   \Bigg\{ \bigg( \frac{1}{\lambda(t)} {\mathcal A}^\infty_j \bigg(t, \frac{\cdot-\bar{x}(t)}{\lambda(t)} \bigg), \frac{1}{\lambda(t)^2} \partial_t {\mathcal A}^\infty_j \bigg(t, \frac{\cdot-\bar{x}(t)}{\lambda(t)} \bigg) \bigg) \, \colon \, t \in I \Bigg\}
  \]
  for $j = 1, \ldots, 4$ and 
  \[
   \Bigg\{ \bigg( \frac{1}{\lambda(t)} \Phi^\infty \bigg(t, \frac{\cdot-\bar{x}(t)}{\lambda(t)} \bigg), \frac{1}{\lambda(t)^2} \partial_t \Phi^\infty \bigg(t, \frac{\cdot-\bar{x}(t)}{\lambda(t)} \bigg) \bigg) \, \colon \, t \in I \Bigg\}
  \]
  is pre-compact in $\dot{H}^1_x(\R^4) \times L^2_x(\R^4)$.

\medskip

\noindent {\it The Kenig-Merle rigidity argument.} In the final Section~\ref{sec:rigidity_argument}, we rule out the existence of such a minimal blowup solution $(\cA^\infty, \Phi^\infty)$ with the compactness property by following the scheme of the Kenig-Merle rigidity argument \cite{KM}. The idea is to infer from the compactness property and the minimal energy property of $(\cA^\infty, \Phi^\infty)$ the existence of either a static solution to (MKG-CG) or else the existence of a self-similar blowup solution to (MKG-CG) and to then exclude the existence of both of these objects. 

A crucial step in the Kenig-Merle rigidity argument is to conclude that the momentum of $(\cA^\infty, \Phi^\infty)$ must vanish. The proof of this hinges on the relativistic invariance of the Maxwell-Klein-Gordon equation and the transformation behavior of the Maxwell-Klein-Gordon energy functional under Lorentz transformations. This step is technically difficult for the Maxwell-Klein-Gordon equation, because the $S^1$ norm is much more complicated than the Strichartz norms used in \cite{KM}. 

We then distinguish between the lifespan $I$ of $(\cA^\infty, \Phi^\infty)$ being finite in at least one time direction or not. If $I$ is infinite, we face the possibility of a static solution, which we rule out using virial type identities for the Maxwell-Klein-Gordon equation, the vanishing momentum condition for $(\cA^\infty, \Phi^\infty)$ and a Vitali covering argument from \cite{KS}. If instead $I$ is finite at one end, we reduce to a self-similar blowup scenario. We then uncover a Lyapunov functional for solutions to the Maxwell-Klein-Gordon equation in self-similar variables.  This is the key ingredient, which allows us to also rule out this scenario. The derivation of this Lyapunov functional appears significantly more complicated than in \cite{KM} or \cite{KS} and we use the trick of working in a Cronstrom-type gauge to simplify the computations.

\subsection{Overview of the paper}

We now give an overview of the structure of this paper. The two main steps of the proof of Theorem~\ref{thm:TheMainTheorem} are the modified Bahouri-G\'erard procedure in Section~\ref{sec:concentration_compactness_step} and the rigidity argument in Section~\ref{sec:rigidity_argument}. The necessary technical preparations are carried out in the sections leading up to Section~\ref{sec:concentration_compactness_step}. 
\begin{itemize}
\item In Section~\ref{sec:preliminaries} we lay out the functional framework following \cite{KST}.
\item In Section~\ref{sec:magnetic_wave_equation} we prove key estimates for the linear magnetic wave equation $\Box_A^p u = f$.
\item In Section~\ref{sec:breakdown} we state the property of the $S^1$ norm as a regularity controlling device. 
\item In Section~\ref{sec:concept_of_weak_evolution} we show how to unambiguously locally evolve Coulomb energy class data $(A_x, \phi)[0]$ via approximation by smoothed data and truncation in physical space to reduce to the admissible setup. Here one needs to pay close attention to the fact that solutions to (MKG-CG) do not obey as good a perturbation theory with respect to the $S^1$ spaces as, say, critical wave maps in a suitable gauge, due to a low frequency divergence. Hence, one needs to be very careful about the correct choice of smoothing, using low frequency truncations of the data. Moreover, to ensure the existence of an energy class local evolution of Coulomb energy class data $(A_x, \phi)[0]$ on a non-trivial time slice around $t = 0$, we need to prove uniform $S^1$ norm bounds for the approximations, which we accomplish similarly to the procedure in \cite{KS} via localization in physical space, see Proposition~\ref{prop:joint_time_interval}. We also introduce the concept of the ``lifespan'' of such an energy class solution and the definition of its $S^1$ norm.
\item In Section~\ref{sec:how_to_arrive_at} we then state that energy class data $(A_x, \phi)[0]$ obtained as the limit of the data of an essentially singular sequence, which will be the outcome of the modified Bahouri-G\'erard procedure, lead to a singular solution $(A, \phi)$ in the sense that 
\[
 \sup_{J \subset I} \big\| (A_x, \phi) \big\|_{S^1(J \times \R^4)} = +\infty,
\]
where $I$ denotes its lifespan. The proof of this as well as a number of further technical assertions will be relegated to Subsection~\ref{subsec:interlude}.
\item In Section~\ref{sec:concentration_compactness_step} we carry out the modified Bahouri-G\'erard procedure. In Subsection~\ref{subsec:general_considerations} and Subsection~\ref{subsec:setting_up_induction_on_freq_scales} we extract the ``frequency atoms'' mimicking closely the procedure in \cite{KS}. Then we show in Subsection~\ref{subsec:evolving_non_atomic} how the lowest frequency ``non-atomic'' part of the low frequency approximation induction can be globally evolved with good $S^1$ norm bounds. In Subsection~\ref{subsec:interlude} we prove several technical assertions that all use the core perturbative result from Step 3 of the proof of Proposition~\ref{prop:bootstrap}. In Subsection~\ref{subsec:adding_first_atom}, we add the first ``large'' frequency atom by extracting concentration profiles and invoking the induction on energy hypothesis that all profiles have energy strictly less than $E_{crit}$. The end result of the modified Bahouri-G\'erard procedure is obtained in Subsection~\ref{subsec:conclusion}, see in particular Theorem~\ref{thm:compact_orbit}. We then have a minimal blowup solution $(\mathcal{A}^\infty, \Phi^\infty)$ with the required compactness property. 
\item In Section~\ref{sec:rigidity_argument} we rule out the existence of a minimal blowup solution $(\cA^\infty, \Phi^\infty)$ with the compactness property. To this end we largely follow the scheme of the rigidity argument by Kenig-Merle \cite{KM}. In Subsection~\ref{subsec:rigidity0} we derive several energy and virial identities for energy class solutions to (MKG-CG). Then we prove some preliminary properties of the minimal blowup solution $(\cA^\infty, \Phi^\infty)$, in particular that its momentum must vanish. Denoting by $I$ the lifespan of $(\cA^\infty, \Phi^\infty)$, we distinguish between $I^+ := I \cap [0,\infty)$ being a finite or an infinite time interval. In the next Subsection~\ref{subsec:rigidity1}, we exclude the existence of a minimal blowup solution $(\cA^\infty, \Phi^\infty)$ with infinite time interval $I^+$ using the virial identities, the fact that the momentum of $(\cA^\infty, \Phi^\infty)$ must vanish and an additional Vitali covering argument introduced in \cite{KS}. Moreover, we reduce the case of finite lifespan $I^+$ to a self-similar blowup scenario. In the last Subsection~\ref{subsec:rigidity2}, we then derive a suitable Lyapunov functional for the Maxwell-Klein-Gordon system in self-similar variables, which will enable us to also rule out the self-similar case and thus finishes the rigidity argument. Finally, we address the proof of the scattering assertion in Theorem~\ref{thm:TheMainTheorem}.
\end{itemize}
\nopagebreak

We remark that we will often abuse notation and denote the spatial components $A_x$ of the connection form simply by $A$.

\nopagebreak

\medskip

\noindent {\it Acknowledgments}: The authors are grateful to the referee for valuable corrections and suggestions.

\section{Function spaces and technical preliminaries} \label{sec:preliminaries}

We will be working with the same function spaces that were used for the small data energy critical global well-posedness result for the MKG-CG system \cite{KST} together with their time-localized versions. In this section we briefly recall their definitions. For a more detailed discussion of these spaces we refer to Section 3 in \cite{KST} and \cite{Tao_Wave_MapII}, \cite{Sterbenz-Tataru_1}, \cite{KS}.

In this work we only rely on the precise fine structure of these spaces in that we frequently use the multilinear estimates from \cite{KST} to reduce to ``sufficiently generic'' situations where a divisibility argument works, i.e. when all inputs are approximately at the same frequency and have angular separation between their frequency supports. 

\medskip

In order to introduce various Littlewood-Paley projection operators, we pick a non-negative even bump function $\varphi_0 \in C^\infty(\R)$ satisfying $\varphi_0(y) = 1$ for $|y| \leq 1$ and $\varphi_0(y) = 0$ for $|y| > 2$ and set $\varphi(y) = \varphi_0(y) - \varphi_0(2 y)$. Then we define the standard Littlewood-Paley projection operators for $k \in \Z$ by
\[
 \widehat{P_k f}(\xi) = \varphi \big( 2^{-k} |\xi| \big) \hat{f}(\xi).
\]
We use the concept of modulation to measure proximity of the space-time Fourier support to the light cone and define for $j \in \Z$ the projection operators
\begin{align*}
 \cF \big(Q_j f\big)(\tau, \xi) &= \varphi \big( 2^{-j} | |\tau| - |\xi| | \big) \, \cF(f)(\tau, \xi), \\
 \cF \big(Q_j^\pm f\big)(\tau, \xi) &= \varphi \big( 2^{-j} | |\tau| - |\xi| | \big) \, \chi_{\{ \pm \tau > 0 \}} \, \cF(f)(\tau, \xi), 
\end{align*}
where $\cF$ denotes the space-time Fourier transform. Occasionally, we also need multipliers $S_l$ to restrict the space-time frequency and correspondingly set for $l \in \Z$,
\[
 \cF \big(S_l f\big)(\tau, \xi) = \varphi \big( 2^{-l} |(\tau, \xi)| \big) \, \cF(f)(\tau, \xi).
\]
We also use projection operators $P_l^\omega$ to localize the homogeneous variable $\frac{\xi}{|\xi|}$ to caps $\omega \subset \mathbb{S}^3$ of diameter $\sim 2^l$ for integers $l < 0$ via smooth cutoffs. We assume that for each such $l < 0$ these cutoffs form a smooth partition of unity subordinate to a uniformly finitely overlapping covering of $\mathbb{S}^3$ by caps $\omega$ of diameter $\sim 2^l$.

\medskip

With these projection operators in hand we introduce the convention that for any norm $\|\cdot\|_S$ and any $p \in [1,\infty)$, 
\[
 \|F\|_{\ell^p S} = \biggl( \sum_{k \in \Z} \|P_k F\|_S^p \biggr)^{\frac{1}{p}}.
\]
Next we define the $X^{s,b}$ type norms applied to functions at spatial frequency $\sim 2^k$,
\[
 \|F\|_{X^{s,b}_p} = 2^{s k} \biggl( \sum_{j\in\Z} \Bigl( 2^{b j} \|Q_j P_k F\|_{L^2_t L^2_x} \Bigr)^p \biggr)^{\frac{1}{p}}
\]
for $s, b \in \R$ and $p \in [1,\infty)$ with the obvious analogue for $p = \infty$,
\[
 \|F\|_{X^{s,b}_\infty} = 2^{s k} \sup_{j \in \Z} 2^{b j} \|Q_j P_k F\|_{L^2_t L^2_x}.
\]

\medskip

We will mainly use three function spaces $N, N^\ast, \text{ and } S$. Their dyadic subspaces $N_k, N_k^\ast$ and $S_k$ satisfy 
\[
 N_k = L^1_t L^2_x + X_1^{0, -\frac{1}{2}}, \quad N_k^\ast = L^\infty_t L^2_x \cap X_\infty^{0, \frac{1}{2}}, \quad X_1^{0, \frac{1}{2}} \subseteq S_k \subseteq N_k^\ast.
\]
Then we have
\[
 \|F\|_N^2 = \sum_{k \in \Z} \|P_k F\|_{N_k}^2, \quad \|F\|_{N^\ast}^2 = \sum_{k \in \Z} \|P_k F\|_{N_k^\ast}^2.
\]
The space $S_k$ is defined by
\[
 \|\phi\|_{S_k}^2 = \|\phi\|_{S_k^{Str}}^2 + \|\phi\|_{S_k^{ang}}^2 + \|\phi\|_{X^{0,\frac{1}{2}}_\infty}^2,
\]
where
\begin{equation*}
 \begin{split}
  S_k^{Str} &= \bigcap_{\frac{1}{q} + \frac{3/2}{r} \leq \frac{3}{4}} 2^{(\frac{1}{q} + \frac{4}{r} - 2)k} L^q_t L^r_x, \\
  \|\phi\|_{S_k^{ang}}^2 &= \sup_{l < 0} \sum_\omega \|P_l^\omega Q_{< k + 2l} \phi \|_{S_k^\omega(l)}^2
 \end{split}
\end{equation*}
and the angular sector norms $S_k^\omega(l)$ are defined below. 

\medskip

To introduce the angular sector norms $S_k^\omega(l)$ we first define the plane wave space
\[
 \|\phi\|_{PW^\pm_\omega(l)} = \inf_{\phi = \int \phi^{\omega'}} \int_{|\omega - \omega'| \leq 2^l} \|\phi^{\omega'}\|_{L^2_{\pm \omega'} L^\infty_{(\pm \omega')^\perp}} \, d\omega'
\]
and the null energy space
\[
 \|\phi\|_{NE} = \sup_\omega \|\slashed{\nabla}_\omega \phi\|_{L^\infty_\omega L^2_{\omega^\perp}},
\]
where the norms are with respect to $\ell_\omega^\pm = t \pm \omega \cdot x$ and the transverse variable, while $\slashed{\nabla}_\omega$ denotes spatial differentiation in the $(\ell_\omega^+)^\perp$ plane. We now set 
\begin{equation*}
 \begin{split}
  \|\phi\|_{S_k^\omega(l)}^2 &= \|\phi\|_{S_k^{Str}}^2 + 2^{-2k} \|\phi\|_{NE}^2 + 2^{-3k} \sum_{\pm} \|Q^\pm \phi\|_{PW_\omega^{\mp}(l)}^2 \\
  &\quad \quad + \sup_{\substack{ k' \leq k, l' \leq 0, \\ k+2l \leq k' + l' \leq k+l} } \sum_{{\mathcal C}_{k'}(l')} \bigg( \|P_{{\mathcal C}_{k'}(l')} \phi\|_{S_k^{Str}}^2 + 2^{-2k} \|P_{{\mathcal C}_{k'}(l')} \phi\|_{NE}^2 \\
  &\quad \quad \quad \quad + 2^{-2k'-k} \|P_{{\mathcal C}_{k'}(l')} \phi\|_{L^2_t L^\infty_x}^2 + 2^{-3(k'+l')} \sum_{\pm} \|Q^{\pm} P_{{\mathcal C}_{k'}(l')} \phi \|_{PW_\omega^\mp(l)}^2 \bigg),
 \end{split}
\end{equation*}
where $P_{{\mathcal C}_{k'}(l')}$ is a projection operator to a radially directed block ${\mathcal C}_{k'}(l')$ of dimensions $2^{k'} \times (2^{k'+l'})^3$.

\medskip

Then we define
\[
 \|\phi\|_{S^1}^2 = \sum_{k \in \Z} \|\nabla_{t,x} P_k \phi\|_{S_k}^2 + \| \Box \phi \|^2_{\ell^1 L^2_t \dot{H}^{-\frac{1}{2}}_x}
\]
and the higher derivative norms
\[
 \|\phi\|_{S^N} := \|\nabla_{t,x}^{N-1} \phi\|_{S^1}, \quad N \geq 2.
\]
Moreover, we introduce
\[
 \|u\|_{S_k^\sharp} = \|\nabla_{t,x} u\|_{L^\infty_t L^2_x} + \|\Box u\|_{N_k}.
\]
On occasion we need to separate the two characteristic cones $\{ \tau = \pm |\xi| \}$. To this end we define
\begin{eqnarray*}
 N_{k,\pm},&  &N_k = N_{k,+} \cap N_{k, -} \\
 S_{k,\pm}^\sharp,&  &S_k^\sharp = S_{k,+}^\sharp + S_{k,-}^\sharp \\
 N_{k, \pm}^\ast,&  &N_k^\ast = N_{k,+}^\ast + N_{k,-}^\ast.
\end{eqnarray*}
We will also use an auxiliary space of $L^1_t L^\infty_x$ type,
\begin{equation*}
 \|\phi\|_{Z} = \sum_{k \in \Z} \|P_k \phi\|_{Z_k}, \quad \|\phi\|_{Z_k}^2 = \sup_{l < C} \sum_\omega 2^l \|P_l^\omega Q_{k+2l} \phi \|_{L^1_t L^\infty_x}^2.
\end{equation*}
Finally, to control the component $A_0$, we define
\[
 \|A_0\|_{Y^1}^2 = \|\nabla_{t,x} A_0\|_{L^\infty_t L^2_x}^2 + \|A_0\|_{L^2_t \dot{H}^{3/2}_x}^2 + \|\partial_t A_0\|_{L^2_t \dot{H}^{1/2}_x}^2
\]
and the higher derivative norms
\[
 \|A_0\|_{Y^N} = \|\nabla_{t,x}^{N-1} A_0\|_{Y^1}, \quad N \geq 2.
\]

\medskip

\noindent The link between the $S$ and $N$ spaces is given by the following energy estimate from \cite{KST},
\[
 \|\nabla_{t,x} \phi \|_{S} \lesssim \|\nabla_{t,x} \phi(0)\|_{L^2_x} + \| \Box \phi \|_{N}.
\]

\medskip

We will need to work with time-localized versions of the $S_k$ and $N_k$ spaces. For any compact interval $I \subset \R$ and $k \in \Z$, we define
\[
 \|\psi\|_{S_k(I\times\R^4)} := \inf_{\tilde{\psi}|_I = \psi|_I}  \|P_k \tilde{\psi}\|_{S_k(\R\times\R^4)}
\]
with $\psi$ and $\tilde{\psi}$ Schwartz functions. Analogously, we define $N_k(I\times\R^4)$.

\medskip

The following lemma shows that the $S_k$ and $Z_k$ spaces are compatible with time cutoffs. We will frequently use this fact without further mentioning.
\begin{lem} \label{lem:time_cutoff_compatible_with_S_norm}
 Let $\chi_I$ be a smooth cutoff to a time interval $I \subset \R$. Then it holds for all $k \in \Z$ that
 \[
  \bigl\|P_k (\chi_I \phi)\bigr\|_{S_k(\R\times\R^4)} \lesssim \bigl\|P_k \phi\bigr\|_{S_k(\R\times\R^4)}
 \]
 and 
 \[
  \bigl\|P_k (\chi_I\phi)\bigr\|_{Z_k(\R\times\R^4)} \lesssim \bigl\| P_k\phi \bigr\|_{Z_k(\R\times\R^4)}.
 \]
\end{lem}
\begin{proof} This is obvious for the Strichartz type norms. It remains to show it for the $X_\infty^{0,\frac{1}{2}}$ and $S_k^{ang}$ components. We start with the former. For fixed $j \in \Z$, we have 
\[
Q_j\big(\chi_I\phi\big) = Q_j\big(Q_{j+O(1)}(\chi_I) Q_{\leq j-C}\phi\big) + Q_j\big(\chi_IQ_{> j-C}(\phi)\big).
\]
Using the bound
\[
 \big\|Q_{j+O(1)}(\chi_I)\big\|_{L_t^2}\lesssim 2^{-\frac{1}{2} j}, 
\]
we obtain
\begin{align*}
2^{\frac{1}{2}j} \big\|Q_j\big(Q_{j+O(1)}(\chi_I) P_k Q_{\leq j-C}\phi\big)\big\|_{L^2_t L^2_x} \lesssim 2^{\frac{1}{2}j} \big\|Q_{j+O(1)}(\chi_I)\big\|_{L_t^2}\big\| P_k Q_{\leq j-C}\phi\big\|_{L_t^\infty L_x^2} \lesssim \big\|P_k\phi\big\|_{L_t^\infty L_x^2}.
\end{align*}
Moreover, we find
\begin{align*}
2^{\frac{1}{2} j} \big\|Q_j\big(\chi_IP_k Q_{> j-C}(\phi)\big)\big\|_{L^2_t L^2_x}&\lesssim 2^{\frac{1}{2} j}\big\|\chi_I\big\|_{L^\infty_t L^\infty_x}\big\|P_kQ_{> j-C}(\phi)\big\|_{L_{t,x}^2} \lesssim \big\|P_k\phi\big\|_{X^{0,\frac{1}{2}}_{\infty}}.
\end{align*}
Thus, we have
\[
 \bigl\| P_k ( \chi_I \phi ) \bigr\|_{X_\infty^{0,\frac{1}{2}}} \lesssim \bigl\|P_k \phi \bigr\|_{S_k}.
\]

Next, we consider the $S_k^{ang}$ component, which is given by
\[
 \|\phi\|_{S^{ang}_k} = \sup_{l<0}\sum_{\omega}\big\|P_l^{\omega}Q_{<k+2l}\phi\big\|^2_{S_k^{\omega}(l)}.
\]
We write 
\begin{align*}
P_l^{\omega}Q_{<k+2l}(\chi_I\phi) = P_l^{\omega}Q_{<k+2l}(\chi_I Q_{<k+2l+C}\phi) + P_l^{\omega}Q_{<k+2l}(\chi_I Q_{\geq k+2l+C}\phi)
\end{align*}
Then the first term on the right hand side is bounded by 
\begin{align*}
\big\|P_l^{\omega}Q_{<k+2l}(\chi_I Q_{<k+2l+C}\phi)\big\|_{S_k^{\omega}(l)}\lesssim \big\|P_l^{\omega}Q_{<k+2l+C}\phi\big\|_{S_k^{\omega}(l)},
\end{align*}
where we have used the fact that the operator $P_l^{\omega}Q_{<k+2l}$ is disposable. For the second term above, we use that 
\begin{align*}
\sum_{\omega}\big\|P_l^{\omega}Q_{<k+2l}(\chi_I Q_{\geq k+2l+C}\phi)\big\|^2_{S_k^{\omega}(l)} \lesssim \big\|P_kQ_{<k+2l}(\chi_I Q_{\geq k+2l+C}\phi)\big\|_{X^{0,\frac{1}{2}}_{1}}^2 \lesssim \big\|\phi\big\|_{X^{0,\frac{1}{2}}_{\infty}}^2.
\end{align*}

For the $Z_k$ space, fix a scale $l<0$ and consider the expression 
\[
\sum_{\omega}2^l\big\|P_l^{\omega}Q_{k+2l}\big(\chi_I\phi\big)\big\|_{L_t^1 L_x^\infty}^2.
\]
Write 
\begin{align*}
P_l^{\omega}Q_{k+2l}\big(\chi_I\phi\big) &= P_l^{\omega}Q_{k+2l}\big(Q_{<k+2l-C}(\chi_I)\phi\big) + P_l^{\omega}Q_{k+2l}\big(Q_{\geq k+2l-C}(\chi_I)\phi\big).
\end{align*}
For the first term on the right hand side, we have 
\[
\big\| P_l^{\omega}Q_{k+2l}\big(Q_{<k+2l-C}(\chi_I)\phi\big)\big\|_{L_t^1 L_x^\infty}\lesssim \big\|P_l^{\omega}Q_{k+2l+O(1)}\phi\big\|_{L_t^1L_x^\infty},
\]
which leads to an acceptable contribution. For the second term on the right hand side, we use
\begin{align*}
 \big\|P_l^{\omega}Q_{k+2l}\big(Q_{\geq k+2l-C}(\chi_I)\phi\big)\big\|_{L_t^1L_x^\infty}\lesssim \big\|Q_{\geq k+2l-C}(\chi_I)\big\|_{L_t^2} 2^{\frac{1}{2} l + \frac{1}{2} k} \big( 2^{-\frac{1}{2} l - \frac{1}{2} k} \big\|P_l^{\omega}\phi_k\big\|_{L_t^2L_x^\infty} \big)
\end{align*}
It follows that 
\begin{align*}
 2^{\frac{1}{2} l} \big\|P_l^{\omega}Q_{k+2l}\big(Q_{\geq k+2l-C}(\chi_I)\phi\big)\big\|_{L_t^1L_x^\infty} \lesssim \big( 2^{ -\frac{1}{2} l - \frac{1}{2} k} \big\|P_l^{\omega}\phi_k\big\|_{L_t^2L_x^\infty} \big), 
\end{align*}
which can be square-summed over $\omega$, see (9) in \cite{KST}. 
\end{proof}

\section{Microlocalized magnetic wave equation} \label{sec:magnetic_wave_equation}

In this section we assume that the spatial components of the connection form $A$ are solutions to the linear wave equation $\Box A^j = 0$ on $\R_t \times \R^4_x$ for $j= 1, \ldots, 4$ and that $A$ is in Coulomb gauge. We define the magnetic wave operator
\begin{equation} \label{equ:magnetic_wave_operator}
 \Box_{A}^p = \Box + 2i \sum_{k \in \Z} P_{\leq k-C} A^j P_k \partial_j.
\end{equation}
The goal of this section is to derive the following linear estimate for the magnetic wave operator~$\Box_A^p$.

\begin{thm} \label{thm:linear_estimates_magnetic_wave_equation}
 Suppose that $\Box A^j = 0$ on $\R_t \times \R^4_x$ for $j=1,\ldots, 4$ and that $A$ is in Coulomb gauge. For all $f \in N(\R \times \R^4)$ and $(g,h) \in \dot{H}^1_x(\R^4) \times L^2_x(\R^4)$, the solution to the magnetic wave equation
 \begin{equation} \label{equ:linear_magnetic_wave_equation}
  \left\{ \begin{aligned}
   \Box_A^p \phi &= f \text{ on } \R \times \R^4, \\
   (\phi, \phi_t)|_{t=0} &= (g,h) 
  \end{aligned} \right.
 \end{equation}
 exists globally and satisfies
 \begin{equation} \label{equ:linear_estimates_magnetic_wave_equation}
  \|\phi\|_{S^1(\R\times\R^4)} \leq C \Bigl( \|g\|_{\dot{H}^1_x} + \|h\|_{L^2_x} + \|f\|_{(N \cap \ell^1 L^2_t \dot{H}^{-\frac{1}{2}}_x)(\R\times\R^4)} \Bigr),
 \end{equation}
 where the constant $C > 0$ depends only on $\|\nabla_{t,x} A\|_{L^2_x}$ and grows at most polynomially in $\|\nabla_{t,x} A\|_{L^2_x}$.
\end{thm}
\begin{proof}
By time reversibility it suffices to prove the existence of the solution $\phi$ on the time interval $[0,\infty)$. Let $\varepsilon > 0$ be a sufficiently small constant to be fixed later. We may cover the time interval $[0,\infty)$ by finitely many consecutive closed intervals $I_1, \ldots, I_J$ with the following properties. The number of intervals $J$ depends only on $\|\nabla_{t,x} A\|_{L^2_x}$ and $\varepsilon$, the intervals $I_j$ overlap at most two at a time, consecutive intervals have intersection with non-empty interior and $[0, \infty) = \cup_{j=1}^\infty I_j$. Most importantly, the intervals $I_j$ are chosen such that a finite number of suitable space-time norms of the magnetic potential $A$ that will be specified later are less than $\varepsilon$ uniformly on all intervals $I_j$.

\medskip

We first construct suitable local solutions to the magnetic wave equation \eqref{equ:linear_magnetic_wave_equation} on the intervals $I_j$. The precise statement is summarized in the following theorem. Its proof will be given further below and is based on a parametrix construction. The accuracy of the parametrix crucially relies on the above mentioned smallness of suitable space-time norms of the magnetic potential $A$ on the intervals $I_j$. We use the notation $I_j = [T_j^{(l)}, T_j^{(r)}]$ for the left and right endpoints of $I_j$.

\begin{thm} \label{thm:local_solutions_magnetic_wave_equation}
 Let $f \in N(\R \times \R^4)$ and $(\tilde{g},\tilde{h}) \in \dot{H}^1_x(\R^4) \times L^2_x(\R^4)$. For $j = 1, \ldots, J$ there exists a solution $\phi^{(j)} \in S^1(\R \times \R^4)$ to
 \begin{equation} \label{equ:magnetic_wave_equation_on_I_j}
  \left\{ \begin{aligned}
   \Box_A^p \phi^{(j)} &= f \text{ on } I_j \times \R^4, \\
   (\phi^{(j)}, \phi^{(j)}_t)|_{t= T_j^{(l)}} &= (\tilde{g}, \tilde{h})
  \end{aligned} \right.
 \end{equation}
 in the sense that $\| \chi_{I_j} (\Box_A^p \phi^{(j)} - f) \|_{N(\R\times\R^4)} = 0$ for a sharp cutoff $\chi_{I_j}$ to the time interval $I_j$. Moreover, it holds that
 \begin{equation} \label{equ:estimates_on_phi_on_I_j}
  \|\phi^{(j)}\|_{S^1(\R \times \R^4)} \leq C \Bigl( \|\tilde{g}\|_{\dot{H}^1_x} + \|\tilde{h}\|_{L^2_x} + \|f\|_{(N \cap \ell^1 L^2_t \dot{H}^{-\frac{1}{2}}_x)(\R \times \R^4)} \Bigr),
 \end{equation}
 where the constant $C > 0$ depends only on $\|\nabla_{t,x} A\|_{L^2_x}$.
\end{thm}

\medskip

Finally, we obtain the solution $\phi$ to the magnetic wave equation \eqref{equ:linear_magnetic_wave_equation} on $[0,\infty)\times\R^4$ by patching together suitable local solutions on the intervals $I_j$. Given $(g, h) \in \dot{H}^1_x(\R^4) \times L^2_x(\R^4)$ and $f \in N(\R \times \R^4)$, Theorem \ref{thm:local_solutions_magnetic_wave_equation} yields a solution $\phi^{(1)} \in S^1(\R \times \R^4)$ on $I_1 = [0, T_1^{(r)}]$ to
\begin{equation*}
 \left\{ \begin{aligned}
  \Box_A^p \phi^{(1)} &= f \text{ on } I_1 \times \R^4, \\
  (\phi^{(1)}, \phi^{(1)}_t)|_{t=0} &= (g,h).  
 \end{aligned} \right.
\end{equation*}
Next, we obtain a solution $\phi^{(2)} \in S^1(\R \times \R^4)$ on $I_2 = [T_2^{(l)}, T_2^{(r)}]$ to
\begin{equation*}
 \left\{ \begin{aligned}
  \Box_A^p \phi^{(2)} &= f \text{ on } I_2 \times \R^4, \\
  (\phi^{(2)}, \phi^{(2)}_t)|_{t=T_2^{(l)}} &= (\phi^{(1)}(T_2^{(l)}), \phi_t^{(1)}(T_2^{(l)})),
 \end{aligned} \right.
\end{equation*}
where we recall that $I_1 \cap I_2 \neq \emptyset$ with $T_2^{(l)} < T_1^{(r)}$. We proceed analogously for the remaining intervals $I_3, \ldots, I_J$. By uniqueness, we must have $\phi^{(j)}|_{I_j \cap I_{j+1}} = \phi^{(j+1)}|_{I_j \cap I_{j+1}}$ for $j = 1, \ldots, J-1$. We choose a smooth partition of unity $\{\chi_j\}$ subordinate to the cover $\{I_j\}$ such that $\text{supp}(\chi_j) \subset I_j$ and $\text{supp}(\chi_j') \subset \subset  (I_{j-1} \cap I_j) \cup (I_j \cap I_{j+1})$. We then define
\[
 \phi = \sum_{j=1}^J \chi_j \phi^{(j)}.
\]
Since we have $\chi_j' + \chi_{j+1}' = 0$ on $I_j \cap I_{j+1}$ for $j = 1, \ldots, J-1$, it follows that
\[
 \sum_{j=1}^J \chi_j' \phi^{(j)} = 0 \text{ on } \R_t \times \R^4_x
\]
and hence,
\[
 \nabla_{t,x} \sum_{j=1}^J \chi_j \phi^{(j)} = \sum_{j=1}^J \chi_j \nabla_{t,x} \phi^{(j)} \text{ on } \R_t \times \R^4_x.
\]
Similarly, we find that
\[
 \Box \sum_{j=1}^J \chi_j \phi^{(j)} = \sum_{j=1}^J \chi_j \Box \phi^{(j)}.
\]
Using Lemma \ref{lem:time_cutoff_compatible_with_S_norm} and estimate \eqref{equ:estimates_on_phi_on_I_j}, we thus conclude that
\begin{equation*}
 \begin{split}
  \|\phi\|_{S^1(\R\times\R^4)} &= \bigl\| \sum_{j=1}^J \chi_j \phi^{(j)} \bigr\|_{S^1(\R\times\R^4)} \\
  &\lesssim \sum_{j=1}^J \|\phi^{(j)}\|_{S^1(\R\times\R^4)} \\
  &\lesssim C(\|\nabla_{t,x} A\|_{L^2_x}) \Bigl( \sum_{j=1}^J \|\phi^{(j)}(T_j^{(l)}) \|_{\dot{H}^1_x} + \|\partial_t \phi^{(j)}(T_j^{(l)})\|_{L^2_x} + \|f\|_{N(\R\times\R^4)} \Bigr) \\
  &\lesssim C(J) C(\|\nabla_{t,x} A\|_{L^2_x}) \bigl( \|g\|_{\dot{H}^1_x} + \|h\|_{L^2_x} + \|f\|_{N(\R\times\R^4)} \bigr).
 \end{split}
\end{equation*}
Since $J$ depends only on the size of $\|\nabla_{t,x} A\|_{L^2_x}$ and $\varepsilon$, we obtain the desired estimate \eqref{equ:linear_estimates_magnetic_wave_equation}.
\end{proof}

We proceed with the proof of Theorem \ref{thm:local_solutions_magnetic_wave_equation}.

\begin{proof}[Proof of Theorem \ref{thm:local_solutions_magnetic_wave_equation}]
We begin by considering for every $k \in \Z$ the frequency localized problem
\begin{equation} \label{equ:frequency_localized_magnetic_wave_equation_on_I_j}
 \left\{ \begin{aligned}
  \Box_{A_{<k}}^p \phi^{(j)}_k &= f_k \text{ on } I_j \times \R^4, \\
  (\phi^{(j)}_k, \partial_t \phi^{(j)}_k)|_{t=T_j^{(l)}} &= (\tilde{g}_k, \tilde{h}_k),
 \end{aligned} \right.
\end{equation}
where $\Box_{A_{<k}}^p = \Box + 2 i P_{\leq k-C} A^j P_k \partial_j$. Let $\chi_{I_j}$ denote a sharp cutoff to the time interval $I_j$. We first want to construct an approximate solution $\phi^{(j)}_{app, k}$ to \eqref{equ:frequency_localized_magnetic_wave_equation_on_I_j} that satisfies
\begin{equation} \label{equ:S_estimate_on_phi_app_k_j}
 \|\phi^{(j)}_{app,k}\|_{S_k^1(\R \times \R^4)} \leq C \Bigl( \|\tilde{g}_k\|_{\dot{H}^1_x} + \|\tilde{h}_k\|_{L^2_x} + \|f_k\|_{N_k(\R \times\R^4)} \Bigr)
\end{equation}
and
\begin{equation} \label{equ:error_estimate_on_phi_app_k_j}
 \begin{split}
  &\|\phi^{(j)}_{app,k}(T_j^{(l)}) - \tilde{g}_k\|_{\dot{H}^1_x} + \|\partial_t \phi^{(j)}_{app,k}(T_j^{(l)}) - \tilde{h}_k\|_{L^2_x} + \|\chi_{I_j} (\Box_{A_{<k}}^p \phi_{app,k} - f_k)\|_{(N_k \cap L^2_t \dot{H}^{-\frac{1}{2}}_x)(\R \times \R^4)} \\
  &\quad \quad \quad \lesssim \varepsilon \Bigl( \|\tilde{g}_k\|_{\dot{H}^1_x} + \|\tilde{h}_k\|_{L^2_x} + \|f_k\|_{(N_k \cap L^2_t \dot{H}^{-\frac{1}{2}}_x)(\R \times \R^4)} \Bigr).
 \end{split}
\end{equation}
To this end we split
\[
 f_k = f_k^{hyp} + f_k^{ell},
\]
where $f_k^{hyp}$ is supported in the region $| |\tau| - |\xi| | \lesssim 2^k$. We note that it holds that
\begin{equation} \label{equ:S_estimate_Box_inverse_f_k_ell}
 \| \Box^{-1} f_k^{ell} \|_{S_k^1(\R\times\R^4)} \lesssim \|f_k^{ell}\|_{N_k(\R\times\R^4)}.
\end{equation}
Theorem \ref{thm:approx_local_solution_freq_localized_magnetic_wave_equation} below then yields an approximate solution $\tilde{\phi}_{app,k}^{(j)}$ to 
\begin{equation}
 \left\{ \begin{aligned}
  \Box \tilde{\phi}_{app,k}^{(j)} &= f_k^{hyp} \text{ on } I_j \times \R^4, \\
  (\tilde{\phi}_{app,k}^{(j)}, \partial_t \tilde{\phi}_{app,k}^{(j)})|_{t = T_j^{(l)}} &= (\tilde{g}_k, \tilde{h}_k) - ( (\Box^{-1} f_k^{ell})(T_j^{(l)}), (\partial_t \Box^{-1} f_k^{ell})(T_j^{(l)}) ) 
 \end{aligned} \right.
\end{equation}
that satisfies
\begin{equation} \label{equ:S_estimate_tilde_phi_app_k_j}
 \bigl\| \tilde{\phi}_{app,k}^{(j)} \bigr\|_{S^1_k(\R\times\R^4)} \lesssim \|\tilde{g}_k\|_{\dot{H}^1_x} + \|\tilde{h}_k\|_{L^2_x} + \|f_k\|_{N_k(\R\times\R^4)}
\end{equation}
and
\begin{equation} \label{equ:error_estimate_tilde_phi_app_k_j}
 \begin{split}
  &\bigl\| \tilde{\phi}_{app,k}^{(j)}(T_j^{(l)}) - \bigl( \tilde{g}_k - (\Box^{-1} f_k^{ell})(T_j^{(l)}) \bigr) \bigr\|_{\dot{H}^1_x} + \bigl\| \partial_t \tilde{\phi}_{app,k}^{(j)}(T_j^{(l)}) - \bigl( \tilde{h}_k - (\partial_t \Box^{-1} f_k^{ell})(T_j^{(l)}) \bigr) \bigr\|_{L^2_x} \\
  &\quad + \bigl\| \chi_{I_j} \bigl( \Box_{A_{<k}^p} \tilde{\phi}_{app,k}^{(j)} - f_k^{hyp} \bigr) \bigr\|_{N_k(\R\times\R^4)} \\
  &\lesssim \varepsilon \bigl( \|\tilde{g}_k\|_{\dot{H}^1_x} + \|\tilde{h}_k\|_{L^2_x} + \|f_k\|_{N_k(\R\times\R^4)}  \bigr).
 \end{split}
\end{equation}
We remark that because of scaling invariance Theorem \ref{thm:approx_local_solution_freq_localized_magnetic_wave_equation} below is only formulated for the case $k=0$. Next we set
\[
 \phi_{app, k}^{(j)} = \tilde{\phi}_{app, k}^{(j)} + (\Box^{-1} f_k^{ell}) 
\]
and find that
\begin{equation} \label{equ:error_estimate_phi_app_k_j_derivation}
 \begin{split}
  &\bigl\| \chi_{I_j} \bigl( \Box_{A_{<k}}^p \phi_{app,k}^{(j)} - f_k \bigr) \bigr\|_{(N_k \cap L^2_t \dot{H}^{-\frac{1}{2}}_x)(\R\times\R^4)} \\
  &\lesssim \bigl\| \chi_{I_j} \bigl( \Box_{A_{<k}}^p \tilde{\phi}_{app,k}^{(j)} - f_k^{hyp} \bigr) \bigr\|_{(N_k \cap L^2_t \dot{H}^{-\frac{1}{2}}_x)(\R\times\R^4)} + \bigl\| \chi_{I_j} A_{<k}^j P_k \partial_j (\Box^{-1} f_k^{ell}) \bigr\|_{(N_k \cap L^2_t \dot{H}^{-\frac{1}{2}}_x)(\R\times\R^4)} \\
  &\lesssim \varepsilon \|f_k\|_{(N_k \cap L^2_t \dot{H}^{-\frac{1}{2}})(\R\times\R^4)}.
 \end{split}
\end{equation}
Here we used that the intervals $I_j$ can be chosen such that uniformly for all $j = 1, \ldots, J$,
\[
 \|A\|_{L^2_t L^8_x(I_j \times \R^4)} \leq \varepsilon
\]
and thus,
\begin{align*}
 \bigl\| \chi_{I_j} A_{<k}^j P_k \partial_j (\Box^{-1} f_k^{ell}) \bigr\|_{(N_k \cap L^2_t \dot{H}^{-\frac{1}{2}}_x)(\R\times\R^4)} &\leq \bigl\| \chi_{I_j} A_{<k}^j P_k \partial_j (\Box^{-1} f_k^{ell}) \bigr\|_{(L^1_t L^2_x \cap L^2_t \dot{H}^{-\frac{1}{2}}_x)(\R\times\R^4)} \\
 &\lesssim \bigl\| \chi_{I_j} A_{<k} \bigr\|_{(L^2_t L^\infty_x)(\R\times\R^4)} \bigl\| P_k \nabla_x (\Box^{-1} f_k^{ell}) \bigr\|_{(L^2_t L^2_x \cap L^\infty_t \dot{H}^{-\frac{1}{2}}_x)(\R\times\R^4)} \\
 &\lesssim \|A\|_{L^2_t L^8_x(I_j \times \R^4)} 2^{\frac{1}{2} k} \bigl\| P_k \nabla_x (\Box^{-1} f_k^{ell}) \bigr\|_{(L^2_t L^2_x \cap L^\infty_t \dot{H}^{-\frac{1}{2}}_x)(\R\times\R^4)} \\
 &\lesssim \varepsilon \bigl\| \Box^{-1} f_k^{ell} \bigr\|_{S^1_k(\R\times\R^4)} \\
 &\lesssim \varepsilon \| f_k \|_{N_k(\R\times\R^4)}.
\end{align*}
From \eqref{equ:S_estimate_Box_inverse_f_k_ell}, \eqref{equ:S_estimate_tilde_phi_app_k_j}, \eqref{equ:error_estimate_tilde_phi_app_k_j}, and \eqref{equ:error_estimate_phi_app_k_j_derivation} it now follows immediately that $\phi_{app, k}^{(j)}$ is an approximate solution to \eqref{equ:frequency_localized_magnetic_wave_equation_on_I_j} that satisfies the estimates \eqref{equ:S_estimate_on_phi_app_k_j} and \eqref{equ:error_estimate_on_phi_app_k_j}. 

\medskip

Finally, we reassemble the approximate solutions $\phi_{app, k}^{(j)}$ to the frequency localized problems \eqref{equ:frequency_localized_magnetic_wave_equation_on_I_j} to a full approximate solution $\phi^{(j)}_{app} = \sum_{k \in \Z} \phi^{(j)}_{app,k}$ to \eqref{equ:magnetic_wave_equation_on_I_j} satisfying
\begin{equation*}
 \begin{split}
  &\|\phi^{(j)}_{app}(T_j^{(l)}) - \tilde{g}\|_{\dot{H}^1_x} + \|\partial_t \phi^{(j)}_{app}(T_j^{(l)}) - \tilde{h}\|_{L^2_x} + \|\chi_{I_j} (\Box_{A}^p \phi^{(j)}_{app} - f) \|_{(N \cap \ell^1 L^2_t \dot{H}^{-\frac{1}{2}}_x)(\R \times \R^4)} \\
  &\quad \quad \quad \lesssim \varepsilon \bigl( \|\tilde{g}\|_{\dot{H}^1_x} + \|\tilde{h}\|_{L^2_x} + \|f\|_{(N \cap \ell^1 L^2_t \dot{H}^{-\frac{1}{2}}_x)(\R \times \R^4)} \bigr)
 \end{split}
\end{equation*}
and 
\[
 \bigl\| \phi_{app}^{(j)} \bigr\|_{S^1(\R\times\R^4)} \lesssim \|\tilde{g}\|_{\dot{H}^1_x} + \|\tilde{h}\|_{L^2_x} + \|f\|_{(N \cap \ell^1 L^2_t \dot{H}^{-\frac{1}{2}}_x)(\R \times \R^4)}.
\]
Applying this procedure iteratively to the successive errors, we obtain an exact solution $\phi^{(j)}$ to \eqref{equ:magnetic_wave_equation_on_I_j} satisfying \eqref{equ:estimates_on_phi_on_I_j}. 
\end{proof}

We now turn to the heart of the matter, namely the construction of the approximate solutions to the frequency localized magnetic wave equations.

\begin{thm} \label{thm:approx_local_solution_freq_localized_magnetic_wave_equation}
 Let $(\tilde{g}, \tilde{h}) \in \dot{H}^1_x(\R^4) \times L^2_x(\R^4)$ and $\tilde{f} \in N(\R\times\R^4)$. Assume that $\tilde{f}, \tilde{g}, \tilde{h}$ are frequency localized at $|\xi| \sim 1$ and that $\tilde{f}$ is localized at modulation $||\tau| - |\xi|| \lesssim 1$. For $j = 1, \ldots, J$ there exists an approximate solution $\tilde{\phi}_{app}^{(j)}$ to
 \begin{equation} \label{equ:zero_frequency_localized_magnetic_wave_equation}
  \left\{ \begin{aligned}
   \Box_{A_{<0}}^p \phi &= \tilde{f} \text{ on } I_j \times \R^4, \\
   (\phi, \phi_t)|_{t = T_j^{(l)}} &= (\tilde{g}, \tilde{h})
  \end{aligned} \right.
 \end{equation}
 in the sense that
 \begin{equation} \label{equ:S_estimate_tilde_phi_app}
  \bigl\| \tilde{\phi}_{app}^{(j)} \bigr\|_{S_0(\R\times\R^4)} \lesssim \|\tilde{g}\|_{L^2_x} + \|\tilde{h}\|_{L^2_x} + \|\tilde{f}\|_{N_0(\R\times\R^4)}
 \end{equation}
 and
 \begin{equation} \label{equ:error_estimate_tilde_phi}
  \begin{split}
  &\bigl\| \tilde{\phi}_{app}^{(j)}(T_j^{(l)}) - \tilde{g} \bigr\|_{L^2_x} + \bigl\| \partial_t \tilde{\phi}_{app}^{(j)}(T_j^{(l)}) - \tilde{h} \bigr\|_{L^2_x} + \bigl\| \chi_{I_j} \bigl( \Box_{A_{<0}}^p \tilde{\phi}_{app}^{(j)} - \tilde{f} \bigr) \bigr\|_{N_0(\R\times\R^4)} \\
  &\quad \lesssim \varepsilon \bigl( \|\tilde{g}\|_{L^2_x} + \|\tilde{h}\|_{L^2_x} + \|\tilde{f}\|_{N_0(\R\times\R^4)} \bigr),
  \end{split}
 \end{equation}
 where $\chi_{I_j}$ denotes a sharp cutoff to the time interval $I_j$.
\end{thm}
\begin{proof}
In order to prove estimates and construct a parametrix for the frequency localized magnetic wave equation \eqref{equ:zero_frequency_localized_magnetic_wave_equation} we adapt the scheme in Section 6 of \cite{KST} to our time-localized setting. We will use frequency localized renormalization operators $e_{<0}^{-i \psi_\pm}(t,x,D)$ and $e_{<0}^{+i\psi_\pm}(D,y,s)$, where $P(x,D)$ denotes the left quantization and $P(D,y)$ the right quantization of a pseudodifferential operator $P$ and where the subscript $<0$ denotes the space-time frequency localization of the symbol at frequencies $\ll 1$. For the definition of the phase correction $\psi_{\pm}$ in the renormalization operator $e_{<0}^{+i \psi_\pm}(D,y,s)$ we need to introduce some notation. 

\medskip

For any $\xi \in \R^4 \backslash \{0\}$ we set
\[
 \omega = \frac{\xi}{|\xi|}, \quad L_{\pm}^\omega := \pm \partial_t + \omega \cdot \nabla_x, \quad \Delta_{\omega^\perp} := \Delta - (\omega \cdot \nabla_x)^2.
\]
Moreover, for any $\omega \in \mathbb{S}^3$ and any angle $0 < \theta \lesssim 1$, we define the sector projection $\Pi_{>\theta}^\omega$ in frequency space by the formula
\[
 \widehat{\Pi_{> \theta}^\omega f}(\zeta) := \Bigl( 1 - \eta\Bigl(\frac{\angle (\zeta, \omega)}{\theta}\Bigr) \Bigr) \Bigl( 1 - \eta\Bigl(\frac{\angle (-\zeta, \omega)}{\theta}\Bigr) \Bigr) \hat{f}(\zeta),
\]
where $\eta(y)$ is a bump function on $\R$ which equals $1$ when $|y| < \frac{1}{2}$ and vanishes for $|y| > 1$, and $\angle(\zeta, \omega)$ is the angle between $\zeta$ and $\omega$. Thus, $\Pi_{>\theta}^\omega$ restricts $f$ smoothly (except at the frequency origin) to the sector of frequencies $\zeta$ whose angle with both $\omega$ and $-\omega$ is $\gtrsim \theta$. Similarly, we define the Fourier multipliers $\Pi_\theta^\omega$, $\Pi_{\leq \theta}^\omega$ and $\Pi^\omega_{\theta_1 > \cdot > \theta_2}$. 

\medskip

Let $C_1, C_2 > 0$ be constants to be chosen sufficiently large later on depending on the size of $\|\nabla_{t,x} A\|_{L^2_x}$ and let $\sigma > 0$ be a constant to be chosen sufficiently small. We then define the phase correction $\psi_\pm$ by
\begin{equation} \label{equ:definition_phase}
 \psi_{\pm}(t,x,\xi) = \sum_{-C_2 \leq k < 0} L^\omega_\pm \Delta_{\omega^\perp}^{-1} \Pi^\omega_{2^{-C_1} > \cdot > 2^{\sigma k}} A_k \cdot \omega + \sum_{k < -C_2} L^\omega_{\pm} \Delta_{\omega^\perp}^{-1} \Pi^\omega_{> 2^{\sigma k}} A_k \cdot \omega.
\end{equation}
Note that the first sum effectively only starts at $k \lesssim - \frac{C_1}{\sigma}$. See Section 6 in \cite{KST} for a motivation for such a choice of phase correction. We emphasize that this phase slightly differs from the one used in \cite{KST}, because for intermediate frequencies $-C_2 \leq k < 0$ the high angles are cut off. 

\medskip

We define the approximate solution $\tilde{\phi}_{app}^{(j)}$ to \eqref{equ:zero_frequency_localized_magnetic_wave_equation} by
\begin{equation*}
 \begin{split}
  \tilde{\phi}^{(j)}_{app} &= \chi_{I_j}(t) \frac{1}{2} \sum_\pm \biggl\{ e_{<0}^{-i\psi_\pm}(t,x,D) \frac{1}{|D|} e^{\pm i (t - T_j^{(l)}) |D|} e_{<0}^{ i \psi_\pm} (D,y,T_j^{(l)}) (|D| \tilde{g} \pm (-i) \tilde{h}) \\
  &\quad \quad \quad \quad \quad \quad \quad \pm e_{<0}^{-i\psi_\pm}(t,x,D) \frac{1}{|D|} K_j^{\pm} e_{<0}^{ i \psi_\pm}(D,y,s) (-i) \tilde{f} \biggr\},
 \end{split}
\end{equation*}
where
\[
 K_j^{\pm} \tilde{f}(t) = \int_{T_j^{(l)}}^t e^{\pm i (t-s) |D|} \tilde{f}(s) \, ds.
\]

In order to prove the estimates \eqref{equ:S_estimate_tilde_phi_app} and \eqref{equ:error_estimate_tilde_phi} we establish the following crucial time-localized mapping properties of the renormalization operator $e_{<0}^{\pm i \psi_\pm}(t,x,D)$.
\begin{thm} \label{thm:mapping_properties_renormalization_operator}
 For $j = 1, \ldots, J$, the frequency localized renormalization operators have the following mapping properties with $Z \in \{ N_0(\R \times \R^4), L^2_x(\R^4), N_0^\ast(\R \times\R^4) \}$,
 \begin{align}
  \chi_{I_j} e_{<0}^{\pm i \psi_\pm}: \quad & Z \to Z, \label{equ:Z_to_Z_bound_for_e}\\
  \chi_{I_j} \partial_t e_{<0}^{\pm i \psi_\pm}: \quad & Z \to \varepsilon Z, \label{equ:Z_to_Z_bound_for_time_derivative_e}\\
  \chi_{I_j} (e_{<0}^{-i\psi_\pm}(t,x,D) e_{<0}^{+i \psi_\pm}(D,y,t) - 1): \quad & Z \to \varepsilon Z, \label{equ:renormalization_error_estimate_1} \\
  \chi_{I_j} (e_{<0}^{-i\psi_\pm}(t,x,D) \Box - \Box_{A_{<0}}^p e_{<0}^{-i\psi_\pm}(t,x,D)): \quad & N_{0,\pm}^\ast(\R \times \R^4) \to \varepsilon N_{0,\pm}(\R \times \R^4), \label{equ:renormalization_error_estimate_2} \\
  \chi_{I_j} e_{<0}^{-i \psi_\pm}(t,x,D): \quad & S_0^\sharp(\R \times \R^4) \to S_0(\R \times \R^4), \label{equ:dispersive_estimate_for_e}
 \end{align}
 where $\chi_{I_j}$ denotes a sharp cutoff to the time interval $I_j$. In the estimates \eqref{equ:Z_to_Z_bound_for_e} and \eqref{equ:Z_to_Z_bound_for_time_derivative_e}, the operator $e_{<0}^{\pm i \psi_\pm}$, respectively $\partial_t e_{<0}^{\pm i \psi_\pm}$, stands for both left and right quantization.
\end{thm}

The estimates \eqref{equ:S_estimate_tilde_phi_app} and \eqref{equ:error_estimate_tilde_phi} then follow by adapting the manipulations in the proof of Theorem 4 in \cite{KST} to our time-localized setting. 
\end{proof}

The remainder of this section is devoted to the proof of Theorem \ref{thm:mapping_properties_renormalization_operator}. To this end we will adapt the general scheme of Sections 7 -- 11 in \cite{KST} to our large data setting. The accuracy of the approximate solution $\tilde{\phi}_{app}^{(j)}$ relies on the error estimates \eqref{equ:Z_to_Z_bound_for_time_derivative_e}, \eqref{equ:renormalization_error_estimate_1} and \eqref{equ:renormalization_error_estimate_2}. While in \cite{KST} the small energy assumption can be used to achieve smallness in the corresponding error estimates, we have to argue more carefully here, using the high angle cut-off in the definition of the phase correction and smallness of suitable space-time norms of $A$ on sufficiently small time intervals, namely the intervals $I_j$.

\subsection{Decomposable function spaces}

We begin by reviewing the notion of decomposable function spaces and estimates from \cite{RT}, \cite{Krieger-Sterbenz}, and \cite{KST}.

\medskip

Let $c(t,x,D)$ be a pseudodifferential operator whose symbol $c(t,x,\xi)$ is homogeneous of degree $0$ in $\xi$. Assume that $c$ has a representation 
\[
 c = \sum_{\theta \in 2^{-\N}} c^{(\theta)}.
\]
Let $1 \leq q,r \leq \infty$. For every $\theta \in 2^{-\N}$, we define 
\begin{equation*}
 \|c^{(\theta)}\|_{D_\theta (L^q_t L^r_x)(\R\times\R^4)} = \Biggl\| \biggl( \sum_{l=0}^{100} \sum_{\Gamma_\theta^\nu} \sup_{\omega \in \Gamma_\theta^\nu} \, \bigl\| b_\theta^\nu(\omega) (\theta \nabla_\xi)^l c^{(\theta)} \bigr\|_{L^r_x}^2 \biggr)^{1/2} \Biggr\|_{L^q_t(\R)},
\end{equation*}
where $\{\Gamma_\theta^\nu\}_{\nu \in \mathbb{S}^3}$ is a uniformly finitely overlapping covering of $\mathbb{S}^3$ by caps of diameter $\sim \theta$ and $\{ b_\theta^\nu \}_{\nu \in \mathbb{S}^3}$ is a smooth partition of unity subordinate to the covering $\{\Gamma_\theta^\nu\}_{\nu \in \mathbb{S}^3}$. Then we define the decomposable norm
\begin{equation*}
 \|c\|_{D(L^q_t L^r_x)(\R \times \R^4)} = \inf_{c = \sum_{\theta } c^{(\theta)}} \sum_{\theta \in 2^{-\mathbb{N}}} \|c^{(\theta)}\|_{D_\theta(L^q_t L^r_x)(\R \times \R^4)}.
\end{equation*}

We will repeatedly use the following decomposable estimates.
\begin{lem}[\protect{\cite[Lemma 7.1]{KST}}] \label{lem:decomposability_lemma_KST}
 Let $P(t,x,D)$ be a pseudodifferential operator with symbol $p(t,x,\xi)$. Suppose that $P$ satisfies the fixed-time estimate
 \[
  \sup_{t \in \R} \|P(t,x,D)\|_{L^2_x \to L^2_x} \lesssim 1.
 \]
 Let $1 \leq q, q_1, q_2, r, r_1 \leq \infty$ such that $\frac{1}{q} = \frac{1}{q_1} + \frac{1}{q_2}$ and $\frac{1}{r} = \frac{1}{r_1} + \frac{1}{2}$. For any symbol $c(t,x,\xi) \in D(L^{q_1}_t L^{r_1}_x)(\R \times \R^4)$ that is zero homogeneous in $\xi$, we have
 \begin{equation*}
  \|(c p)(t,x,D) \phi\|_{L^{q}_t L^{r}_x(\R\times\R^4)} \lesssim \|c\|_{D(L^{q_1}_t L^{r_1}_x)(\R\times\R^4)} \|\phi\|_{L^{q_2}_t L^2_x(\R\times\R^4)}.
 \end{equation*}
\end{lem}

By duality we obtain decomposable estimates for right quantizations.
\begin{lem} \label{lem:decomposability_lemma_right_quantizations}
 Let $P$ be a pseudodifferential operator with symbol $p(t,x,\xi)$. Suppose that $P$ satisfies the fixed-time estimate
 \[
  \sup_{t \in \R} \|P(t,x,D)\|_{L^2_x \to L^2_x} \lesssim 1.
 \]
 Let $1 \leq q < \infty$ and $1 \leq q_1, q_2 \leq \infty$ such that $\frac{1}{q} = \frac{1}{q_1} + \frac{1}{q_2}$. For any symbol $c(t,x,\xi) \in D(L^{q_1}_t L^\infty_x)(\R \times \R^4)$ that is zero homogeneous in $\xi$, the right-quantized operator $(\overline{c} \, \overline{p})(D,y,t)$ has the following mapping property
 \begin{equation*}
  \bigl\| (\overline{c} \, \overline{p})(D, y, t) \phi \bigr\|_{L^q_t L^2_x(\R\times\R^4)} \lesssim \|c\|_{D(L^{q_1}_t L^\infty_x)(\R\times\R^4)} \|\phi\|_{L^{q_2}_t L^2_x(\R\times\R^4)}.
 \end{equation*}
\end{lem}
\begin{proof}
 Let $1 < q' \leq \infty$ be the the conjugate exponent to $q$ and define $\frac{1}{\tilde{q}} = \frac{1}{q_1} + \frac{1}{q'}$. By duality, H\"older's inequality and Lemma \ref{lem:decomposability_lemma_KST}, we have
 \begin{align*}
  \bigl\| (\overline{c} \, \overline{p})(D,y,t) \phi \bigr\|_{L^q_t L^2_x} &= \sup_{\|\psi\|_{L^{q'}_t L^2_x} \leq 1} \bigl\langle \psi, (\overline{c} \, \overline{p})(D,y,t) \phi \bigr\rangle \\
  &= \sup_{\|\psi\|_{L^{q'}_t L^2_x} \leq 1} \bigl\langle (c p)(t,x,D) \psi, \phi \bigr\rangle \\
  &\leq \sup_{\|\psi\|_{L^{q'}_t L^2_x} \leq 1} \bigl\| (c p)(t,x,D) \psi \bigr\|_{L^{\tilde{q}}_t L^2_x} \|\phi\|_{L^{q_2}_t L^2_x} \\
  &\lesssim \sup_{\|\psi\|_{L^{q'}_t L^2_x} \leq 1} \|c\|_{D(L^{q_1}_t L^\infty_x)} \|\psi\|_{L^{q'}_t L^2_x} \|\phi\|_{L^{q_2}_t L^2_x} \\
  &\lesssim \|c\|_{D(L^{q_1}_t L^\infty_x)} \|\phi\|_{L^{q_2}_t L^2_x}.
 \end{align*}
\end{proof}

From \cite[Lemma 10.2]{Krieger-Sterbenz} we have the following H\"older-type estimate for decomposable norms
\begin{equation} \label{equ:hoelder_type_estimate_decomposable_norms}
 \Bigl\| \prod_{i=1}^m c_i \Bigr\|_{D(L^q_t L^r_x)} \lesssim \prod_{i=1}^m \|c_i\|_{D(L^{q_i}_t L^{r_i}_x)},
\end{equation}
where $m \in \N$, $1 \leq q,r, q_i, r_i \leq \infty$ for $i = 1, \ldots, m$ and $(\frac{1}{q}, \frac{1}{r}) = \sum_{i=1}^m (\frac{1}{q_i}, \frac{1}{r_i})$.

\subsection{Some symbol bounds for phases}

Recall that the magnetic potential $A$ is assumed to be supported at frequencies $\lesssim 1$. For any integer $k < 0$ and any dyadic angle $0 < \theta \lesssim 1$, we use the notation
\begin{equation*}
 \psi_k^{(\theta)}(t,x,\xi) = L^\omega_\pm \Delta_{\omega^\perp}^{-1} \Pi_\theta^\omega A_k \cdot \omega
\end{equation*}
and
\begin{equation*}
 \psi_{<k} = \sum_{l < k} P_l \psi_\pm.
\end{equation*}

\begin{lem}
 For any $t,s \in \R$, $x,y \in \R^4$, $\xi \in \R^4$ and any integer $k<0$, it holds that
 \begin{equation} \label{equ:symbol_bound_difference_psi}
  |\psi_\pm(t,x,\xi) - \psi_\pm(s,y,\xi)| \lesssim ( 2^{-C_1/2} + 2^{-C_2} ) \|\nabla_{t,x} A(0)\|_{L^2_x} ( |t-s| + |x-y| ). 
 \end{equation}
 Moreover, we have for any multi-index $\alpha \in \N_0^4$ with $1 \leq |\alpha| \leq \sigma^{-1}$ that
 \begin{equation} \label{equ:symbol_bound_difference_xi_derivative_of_psi}
  |\nabla_\xi^{\alpha} ( \psi(t,x,\xi) - \psi(s,y,\xi) )| \lesssim \langle |t-s| + |x-y| \rangle^{\sigma (|\alpha| - \frac{1}{2})} \|\nabla_{t,x} A(0)\|_{L^2_x}.
 \end{equation}
\end{lem}
\begin{proof}
 For any $t \in \R, x \in \R^4$, $\xi \in \R^4$ and any integer $k < 0$, we obtain that
\begin{equation*}
 \begin{split}
  |\psi_k^{(\theta)}(t,x,\xi)| &\leq \| L_\pm^\omega \Delta_\omega^{-1} \Pi_\theta^\omega P_k A \cdot \omega \|_{L^\infty_x} \\
  &\lesssim (\theta^3 2^{4k})^{\frac{1}{2} - \frac{1}{\infty}} \| L_\pm^\omega \Delta_\omega^{-1} \Pi_\theta^\omega P_k A \cdot \omega \|_{L^2_x} \\
  &\lesssim \theta^{3/2} 2^{2k} \theta 2^{-2k} \theta^{-2} \|\Pi_\theta^\omega P_j L^\omega_\pm A\|_{L^2_x} \\
  &\lesssim \theta^{1/2} \|\nabla_{t,x} A_k\|_{L^2_x},
 \end{split}
\end{equation*}
where we used Bernstein's inequality, the Coulomb gauge of $A$ and that $|\widehat{\Delta^{-1}_{\omega^\perp}}(\xi)| \sim 2^{-2k} \theta^{-2}$ on the frequency support of $\Pi_\theta^\omega P_k$. Similarly, we find
\[
 |\nabla_{t,x} \psi_k^{(\theta)}(t,x,\xi)| \lesssim 2^{k} \theta^{1/2} \|\nabla_{t,x} A_k\|_{L^2_x}.
\]
Thus, we have
\begin{equation*}
 \begin{split}
  &|\psi_\pm(t,x,\xi) - \psi_\pm(s,y,\xi)| \\
  &\leq \sum_{-C_2 \leq k < 0} \sum_{2^{\sigma k} < \theta < 2^{-C_1}} |\psi_k^{(\theta)}(t,x,\xi) - \psi_k^{(\theta)}(s,y,\xi)| + \sum_{k < -C_2} \sum_{2^{\sigma k} < \theta} |\psi_k^{(\theta)}(t,x,\xi)-\psi_k^{(\theta)}(s,y,\xi)| \\
  &\leq \biggl( \sum_{-C_2 \leq k < 0} \sum_{2^{\sigma k} < \theta < 2^{-C_1}} 2^k \theta^{1/2} + \sum_{k < -C_2} \sum_{2^{\sigma k} < \theta} 2^k \theta^{1/2} \biggr)^{1/2} \|\nabla_{t,x} A\|_{L^2_x} (|x-y| + |t-s|) \\
  &\leq (2^{-C_1/2} + 2^{-C_2}) \|\nabla_{t,x} A\|_{L^2_x} (|x-y| + |t-s|).
 \end{split}
\end{equation*}

We now turn to the proof of \eqref{equ:symbol_bound_difference_xi_derivative_of_psi}. To this end we note that differentiating with respect to $\xi$ yields $\theta^{-1}$ factors, i.e. for any $\alpha \in \N^4_0$ it holds that 
\[
 |\nabla_\xi^\alpha \psi_k^{(\theta)}(t,x,\xi)| \lesssim \theta^{\frac{1}{2}-|\alpha|} \|\nabla_{t,x} A_k\|_{L^2_x}
\]
and
\[
 |\nabla_{t,x} \nabla_\xi^\alpha \psi_k^{(\theta)}(t,x,\xi)| \lesssim 2^k \theta^{\frac{1}{2} - |\alpha|} \|\nabla_{t,x} A_k\|_{L^2_x}.
\]
For any $1 \leq |\alpha| \leq \sigma^{-1}$ and $l < 0$ we then obtain
\begin{equation*}
 \begin{split}
  &|\nabla_\xi^\alpha( \psi_\pm(t,x,\xi) - \psi_\pm(s,y,\xi) )| \\
  &\lesssim \sum_{k < l} \sum_{2^{\sigma k} < \theta} 2^k \theta^{\frac{1}{2} - |\alpha|} \|\nabla_{t,x} A\|_{L^2_x} (|x-y| + |t-s|)  + \sum_{k \geq l} \sum_{2^{\sigma k} < \theta} \theta^{\frac{1}{2} - |\alpha|} \|\nabla_{t,x} A\|_{L^2_x} \\
  &\lesssim 2^{l(1-\sigma(|\alpha|-\frac{1}{2}))} \|\nabla_{t,x} A\|_{L^2_x} (|x-y| + |t-s|) + 2^{-\sigma l (|\alpha| - \frac{1}{2})} \|\nabla_{t,x} A\|_{L^2_x}.
 \end{split}
\end{equation*}
Optimizing the choice of $l < 0$ we find that
\begin{equation*}
 |\nabla_\xi^\alpha( \psi_\pm(t,x,\xi) - \psi_\pm(s,y,\xi) )| \lesssim \langle |t-s| + |x-y| \rangle^{\sigma (|\alpha|-\frac{1}{2})} \|\nabla_{t,x} A\|_{L^2_x}.
\end{equation*}
\end{proof}

We will frequently use the following bounds on decomposable norms of the phase $\psi_\pm$.
\begin{lem}[\protect{\cite[Lemma 7.3]{KST}}] \label{lem:symbol_bounds_psi_k_theta}
 Let $2 \leq q,r \leq \infty$ with $\frac{2}{q} + \frac{3}{r} \leq \frac{3}{2}$. For any integer $k < 0$ and any dyadic angle $\theta \in 2^{-\N}$ the component $\psi_k^{(\theta)} = L_\pm^\omega \Delta_{\omega^\perp}^{-1} \Pi_\theta^\omega A_k \cdot \omega$ satisfies
 \begin{equation} \label{equ:psi_k_theta_decomposable_estimate}
  \bigl\|(\psi_k^{(\theta)}, 2^{-k} \nabla_{t,x} \psi_k^{(\theta)})\bigr\|_{D_\theta(L^q_t L^r_x)(\R\times\R^4)} \lesssim 2^{-(\frac{1}{q} + \frac{4}{r})k} \theta^{\frac{1}{2}-\frac{2}{q}-\frac{3}{r}} \|\nabla_{t,x} A\|_{L^2_x}.
 \end{equation}
\end{lem}

\subsection{Oscillatory integral estimates}

In order to prove the mapping properties in Theorem \ref{thm:mapping_properties_renormalization_operator}, we need pointwise kernel bounds for operators of the form
\[
 T_a = e^{-i \psi_\pm}(t,x,D) a(D) e^{\pm i(t-s)|D|} e^{i \psi_\pm}(D,y,s),
\]
where $a$ is localized at frequency $|\xi| \sim 1$. The kernel of $T_a$ is given by the oscillatory integral
\begin{equation*}
 K_a(t,x;s,y) = \int_{\R^4} e^{-i(\psi_\pm(t,x,\xi) - \psi_\pm(s,y,\xi))} e^{i(t-s)|\xi|} e^{i(x-y)\cdot \xi} a(\xi) \, d\xi,
\end{equation*}
where $a$ is a smooth bump function with support on the annulus $|\xi| \sim 1$.

\begin{lem}
 For any $t,s \in \R$, $x,y \in \R^4$ and any integer $1 \leq N \leq \sigma^{-1}$, we have
 \begin{equation} \label{equ:fixed_time_decay_kernel_K}
  |K_a(t,x;t,y)| \lesssim \frac{\|\nabla_{t,x} A\|_{L^2_x}}{\langle |x-y| \rangle^{N(1-\sigma)}}
 \end{equation}
 and
 \begin{equation} \label{equ:fixed_time_decay_K_minus_a}
  |K_a(t,x;t,y) - \check{a}(x-y)| \lesssim \min \Bigl\{ ( 2^{-C_1/2} + 2^{-C_2} ), \frac{1}{|x-y|^{N(1-\sigma)}} \Bigr\} \|\nabla_{t,x} A\|_{L^2_x}.
 \end{equation}
 Moreover, it holds that
 \begin{equation} \label{equ:dispersive_estimate_kernel_K}
  |K_a(t,x;s,y)| \lesssim \langle t-s \rangle^{-\frac{3}{2}} \langle |t-s| - |x-y| \rangle^{-N} \|\nabla_{t,x} A\|_{L^2_x}^2.
 \end{equation} 
 \end{lem}
\begin{proof}
 If $|x-y| \lesssim 1$ we obtain by taking absolute values that
 \[
  |K_a(t,x;t,y)| \lesssim 1.
 \]
 If instead $|x-y| \gg 1$, we use \eqref{equ:symbol_bound_difference_xi_derivative_of_psi} and integrate by parts repeatedly to obtain for any $1 \leq N \leq \sigma^{-1}$ that
 \[
  |K_a(t,x;t,y)| \lesssim \frac{\|\nabla_{t,x} A\|_{L^2_x}}{|x-y|^{N (1-\sigma)}}.
 \]
 This proves \eqref{equ:fixed_time_decay_kernel_K}. In order to show \eqref{equ:fixed_time_decay_K_minus_a}, we integrate by parts repeatedly for $|x-y| \gg 1$, while for $|x-y| \lesssim 1$, we use \eqref{equ:symbol_bound_difference_psi} to estimate
 \begin{equation*}
  \begin{split}
   |K_a(t,x;t,y) - \check{a}(x-y)| &\leq \int_{\R^4} |e^{i(\psi_\pm(t,x,\xi)-\psi_\pm(t,y,\xi))} - 1| \, a(\xi) \, d\xi \\
   &\leq \int_{\R^4} |\psi_\pm(t,x,\xi) - \psi_\pm(t,y,\xi)| \, a(\xi) \, d\xi \\
   &\lesssim ( 2^{-C_1/2} + 2^{-C_2} ) \|\nabla_{t,x} A\|_{L^2_x} |x-y|.
  \end{split}
 \end{equation*}
 Finally, the proof of \eqref{equ:dispersive_estimate_kernel_K} proceeds along the lines of Proposition 6(a) in \cite{KST}. We only have to argue a bit more that away from the cone for sufficiently large $|t-s|$ the phase is still non-degenerate. But this is because away from the cone
 \begin{equation*}
  \begin{split}
   &\bigl|-i \nabla_\xi(\psi_\pm(t,x,\xi) - \psi_\pm(s,y,\xi)) + i(t-s) \frac{\xi}{|\xi|} + i (x-y) \bigr| \\
   &\geq  c \langle |t-s| + |x-y| \rangle - C \|\nabla_{t,x} A\|_{L^2_x} \langle |t-s| + |x-y| \rangle^{ \frac{1}{2} \sigma}
  \end{split}
 \end{equation*}
 and we choose $0 < \sigma \ll 1$ sufficiently small.
\end{proof}
To deal with the frequency localized operators $e_{<0}^{\pm i \psi_\pm}(t,x,D)$ and $e_{<0}^{i\psi_\pm}(D,y,s)$, we need to produce similar estimates for the kernel $K_{a, <0}$ of the operator
\[
 T_{a, <0} = e_{<0}^{- i \psi_\pm}(t,x,D) a(D) e^{\pm i (t-s) |D|} e_{<0}^{i \psi_\pm}(D,y,s).
\]
Noting that the frequency localized symbol $e_{<0}^{\pm i \psi_\pm}$ can be represented as
\[
 e_{<0}^{\pm i \psi_\pm} = \int_{\R^{1+4}_z} m(z) e^{\pm i T_z \psi_\pm} \, dz,
\]
where $m(z)$ is an integrable bump function on the unit scale and $T_z$ denotes space-time translation in the direction $z \in \R^{1+4}$, the transition to these frequency localized operators can be made just as in Proposition 7 in \cite{KST}. We obtain the following estimates for $K_{a,<0}$.
\begin{lem}
 For any $t,s \in \R$, $x,y \in \R^4$ and any integer $1 \leq N \leq \sigma^{-1}$, we have
 \begin{equation} \label{equ:fixed_time_decay_kernel_K_0}
  |K_{a, <0}(t,x;t,y)| \lesssim \frac{\|\nabla_{t,x} A\|_{L^2_x}}{\langle |x-y| \rangle^{N(1-\sigma)}}
 \end{equation}
 and
 \begin{equation} \label{equ:fixed_time_decay_K_0_minus_a}
  |K_{a,<0}(t,x;t,y) - \check{a}(x-y)| \lesssim \min \Bigl\{ ( 2^{-C_1/2} + 2^{-C_2} ), \frac{1}{|x-y|^{N(1-\sigma)}} \Bigr\} \|\nabla_{t,x} A\|_{L^2_x}.
 \end{equation}
 Moreover, it holds that
 \begin{equation} \label{equ:dispersive_estimate_kernel_K_0}
  |K_{a, <0}(t,x;s,y)| \lesssim \langle t-s \rangle^{-\frac{3}{2}} \langle |t-s| - |x-y| \rangle^{-N} \|\nabla_{t,x} A\|_{L^2_x}^2.
 \end{equation} 
\end{lem}

\subsection{Fixed-time $L^2_x$ estimates in Theorem \ref{thm:mapping_properties_renormalization_operator}}

In this subsection we prove the fixed-time $L^2_x$ estimates in Theorem \ref{thm:mapping_properties_renormalization_operator} using the above oscillatory integral estimates. To obtain a small factor $\varepsilon$ in the estimates \eqref{equ:Z_to_Z_bound_for_e} and \eqref{equ:Z_to_Z_bound_for_time_derivative_e}, we additionally have to fix the constants $C_1, C_2 > 0$ in the definition of the phase correction $\psi_\pm$ sufficiently large. 

\begin{lem} \label{lem:renormalization_operator_fixed_time_1}
 For any $t \in \R$, we have
 \begin{equation} \label{equ:fixed_time_L_2_bound}
  \bigl\|e_{<0}^{\pm i \psi_\pm}(t,x,D) P_0 \phi \bigr\|_{L^2_x} \lesssim \|\nabla_{t,x}A\|_{L^2_x} \|P_0 \phi\|_{L^2_x}
 \end{equation}
 and
 \begin{equation} 
  \bigl\|e_{<0}^{\pm i \psi_\pm}(D,y,t) P_0 \phi \bigr\|_{L^2_x} \lesssim \|\nabla_{t,x}A\|_{L^2_x} \|P_0 \phi\|_{L^2_x}
 \end{equation}
\end{lem}
\begin{proof}
 The claim follows immediately from the kernel bound \eqref{equ:fixed_time_decay_kernel_K_0} and a  $T T^\ast$-argument.
\end{proof}

\begin{lem} \label{lem:renormalization_operator_fixed_time_2}
 For any $\varepsilon > 0$ the constants $C_1, C_2 > 0$ in the definition of the phase correction $\psi_\pm$ can be chosen sufficiently large (depending on the size of $\varepsilon^{-1}$ and $\|\nabla_{t,x} A\|_{L^2_x}$) such that we have 
 \begin{equation} \label{equ:fixed_time_L_2_bound_on_derivative}
  \bigl\| (\nabla_{t,x} e_{<0}^{-i \psi_\pm})(t,x,D) P_0 \phi \bigr\|_{L^2_x} \lesssim \varepsilon \|P_0 \phi\|_{L^2_x}.
 \end{equation}
\end{lem}
\begin{proof}
 Using Lemma \ref{lem:decomposability_lemma_KST}, Lemma \ref{lem:symbol_bounds_psi_k_theta}, and \eqref{equ:fixed_time_L_2_bound} we obtain that
 \begin{equation*}
  \begin{split}
   \bigl\| (\nabla_{t,x} e_{<0}^{-i\psi_\pm})(t,x,D) P_0 \phi \bigr\|_{L^2_x} &= \bigl\|(\nabla_{t,x} \psi_\pm) e_{<0}^{-i\psi_\pm}(t,x,D) P_0 \phi \bigr\|_{L^2_x} \\
   &\leq \sum_{-C_2 \leq k < 0} \sum_{2^{\sigma k} < \theta < 2^{-C_1}} \bigl\|(\nabla_{t,x} \psi_k^{(\theta)}) e_{<0}^{-i \psi_\pm}(t,x,D) P_0 \phi \bigr\|_{L^2_x} \\
   &\quad + \sum_{k < -C_2} \sum_{2^{\sigma k} < \theta} \bigl\|(\nabla_{t,x} \psi_k^{(\theta)}) e_{<0}^{-i \psi_\pm}(t,x,D) P_0 \phi \bigr\|_{L^2_x} \\
   &\lesssim \sum_{-C_2 \leq k < 0} \sum_{2^{\sigma k} < \theta < 2^{-C_1}} \|\nabla_{t,x} \psi_k^{(\theta)}\|_{D_\theta(L^\infty_t L^\infty_x)} \|\nabla_{t,x} A\|_{L^2_x} \|P_0 \phi\|_{L^2_x} \\
   &\quad + \sum_{k < -C_2} \sum_{2^{\sigma k} < \theta} \|\nabla_{t,x} \psi_k^{(\theta)}\|_{D_\theta(L^\infty_t L^\infty_x)} \|\nabla_{t,x} A\|_{L^2_x} \|P_0 \phi\|_{L^2_x} \\
   &\lesssim \biggl( \sum_{-C_2 \leq k < 0} \sum_{2^{\sigma k} < \theta < 2^{-C_1}} 2^k \theta^{1/2} + \sum_{k < -C_2} \sum_{2^{\sigma k} < \theta} 2^k \theta^{1/2} \biggr) \|\nabla_{t,x} A\|_{L^2_x}^2 \|P_0 \phi\|_{L^2_x} \\
   &\lesssim (2^{-C_1/2} + 2^{-C_2}) \|\nabla_{t,x} A\|_{L^2_x}^2 \|P_0 \phi\|_{L^2_x},
  \end{split}
 \end{equation*}
 from which the assertion follows.
\end{proof}

\begin{lem} \label{lem:renormalization_operator_fixed_time_3}
 For any $\varepsilon > 0$ the constants $C_1, C_2 > 0$ in the definition of the phase correction $\psi_\pm$ can be chosen sufficiently large (depending on the size of $\varepsilon^{-1}$ and $\|\nabla_{t,x} A\|_{L^2_x}$) such that we have 
 \begin{equation} \label{equ:fixed_time_L_2_bound_renormalization_error_1}
  \bigl\| (e_{<0}^{-i \psi_\pm}(t,x,D) e_{<0}^{i\psi_\pm}(D,y,t) - 1) P_0 \phi \bigr\|_{L^2_x} \lesssim \varepsilon \|P_0 \phi\|_{L^2_x}.
 \end{equation}
\end{lem}
\begin{proof}
 The integral kernel of
  \[   
   (e_{<0}^{-i \psi_\pm}(t,x,D) e_{<0}^{i\psi_\pm}(D,y,t) - 1) a(D)
  \]
 is given by $K_{a, <0}(t,x;t,y) - \check{a}(x-y)$. Using \eqref{equ:fixed_time_decay_K_0_minus_a} we find that
 \begin{equation*}
  \begin{split}
   &\sup_y \int_{\R^4} |K_{a, <0}(t,x;t,y) - \check{a}(x-y)| \, dx \\
   &\lesssim \int_{\R^4} \min \Bigl\{ (2^{-C_1/2} + 2^{-C_2}), \frac{1}{|x|^{N(1-\sigma)}} \Bigr\} \|\nabla_{t,x} A\|_{L^2_x} \, dx \\
   &\lesssim \inf_{R>0} \Bigl\{ (2^{-C_1/2} + 2^{-C_2}) R^4 + \frac{1}{R^{N(1-\sigma)-4}} \Bigr\} \|\nabla_{t,x} A\|_{L^2_x}.
  \end{split}
 \end{equation*}
 Choosing $C_1, C_2 > 0$ sufficiently large depending on the size of $\varepsilon^{-1}$ and $\|\nabla_{t,x} A\|_{L^2_x}$, we obtain that
 \[
  \sup_y \int_{\R^4} |K_{a, <0}(t,x;t,y) - \check{a}(x-y)| \, dx \leq \varepsilon
 \]
 and similarly for $\sup_x \int_{\R^4} |K_a(t,x;t,y) - \check{a}(x-y)| \, dy$. The assertion then follows from Schur's lemma.
\end{proof}

\begin{rem}
The fixed-time $L^2_x$ bounds from Lemma \ref{lem:renormalization_operator_fixed_time_1}, Lemma \ref{lem:renormalization_operator_fixed_time_2} and Lemma \ref{lem:renormalization_operator_fixed_time_3} in fact hold for the operators $e_{<k}^{\pm i \psi_{<l}}(t,x,D)$, $e_{k}^{\pm i \psi_{<l}}(t,x,D)$, and $e^{\pm i \psi_{<l}}(t,x,D)$ for any $k,l < 0$. The proofs in this and the previous subsection can be easily adapted to obtain this assertion.
\end{rem}

\subsection{Modulation localized estimates}

All implicit constants in this subsection may depend on the size of $\|\nabla_{t,x} A\|_{L^2_x}$.

\begin{prop} \label{prop:modulation_localized_estimate}
 For any $\varepsilon > 0$ the intervals $I_j$ can be chosen such that uniformly for all $j = 1, \ldots, J$ and all integers $k \leq k' \pm O(1) < 0$, it holds that
 \begin{equation}
  \bigl\| Q_k \bigl( \chi_{I_j} e_{k'}^{\pm i \psi_\pm}(t,x,D) P_0 \phi \bigr) \bigr\|_{L^2_t L^2_x(\R\times\R^4)} \lesssim \varepsilon 2^{-\frac{1}{2} k} 2^{\delta (k-k')} \|P_0 \phi\|_{N_0^\ast(\R\times\R^4)}. 
 \end{equation}
\end{prop}

 In the proof of Proposition \ref{prop:modulation_localized_estimate} we will use the following result whose proof will be given later.
\begin{lem} \label{lem:fixed_freq_in_symbol_space_time_estimates}
 Let $1 \leq q \leq p \leq \infty$. For any $\varepsilon > 0$ the intervals $I_j$ can be chosen such that uniformly for all $j = 1, \ldots, J$ and all integers $k + C \leq l \leq 0$, the following operator bound holds
 \begin{equation*}
  \bigl\| \chi_{I_j} (e_l^{\pm i \psi_{<k}})(t,x,D) \bigr\|_{L^p_t L^2_x(\R\times\R^4) \to L^q_t L^2_x(\R\times\R^4)} \lesssim \varepsilon 2^{4(k-l)} 2^{(\frac{1}{p}-\frac{1}{q})k}.
 \end{equation*}
\end{lem}

\begin{proof}[Proof of Proposition \ref{prop:modulation_localized_estimate}]
 In the following we denote an interval $I_j$ just by $I$ and $e_{k'}^{\pm i \psi_\pm}$ stands for the left quantization $e_{k'}^{\pm i \psi_\pm}(t,x,D)$.

 \medskip

 We first reduce to the case $k = k' \pm O(1)$. To this end we will use that Proposition 9 and Lemma 10 in \cite{KST} hold without the $\varepsilon$ smallness gain also for large energies. We split
 \begin{equation} \label{equ:modulation_localized_estimate_reduce_to_k_equal_k_prime}
  \begin{split}
   Q_k \bigl( \chi_I e_{k'}^{\pm i \psi_\pm}  P_0 \phi \bigr) &= Q_k \bigl( Q_{<k-C}(\chi_I) e_{k'}^{\pm i \psi_\pm}  P_0 \phi \bigr) + Q_k \bigl( Q_{[k-C,k+C]}(\chi_I) e_{k'}^{\pm i \psi_\pm}  P_0 \phi \bigr) \\
   &\quad + Q_k \bigl( Q_{>k+C}(\chi_I) e_{k'}^{\pm i \psi_\pm}  P_0 \phi \bigr).
  \end{split}
 \end{equation}
 For the first term we obtain
 \begin{align*}
  2^{\frac{1}{2} k} \bigl\| Q_k \bigl( Q_{<k-C}(\chi_I) e_{k'}^{\pm i \psi_\pm}  P_0 \phi \bigr) \bigr\|_{L^2_t L^2_x} &= 2^{\frac{1}{2} k} \bigl\| Q_k \bigl( Q_{<k-C}(\chi_I) Q_{k+O(1)} e_{k'}^{\pm i \psi_\pm}  P_0 \phi \bigr) \bigr\|_{L^2_t L^2_x} \\
  &\lesssim 2^{\frac{1}{2} k} \bigl\| Q_{k+O(1)} e_{k'}^{\pm i \psi_\pm}  P_0 \phi \bigr\|_{L^2_t L^2_x} \\
  &\lesssim 2^{\delta (k-k')} \|P_0 \phi\|_{N_0^\ast},
 \end{align*}
 where we used Proposition 9 from \cite{KST}. We estimate the second term from \eqref{equ:modulation_localized_estimate_reduce_to_k_equal_k_prime} by
 \begin{align*}
  2^{\frac{1}{2} k} \bigl\| Q_k \bigl( Q_{[k-C,k+C]}(\chi_I) e_{k'}^{\pm i \psi_\pm}  P_0 \phi \bigr) \bigr\|_{L^2_t L^2_x} &\lesssim 2^{\frac{1}{2} k} \bigl\| Q_{[k-C,k+C]}(\chi_I) \bigr\|_{L^{3}_t(\R)} \bigl\|e_{k'}^{\pm i \psi_\pm}  P_0 \phi \bigr\|_{L^6_t L^2_x} \\
  &\lesssim 2^{\frac{1}{2} k} 2^{-\frac{1}{3} k} \bigl\|e_{k'}^{\pm i \psi_\pm}  P_0 \phi \bigr\|_{L^6_t L^2_x}.
 \end{align*}
 Using continuous Littlewood-Paley resolutions to decompose the group element we have
 \[
  e_{k'}^{\pm i \psi_\pm} = e_{k'}^{\pm i \psi_{<k'-C}} \pm i \int_{l > k'-C} S_{k'} \bigl( \psi_{l} e^{\pm i \psi_{<l}} \bigr) \, dl.
 \]
 By Lemma 10 in \cite{KST} and the decomposable estimates \eqref{equ:psi_k_theta_decomposable_estimate} we find
 \begin{align*}
  &\bigl\|e_{k'}^{\pm i \psi_\pm}  P_0 \phi \bigr\|_{L^6_t L^2_x} \\
  &\lesssim \bigl\| e_{k'}^{\pm i \psi_{<k'-C}}  P_0 \phi \bigr\|_{L^6_t L^2_x} + \int_{l > k'-C} \bigl\| S_{k'} \bigl( \psi_l e^{\pm i \psi_{<l}} \bigr)  P_0 \phi \bigr\|_{L^6_t L^2_x} \, dl\\
  &\lesssim 2^{-\frac{1}{6} k'} \|P_0 \phi\|_{L^\infty_t L^2_x} + \int_{l > k'-C} \|\psi_l\|_{D(L^6_t L^\infty_x)} \|e^{\pm i \psi_{<l}}  P_0 \phi \|_{L^\infty_t L^2_x} \, dl \\
  &\lesssim 2^{-\frac{1}{6} k'} \|P_0 \phi\|_{L^\infty_t L^2_x} + \int_{l > k'-C} 2^{-\frac{1}{6} l} \|\nabla_{t,x} A\|_{L^2_x} \|P_0 \phi\|_{L^\infty_t L^2_x} \, dl \\
  &\lesssim 2^{-\frac{1}{6} k'} \|P_0 \phi\|_{N_0^\ast}.
 \end{align*}
 Thus, the second term from \eqref{equ:modulation_localized_estimate_reduce_to_k_equal_k_prime} is bounded by
 \[
  2^{\frac{1}{2} k} \bigl\| Q_k \bigl( Q_{[k-C,k+C]}(\chi_I) e_{k'}^{\pm i \psi_\pm}  P_0 \phi \bigr) \bigr\|_{L^2_t L^2_x} \lesssim 2^{\frac{1}{6} (k-k')} \|P_0 \phi\|_{N_0^\ast}.
 \]
 To estimate the third term from \eqref{equ:modulation_localized_estimate_reduce_to_k_equal_k_prime} we first use Bernstein in time to obtain
 \begin{align*}
  &2^{\frac{1}{2} k} \bigl\| Q_k \bigl( Q_{>k+C}(\chi_I) e_{k'}^{\pm i \psi_\pm}  P_0 \phi \bigr) \bigr\|_{L^2_t L^2_x} \\
  &\leq 2^{\frac{1}{2} k} \sum_{j \geq k+C} \bigl\| Q_k \bigl( Q_j(\chi_I) Q_{j + O(1)} (e_{k'}^{\pm i \psi_\pm}  P_0 \phi) \bigr) \bigr\|_{L^2_t L^2_x} \\
  &\lesssim 2^k \sum_{j \geq k+C} \|Q_j(\chi_I)\|_{L^2_t} \bigl\| Q_{j+O(1)} e_{k'}^{\pm i \psi_\pm}  P_0 \phi \bigr\|_{L^2_t L^2_x} \\
  &\lesssim 2^k \sum_{j = k+C}^{k'+C} 2^{-\frac{1}{2} j} \bigl\| Q_{j+O(1)} e_{k'}^{\pm i \psi_\pm}  P_0 \phi \bigr\|_{L^2_t L^2_x} + 2^k \sum_{j > k'+C} 2^{-\frac{1}{2} j} \bigl\| Q_{j+O(1)} e_{k'}^{\pm i \psi_\pm}  P_0 \phi \bigr\|_{L^2_t L^2_x}.
 \end{align*}
 For the first sum we use Proposition 9 from \cite{KST}, for the second sum we first note that there is no modulation interference since $k' < j-C$ and then use the fixed-time $L^2_x \to L^2_x$ estimate for $e_{k'}^{\pm i \psi_\pm}$. Hence,
 \begin{align*}
  &\lesssim 2^k \sum_{j = k+C}^{k'+C} 2^{-\frac{1}{2} j} 2^{-\frac{1}{2} j} 2^{\delta(j-k')} \|\nabla_{t,x} A\|_{L^2_x} \|P_0 \phi\|_{N_0^\ast} + 2^k \sum_{j > k'+C} 2^{-j} \|P_0 \phi\|_{X_\infty^{0, \frac{1}{2}}} \\
  &\lesssim ( 2^{\delta(j-k')} + 2^{k-k'} ) \|P_0 \phi\|_{N_0^\ast}.
 \end{align*}
 Putting things together we find that
 \[
  2^{\frac{1}{2} k} \bigl\| Q_k \bigl( \chi_I e_{k'}^{\pm i \psi_\pm}  P_0 \phi \bigr) \bigr\|_{L^2_t L^2_x} \lesssim 2^{\delta (k-k')} \|P_0 \phi\|_{N_0^\ast}
 \]
 and for sufficiently large $|k| \gg |k'|$ we therefore trivially gain a smallness factor $\varepsilon$ from $2^{\frac{1}{2} \delta (k-k')}$.

 \medskip

 We are thus reduced to the case $k = k' \pm O(1)$ and it remains to show that
 \begin{equation*}
  2^{\frac{1}{2} k} \bigl\| Q_k \bigl( \chi_I e_k^{\pm i \psi_\pm}  P_0 \phi \bigr) \bigr\|_{L^2_t L^2_x} \lesssim \varepsilon \|P_0 \phi\|_{N_0^\ast}.
 \end{equation*}
 As in the proof of Proposition 9 in \cite{KST} we expand the untruncated group element 
 \begin{equation*}
  \begin{split}
   e^{\pm i \psi} &= e^{\pm i \psi_{<k-C}} \pm i \int_{l > k-C} \psi_l e^{\pm i \psi_{<k-C}} \, dl - \iint_{l,l' > k-C} \psi_l \psi_{l'} e^{\pm i \psi_{<k-C}} \, dl' \, dl \\
   &\quad \quad \mp i \iiint_{l,l',l'' > k-C} \psi_l \psi_{l'} \psi_{l''} e^{\pm i \psi_{<l''}} \, dl'' \, dl' \, dl \\
   &= {\mathcal Z} + {\mathcal L} + {\mathcal Q} + {\mathcal C}
  \end{split}
 \end{equation*} 
 and estimate each of these terms seperately.

 \medskip

 \noindent {\it Zero order term $\mathcal{Z}$:} From Lemma \ref{lem:fixed_freq_in_symbol_space_time_estimates} we immediately obtain that
 \[
  \bigl\| Q_k \bigl( \chi_I e_k^{\pm i \psi_{<k-C}}  P_0 \phi \bigr) \bigr\|_{L^2_t L^2_x} \lesssim \varepsilon 2^{- \frac{1}{2} k} \|P_0 \phi\|_{N_0^\ast}.
 \]

 \medskip

 \noindent {\it Linear term $\mathcal{L}$:} We have to show that
 \[
  \biggl\| Q_k \Bigl( \chi_I \int_{l > k-C} S_k \bigl( \psi_l e^{\pm i \psi_{<k-C}} \bigr)  P_0 \phi \, dl \Bigr) \biggr\|_{L^2_t L^2_x} \lesssim \varepsilon 2^{-\frac{1}{2} k} \|P_0 \phi\|_{N_0^\ast}.
 \]
 To this end we decompose $\psi_l$ into a small and a large angular part
 \begin{equation} \label{equ:split_into_small_and_large_angular_part}
  \psi_l = \sum_{2^{\sigma l} < \theta < 2^{-C_3}} \psi_l^{(\theta)} + \sum_{2^{-C_3} \leq \theta \lesssim 1} \psi_l^{(\theta)}.
 \end{equation}
 In order to bound the small angular part we split
 \[
  \chi_I = Q_{\geq k-C}(\chi_I) + Q_{<k-C}(\chi_I).
 \]
 Using Lemma \ref{lem:decomposability_lemma_KST}, we estimate the first term by
 \begin{align*}
  &\Biggl\| Q_k \biggl( Q_{\geq k-C}(\chi_I) \int_{l > k-C} S_k \Bigl( \sum_{2^{\sigma l} < \theta < 2^{-C_3}} (\psi_l^{(\theta)}) e^{\pm i\psi_{<k-C}} \Bigr)  P_0 \phi \, dl \biggr) \Biggr\|_{L^2_t L^2_x} \\
  &\lesssim \bigl\|Q_{\geq k-C}(\chi_I)\bigr\|_{L^3_t} \int_{l > k-C} \sum_{2^{\sigma l} < \theta < 2^{-C_3}} \| \psi_l^{(\theta)} \|_{D_\theta(L^6_t L^\infty_x)} \bigl\| e^{\pm i \psi_{<k-C}}  P_0 \phi \bigr\|_{L^\infty_t L^2_x} \,dl \\
  &\lesssim 2^{-\frac{1}{3} k} \int_{l > k-C} \sum_{2^{\sigma l} < \theta < 2^{-C_3}} \theta^{\frac{1}{6}} 2^{-\frac{1}{6} l} \|\nabla_{t,x} A\|_{L^2_x} \|P_0 \phi\|_{L^\infty_t L^2_x} \, dl \\
  &\lesssim 2^{-\frac{1}{6} C_3} 2^{-\frac{1}{2} k} \|P_0 \phi\|_{N_0^\ast}.
 \end{align*}
 For the second term we have
 \begin{align*}
  &Q_k \biggl( Q_{<k-C}(\chi_I) \int_{l > k-C} S_k \Bigl( \sum_{2^{\sigma l} < \theta < 2^{-C_3}} \psi_l^{(\theta)} e^{\pm i \psi_{<k-C}} \Bigr)  P_0 \phi \, dl \biggr) \\
  &= Q_k \biggl( Q_{<k-C}(\chi_I) \, Q_{k+O(1)} \int_{l > k-C} S_k \Bigl( \sum_{2^{\sigma l} < \theta < 2^{-C_3}} \psi_l^{(\theta)} e^{\pm i \psi_{<k-C}} \Bigr)  P_0 \phi \, dl \biggr).
 \end{align*}
 Then since $\psi_{l}^{(\theta)}$ is a free wave, we can write this as 
 \[
  Q_k \biggl( Q_{<k-C}(\chi_I) \, Q_{k+O(1)} \int_{l > k-C} S_k \Bigl( \sum_{2^{\sigma l} < \theta < 2^{-C_3}} \psi_l^{(\theta)} e^{\pm i \psi_{<k-C}} \Bigr) Q_{k + O(1)} P_0 \phi \, dl \biggr)
 \]
 and estimate by
 \begin{align*}
  &\Biggl\| Q_k \biggl( Q_{<k-C}(\chi_I) \, Q_{k+O(1)} \int_{l > k-C} S_k \Bigl( \sum_{2^{\sigma l} < \theta < 2^{-C_3}} \psi_l^{(\theta)} e^{\pm i \psi_{<k-C}} \Bigr) Q_{k + O(1)} P_0 \phi \, dl \biggr) \Biggr\|_{L^2_t L^2_x} \\
  &\lesssim \int_{l > k-C} \sum_{2^{\sigma l} < \theta < 2^{-C_3}} \bigl\| \psi_l^{(\theta)} \bigr\|_{D_\theta(L^6_t L^\infty_x)} \bigl\| e^{\pm i \psi_{<k-C}} Q_{k+O(1)} P_0 \phi \bigr\|_{L^3_t L^2_x} \, dl \\
  &\lesssim \int_{l > k-C} \sum_{2^{\sigma l} < \theta < 2^{-C_3}} 2^{-\frac{1}{6} l} 2^{-\frac{1}{6} \theta} \|\nabla_{t,x} A\|_{L^2_x} \bigl\| Q_{k+O(1)} P_0 \phi \bigr\|_{L^3_t L^2_x} \, dl \\
  &\lesssim 2^{-\frac{1}{6} C_3} 2^{-\frac{1}{6} k} 2^{\frac{1}{6} k} \bigl\| Q_{k+O(1)} P_0 \phi \bigr\|_{L^2_t L^2_x} \\
  &\lesssim 2^{-\frac{1}{6} C_3} 2^{-\frac{1}{2} k} \|P_0 \phi \|_{N_0^\ast}.
 \end{align*}
 Here we used Lemma \ref{lem:decomposability_lemma_KST}, the fixed-time $L^2_x \to L^2_x$ estimate for $e^{\pm i T_z \psi_{<k-C}}$ and then Bernstein in time.

 The large angular part in \eqref{equ:split_into_small_and_large_angular_part} has to be estimated more carefully. Noting that the symbol localization $S_k \bigl( \psi_l e^{\pm i \psi_{<k-C}} \bigr)$ can be represented as
 \begin{equation}
  S_k \bigl( \psi_l e^{\pm i \psi_{<k-C}} \bigr) = \int_{\R^{1+4}_z} m_k(z) (T_z \psi_l) e^{\pm i T_z \psi_{<k-C}} \, dz,
 \end{equation}
 where $m_k$ is an integrable bump function at scale $2^{-k}$ and $T_z$ denotes translation in space-time direction $z \in \R^{1+4}$, we derive the following key estimate for the large angular part
 \begin{align*}
  &\Biggl\| Q_k \biggl( \chi_I \int_{l > k-C} \int_{\R^{1+4}_z} m_k(z) \Bigl( \sum_{2^{-C_3} \leq \theta \lesssim 1} (T_z \psi_l^{(\theta)}) \Bigr)  e^{\pm i T_z \psi_{<k-C}}  P_0 \phi \, dz \, dl \biggr) \Biggr\|_{L^2_t L^2_x} \\
  &\lesssim C_3^{1/2} 2^{-\frac{1}{2} k} \Biggl\| \biggl( \int_{l > k-C} 2^{-\frac{1}{2} (l-k)} \int_{\R^{1+4}_z} |m_k(z)| \sum_{2^{-C_3} \leq \theta \lesssim 1} \sum_{\Gamma_\theta^\nu} \sup_{\omega \in \Gamma_\theta^\nu} \Bigl( 2^{-\frac{5}{6} l} \bigl\|b_\theta^\nu(\omega) \Pi_\theta^\omega \nabla_{t,x} T_z A_l \bigr\|_{L^6_x}  \Bigr)^2 \, dz \, dl \biggl)^{\frac{1}{2}} \Biggr\|_{L^2_t(I)} \times \\
  &\qquad \qquad \qquad \times \biggl( \int_{l>k-C} 2^{-\frac{1}{2} (l-k)} \int_{\R^{1+4}_z} |m_k(z)| \, \bigl\| e^{\pm i T_z \psi_{<k-C}}  P_0 \phi \bigr\|_{L^\infty_t L^2_x}^2 \, dz \, dl \biggr)^{\frac{1}{2}} \\
  &\lesssim C_3^{1/2} 2^{-\frac{1}{2} k} \times  \\
  &\quad \times \Biggl\| \biggl( \int_{l > k-C} 2^{-\frac{1}{2} (l-k)} \int_{\R^{1+4}_z} |m_k(z)| \sum_{2^{-C_3} \leq \theta \lesssim 1} \sum_{\Gamma_\theta^\nu} \sup_{\omega \in \Gamma_\theta^\nu} \Bigl( 2^{-\frac{5}{6} l} \bigl\|b_\theta^\nu(\omega) \Pi_\theta^\omega \nabla_{t,x} T_z A_l \bigr\|_{L^6_x}  \Bigr)^2 \, dz \, dl \biggl)^{\frac{1}{2}} \Biggr\|_{L^2_t(I)} \|P_0 \phi\|_{N_0^\ast}.
 \end{align*}
 This estimate can be proven by carefully opening up the proof of the decomposable estimates in Lemma \ref{lem:decomposability_lemma_KST}. We emphasize that \emph{uniformly} for all integers $k < 0$, the quantity
 \begin{align*}
  &\Biggl\| \biggl( \int_{l > k-C} 2^{-\frac{1}{2} (l-k)} \int_{\R^{1+4}_z} |m_k(z)| \sum_{2^{-C_3} \leq \theta \lesssim 1} \sum_{\Gamma_\theta^\nu} \sup_{\omega \in \Gamma_\theta^\nu} \Bigl( 2^{-\frac{5}{6} l} \bigl\|b_\theta^\nu(\omega) \Pi_\theta^\omega \nabla_{t,x} T_z A_l \bigr\|_{L^6_x}  \Bigr)^2 \, dz \, dl \biggl)^{\frac{1}{2}} \Biggr\|_{L^2_t(\R)} \\
  &\leq \Biggl\| \biggl( \int_{\R} \sum_{k < l+C} 2^{-\frac{1}{2} (l-k)} \int_{\R^{1+4}_z} |m_k(z)| \sum_{2^{-C_3} \leq \theta \lesssim 1} \sum_{\Gamma_\theta^\nu} \sup_{\omega \in \Gamma_\theta^\nu} \Bigl( 2^{-\frac{5}{6} l} \bigl\|b_\theta^\nu(\omega) \Pi_\theta^\omega \nabla_{t,x} T_z A_l \bigr\|_{L^6_x}  \Bigr)^2 \, dz \, dl \biggl)^{\frac{1}{2}} \Biggr\|_{L^2_t(\R)} 
 \end{align*}
 is bounded by $\|\nabla_{t,x} A\|_{L^2_x}$ by Strichartz estimates.

 By first fixing $C_3 > 0$ sufficiently large and then suitably choosing the intervals $I_j$, the estimate of the linear term $\mathcal{L}$ follows.

 \medskip

 \noindent {\it Quadratic and cubic terms $\mathcal{Q}$ and $\mathcal{C}$:} Using the above ideas these can be estimated similarly. We omit the details.
\end{proof}

\begin{proof}[Proof of Lemma \ref{lem:fixed_freq_in_symbol_space_time_estimates}]
 As in \cite[Lemma 10]{KST} we write the symbol as
 \[
  S_l e^{\pm i \psi_{<k}} = (\pm i)^5 2^{-5l} \prod_{r=1}^5 [S_l^{(r)} \partial_t \psi_{<k}] e^{\pm i \psi_{<k}},
 \]
 where the product denotes a nested multiplication by $S_l \partial_t \psi_{<k}$ for a series of frequency cutoffs $S_l^{(r+1)} S_l^{(r)} = S_l^{(r)} \approx S_l$ with expanding widths. Then we have
 \[
  S_l^{(r)} \partial_t \psi_{<k} = \int_{\R^{1+4}_{z_r}} m_l^{(r)}(z_r) (T_{z_r} \partial_t \psi_{<k}) \, dz_r,
 \]
 where $m_l^{(r)}$ is an integrable bump function at scale $2^{-l}$ and $T_{z_r}$ denotes translation in space-time direction $z_r \in \R^{1+4}$. The claim now reduces to proving that the intervals $I_j$ can be chosen such that uniformly for $j = 1, \ldots, J$ and all integers $k \leq l - C$, it holds that
 \begin{equation} \label{equ:fixed_freq_in_symbol_space_time_estimate_main}
  \Biggl\| \chi_{I_j} \biggl( \prod_{r=1}^5 \int_{\R^{1+4}_{z_r}} m_l^{(r)}(z_r) (T_{z_r} \partial_t \psi_{<k}) \biggr) e^{\pm i \psi_{<k}}(t,x,D) \Biggr\|_{L^p_t L^2_x \to L^q_t L^2_x} \lesssim \varepsilon 2^{4k} 2^l 2^{(\frac{1}{p}-\frac{1}{q})k}.
 \end{equation}
 To this end we show that the intervals $I_j$ can be chosen such that uniformly for $j = 1, \ldots, J$, all integers $k \leq l - C$ and all integers $k_1, \ldots, k_5 < k$, we have the operator bound
 \begin{equation} \label{equ:fixed_freq_in_symbol_space_time_estimate_reduced}
  \begin{split}
   &\Biggl\| \chi_{I_j} \biggl( \prod_{r=1}^5 \int_{\R^{1+4}_{z_r}} m_l^{(r)}(z_r) (T_{z_r} \partial_t \psi_{k_r}) \biggr) e^{\pm i \psi_{<k}}(t,x,D) \Biggr\|_{L^p_t L^2_x \to L^q_t L^2_x} \lesssim \varepsilon 2^l 2^{-\frac{1}{5\tilde{q}} k_1} 2^{(1-\frac{1}{5 \tilde{q}}) k_2} \cdots 2^{(1-\frac{1}{5\tilde{q}}) k_5},
  \end{split}
 \end{equation}
 where $\frac{1}{\tilde{q}} = \frac{1}{q}-\frac{1}{p}$. By summing over dyadic frequencies, the estimate \eqref{equ:fixed_freq_in_symbol_space_time_estimate_main} then follows.

 \medskip

 In order to prove \eqref{equ:fixed_freq_in_symbol_space_time_estimate_reduced}, we split $\partial_t \psi_{k_1}$ into a small and a large angular part
 \[
  \partial_t \psi_{k_1} = \sum_{2^{\sigma k_1} < \theta_1 < 2^{-C_3}} \partial_t \psi_{k_1}^{(\theta_1)} + \sum_{2^{-C_3} \leq \theta_1 \lesssim 1} \partial_t \psi_{k_1}^{(\theta_1)}
 \]
 for some constant $C_3 > 0$ to be chosen sufficiently large later in the proof. We estimate the small angular part using H\"older-type estimates for decomposable function spaces \eqref{equ:hoelder_type_estimate_decomposable_norms} and the bounds \eqref{equ:psi_k_theta_decomposable_estimate} for the phase,
 \begin{align*}
  & \Biggl\| \chi_{I_j} \biggl( \sum_{2^{\sigma k_1} < \theta_1 < 2^{-C_3}} \int_{\R^{1+4}_{z_1}} m_l^{(1)}(z_1) (T_{z_1} \partial_t \psi_{k_1}^{(\theta_1)}) \, dz_1 \biggr) \biggl( \prod_{r=2}^5 \int_{\R^{1+4}_{z_r}} m_l^{(r)}(z_r) (T_{z_r} \partial_t \psi_{k_r}) \, dz_r \biggr) e^{\pm i \psi_{<k}}(t,x,D) \phi \Biggr\|_{L^q_t L^2_x} \\
  &\lesssim \biggl( \sum_{2^{\sigma k_1} < \theta_1 < 2^{-C_3}} \int_{\R^{1+4}_{z_1}} |m_l^{(1)}(z_1)| \bigl\|T_{z_1} \partial_t \psi_{k_1}^{(\theta_1)} \bigr\|_{D_{\theta_1}(L^{5 \tilde{q}}_t L^\infty_x)} \, dz_1 \biggr) \biggl( \prod_{r=2}^5 \int_{\R_{z_r}^{1+4}} |m_l^{(r)}(z_r)| \bigl\| T_{z_r} \partial_t \psi_{k_r} \bigr\|_{D(L^{5 \tilde{q}}_t L^\infty_x)} \, dz_r \biggr) \|\phi\|_{L^p_t L^2_x} \\
  &\lesssim \biggl( \sum_{2^{\sigma k_1} < \theta_1 < 2^{-C_3}} 2^{(1-\frac{1}{5 \tilde{q}}) k_1} \theta_1^{\frac{1}{2} - \frac{2}{5\tilde{q}}} \|\nabla_{t,x} A\|_{L^2_x} \biggr) \biggl( \prod_{r=2}^5 2^{(1-\frac{1}{5\tilde{q}}) k_r} \|\nabla_{t,x} A\|_{L^2_x} \biggr) \|\phi\|_{L^p_t L^2_x} \\
  &\lesssim 2^{-C_3 (\frac{1}{2} - \frac{2}{5 \tilde{q}})} 2^l 2^{-\frac{1}{5\tilde{q}} k_1} 2^{(1-\frac{1}{5 \tilde{q}}) k_2} \cdots 2^{(1-\frac{1}{5 \tilde{q}}) k_5} \|\nabla_{t,x} A\|_{L^2_x}^5 \|\phi\|_{L^p_t L^2_x}.
 \end{align*}
 Here we dropped the time cutoff $\chi_{I_j}$ and used the space-time translation invariance of the decomposable function spaces. For the large angular part we establish the crucial estimate
 \begin{align*}
  &\Biggl\| \chi_{I_j} \biggl( \sum_{2^{-C_3} \leq \theta_1 \lesssim 1} \int_{\R_{z_1}^{1+4}} m_l^{(1)}(z_1) (T_{z_1} \partial_t \psi_{k_1}^{(\theta_1)}) \, dz_1 \biggr) \biggl( \prod_{r=2}^5 \int_{\R^{1+4}_{z_r}} m_l^{(r)}(z_r) (T_{z_r} \partial_t \psi_{k_r}) \, dz_r \biggr) e^{\pm i \psi_{<k}}(t,x,D) \phi \Biggr\|_{L^q_t L^2_x} \\
  &\lesssim 2^l 2^{-\frac{1}{5 \tilde{q}} k_1} 2^{(1-\frac{1}{5 \tilde{q}}) k_2} \cdots 2^{(1-\frac{1}{5 \tilde{q}}) k_5} \|\nabla_{t,x} A\|_{L^2_x}^4 C_3^{1/2} \times \\
  &\quad \times \Biggl\| \biggl( 2^{k_1 - l} \int_{\R^{1+4}_{z_1}} |m_l^{(1)}(z_1)| \sum_{2^{-C_3} \leq \theta_1 \lesssim 1} \sum_{\Gamma_{\theta_1}^{\nu_1}} \sup_{\omega \in \Gamma_{\theta_1}^{\nu_1}} \Bigl( 2^{(\frac{1}{5\tilde{q}} + \frac{4}{r_0}-2) k_1} \|b_{\theta_1}^{\nu_1}(\omega) \Pi_{\theta_1}^\omega \nabla_{t,x} T_{z_1} A_{k_1} \|_{L^{r_0}_x} \Bigr)^2 \, dz_1 \biggr)^{1/2} \Biggr\|_{L^{5 \tilde{q}}_t(I_j)} \|\phi\|_{L^p_t L^2_x},
 \end{align*}
 where $r_0 \geq 2$ is such that the exponent pair $(5 \tilde{q}, r_0)$ is sharp wave admissible. This estimate can be proven by carefully opening up the proof of Lemma \ref{lem:decomposability_lemma_KST} and of H\"older-type estimates for decomposable function spaces \eqref{equ:hoelder_type_estimate_decomposable_norms}.

 \medskip

 Noting that \emph{uniformly} for all integers $k_1 \leq l - C$, the quantity
 \begin{align*}
  &\Biggl\| \biggl( 2^{k_1 - l} \int_{\R^{1+4}_{z_1}} |m_l^{(1)}(z_1)| \sum_{2^{-C_3} \leq \theta_1 \lesssim 1} \sum_{\Gamma_{\theta_1}^{\nu_1}} \sup_{\omega \in \Gamma_{\theta_1}^{\nu_1}} \Bigl( 2^{(\frac{1}{5\tilde{q}} + \frac{4}{r_0}-2) k_1} \|b_{\theta_1}^{\nu_1}(\omega) \Pi_{\theta_1}^\omega \nabla_{t,x} T_{z_1} A_{k_1} \|_{L^{r_0}_x} \Bigr)^2 \, dz_1 \biggr)^{1/2} \Biggr\|_{L^{5 \tilde{q}}_t(\R)} \\
  &\leq \Biggl\| \biggl( \sum_{l \in \Z} \sum_{k_1 \leq l-C} 2^{k_1 - l} \int_{\R^{1+4}_{z_1}} |m_l^{(1)}(z_1)| \sum_{2^{-C_3} \leq \theta_1 \lesssim 1} \sum_{\Gamma_{\theta_1}^{\nu_1}} \sup_{\omega \in \Gamma_{\theta_1}^{\nu_1}} \Bigl( 2^{(\frac{1}{5\tilde{q}} + \frac{4}{r_0}-2) k_1} \|b_{\theta_1}^{\nu_1}(\omega) \Pi_{\theta_1}^\omega \nabla_{t,x} T_{z_1} A_{k_1} \|_{L^{r_0}_x} \Bigr)^2 \, dz_1 \biggr)^{1/2} \Biggr\|_{L^{5 \tilde{q}}_t(\R)}
 \end{align*}
 is bounded by $\|\nabla_{t,x} A\|_{L^2_x}$ by Strichartz estimates, the assertion follows by first choosing $C_3 > 0$ sufficiently large and then suitably choosing the intervals $I_j$.
\end{proof}

\begin{prop} \label{prop:modulation_localized_estimate_right_quantization}
 For any $\varepsilon > 0$ the intervals $I_j$ can be chosen such that uniformly for all $j = 1, \ldots, J$ and all integers $k \leq k' \pm O(1) < 0$, it holds that
 \begin{equation}
  \bigl\| Q_k \bigl( \chi_{I_j} e_{k'}^{-i \psi_\pm}(D,y,t)  P_0 \phi \bigr) \bigr\|_{L^2_t L^2_x(\R\times\R^4)} \lesssim \varepsilon 2^{-\frac{1}{2} k} 2^{\delta (k-k')} \|P_0 \phi\|_{N_0^\ast(\R\times\R^4)}. 
 \end{equation}
\end{prop}
\begin{proof}
 The proof proceeds analogously to the one of Proposition \ref{prop:modulation_localized_estimate} using Lemma \ref{lem:decomposability_lemma_right_quantizations} in place of Lemma \ref{lem:decomposability_lemma_KST}.
\end{proof}

\subsection{Proof of the $N_0 \to N_0$ and $N_0^\ast \to N_0^\ast$ bounds \eqref{equ:Z_to_Z_bound_for_e} for $\chi_{I_j} e_{<0}^{\pm i \psi_\pm}$}

\begin{prop} \label{prop:N_0_to_N_0_for_e}
 For $j = 1, \ldots, J$ it holds that
 \begin{equation} \label{equ:N_0_to_N_0_for_left_quantized_e}
  \bigl\| \chi_{I_j} e_{<0}^{\pm i \psi_\pm}(t,x,D) P_0 \phi \bigr\|_{N_0(\R\times\R^4)} \lesssim \bigl\| P_0 \phi \bigr\|_{N_0(\R\times\R^4)}
 \end{equation}
 and
 \begin{equation} \label{equ:N_0_to_N_0_for_right_quantized_e}
  \bigl\| \chi_{I_j} e_{<0}^{\pm i \psi_\pm}(D,y,t) P_0 \phi \bigr\|_{N_0(\R\times\R^4)} \lesssim \bigl\| P_0 \phi \bigr\|_{N_0(\R\times\R^4)}.
 \end{equation}
\end{prop}
\begin{proof}
 We begin with the proof of \eqref{equ:N_0_to_N_0_for_left_quantized_e}. To simplify the notation we denote an interval $I_j$ just by $I$ in what follows. If $\phi$ is an $L^1_t L^2_x$ atom, the claim follows immediately from the fixed-time $L^2_x \to L^2_x$ estimate for $e_{<0}^{\pm i \psi_\pm}(t,x,D)$. The key point is therefore to show that if $\phi$ is an $X_1^{0, -\frac{1}{2}}$ atom at modulation $k$, then we have
 \begin{equation*}
  \bigl\| \chi_I e_{<0}^{\pm i \psi_\pm}(t,x,D) Q_k P_0 \phi \bigr\|_{N_0} \lesssim 2^{-\frac{1}{2} k} \|P_0 \phi\|_{L^2_t L^2_x}.
 \end{equation*}
 By duality, this is equivalent to proving
 \begin{equation*}
  \bigl\| Q_k \bigl( \chi_I e_{<0}^{\pm i \psi_\pm}(D,y,t) P_0 \phi \bigr) \bigr\|_{L^2_t L^2_x} \lesssim 2^{-\frac{1}{2} k} \|P_0 \phi\|_{N_0^\ast}.
 \end{equation*}
 As in \cite[Proposition 9.1]{KST} we now write
 \begin{align*}
  Q_k \bigl( \chi_I e_{<0}^{\pm i \psi_\pm}(D,y,t) P_0 \phi \bigr) &= Q_k \bigl( \chi_I e_{<k-C}^{\pm i \psi_\pm}(D,y,t) P_0 \phi \bigr) + Q_k \bigl( \chi_I (e_{<0}^{\pm i \psi_\pm} - e_{<k-C}^{\pm i \psi_\pm})(D,y,t) P_0 \phi \bigr) \\
  &= Q_k \bigl( Q_{<k-C}(\chi_I) e_{<k-C}^{\pm i \psi_\pm}(D,y,t) P_0 \phi \bigr) + Q_k \bigl( Q_{\geq k-C}(\chi_I) e_{<k-C}^{\pm i \psi_\pm}(D,y,t) P_0 \phi \bigr) \\
  &\quad \quad + Q_k \bigl( \chi_I (e_{<0}^{\pm i \psi_\pm} - e_{<k-C}^{\pm i \psi_\pm})(D,y,t) P_0 \phi \bigr). 
 \end{align*}
 For the first term we obtain 
 \[
  2^{\frac{1}{2} k} \bigl\| Q_k \bigl( Q_{<k-C}(\chi_I) e_{<k-C}^{\pm i \psi_\pm}(D,y,t) P_0 \phi \bigr) \bigr\|_{L^2_t L^2_x} \lesssim \| P_0 \phi \|_{X_\infty^{0, \frac{1}{2}}} \lesssim \|P_0 \phi \|_{N_0^\ast},
 \]
 because the output modulation directly transfers to $\phi$. We estimate the second term by
 \begin{align*}
  2^{\frac{1}{2} k} \bigl\| Q_k \bigl( Q_{\geq k-C}(\chi_I) e_{<k-C}^{\pm i \psi_\pm}(D,y,t) P_0 \phi \bigr) \bigr\|_{L^2_t L^2_x} &\lesssim 2^{\frac{1}{2} k} \bigl\| Q_{\geq k - C}(\chi_I) \bigr\|_{L^2_t} \bigl\| e_{<k-C}^{\pm i \psi_\pm}(D,y,t) P_0 \phi \bigr\|_{L^\infty_t L^2_x} \\
  &\lesssim \|P_0 \phi\|_{N_0^\ast},
 \end{align*}
 where we used that $\bigl\| Q_{\geq k - C}(\chi_I) \bigr\|_{L^2_t} \lesssim 2^{-\frac{1}{2} k}$. To deal with the last term we use Proposition \ref{prop:modulation_localized_estimate_right_quantization}. 
 
 The proof of \eqref{equ:N_0_to_N_0_for_right_quantized_e} works similarly using Proposition \ref{prop:modulation_localized_estimate}.
\end{proof}

In a similar vein we obtain the following $N_0^\ast \to N_0^\ast$ bounds.
\begin{prop} \label{prop:N_0_ast_to_N_0_ast_for_e}
 For $j = 1, \ldots, J$ it holds that
 \begin{equation} \label{equ:N_0_ast_to_N_0_ast_for_left_quantized_e}
  \bigl\| \chi_{I_j} e_{<0}^{\pm i \psi_\pm}(t,x,D) P_0 \phi \bigr\|_{N_0^\ast(\R\times\R^4)} \lesssim \bigl\| P_0 \phi \bigr\|_{N_0^\ast(\R\times\R^4)}
 \end{equation}
 and
 \begin{equation} \label{equ:N_0_ast_to_N_0_ast_for_right_quantized_e}
  \bigl\| \chi_{I_j} e_{<0}^{\pm i \psi_\pm}(D,y,t) P_0 \phi \bigr\|_{N_0^\ast(\R\times\R^4)} \lesssim \bigl\| P_0 \phi \bigr\|_{N_0^\ast(\R\times\R^4)}.
 \end{equation}
\end{prop}

\subsection{Proof of the $N_0 \to \varepsilon N_0$ and $N_0^\ast \to \varepsilon N_0^\ast$ bounds \eqref{equ:Z_to_Z_bound_for_time_derivative_e} for $\chi_{I_j} \partial_t e_{<0}^{\pm i \psi_\pm}$}

\begin{prop}
 For any $\varepsilon > 0$ the intervals $I_j$ can be chosen such that uniformly for all $j = 1, \ldots, J$ it holds that
 \begin{equation}
  \bigl\| \chi_{I_j} \partial_t e_{<0}^{\pm i \psi_\pm}(t,x,D) P_0 \phi \bigr\|_{N_0(\R\times\R^4)} \lesssim \varepsilon \|P_0 \phi\|_{N_0(\R\times\R^4)}
 \end{equation}
 and 
 \begin{equation}
  \bigl\| \chi_{I_j} \partial_t e_{<0}^{\pm i \psi_\pm}(D, y,t) P_0 \phi \bigr\|_{N_0(\R\times\R^4)} \lesssim \varepsilon \|P_0 \phi\|_{N_0(\R\times\R^4)}.
 \end{equation}
\end{prop}
\begin{proof}
 We proceed as in the proof of Proposition \ref{prop:N_0_to_N_0_for_e} using the $L^2_x \to \varepsilon L^2_x$ bound for $\partial_t e_{<0}^{\pm i \psi_\pm}$ and that we have for $k \leq k' \pm O(1)$,
 \[
  \bigl\| Q_k \bigl( \chi_{I_j} \partial_t e_{k'}^{-i \psi_\pm} P_0 \phi \bigr) \bigr\|_{L^2_t L^2_x(\R\times\R^4)} \lesssim \varepsilon 2^{-\frac{1}{2} k} 2^{\delta (k-k')} \|P_0 \phi\|_{N_0^\ast(\R\times\R^4)}
 \]
 for both left and right quantization. The latter estimate can be proven similarly to the proof of Proposition \ref{prop:modulation_localized_estimate}.
\end{proof}

\begin{prop}
 For any $\varepsilon > 0$ the intervals $I_j$ can be chosen such that uniformly for all $j = 1, \ldots, J$ it holds that
 \begin{equation}
  \bigl\| \chi_{I_j} \partial_t e_{<0}^{\pm i \psi_\pm}(t,x,D) P_0 \phi \bigr\|_{N_0^\ast(\R\times\R^4)} \lesssim \varepsilon \|P_0 \phi\|_{N_0^\ast(\R\times\R^4)}
 \end{equation}
 and 
 \begin{equation}
  \bigl\| \chi_{I_j} \partial_t e_{<0}^{\pm i \psi_\pm}(D, y,t) P_0 \phi \bigr\|_{N_0^\ast(\R\times\R^4)} \lesssim \varepsilon \|P_0 \phi\|_{N_0^\ast(\R\times\R^4)}.
 \end{equation}
\end{prop}

\subsection{Proof of the renormalization error estimate \eqref{equ:renormalization_error_estimate_1}}

\begin{prop}
 For any $\varepsilon > 0$ the intervals $I_j$ can be chosen such that uniformly for all $j = 1, \ldots, J$ we have
 \begin{equation} \label{equ:N_0_to_N_0_renormalization_error_estimate_1}
  \bigl\| \chi_{I_j} \bigl( e_{<0}^{-i \psi_\pm}(t,x,D) e_{<0}^{+i \psi_\pm}(D,y,t) - 1 \bigr) P_0 \phi \bigr\|_{N_0(\R\times\R^4)} \lesssim \varepsilon \|P_0 \phi\|_{N_0(\R\times\R^4)}
 \end{equation}
 and 
 \begin{equation} \label{equ:N_0_ast_to_N_0_ast_renormalization_error_estimate_1}
  \bigl\| \chi_{I_j} \bigl( e_{<0}^{-i \psi_\pm}(t,x,D) e_{<0}^{+i \psi_\pm}(D,y,t) - 1 \bigr) P_0 \phi \bigr\|_{N_0^\ast(\R\times\R^4)} \lesssim \varepsilon \|P_0 \phi\|_{N_0^\ast(\R\times\R^4)}.
 \end{equation}
\end{prop}
\begin{proof}
 We prove the $N_0^\ast \to \varepsilon N_0^\ast$ estimate \eqref{equ:N_0_ast_to_N_0_ast_renormalization_error_estimate_1}. The bound $\eqref{equ:N_0_to_N_0_renormalization_error_estimate_1}$ then follows by duality. The $L^\infty_t L^2_x$ part of \eqref{equ:N_0_ast_to_N_0_ast_renormalization_error_estimate_1} follows immediately from the fixed-time $L^2_x \to \varepsilon L^2_x$ estimate \eqref{equ:fixed_time_L_2_bound_renormalization_error_1}. The $X_\infty^{0, \frac{1}{2}}$ part reduces to showing that we can choose the intervals $I_j$ such that uniformly for $j = 1, \ldots, J$ and all $k \in \Z$,
 \[
  2^{\frac{1}{2} k} \bigr\| Q_k \bigl( \chi_{I_j} \bigl( e_{<0}^{-i\psi_\pm}(t,x,D) e_{<0}^{i\psi_\pm}(D,y,t) - 1 \bigr) P_0 \phi \bigr) \bigr\|_{L^2_t L^2_x(\R \times \R^4)} \lesssim \varepsilon \|P_0 \phi\|_{N_0^\ast(\R\times\R^4)}.
 \]
 We use the notation
 \[
  R_{<k} = e_{<k}^{-i \psi_\pm}(t,x,D) e_{<k}^{i\psi_\pm}(D,y,t)
 \]
 to write
 \begin{equation} \label{equ:expansion_in_proof_of_renormalization_error_estimate_1}
  \begin{split}
   Q_k \bigl( \chi_{I_j} (R_{<0} - 1) P_0 \phi \bigr) &= Q_k \bigl( \chi_{I_j} (R_{<0} - 1) Q_{>k-C} P_0 \phi \bigr) + Q_k \bigl( \chi_{I_j} (R_{<k-C} - 1) Q_{\leq k-C} P_0 \phi \bigr) \\
   &\quad + Q_k \bigl( \chi_{I_j} (R_{<0} - R_{<k-C}) Q_{\leq k-C} P_0 \phi \bigr).
  \end{split}
 \end{equation}
 Using the fixed-time $L^2_x \to \varepsilon L^2_x$ estimate \eqref{equ:fixed_time_L_2_bound_renormalization_error_1} for $(R_{<0} - 1)$, we bound the first term in \eqref{equ:expansion_in_proof_of_renormalization_error_estimate_1} by
 \begin{equation*}
  \begin{split}
   2^{\frac{1}{2} k} \bigl\| Q_k \bigl( \chi_{I_j} (R_{<0} - 1) Q_{> k-C} P_0 \phi \bigr) \bigr\|_{L^2_t L^2_x} &\lesssim 2^{\frac{1}{2} k} \bigl\| (R_{<0} - 1) Q_{> k-C} P_0 \phi \bigr\|_{L^2_t L^2_x} \\
   &\lesssim 2^{\frac{1}{2} k} \varepsilon \| Q_{> k-C} P_0 \phi \|_{L^2_t L^2_x} \\
   &\lesssim \varepsilon \| P_0 \phi \|_{X^{0, \frac{1}{2}}_\infty}.
  \end{split}
 \end{equation*}
 To estimate the second term in \eqref{equ:expansion_in_proof_of_renormalization_error_estimate_1} we observe that we have
 \[
  Q_k \bigl( \chi_{I_j} (R_{<k-C} - 1) Q_{\leq k-C} P_0 \phi \bigr) = Q_k \bigl( (Q_{[k-C, k+C]} \chi_{I_j}) (R_{<k-C} - 1) Q_{\leq k-C} P_0 \phi \bigr)
 \]
 and hence by the fixed-time $L^2_x \to \varepsilon L^2_x$ estimate for $(R_{<k-C} -1)$,
 \begin{equation*}
  \begin{split}
   2^{\frac{1}{2} k} \bigl\| Q_k \bigl( \chi_{I_j} (R_{<k-C} - 1) Q_{\leq k-C} P_0 \phi \bigr) \bigr\|_{L^2_t L^2_x} &\lesssim 2^{\frac{1}{2} k} \bigl\| Q_{[k-C,k+C]}(\chi_{I_j}) \bigr\|_{L^2_t} \bigl\| (R_{<k-C}-1) Q_{\leq k-C} P_0\phi \bigr\|_{L^\infty_t L^2_x} \\
   &\lesssim \varepsilon \|P_0 \phi\|_{L^\infty_t L^2_x}.  
  \end{split}
 \end{equation*}
 Finally, we expand the third term in \eqref{equ:expansion_in_proof_of_renormalization_error_estimate_1} as follows
 \begin{equation*}
  \begin{split}
   Q_k \bigl( \chi_{I_j} (R_{<0} - R_{<k-C}) Q_{\leq k-C} P_0 \bigr) &= Q_k \Bigl( \chi_{I_j} \bigl(e_{<0}^{-i\psi_\pm}(t,x,D) - e_{<k-C}^{-i\psi_\pm}(t,x,D) \bigr) e_{<0}^{i \psi_\pm}(D,y,t) Q_{\leq k-C} P_0 \phi \Bigr) \\
   &\quad + Q_k \Bigl( \chi_{I_j} e_{<k-C}^{-i\psi_\pm}(t,x,D) \bigl( e_{<0}^{i\psi_\pm}(D,y,t) - e_{<k-C}^{i\psi_\pm}(D,y,t) \bigr) Q_{\leq k-C} P_0 \phi \Bigr).
  \end{split}
 \end{equation*}
 To handle the first term in the above expansion we use Proposition \ref{prop:modulation_localized_estimate} and the $N_0^\ast \to N_0^\ast$ estimate \eqref{equ:N_0_ast_to_N_0_ast_for_right_quantized_e} for $e_{<0}^{i \psi_\pm}(D,y,t)$ to find that
 \begin{align*}
  &2^{\frac{1}{2} k} \bigl\| Q_k \bigl( \chi_{I_j} \bigl( e_{<0}^{-i\psi_\pm}(t,x,D) - e_{<k-C}^{-i\psi_\pm}(t,x,D) \bigr) e_{<0}^{i \psi_\pm}(D,y,t) Q_{\leq k-C} P_0 \phi \bigr\|_{L^2_t L^2_x} \\
  &\lesssim \sum_{k' = k-C}^{k' = 0} 2^{\frac{1}{2} k} \bigl\| Q_k \bigl( \chi_{I_j} e_{k'}^{-i \psi_\pm}(t,x,D) e_{<0}^{i \psi_\pm}(D,y,t) Q_{\leq k-C} P_0 \phi \bigr\|_{L^2_t L^2_x} \\
  &\lesssim \sum_{k' = k-C}^{k' = 0} \varepsilon 2^{\delta (k-k')} \bigl\| e_{<0}^{i \psi_\pm}(D,y,t) Q_{\leq k-C} P_0 \phi \bigr\|_{N_0^\ast} \\
  &\lesssim \varepsilon \| P_0 \phi \|_{N_0^\ast}.
 \end{align*}
 Observing that
 \begin{align*}
  & Q_k \Bigl( \chi_{I_j} e_{<k-C}^{-i\psi_\pm}(t,x,D) \bigl( e_{<0}^{i \psi_\pm}(D,y,t) - e_{<k-C}^{i \psi_\pm}(D,y,t) \bigr) Q_{\leq k-C} P_0 \phi \Bigr) \\
  &= Q_k \Bigl( e_{<k-C}^{-i\psi_\pm}(t,x,D) Q_{k+O(1)} \Bigl( \chi_{I_j} \bigl( e_{<0}^{i\psi_\pm}(D,y,t) - e_{<k-C}^{i\psi_\pm}(D,y,t) \bigr) Q_{\leq k-C} P_0 \phi \Bigr) \Bigr), 
 \end{align*}
 we estimate the second term analogously using the fixed-time $L^2_x \to L^2_x$ estimate for $e_{<k-C}^{-i\psi_\pm}(t,x,D)$ and Proposition \ref{prop:modulation_localized_estimate_right_quantization}. 
\end{proof}

\subsection{Proof of the renormalization error estimate \eqref{equ:renormalization_error_estimate_2}}

This estimate can be proven by adapting the proof in \cite[Section 10.2]{KST} to our large data setting using similar ideas as above. The additional errors generated by the high-angle cutoff for intermediate frequencies in the definition of the phase correction $\psi_\pm$ can be controlled by divisibility of suitable space-time norms of $A$.

\subsection{Proof of the dispersive estimate \eqref{equ:dispersive_estimate_for_e}}

Since the $S$ space is compatible with time localizations by Lemma \ref{lem:time_cutoff_compatible_with_S_norm}, the dispersive estimate \eqref{equ:dispersive_estimate_for_e} follows immediately from the estimate (83) in~\cite{KST}.

\section{Breakdown criterion} \label{sec:breakdown}

\begin{defn} \label{defn:S_norm_breakdown}
 Let $T_0, T_1 > 0$. For any admissible solution $(A,\phi)$ to the MKG-CG system on $(-T_0, T_1) \times \R^4$, we define 
 \begin{equation*}
   \| (A,\phi) \|_{S^1((-T_0, T_1)\times\R^4)} := \sup_{\substack{0 < T < T_0, \\ 0 < T' < T_1}} \biggl( \sum_{j=1}^4 \|A_j\|_{S^1([-T,T']\times\R^4)}^2 + \|\phi\|_{S^1([-T,T']\times\R^4)}^2 \biggr)^{\frac{1}{2}}. 
 \end{equation*}
\end{defn}

We establish the following blowup criterion for admissible solutions to the MKG-CG system.
\begin{prop} \label{prop:breakdown_criterion}
 Let $(-T_0, T_1)$ be the maximal interval of existence of an admissible solution $(A, \phi)$ to the MKG-CG system. If $\| (A, \phi) \|_{S^1((-T_0, T_1)\times\R^4)} < \infty$, then it must hold that $T_0 = T_1 = \infty$.
\end{prop}

The idea of the proof of Proposition \ref{prop:breakdown_criterion} is to establish an a priori bound on a subcritical norm
\[
 \sup_{t \in (-T_0, T_1)} \sum_{j=1}^4 \, \bigl\| A_j[t] \bigr\|_{H^s_x \times H^{s-1}_x} + \bigl\| \phi[t] \bigr\|_{H^s_x \times H^{s-1}_x} < \infty
\]
for some $s > 1$. By the local well-posedness result \cite{Selberg} for the MKG-CG system it then follows that the solution can be smoothly extended beyond the time interval $(-T_0, T_1)$. To this end, we will use Tao's device of frequency envelopes. For sufficiently small $\sigma > 0$ we define for all $k \in \Z$,
\[
 c_k := \biggl( \sum_{l \in \Z} 2^{-\sigma |k-l|} \Bigl( \sum_{j=1}^4  \bigl\| P_l A_j[0] \bigr\|_{\dot{H}^1_x \times L^2_x}^2 + \bigl\| P_l \phi[0] \bigr\|_{\dot{H}^1_x \times L^2_x}^2 \Bigr) \biggr)^{\frac{1}{2}}.
\]
Proposition \ref{prop:breakdown_criterion} is then a consequence of the following result.

\begin{prop} \label{prop:breakdown_criterion_frequency_envelopes}
 Let $(-T_0, T_1)$ be the maximal interval of existence of an admissible solution $(A, \phi)$ to the MKG-CG system. If $\| (A, \phi) \|_{S^1((-T_0, T_1)\times\R^4)} < \infty$, there exists $C = C(\| (A, \phi) \|_{S^1((-T_0, T_1)\times\R^4)})~<~\infty$ such that for all $k \in \Z$,
 \begin{equation}
  \bigl\|  P_k A \bigr\|_{S^1_k((-T_0, T_1)\times\R^4)} + \bigl\| P_k \phi \bigr\|_{S^1_k((-T_0, T_1)\times\R^4)} \leq C c_k.
 \end{equation}
\end{prop}
\begin{proof}
 A sketch of the proof is given in Subsection \ref{subsec:interlude}.
\end{proof}

\section{A concept of weak evolution} \label{sec:concept_of_weak_evolution}

In order to implement the contradiction argument after the concentration compactness step, we have to define the notion of a solution to the MKG-CG system that is merely of energy class. In the context of critical wave maps in \cite{KS} this is achieved by first approximating an energy class datum by Schwartz class data in the energy topology. One then defines the energy class solution as a suitable limit of the associated Schwartz class solutions. Using perturbation theory, one shows that this limit is well-defined and independent of the approximating sequence. 

For the MKG-CG system we have to argue more carefully, because it appears that the strong perturbative step in the context of the critical wave maps in \cite{KS} is not available due to a low frequency divergence. However, the problem with evolving irregular data is really a ``high frequency issue'' and it appears that truncating high frequencies away does not lead to the same problems as a general perturbative step. More concretely, consider Coulomb energy class data at time $t=0$. By truncating in frequency, we can assume that the frequency support of either input is at $|\xi| \leq K$ for some $K > 0$. Then the problem becomes to show that we can add high-frequency perturbations to the data, i.e. supported in frequency space at $|\xi| > K$ at time $t = 0$, and to obtain a perturbed global evolution. 

\begin{prop} \label{prop:perturbation}
 Let $(A, \phi)$ be an admissible solution to the MKG-CG system on $[-T_0, T_1] \times \R^4$ for some $T_0, T_1 > 0$. Assume that $(A, \phi)[0]$ have frequency support at $|\xi| \leq K$ for some $K > 0$ and that $\|(A, \phi)\|_{S^1([-T_0, T_1]\times\R^4)} = L < \infty$. Then there exists $\delta_1(L) > 0$ with the following property: Let $(A + \delta A, \phi + \delta \phi)$ be any other admissible solution to the MKG-CG system defined locally around $t=0$ such that
 \[
  E(\delta A, \delta \phi)(0) = \delta_0 < \delta_1(L)
 \]
 and such that $(\delta A, \delta \phi)[0]$ have frequency support at $|\xi| > K$. Then $(A + \delta A, \phi + \delta \phi)$ extends to an admissible solution to the MKG-CG system on the whole time interval $[-T_0, T_1]$ and satisfies
 \[
  \|(A + \delta A, \phi + \delta \phi)\|_{S^1([-T_0, T_1]\times\R^4)} \leq \tilde{L}(L, \delta_0).
 \]
 Moreover, we have 
 \[
  \|(\delta A, \delta \phi)\|_{S^1([-T_0, T_1]\times\R^4)} \rightarrow 0
 \]
 as $\delta_0 \rightarrow 0$.
\end{prop}
\begin{proof}
 A sketch of the proof is given in Subsection \ref{subsec:interlude}.
\end{proof}

The above high-frequency perturbation result suggests that we could define the MKG-CG evolution of energy class Coulomb data as a suitable limit of the evolutions of low frequency approximations of the energy class data. More precisely, for Coulomb data $(A, \phi)[0] \in \dot{H}^1_x \times L^2_x$, we pick a sequence of smoothings $(A_n, \phi_n)[0]$ by truncating the frequency support of $(A, \phi)[0]$ so that
\[
 \lim_{n \rightarrow \infty} (A_n, \phi_n)[0] = (A,\phi)[0]
\]
in the sense of $\dot{H}^1_x \times L^2_x$. Here the rather technical issue appears whether there exists a smooth (local) solution $(A_n, \phi_n)$ to the MKG-CG system with initial data $(A_n, \phi_n)[0]$. The hypothesis $(A, \phi)[0] \in \dot{H}^1_x \times L^2_x$ does not guarantee that $A(0)$ and $\phi(0)$ are $L^2$ integrable in the low frequencies. For this reason we cannot directly invoke the local well-posedness result \cite{Selberg} to obtain a smooth local solution. The natural way around this is to localize in physical space. This will be explained in more detail in Subsection \ref{subsec:localizing_in_physical_space} below. 

For each smooth local solution $(A_n, \phi_n)$ to the MKG-CG system with initial data $(A_n, \phi_n)[0]$ we then define
\[
 I_n := \cup_{\tilde{I} \in {\mathcal A}_n} \tilde{I},
\]
where
\[
 {\mathcal A}_n := \bigl\{ \tilde{I} \subset \R \text{ open interval with } 0 \in \tilde{I} : \sup_{J \subset \tilde{I}, J closed} \|(A_n, \phi_n)\|_{S^1(J\times\R^4)} < \infty \bigr\}.
\]
We call $I_n$ the maximal lifespan of the solution $(A_n, \phi_n)$. 

In order to define a canonical evolution of Coulomb energy class data, we have to show that the low frequency approximations $(A_n, \phi_n)$ exist on some joint time interval and satisfy uniform $S^1$ norm bounds there.

\begin{prop} \label{prop:joint_time_interval}
 Let $(A, \phi)[0]$ be Coulomb energy class data and let $\bigl\{ (A_n, \phi_n)[0] \bigr\}_n$ be a sequence of smooth low frequency truncations of $(A, \phi)[0]$ such that
 \begin{equation*}
  \lim_{n \rightarrow \infty} (A_n, \phi_n)[0] = (A, \phi)[0]
 \end{equation*}
 in the sense of $\dot{H}^1_x \times L^2_x$. Denote by $(A_n, \phi_n)$ the smooth solutions to the MKG-CG system with initial data $(A_n, \phi_n)[0]$ and with maximal intervals of existence $I_n$. Then there exists a time $T_0 \equiv T_0(A,\phi)~>~0$ such that $[-T_0, T_0] \subset I_n$ for all sufficiently large $n$ and 
 \begin{equation*}
  \limsup_{n \rightarrow \infty} \|(A_n, \phi_n)\|_{S^1([-T_0, T_0]\times\R^4)} \leq C(A, \phi),
 \end{equation*}
 where $C(A,\phi) > 0$ is a constant that depends only on the energy class data $(A,\phi)[0]$.
\end{prop}
\begin{proof}
 The proof is given in Subsection \ref{subsec:proof_of_joint_time_interval_prop} below.
\end{proof}

Using Proposition \ref{prop:perturbation} and Proposition \ref{prop:joint_time_interval}, we may introduce the following notion of energy class solutions to the MKG-CG system that we outlined above.
\begin{defn} \label{defn:energy_class_solution}
 Let $(A, \phi)[0]$ be Coulomb energy class data and let $\{ (A_n, \phi_n)[0] \}_n$ be a sequence of smooth low frequency truncations of $(A, \phi)[0]$ such that
 \begin{equation*}
  \lim_{n \rightarrow \infty} (A_n, \phi_n)[0] = (A, \phi)[0]
 \end{equation*}
 in the sense of $\dot{H}^1_x \times L^2_x$. We denote by $(A_n, \phi_n)$ the smooth solutions to the MKG-CG system with initial data $(A_n, \phi_n)[0]$ and define $I = (-T_0, T_1) = \cup \tilde{I}$ to be the union of all open time intervals $\tilde{I}$ containing $0$ with the property that
 \[
  \sup_{J \subset \tilde{I}, J closed} \liminf_{n \rightarrow \infty} \|(A_n, \phi_n)\|_{S^1(J \times \R^4)} < \infty.
 \]
 Then we define the MKG-CG evolution of $(A,\phi)[0]$ on $I \times \R^4$ to be
 \[
  (A,\phi)[t] := \lim_{n \rightarrow \infty} (A_n, \phi_n)[t], \quad t \in I,
 \]
 where the limit is taken in the energy topology. We call $I$ the maximal lifespan of $(A, \phi)$. For any closed interval $J \subset I$, we set
 \begin{equation*}
  \|(A, \phi)\|_{S^1(J\times\R^4)} := \lim_{n \rightarrow \infty} \|(A_n, \phi_n)\|_{S^1(J \times \R^4)}.
 \end{equation*}
\end{defn}

We obtain the following characterization of the maximal lifespan $I$ of an energy class solution.
\begin{lem} \label{lem:characerization_maximal_lifespan}
 Let $(A, \phi)$, $(A_n, \phi_n)$ and $I$ be as in Definition \ref{defn:energy_class_solution}. Assume in addition that $I \neq (-\infty, \infty)$. Then
 \begin{equation*}
  \sup_{J \subset I, J closed} \liminf_{n \rightarrow \infty} \|(A_n, \phi_n)\|_{S^1(J\times\R^4)} = \infty.
 \end{equation*}
\end{lem}

We call an energy class solution $(A, \phi)$ with maximal lifespan $I$ \emph{singular}, if either $I \neq \R$, or if $I = \R$ and 
\[
 \sup_{J \subset I, J closed} \|(A, \phi)\|_{S^1(J\times\R^4)} = \infty.
\]

\subsection{Proof of Proposition \ref{prop:joint_time_interval}} \label{subsec:proof_of_joint_time_interval_prop}

A natural idea is to localize the data $(A_n, \phi_n)[0]$ in physical space to ensure smallness of the energy and to then try to ``patch together'' the local solutions obtained from the small energy global well-posedness result \cite{KST}. The problem is that the MKG-CG system does not have the finite speed of propagation property due to non-local terms in the equation for the magnetic potential $A$. To overcome this difficulty, we exploit that the Maxwell-Klein-Gordon system enjoys gauge invariance.

\medskip

We first describe how we suitably localize the data $(A_n, \phi_n)[0]$ in physical space to obtain admissible Coulomb data with small energy that can be globally evolved by \cite{KST}. Let $\chi \in C_c^\infty(\R^4)$ be a smooth cutoff function with support in the ball $B(0, \frac{3}{2})$ and such that $\chi \equiv 1$ on $B(0, \frac{5}{4})$. For $x_0 \in \R^4$ and $r_0 > 0$, we set $\chi_{\{|x - x_0| \lesssim r_0\}}(x) := \chi(\frac{x-x_0}{r_0})$. Then we define
\begin{equation} \label{equ:definition_of_small_energy_data_beginning}
 \gamma_n(0,\cdot) := \Delta^{-1} \partial_j \bigl( \chi_{\{|x - x_0| \lesssim r_0\}}(\cdot) A^j_n(0,\cdot) \bigr)
\end{equation}
and for $j = 1, \ldots, 4$,
\begin{equation} 
 \tilde{A}_{n,j}(0, \cdot) := \chi_{\{|x - x_0| \lesssim r_0\}}(\cdot) A_{n,j}(0, \cdot) - \partial_j \gamma_n(0,\cdot).
\end{equation}
We determine $\tilde{A}_{n,0}(0, \cdot)$ as the solution to the elliptic equation
\begin{equation} \label{equ:definition_of_small_energy_data_3}
 \Delta \tilde{A}_{n,0} = - \Im \bigl( \chi_{\{|x - x_0| \lesssim r_0\}} \phi_n \overline{\chi_{\{|x - x_0| \lesssim r_0\}} \partial_t \phi_n} \bigr) + |\chi_{\{|x - x_0| \lesssim r_0\}} \phi_n|^2 A_{n,0} \quad \text{ on } \R^4,
\end{equation}
where $\phi_n$ and $A_{n,0}$ are evaluated at time $t=0$. We note that $\tilde{A}_n$ is in Coulomb gauge. Then we set
\begin{equation}
  \partial_t \gamma_n(0,\cdot) := A_{n,0}(0,\cdot) - \tilde{A}_{n,0}(0,\cdot)
\end{equation}
and define $\partial_t \tilde{A}_{n,j}(0, \cdot)$ for $j = 1, \ldots, 4$, first just on $B(x_0, \frac{5}{4} r_0)$, by setting
\begin{equation}
 \partial_t \tilde{A}_{n,j}|_{B(x_0, \frac{5}{4} r_0)}(0, \cdot) := \bigl( \partial_t A_{n,j}(0,\cdot) - \partial_j \partial_t \gamma_n(0,\cdot) \bigr)\big|_{B(x_0, \frac{5}{4} r_0)}.
\end{equation}
We observe that $\Delta( A_{n,0}(0, \cdot) - \tilde{A}_{n,0}(0,\cdot)) = 0$ on $B(x_0, \frac{5}{4} r_0)$ by the definition of $\tilde{A}_{n,0}(0,\cdot)$. Thus, the data $\partial_t \tilde{A}_{n}|_{B(x_0, \frac{5}{4} r_0)}(0,\cdot)$ satisfy the Coulomb compatibility condition $\partial_j (\partial_t \tilde{A}_n^j)(0,\cdot) = 0$ on $B(x_0, \frac{5}{4} r_0)$. Using \cite[Proposition 2.1]{KMPT}, we extend $(\partial_t \tilde{A}_{n,j})(0,\cdot)|_{B(x_0, \frac{5}{4} r_0)}$ to the whole of $\R^4$ while maintaining the Coulomb compatibility condition and such that $\|\partial_t \tilde{A}_{n,j}\|_{L^2_x(\R^4)} \lesssim \|\partial_t \tilde{A}_{n,j}\|_{L^2_x(B(x_0, \frac{5}{4} r_0))}$. Finally, we define
\begin{equation}
 \tilde{\phi}_n(0,\cdot) := e^{i \gamma_n(0,\cdot)} \chi_{\{|x - x_0| \lesssim r_0\}}(\cdot) \phi_n(0, \cdot)
\end{equation}
and
\begin{equation} \label{equ:definition_of_small_energy_data_end}
 \partial_t \tilde{\phi}_n(0,\cdot) := i \partial_t \gamma_n(0,\cdot) e^{i \gamma_n(0,\cdot)} \chi_{\{|x - x_0| \lesssim r_0\}}(\cdot) \phi_n(0, \cdot) + e^{i\gamma_n(0,\cdot)} \chi_{\{|x - x_0| \lesssim r_0\}}(\cdot) \partial_t \phi_n(0, \cdot).
\end{equation}
In the next lemma we prove that by choosing $r_0 > 0$ sufficiently small, we can ensure that the Coulomb data $(\tilde{A}_n, \tilde{\phi}_n)[0]$ have small energy for all sufficiently large $n$. Here we exploit that the convergence $(A_n, \phi_n)[0] \rightarrow (A,\phi)[0]$ in the energy topology as $n \rightarrow \infty$ implies a uniform non-concentration property of the energy of the data $(A_n, \phi_n)[0]$. We denote by $\varepsilon_0 > 0$ the small energy threshold of the small energy global well-posedness result \cite{KST} for the MKG-CG system.
\begin{lem} \label{lem:small_energy_on_small_ball}
 Let $(\tilde{A}_n, \tilde{\phi}_n)$ be defined as in \eqref{equ:definition_of_small_energy_data_beginning} -- \eqref{equ:definition_of_small_energy_data_end}. Given $\varepsilon_0 > 0$ there exists $r_0 > 0$ such that uniformly for all $x_0 \in \R^4$ and for all sufficiently large $n$, it holds that
 \[
  E(\tilde{A}_n, \tilde{\phi}_n) < \varepsilon_0.
 \]
\end{lem}
\begin{proof}
 We start with the components $\tilde{A}_{n,j}$. Suppressing that $A_n$ is evaluated at time $t = 0$, we have for $j = 1, \ldots, 4$ that
 \begin{equation} \label{equ:energy_estimate_A_tilde_n}
  \begin{split}
   \|\nabla_x \tilde{A}_{n,j}\|_{L^2_x(\R^4)}^2 &\lesssim \| \nabla_x ( \chi_{\{|x-x_0| \lesssim r_0\}} A_{n,j} ) \|_{L^2_x(\R^4)}^2 + \|\nabla_x \partial_j \gamma_n\|_{L^2_x(\R^4)}^2 \\
   &\lesssim \sum_{i = 1}^4 \|\nabla_x (\chi_{\{|x-x_0| \lesssim r_0\}} A_{n,i})\|_{L^2_x(\R^4)}^2 \\
   &\lesssim \sum_{i = 1}^4 \frac{1}{r_0^2} \int_{B(x_0, \frac{3}{2} r_0)} |A_{n,i}(x)|^2 \, dx + \int_{B(x_0,\frac{3}{2} r_0)} |\nabla_x A_{n,i}(x)|^2 \, dx \\
   &\lesssim \sum_{i=1}^4 \biggl( \int_{B(x_0,\frac{3}{2} r_0)} |A_{n,i}(x)|^4 \, dx \biggr)^{1/2} + \int_{B(x_0,\frac{3}{2} r_0)} |\nabla_x A_{n,i}(x)|^2 \, dx.
  \end{split}
 \end{equation}
 Next we note that we can pick $r_0 > 0$ such that we have for the energy class data $A$ that
 \[
  \sup_{x_0 \in \R^4} \sum_{i=1}^4 \int_{B(x_0,\frac{3}{2} r_0)} |\nabla_x A_{i}(x)|^2 \, dx +  \int_{B(x_0,\frac{3}{2} r_0)} |A_{i}(x)|^4 \, dx \ll \varepsilon_0.
 \]
 Since $A_n \rightarrow A$ in $\dot{H}^1_x(\R^4)$ as $n \rightarrow \infty$, we also obtain for sufficiently large $n$ that
 \[
  \sup_{x_0 \in \R^4} \sum_{i=1}^4 \int_{B(x_0, \frac{3}{2} r_0)} |\nabla_x A_{n,i}(x)|^2 \, dx +  \int_{B(x_0,\frac{3}{2} r_0)} |A_{n,i}(x)|^4 \, dx \ll \varepsilon_0.
 \]
 From \eqref{equ:energy_estimate_A_tilde_n} we conclude that $\|\nabla_x \tilde{A}_{n,j}\|^2_{L^2_x(\R^4)} \lesssim \varepsilon_0$. In a similar manner we argue that $r_0 > 0$ can be picked such that for all sufficiently large $n$ we also have
 \[
  \sum_{i=1}^4 \|\partial_t \tilde{A}_{n,i}\|_{L^2_x(\R^4)}^2 + \|\nabla_x \tilde{A}_{n,0}\|_{L^2_x(\R^4)}^2 + \|\nabla_{t,x} \tilde{\phi}_{n}\|_{L^2_x(\R^4)}^2 \lesssim \varepsilon_0
 \]
 and hence,
 \[
  E(\tilde{A}_n, \tilde{\phi}_n) \lesssim \varepsilon_0.
 \]
\end{proof}

By Lemma \ref{lem:small_energy_on_small_ball} we can pick $r_0 > 0$ such that the data $(\tilde{A}_n, \tilde{\phi}_n)[0]$ can be globally evolved for sufficiently large $n$ by the small energy global well-posedness result \cite{KST} and we obtain global $S^1$ norm bounds on their evolutions $(\tilde{A}_n, \tilde{\phi}_n)$. For $t > 0$ we then define
\begin{equation} \label{equ:definition_partial_t_gamma_n_t_greater_0}
 \partial_t \gamma_n(t, \cdot) := A_{n,0}(t, \cdot) - \tilde{A}_{n,0}(t, \cdot),
\end{equation}
which implies that
\begin{equation} \label{equ:definition_gamma_n_t_greater_0}
 \gamma_n(t,\cdot) = \gamma_n(0,\cdot) + \int_0^t \bigl( A_{n,0}(s,\cdot) - \tilde{A}_{n,0}(s,\cdot) \bigr) \, ds.
\end{equation}

Our next goal is to relate the evolutions $(\tilde{A}_n, \tilde{\phi}_n)$ and $(A_n, \phi_n)$ on the light cone 
\[
 K_{x_0, \frac{5}{4} r_0} = \bigl\{ (t,x) : 0 \leq t < {\textstyle \frac{5}{4}} r_0, |x-x_0| < {\textstyle{\frac{5}{4}}} r_0 - t \bigr\} 
\]
over the ball $B(x_0, \frac{5}{4} r_0)$. These identities will be the key ingredient to recover $S^1$ norm bounds for $(A_n, \phi_n)$ from those of $(\tilde{A}_n, \tilde{\phi}_n)$. 
\begin{lem} \label{lem:A_tilde_equals_A_minus_gamma}
 Let $(\tilde{A}_n, \tilde{\phi}_n)$ and $\gamma_n$ be defined as in \eqref{equ:definition_of_small_energy_data_beginning} -- \eqref{equ:definition_of_small_energy_data_end} and \eqref{equ:definition_partial_t_gamma_n_t_greater_0} -- \eqref{equ:definition_gamma_n_t_greater_0} such that $E(\tilde{A}_n, \tilde{\phi}_n)~<~\varepsilon_0$. For all sufficiently large $n$ it holds that
 \[
  \tilde{A}_{n,j} = A_{n,j} - \partial_j \gamma_n \text{ on } K_{x_0,\frac{5}{4} r_0} 
 \]
 for $j = 1, \ldots, 4$ and that
 \[
  \tilde{\phi}_n = e^{i \gamma_n} \phi_n \text{ on }  K_{x_0, \frac{5}{4} r_0}.
 \]
\end{lem}
\begin{proof}
To simplify the notation we omit the subscript $n$. Using the equations that $(A, \phi)$, $(\tilde{A}, \tilde{\phi})$, and $\gamma$ satisfy, we obtain that  
\begin{equation} \label{equ:A_tilde_equation}
 \Box \tilde{A}_j = - \Im \bigl( \tilde{\phi} \overline{ \tilde{D}_j \tilde{\phi} } \bigr) + \partial_j \Delta^{-1} \partial^i \Im \bigl( \tilde{\phi} \overline{ \tilde{D}_i \tilde{\phi} } \bigr) \, \text{ on } \R_t \times \R^4_x
\end{equation}
and
\begin{equation} \label{equ:A_minus_partial_gamma_equation}
\Box (A_j - \partial_j \gamma) = - \Im \bigl(\phi \overline{D_j \phi} \bigr) + \partial_j \Delta^{-1} \partial^i \Im \bigl( \tilde{\phi} \overline{\tilde{D}_i \tilde{\phi}} \bigr) - \partial_j \int_0^t \Bigl\{ \Im \bigl( \phi \overline{D_t \phi} \bigr) - \Im \bigl( \tilde{\phi} \overline{D_t \tilde{\phi}} \bigr) \Bigr\} \, ds \, \text{ on } K_{x_0,\frac{5}{4} r_0},
\end{equation}
where we use the notation $\tilde{D}_\alpha = \partial_\alpha + i \tilde{A}_\alpha$. Next we introduce the quantities
\[
 B_j = \tilde{A}_j - (A_j - \partial_j \gamma) \\
\]
and
\[
\psi = \tilde{\phi} - e^{i \gamma} \phi.
\]
From \eqref{equ:A_tilde_equation} and \eqref{equ:A_minus_partial_gamma_equation} we infer that 
\[
 \Box B_j = \Im \bigl( \phi \overline{ D_j \phi} \bigr) - \Im \bigl( \tilde{\phi} \overline{ \tilde{D}_j \tilde{\phi} } \bigr)  - \partial_j \int_0^t \Bigl\{ \Im \bigl( \phi \overline{D_t \phi} \bigr) - \Im \bigl( \tilde{\phi} \overline{ \tilde{D}_t \tilde{\phi} } \bigr) \Bigr\} \, ds \, \text{ on } K_{x_0, \frac{5}{4} r_0}.
\]
The first two terms in the above equation can be rewritten as
\[
 \Im \bigl( \phi \overline{ D_j \phi} \bigr) - \Im \bigl( \tilde{\phi} \overline{ \tilde{D}_j \tilde{\phi} } \bigr) = B_j |\phi|^2 - \Im \bigl( \psi \overline{ (\partial_j + i \tilde{A}_j) (\psi + e^{i\gamma} \phi) } \bigr) - \Im \bigl( e^{i\gamma} \phi \overline{ (\partial_j + i \tilde{A}_j) \psi } \bigr)
\]
and similarly we obtain for the last term that
\[
 \Im \bigl( \phi \overline{D_t \phi} \bigr) - \Im \bigl( \tilde{\phi} \overline{\tilde{D}_t \tilde{\phi}} \bigr) = - \Im \bigl( \psi \overline{ (\partial_t + i \tilde{A_0}) (\psi + e^{i \gamma} \phi) } \bigr) - \Im \bigl( e^{i\gamma} \overline{ (\partial_t + i \tilde{A_0}) \psi} \bigr).
\]
We conclude that the wave equation for $B_j$ on the light cone $K_{x_0,\frac{5}{4} r_0}$ is of the schematic form
\begin{equation} \label{equ:B_wave_equation}
 \begin{split}
  \Box B_j &= f_1 B_j + f_2 |\psi|^2 + f_3 \psi + f_4 \overline{\psi} + f_5 (\partial_j \psi) + f_6 \overline{(\partial_j \psi)} \\
  &\quad \quad + \partial_j \int_0^t \Bigl\{ f_7 |\psi|^2 + f_8 \psi + f_9 \overline{\psi} + f_{10} (\partial_t \psi) + f_{11} \overline{(\partial_t \psi)} \Bigr\} \, ds,
 \end{split}
\end{equation}
where $f_1, \ldots, f_{11}$ are smooth functions on $K_{x_0,\frac{5}{4} r_0}$. To obtain a wave equation for $\psi$, we note that $B_0 = \tilde{A}_0 - (A_0 - \partial_t \gamma) = 0$ by construction and write
\begin{equation*} 
 \begin{split}
  0 &= \Box_{\tilde{A}} \tilde{\phi} - e^{i \gamma} \Box_A \phi \\
    &= \Box_{\tilde{A}} ( \psi + e^{i \gamma} \phi ) - e^{i \gamma} \Box_A \phi \\
    &= \Box_{\tilde{A}} \psi + \Box_{B + A -\partial \gamma} ( e^{i\gamma} \phi) - e^{i \gamma} \Box_A \phi \\
    &= \Box_{\tilde{A}} \psi + \Box_B (e^{i \gamma} \phi) - \Box (e^{i\gamma} \phi) - 2 B^j (A_j - \partial_j \gamma) e^{i \gamma} \phi + \Box_{A - \partial \gamma} ( e^{i \gamma} \phi ) - e^{i \gamma} \Box_A \phi \\
    &= \Box_{\tilde{A}} \psi +  i (\partial^j B_j) (e^{i\gamma} \phi) + 2i B^j \partial_j (e^{i\gamma} \phi) - B^j B_j (e^{i\gamma} \phi) - 2 B^j (A_j - \partial_j \gamma) e^{i \gamma} \phi.
 \end{split}
\end{equation*}
Thus, $\psi$ satisfies a wave equation on the light cone $K_{x_0,\frac{5}{4} r_0}$ of the schematic form 
\begin{equation} \label{equ:psi_wave_equation}
 \Box \psi = f \psi + f_\alpha \partial^\alpha \psi + g_j B^j + g B_j B^j + h (\partial_j B^j),
\end{equation}
where $f, f_\alpha, g, g_j, h$ are smooth functions on $K_{x_0,\frac{5}{4} r_0}$. Since $B[0]$ and $\psi[0]$ vanish on $B(x_0,\frac{5}{4} r_0)$ by our choice of the initial data $(\tilde{A}, \tilde{\phi})[0]$, we conclude from \eqref{equ:B_wave_equation} and \eqref{equ:psi_wave_equation} by a standard energy argument that indeed
\[
 \tilde{A}_j = A_j - \partial_j \gamma \text{ on } K_{x_0, \frac{5}{4} r_0} 
\]
and
\[
 \tilde{\phi} = e^{i \gamma} \phi \text{ on } K_{x_0,\frac{5}{4} r_0}.
\]
\end{proof}

It is clear that given $\varepsilon_0 > 0$, there exists $R > 0$ such that for all sufficiently large $n$, it holds that
\[
 E \bigl( \chi_{\{|x| > R\}}(\cdot) A_n(0,\cdot), \chi_{\{|x| > R\}}(\cdot) \phi_n(0,\cdot) \bigr) < \varepsilon_0.
\]
For our later purposes we have to localize the initial data $(A_n, \phi_n)$ outside the large ball $B(0,R)$ in a scaling invariant way. For any $x_l \in \R^4$ with $|x_l| \sim 2R \, 2^m$ for some $m \in \N$, we set $r_l := 2R \, 2^{m-1}$. Then we define
\[
 \gamma_n^{(l)}(0, \cdot) := \Delta^{-1} \partial_j \bigl( \chi_{\{|x-x_l| \lesssim r_l\}}(\cdot) A_n^j(0,\cdot) \bigr)
\]
and for $j = 1, \ldots, 4$,
\[
 \tilde{A}_{n,j}^{(l)}(0,\cdot) := \chi_{\{|x - x_l| \lesssim r_l\}}(\cdot) A_{n,j}(0,\cdot) - \partial_j \gamma_n^{(l)}(0,\cdot).
\]
We define $\tilde{A}_{n,0}^{(l)}(0,\cdot), \partial_t \tilde{A}_{n,j}^{(l)}(0,\cdot), \tilde{\phi}_n^{(l)}(0,\cdot), \partial_t \tilde{\phi}_n^{(l)}(0,\cdot)$ analogously to \eqref{equ:definition_of_small_energy_data_3} -- \eqref{equ:definition_of_small_energy_data_end} and $\gamma_n^{(l)}(t,\cdot)$, $\partial_t \gamma_n^{(l)}(t,\cdot)$ for $t > 0$ analogously to \eqref{equ:definition_partial_t_gamma_n_t_greater_0} -- \eqref{equ:definition_gamma_n_t_greater_0}. Similarly to Lemma \ref{lem:small_energy_on_small_ball} and Lemma \ref{lem:A_tilde_equals_A_minus_gamma} we obtain
\begin{lem} \label{lem:small_energy_on_complement_of_large_ball}
 Given $\varepsilon_0 > 0$ there exists $R > 0$ such that the initial data $(\tilde{A}_n^{(l)}, \tilde{\phi}_n^{(l)})$ defined as above satisfy for all sufficiently large $n$ that
 \[
  E(\tilde{A}_n^{(l)}, \tilde{\phi}_n^{(l)}) < \varepsilon_0.
 \]
\end{lem}
\noindent and
\begin{lem} \label{lem:A_tilde_equals_A_minus_gamma_outside_big_ball}
 For all sufficiently large $n$ it holds that
 \[
  \tilde{A}_{n,j}^{(l)} = A_{n,j} - \partial_j \gamma_n^{(l)} \text{ on } K_{x_l, \frac{5}{4} r_l}
 \]
 for $j = 1, \ldots, 4$, and that
 \[
  \tilde{\phi}_n^{(l)} = e^{i \gamma_n^{(l)}} \phi_n \text{ on } K_{x_l, \frac{5}{4} r_l},
 \]
 where $K_{x_l, \frac{5}{4} r_l } := \bigl\{ (t,x) : 0 \leq t < \frac{5}{4} r_l, |x-x_l| < \frac{5}{4} r_l - t \bigr\}$.
\end{lem}

We now begin with the proof of Proposition \ref{prop:joint_time_interval} where we suitably ``patch together'' the small energy global evolutions constructed above.
\begin{proof}[Proof of Proposition \ref{prop:joint_time_interval}]
 By time reversibility, it suffices to  only prove the statement in forward time. We pick $r_0 > 0$ sufficiently small and $R > 0$ sufficiently large according to Lemma \ref{lem:small_energy_on_small_ball} and Lemma \ref{lem:small_energy_on_complement_of_large_ball}. Then we cover the ball $B(0,2R) \subset \R^4$ by the supports of finitely many cutoffs $\chi_{\{|x-x_l| \lesssim r_l\}}$ with $r_l = r_0$ for $l = 1, \ldots, L$ for some $L \in \N$. We divide the complement $B(0,2R)^c$ of the ball $B(0,2R)$ into dyadic annulli $A_m := \bigl\{ x \in \R^4 : 2R 2^{m-1} < |x| \leq 2R 2^m \bigr\}$, $m \in \N$, and cover each $A_m$ by the supports of finitely many suitable cutoffs $\chi_{\{|x-x_l| \lesssim r_l\}}(\cdot)$ with $|x_l| \sim 2R 2^m$ and $r_l \sim 2R 2^{m-1}$. This can be carried out in such a way that $\bigl\{ \supp(\chi_{\{|x-x_l|\lesssim r_l\}}) \bigr\}_{l=1}^\infty$ is a uniformly finitely overlapping covering of $\R^4$. We denote by $(\tilde{A}_n^{(l)}, \tilde{\phi}_n^{(l)})$ the associated global solutions to MKG-CG with small energy data given by Lemma \ref{lem:small_energy_on_small_ball}, respectively Lemma \ref{lem:small_energy_on_complement_of_large_ball}. Fix $0 < T_0 \ll r_0$ such that 
 \[ 
  [0,T_0] \times \R^4 \subset \bigcup_{l=1}^\infty K_{x_l, \frac{5}{4} r_l}.
 \]
 Then Lemma \ref{lem:A_tilde_equals_A_minus_gamma} and Lemma \ref{lem:A_tilde_equals_A_minus_gamma_outside_big_ball} imply that the evolutions $(A_n, \phi_n)$ exist on the time interval $[0,T_0]$ uniformly for all sufficiently large $n$. The covering of $\R^4$ by the supports of the cutoffs $\chi_{\{|x-x_l| \lesssim r_l\}}(\cdot)$ can be done in such a way that there exists a uniformly finitely overlapping, smooth partition of unity $\{\chi_l\}_{l\in\N} \subset C_c^\infty(\R \times \R^4)$,
 \begin{equation} \label{equ:partition_of_unity_in_joint_interval_prop}
 1 = \sum_{l = 1}^\infty \chi_l \text{ on } [0,T_0]\times\R^4,
 \end{equation}
 so that each cutoff function $\chi_l(t, x)$ is non-zero only for $t \in [-2 T_0, 2 T_0]$ and satisfies $K_{x_l, r_l} \cap \{t\} \times \R^4 \subset \supp(\chi_l(t, \cdot)) \subset K_{x_l, \frac{9}{8} r_l} \cap \{t\} \times \R^4$ for $t \in [0, 2 T_0]$. 

 \medskip

 In order to obtain uniform $S^1$ norm bounds on the evolutions of $(A_n, \phi_n)$ on $[0,T_0]\times\R^4$, it suffices to establish uniform bounds on the Strichartz and $X_\infty^{0, \frac{1}{2}}$ components of the $S^1$ norms of $(A_n, \phi_n)$ on $[0,T_0] \times \R^4$. These bounds then imply uniform bounds on the full $S^1$ norms of $(A_n, \phi_n)$ on $[0,T_0]\times\R^4$ by a bootstrap argument as in the proof of Proposition~\ref{prop:S_norm_Lorentz}, see the key Observation~1 and Observation~2 there. Since the argument in Proposition~\ref{prop:S_norm_Lorentz} is self-contained, we omit the details here. To facilitate the notation in the following, we introduce the $\tilde{S}^1$ norm
 \[
  \|u\|_{\tilde{S}^1}^2 := \sum_{k \in \Z} \| P_k \nabla_{t,x} u \|_{\tilde{S}_k}^2.
 \]
 Its dyadic subspaces $\tilde{S}_k$ are given by
 \[
  \|u\|_{\tilde{S}_k}^2 := \|u\|_{S_k^{Str}}^2 + \|u\|_{X_\infty^{0, \frac{1}{2}}}^2,
 \]
 where we recall that
 \[
  S_k^{Str} = \bigcap_{\frac{1}{q} + \frac{3/2}{r} \leq \frac{3}{4}} 2^{(\frac{1}{q} + \frac{4}{r} - 2)k} L^q_t L^r_x.
 \]

 \medskip

 We begin by deriving uniform $\tilde{S}^1$ norm bounds on the evolutions $A_n$ on $[0,T_0] \times \R^4$. To this end we define for $i, j = 1, \ldots, 4$ and $l \in \N$, the curvature tensors
 \[
  F_{n, ij} = \partial_i A_{n, j} - \partial_j A_{n, i}
 \]
 and 
 \[
  \tilde{F}^{(l)}_{n, ij} = \partial_i \tilde{A}^{(l)}_{n, j} - \partial_j \tilde{A}^{(l)}_{n, i}.
 \]
 From Lemma \ref{lem:A_tilde_equals_A_minus_gamma} and Lemma \ref{lem:A_tilde_equals_A_minus_gamma_outside_big_ball} we conclude that
 \[
  F_{n, ij} = \sum_{l=1}^\infty \chi_l F_{n, ij} = \sum_{l=1}^\infty \chi_l \tilde{F}^{(l)}_{n, ij} \, \, \text{ on } [0,T_0] \times \R^4.
 \]
 Using the Coulomb gauge, we find for $j = 1, \ldots, 4$ that
 \[
  A_{n,j} = \Delta^{-1} \partial^i F_{n, ij} = \sum_{l=1}^\infty \Delta^{-1} \partial^i \bigl( \chi_l \tilde{F}_{n, ij}^{(l)} \bigr) \, \, \text{ on } [0,T_0] \times \R^4.
 \]
 In order to infer $\tilde{S}^1$ norm bounds on $A_{n,j}$ from the finite $S^1$ norm bounds of the globally defined evolutions $\tilde{A}_n^{(l)}$, we invoke the following almost orthogonality estimate. We defer its proof to the end of this subsection.
 \begin{lem} \label{lem:almost_orthogonality_A}
  There exists a constant $C(A, \phi) > 0$ so that uniformly for all $n$, we have for $j = 1, \ldots, 4$ that
  \begin{equation} \label{equ:almost_orthogonality_A_lemma_estimate}
   \Bigl\| \sum_{l=1}^\infty \Delta^{-1} \partial^i \bigl( \chi_l \tilde{F}_{n,ij}^{(l)} \bigr) \Bigr\|_{\tilde{S}^1([0,T_0]\times\R^4)} \leq C(A,\phi) \Biggl( \sum_{l=1}^\infty \, \bigl\| \Delta^{-1} \partial^i \bigl( \chi_l \tilde{F}_{n,ij}^{(l)} \bigr) \bigr\|_{\tilde{S}^1(\R\times\R^4)}^2 \Biggr)^{1/2}.
  \end{equation}
  The constant $C(A, \phi) > 0$ depends only on the size of $T_0 > 0$, which is determined by the energy class data $(A, \phi)[0]$.
 \end{lem}
 Hence, by \eqref{equ:almost_orthogonality_A_lemma_estimate} we obtain for $j = 1, \ldots, 4$ that
 \begin{equation} \label{equ:almost_orthogonality_A_estimate}
  \begin{aligned}
   \| A_{n,j} \|_{\tilde{S}^1([0,T_0]\times\R^4)} &= \Bigl\| \sum_{l=1}^\infty \Delta^{-1} \partial^i \bigl( \chi_l \tilde{F}_{n,ij}^{(l)} \bigr) \Bigr\|_{\tilde{S}^1([0,T_0]\times\R^4)} \\
   &\leq C(A,\phi) \Biggl( \sum_{l=1}^\infty \, \bigl\| \Delta^{-1} \partial^i \bigl( \chi_l \tilde{F}_{n,ij}^{(l)} \bigr) \bigr\|_{\tilde{S}^1(\R\times\R^4)}^2 \Biggr)^{1/2} \\
   &\lesssim C(A,\phi) \Biggl( \sum_{l=1}^\infty \, \bigl\| \Delta^{-1} \nabla_x \bigl( \chi_l \nabla_x \tilde{A}_n^{(l)} \bigr) \bigr\|_{\tilde{S}^1(\R\times\R^4)}^2 \Biggr)^{1/2}. 
  \end{aligned}
 \end{equation}
 Next we will invoke the following multiplier bound for the $\tilde{S}^1$ norm that will be proven at the end of this subsection.
 \begin{lem} \label{lem:S_1_norm_bounds_for_chi_multiplier_on_A} 
  Let $\chi \in C^{\infty}(\R\times\R^4)$ satisfy 
  \[
   \max_{k = 0, 1, \ldots, 4} \|\nabla_{t,x}^k \chi\|_{L^q_t L^r_x(\R\times\R^4)} \leq D \, \text{ for all } 1 \leq q, r \leq \infty
  \]
  for some $D > 0$. Then there exists a constant $C > 0$ independent of $\chi$ such that for all $\psi \in \tilde{S}^1(\R \times \R^4)$, it holds that
  \begin{equation}
   \bigl\| \Delta^{-1} \nabla_x \bigl( \chi \nabla_x \psi \bigr) \bigr\|_{\tilde{S}^1(\R\times\R^4)} \leq C D \|\psi\|_{\tilde{S}^1(\R\times\R^4)}.
  \end{equation}
 \end{lem}
 By scaling invariance of the $\tilde{S}^1$ norm and the scaling invariant setup of the partition of unity $\{ \chi_l \}_{l\in\N}$, we are in a position to apply Lemma \ref{lem:S_1_norm_bounds_for_chi_multiplier_on_A} uniformly for all multipliers $\chi_l$ to estimate the right-hand side of \eqref{equ:almost_orthogonality_A_estimate} by
 \[
  C(A,\phi) \Biggl( \sum_{l=1}^\infty \, \bigl\| \tilde{A}_n^{(l)} \bigr\|_{\tilde{S}^1(\R\times\R^4)}^2 \Biggr)^{1/2}.
 \]
 By the small energy global well-posedness result \cite{KST}, this is in turn bounded by
 \begin{equation} \label{equ:sum_over_initial_data_A_joint_interval_prop}
  C(A,\phi) \Biggl( \sum_{l=1}^\infty \|\nabla_{t,x} \tilde{\phi}_n^{(l)}(0)\|_{L^2_x(\R^4)}^2 + \|\nabla_{t,x} \tilde{A}^{(l)}_n(0)\|_{L^2_x(\R^4)}^2 \Biggr)^{1/2}.
 \end{equation}
 It remains to square sum in $l$ over the $\dot{H}^1_x \times L^2_x$ norms of the initial data $(\tilde{\phi}_n^{(l)}, \tilde{A}_n^{(l)})[0]$, which we defer to the end of the proof of Proposition~\ref{prop:joint_time_interval}.
 
 \medskip

 To deduce uniform $\tilde{S}^1$ norm bounds on the evolutions $\phi_n$ on $[0,T_0]\times\R^4$, we use Lemma \ref{lem:A_tilde_equals_A_minus_gamma} and Lemma \ref{lem:A_tilde_equals_A_minus_gamma_outside_big_ball} to write
 \begin{equation} \label{equ:rewrite_phi_n_for_almost_orthogonality}
  \phi_n = \sum_{l=1}^\infty \chi_l \phi_n = \sum_{l=1}^\infty \chi_l e^{-i \gamma_n^{(l)}} \tilde{\phi}_n^{(l)} \, \, \text{ on } [0,T_0]\times\R^4.
 \end{equation}
 Next, we apply the following almost orthogonality estimate whose proof we defer to the end of this subsection.
 \begin{lem} \label{lem:almost_orthogonality_phi}
  There exists a constant $C(A, \phi) > 0$ so that uniformly for all $n$,
  \begin{equation} \label{equ:almost_orthogonality_phi_lemma_estimate}
   \Biggl\| \sum_{l=1}^\infty \chi_l e^{-i\gamma_n^{(l)}} \tilde{\phi}_n^{(l)} \Biggr\|_{\tilde{S}^1([0,T_0]\times\R^4)} \leq C(A,\phi) \Biggl( \sum_{l=1}^\infty \, \bigl\| \chi_l e^{-i\gamma_n^{(l)}} \tilde{\phi}_n^{(l)} \bigr\|_{\tilde{S}^1(\R\times\R^4)}^2 \Biggr)^{1/2}.
  \end{equation}
  The constant $C(A, \phi) > 0$ depends only on the size of $T_0 > 0$, which is determined by the energy class data $(A, \phi)[0]$.
 \end{lem}
 Thus, by \eqref{equ:rewrite_phi_n_for_almost_orthogonality} and \eqref{equ:almost_orthogonality_phi_lemma_estimate} we find that
 \begin{equation} \label{equ:almost_orthogonality_phi_estimate}
  \|\phi_n\|_{\tilde{S}^1([0,T_0]\times\R^4)} \leq C(A,\phi) \Biggl( \sum_{l=1}^\infty \, \bigl\| \chi_l e^{-i\gamma_n^{(l)}} \tilde{\phi}_n^{(l)} \bigr\|_{\tilde{S}^1(\R\times\R^4)}^2 \Biggr)^{1/2}.
 \end{equation}
 Here it is not immediate how to obtain $\tilde{S}^1$ norm bounds for $\chi_l e^{-i \gamma_n^{(l)}} \tilde{\phi}_n^{(l)}$ from the finite $S^1$ norm bounds of the globally defined $\tilde{\phi}_n^{(l)}$, because $\gamma_n^{(l)}$ implicitly depends on the unknown quantity $\phi_n$. Indeed, we defined in \eqref{equ:definition_gamma_n_t_greater_0} for $t > 0$ that
 \[
  \gamma_n^{(l)}(t,\cdot) = \gamma_n^{(l)}(0,\cdot) + \int_0^t \bigl( A_{n,0}(s,\cdot) - \tilde{A}_{n,0}^{(l)}(s,\cdot) \bigr) \, ds
 \]
 and we have $\Delta A_{n,0} = - \Im \bigl( \phi_n \overline{D_t \phi_n} \bigr)$. We will overcome this difficulty by exploiting that $\partial_t \gamma_n^{(l)}$ is a harmonic function on every fixed-time slice of $K(x_l, \frac{5}{4} r_l)$ in view of Lemma~\ref{lem:A_tilde_equals_A_minus_gamma}, respectively Lemma~\ref{lem:A_tilde_equals_A_minus_gamma_outside_big_ball}, and its definition
 \[
  \partial_t \gamma_n^{(l)}(t,\cdot) = A_{n, 0}(t,\cdot) - \tilde{A}_{n,0}(t,\cdot) \text{ for } t \geq 0.
 \]
 It therefore enjoys the interior derivative estimates for harmonic functions on every fixed-time slice of $K(x_l, \frac{5}{4} r_l)$. The partition of unity \eqref{equ:partition_of_unity_in_joint_interval_prop} was chosen in such a way that the cutoff functions $\chi_l$ satisfy for all $0 \leq t \leq 2 T_0$ and $x \in \supp(\chi_l(t, \cdot))$ that $B(x, \frac{r_l}{8}) \subset K_{x_l, \frac{5}{4} r_l} \cap \{t\} \times \R^4$. Thus, for all integers $k \geq 0$, we obtain from the interior derivative estimates for harmonic functions that
 \begin{align*}
  \bigl| \chi_l \nabla_x^k \partial_t \gamma_n^{(l)}(t,x) \bigr| &\lesssim \frac{C(k)}{r_l^{4+k}} \bigl\| \bigl(A_{n,0} - \tilde{A}_{n,0}^{(l)} \bigr)(t,\cdot) \bigr\|_{L^1_x(B(x,\frac{r_l}{8}))} \\
  &\lesssim \frac{C(k)}{r_l^{1+k}} \bigl\| \bigl(A_{n,0}-\tilde{A}_{n,0}^{(l)}\bigr)(t, \cdot) \bigr\|_{L^4_x(\R^4)} \\
  &\lesssim \frac{C(k)}{r_l^{1+k}} E(A, \phi)^{1/2}.
 \end{align*}
 We may therefore conclude that
 \begin{equation} \label{equ:gradient_estimate_partial_t_gamma_n}
  r_l^{1+k} \bigl\| \chi_l \nabla_x^k \partial_t \gamma_n^{(l)} \bigr\|_{L^\infty_t L^\infty_x(\R\times\R^4)} \leq C(k,A,\phi).
 \end{equation}
 Similarly, we observe that by Lemma 5.6, 
 \[
  \partial_t^2 \gamma_n^{(l)}(t,\cdot) = \partial_t A_{n,0}(t,\cdot) - \partial_t \tilde{A}_{n,0}(t,\cdot)
 \]
 is harmonic on every fixed-time slice of $K(x_l, \frac{5}{4} r_0)$. The interior derivative estimates for harmonic functions then yield for all integers $k \geq 0$ that
 \[
  \bigl| \chi_l \nabla_x^k \partial_t^2 \gamma_n^{(l)}(t,x) \bigr| \lesssim \frac{C(k)}{r_l^{4+k}} \bigl\| \bigl(\partial_t A_{n,0} - \partial_t \tilde{A}_{n,0} \bigr)(t,\cdot) \bigr\|_{L^1_x(B(x, \frac{r_l}{8}))} \lesssim \frac{C(k)}{r_l^{2+k}} \bigl\| \bigl( \partial_t A_{n,0} - \partial_t \tilde{A}_{n,0})(t,\cdot) \Bigr\|_{L^2_x(\R^4)}.
 \]
 Since we have
 \begin{align*}
  \bigl\| \partial_t A_n(t,\cdot) \bigr\|_{L^2_x(\R^4)} \lesssim \bigl\| \nabla_x \partial_t A_n(t,\cdot) \bigr\|_{L^{\frac{4}{3}}_x} \lesssim \sum_{i=1}^4 \bigl\| \Im \bigl( \phi_n \overline{D_i \phi_n} \bigr) \bigr\|_{L^{\frac{4}{3}}_x} &\lesssim \sum_{i=1}^4 \|\phi_n\|_{L^4_x} \| D_i \phi_n \|_{L^2_x} \lesssim E(A, \phi)
 \end{align*}
 and analogously for $\|\partial_t \tilde{A}_n(t,\cdot)\|_{L^2_x(\R^4)}$, it follows that
 \begin{equation} \label{equ:gradient_estimate_partial_t_squared_gamma_n}
  r_l^{2+k} \bigl\| \chi_l \nabla_x^k \partial_t^2 \gamma_n^{(l)} \bigr\|_{L^\infty_t L^\infty_x(\R \times \R^4)} \lesssim C(k,A,\phi).
 \end{equation}
 Next, we note that $\gamma_n^{(l)}(0,\cdot)$ as defined in \eqref{equ:definition_of_small_energy_data_beginning} is harmonic on the ball $B(x_l, \frac{5}{4} r_l)$. As before, the interior derivative estimates for harmonic functions give for all integers $k \geq 0$ that
 \begin{equation} \label{equ:gradient_estimate_gamma_n_time_zero}
  r_l^k \bigl\| \chi_l \nabla_x^k \gamma_n^{(l)}(0,\cdot) \bigr\|_{L^\infty_x(\R^4)} \leq C(k, A, \phi).
 \end{equation}
 We then obtain from
 \[
  \gamma_n^{(l)}(t,x) = \gamma_n^{(l)}(0,x) + \int_0^t \partial_t \gamma_n^{(l)}(s,x) \, ds
 \]
 that 
 \begin{equation} \label{equ:gradient_estimate_gamma_n}
  r_l^k \bigl\| \chi_l \nabla_x^k \gamma_n^{(l)} \bigr\|_{L^\infty_t L^\infty_x(\R\times\R^4)} \lesssim C(k,A,\phi).
 \end{equation}
 From \eqref{equ:gradient_estimate_partial_t_gamma_n} -- \eqref{equ:gradient_estimate_gamma_n} we conclude that for all integers $k \geq 0$ there exists a constant $C(k,A,\phi) > 0$, depending only on $k$ and the energy class data $(A,\phi)[0]$, so that for all sufficiently large $n$ and all $l \in \N$,
 \begin{equation} \label{equ:L_infty_bounds_on_chi_exp_multiplier}
  \max_{m = 0,1,2} r_l^{k+m} \bigl\| \nabla_x^k \partial_t^m \bigl( \chi_l e^{-i \gamma_n^{(l)}} \bigr) \bigr\|_{L^\infty_t L^\infty_x(\R\times\R^4)} \leq C(k,A,\phi).
 \end{equation}
 Similarly to Lemma \ref{lem:S_1_norm_bounds_for_chi_multiplier_on_A}, we also have the following multiplier bound for the $\tilde{S}^1$ norm.
 \begin{lem} \label{lem:S_1_norm_bounds_for_chi_multiplier_on_phi}
  Let $\chi \in C^\infty(\R\times\R^4)$ satisfy 
  \begin{equation} \label{equ:S_1_norm_bounds_for_chi_multiplier_on_phi_assumed_bounds}
   \max_{k=0, \ldots, 3} \max_{m=0,1,2} \bigl\| \nabla_x^k \partial_t^m \chi \bigr\|_{L^q_t L^r_x(\R\times\R^4)} \leq D \, \text{ for all } 1 \leq q,r \leq \infty
  \end{equation}
  for some $D > 0$. Then there exists a constant $C > 0$ independent of $\chi$ such that for all $\psi \in \tilde{S}^1(\R\times\R^4)$, 
  \begin{equation}
   \|\chi \psi \|_{\tilde{S}^1(\R\times\R^4)} \leq C D \|\psi\|_{\tilde{S}^1(\R\times\R^4)}.
  \end{equation}
 \end{lem}
 In view of \eqref{equ:L_infty_bounds_on_chi_exp_multiplier}, the scaling invariance of the $\tilde{S}^1$ norm and the scaling invariant setup of the partition of unity $\{\chi_l\}_{l\in\N}$, we can apply Lemma \ref{lem:S_1_norm_bounds_for_chi_multiplier_on_phi} uniformly for all multipliers $\chi_l$ to estimate the right hand side of \eqref{equ:almost_orthogonality_phi_estimate} by
 \[
  C(A,\phi) \Biggl( \sum_{l=1}^\infty \, \bigl\| \tilde{\phi}_n^{(l)} \bigr\|_{\tilde{S}^1(\R\times\R^4)}^2 \Biggr)^{1/2}.
 \]
 By the small energy global well-posedness result \cite{KST}, this is in turn bounded by
 \begin{equation} \label{equ:sum_over_initial_data_phi_joint_interval_prop}
  C(A,\phi) \biggl( \sum_{l=1}^\infty \|\nabla_{t,x} \tilde{\phi}_n^{(l)}(0)\|_{L^2_x}^2 + \|\nabla_{t,x} \tilde{A}^{(l)}_n(0)\|_{L^2_x}^2 \biggr)^{1/2}.
 \end{equation}
 It remains to square sum in $l$ over the $\dot{H}^1_x \times L^2_x$ norms of the initial data $(\tilde{\phi}_n^{(l)}, \tilde{A}_n^{(l)})[0]$ in \eqref{equ:sum_over_initial_data_A_joint_interval_prop} and in \eqref{equ:sum_over_initial_data_phi_joint_interval_prop}. Here we have, for example, from the definition
 \[
  \tilde{\phi}_n^{(l)}(0,\cdot) := e^{i \gamma_n^{(l)}(0,\cdot)} \chi_{\{|x-x_l|\lesssim r_l\}}(\cdot) \phi_n(0,\cdot) 
 \]
 that
 \begin{equation} \label{equ:summing_over_initial_data__joint_interval_prop}
  \begin{split} 
   \sum_{l=1}^\infty \int_{\R^4} |\nabla_x \tilde{\phi}_n^{(l)}(0,x)|^2 \, dx &\lesssim \sum_{l=1}^\infty \int_{\R^4} \bigl( |\nabla_x \gamma_n^{(l)}(0,x)|^2 |\chi_{\{|x-x_l|\lesssim r_l\}}(x)|^2 + |\nabla_x \chi_{\{|x-x_l|\lesssim r_l\}}(x)|^2 \bigr) |\phi_n(0,x)|^2 \, dx \\
   &\quad \quad \quad + \sum_{l=1}^\infty \int_{\R^4} |\chi_{\{|x-x_l| \lesssim r_l\}}(x)|^2 |\nabla_x \phi_n(0,x)|^2 \, dx.
  \end{split}
 \end{equation}
 By the construction of the partition of unity, we have uniformly for all $l \in \N$ and $x \in \R^4$ that
 \[
  |\nabla_x \chi_{\{|x-x_l|\lesssim r_l\}}(x)|^2 \lesssim \frac{C(A,\phi)}{|x|^2}
 \]
 and, using also \eqref{equ:gradient_estimate_gamma_n_time_zero}, that
 \[
  |\nabla_x \gamma_n^{(l)}(0,x)|^2 |\chi_{\{|x-x_l|\lesssim r_l\}}(x)|^2 \lesssim \frac{C(A,\phi)}{|x|^2}.
 \]
 By Hardy's inequality and the uniformly finite overlap of the supports of the cutoffs $\chi_{\{|x-x_l|\lesssim r_l\}}(\cdot)$, we conclude that \eqref{equ:summing_over_initial_data__joint_interval_prop} is bounded by 
 \[
   C(A,\phi) \|\nabla_x \phi_n(0,\cdot)\|_{L^2_x}^2 \lesssim C(A,\phi) E(A,\phi)
 \]
 uniformly for all sufficiently large $n$. Proceeding similarly with the other terms in \eqref{equ:sum_over_initial_data_phi_joint_interval_prop}, we finally obtain that \eqref{equ:sum_over_initial_data_phi_joint_interval_prop} is bounded by $C(A,\phi) E(A,\phi)$ uniformly for all sufficiently large $n$. This finishes the proof of Proposition \ref{prop:joint_time_interval}.
\end{proof}

Next, we turn to the proofs of Lemma~\ref{lem:almost_orthogonality_A} and of Lemma~\ref{lem:almost_orthogonality_phi}. We only give the proof of Lemma~\ref{lem:almost_orthogonality_phi}, the other one being similar.

\begin{proof}[Proof of Lemma \ref{lem:almost_orthogonality_phi}]
 In view of the setup of the partition of unity $\{ \chi_l \}_{l \in \N}$, we may assume in this proof that the spatial support of $\chi_l$ is at scale $\sim 2^l$ for $l \in \N$. Moreover, we recall that $\chi_l(t, \cdot)$ is non-zero only for $t \in [-2 T_0, 2 T_0]$.

 \medskip

 We first consider the $S_k^{Str}$ component of the $\tilde{S}^1$ norm. Here we want to show that
 \[
  \sum_{k \in \Z} \, \biggl\| P_k \sum_{l=1}^\infty \nabla_{t,x} \bigl( \chi_l e^{-i\gamma_n^{(l)}} \tilde{\phi}_n^{(l)} \bigr) \biggr\|_{S_k^{Str}([0,T_0]\times\R^4)}^2 \lesssim \sum_{l=1}^\infty \sum_{k \in \Z} \, \bigl\| P_k \nabla_{t,x} \bigl( \chi_l e^{-i\gamma_n^{(l)}} \tilde{\phi}_n^{(l)} \bigr) \bigr\|_{S_k^{Str}(\R\times\R^4)}^2.
 \]
 Let $(q,r)$ be a wave-admissible pair. Then we have 
 \begin{equation} \label{equ:almost_orthogonality_Strichartz_splitting}
  \begin{aligned}
   &\sum_{k \in \Z} 2^{2(\frac{1}{q} + \frac{4}{r} - 2) k} \biggl\| P_k \sum_{l=1}^\infty \nabla_{t,x} \bigl( \chi_l e^{-i\gamma_n^{(l)}} \tilde{\phi}_n^{(l)} \bigr) \biggr\|_{L^q_t L^r_x([0,T_0]\times\R^4)}^2 \\
   &\lesssim \sum_{k \leq 0} 2^{2(\frac{1}{q} + \frac{4}{r} - 2) k} \biggl\| P_k \sum_{l=1}^{-k} \nabla_{t,x} \bigl( \chi_l e^{-i\gamma_n^{(l)}} \tilde{\phi}_n^{(l)} \bigr) \biggr\|_{L^q_t L^r_x([0,T_0]\times\R^4)}^2 \\
   &\quad + \sum_{k \in \Z} 2^{2(\frac{1}{q} + \frac{4}{r} - 2) k} \biggl\| P_k \sum_{l > -k} \nabla_{t,x} \bigl( \chi_l e^{-i\gamma_n^{(l)}} \tilde{\phi}_n^{(l)} \bigr) \biggr\|_{L^q_t L^r_x([0,T_0]\times\R^4)}^2.
  \end{aligned}
 \end{equation}
 In order to bound the first term on the right-hand side of \eqref{equ:almost_orthogonality_Strichartz_splitting}, we introduce slightly fattened cutoff functions $\tilde{\chi}_l \in C_c^\infty(\R\times\R^4)$ such that $\supp(\chi_l) \subset \supp(\tilde{\chi}_l)$. Then we obtain from H\"older's inequality and Bernstein's estimate that
 \begin{align*}
  &\sum_{k \leq 0} 2^{2(\frac{1}{q} + \frac{4}{r} - 2) k} \biggl\| P_k \sum_{l=1}^{-k} \nabla_{t,x} \bigl( \chi_l e^{-i\gamma_n^{(l)}} \tilde{\phi}_n^{(l)} \bigr) \biggr\|_{L^q_t L^r_x([0,T_0]\times\R^4)}^2 \\
  &\lesssim \sum_{k \leq 0} \, \Biggl( T_0^{\frac{1}{q}} 2^{(\frac{1}{q} + 2) k} \biggl\| \sum_{l=1}^{-k} \tilde{\chi}_l \nabla_{t,x} \bigl( \chi_l e^{-i\gamma_n^{(l)}} \tilde{\phi}_n^{(l)} \bigr) \biggr\|_{L^\infty_t L^1_x([0,T_0]\times\R^4)} \Biggr)^2 \\
  &\lesssim \sum_{k \leq 0} \, \Biggl( \sum_{l=1}^{-k} T_0^{\frac{1}{q}} 2^{ 2k}  \| \tilde{\chi}_l \|_{L^\infty_t L^2_x} \bigl\| \nabla_{t,x} \bigl( \chi_l e^{-i\gamma_n^{(l)}} \tilde{\phi}_n^{(l)} \bigr) \bigr\|_{L^\infty_t L^2_x([0,T_0]\times\R^4)} \Biggr)^2.
 \end{align*}
 Since the spatial support of $\tilde{\chi}_l$ is at scale $2^l$, this is bounded by
 \begin{align*}
  T_0^{\frac{2}{q}} \sum_{k \leq 0} \, \Biggl( \sum_{l=1}^{-k} 2^{2 (k+l)} \bigl\| \nabla_{t,x} \bigl( \chi_l e^{-i\gamma_n^{(l)}} \tilde{\phi}_n^{(l)} \bigr) \bigr\|_{L^\infty_t L^2_x([0,T_0]\times\R^4)} \Biggr)^2.
 \end{align*}
 Finally, using Young's inequality, we arrive at the desired bound
 \begin{align*}
  T_0^{\frac{2}{q}} \sum_{l=1}^\infty \, \bigl\| \nabla_{t,x} \bigl( \chi_l e^{-i\gamma_n^{(l)}} \tilde{\phi}_n^{(l)} \bigr) \bigr\|_{L^\infty_t L^2_x([0,T_0]\times\R^4)}^2 \lesssim T_0^{\frac{2}{q}} \sum_{l=1}^\infty \, \bigl\| \chi_l e^{-i\gamma_n^{(l)}} \tilde{\phi}_n^{(l)} \bigr\|_{\tilde{S}^1(\R\times\R^4)}^2. 
 \end{align*}
 Regarding the second term on the right-hand side of \eqref{equ:almost_orthogonality_Strichartz_splitting}, 
 \begin{equation} \label{equ:almost_orthogonality_Strichartz_large_l}
  \sum_{k \in \Z} 2^{2(\frac{1}{q} + \frac{4}{r} - 2) k} \biggl\| \sum_{l > -k} P_k \nabla_{t,x} \bigl( \chi_l e^{-i\gamma_n^{(l)}} \tilde{\phi}_n^{(l)} \bigr) \biggr\|_{L^q_t L^r_x([0,T_0]\times\R^4)}^2,
 \end{equation}
 we note that the spatial support of the cutoff $\chi_l$ is at scale $2^l$, while the projection $P_k$ lives at spatial scale $2^{-k}$. Thus, for $l > -k$ the projection $P_k$ approximately preserves the spatial localizations of the cutoffs~$\chi_l$, up to exponential tails that can be treated easily. Since the family of cutoffs $\{ \chi_l \}_{l \in \N}$ is uniformly finitely overlapping, we may therefore bound \eqref{equ:almost_orthogonality_Strichartz_large_l} schematically by
 \begin{equation*}
  \sum_{k \in \Z} \sum_{l > -k} 2^{2(\frac{1}{q} + \frac{4}{r} - 2) k} \biggl\| P_k \nabla_{t,x} \bigl( \chi_l e^{-i\gamma_n^{(l)}} \tilde{\phi}_n^{(l)} \bigr) \biggr\|_{L^q_t L^r_x([0,T_0]\times\R^4)}^2 \lesssim \sum_{l=1}^\infty \, \bigl\| \chi_l e^{-i\gamma_n^{(l)}} \tilde{\phi}_n^{(l)} \bigr\|_{\tilde{S}^1(\R\times\R^4)}^2,
 \end{equation*}
 which is of the desired form.
 
 \medskip

 It remains to consider the $X_\infty^{0, \frac{1}{2}}$ component of the $\tilde{S}^1$ norm. Here our goal is to prove that
 \[
  \sum_{k \in \Z} \, \biggl\| P_k \sum_{l=1}^\infty \nabla_{t,x} \bigl( \chi_l e^{-i \gamma_n^{(l)}} \tilde{\phi}_n^{(l)} \bigr) \biggr\|_{X_\infty^{0,\frac{1}{2}}([0,T_0]\times\R^4)}^2 \lesssim \sum_{l=1}^\infty \, \bigl\| \chi_l e^{-i\gamma_n^{(l)}} \tilde{\phi}_n^{(l)} \bigr\|_{\tilde{S}^1(\R\times\R^4)}^2.
 \]
 To this end we distinguish between small and large modulations. For small modulations $j \leq 0$, we may just dispose of the projection $Q_j$ and trivially estimate
 \begin{align*}
  \sum_{k \in \Z} \sup_{j \leq 0} 2^j \biggl\| P_k Q_j \sum_{l=1}^\infty \nabla_{t,x} \bigl( \chi_l e^{-i\gamma_n^{(l)}} \tilde{\phi}_n^{(l)} \bigr) \biggr\|_{L^2_t L^2_x(\R\times\R^4)}^2 \lesssim \sum_{k \in \Z} \, \biggl\| \sum_{l=1}^\infty P_k \nabla_{t,x} \bigl( \chi_l e^{-i\gamma_n^{(l)}} \tilde{\phi}_n^{(l)} \bigr) \biggr\|_{L^2_t L^2_x(\R\times\R^4)}^2.
 \end{align*}
 By the space-time support properties of the cutoffs $\chi_l$ and H\"older's inequality in time, this is bounded by
 \begin{align*}
  \lesssim \sum_{k \in \Z} T_0 \biggl\| P_k \sum_{l=1}^\infty \nabla_{t,x} \bigl( \chi_l e^{-i\gamma_n^{(l)}} \tilde{\phi}_n^{(l)} \bigr) \biggr\|_{L^\infty_t L^2_x(\R\times\R^4)}^2
 \end{align*}
 and then we obtain as in the previous considerations on the $S_k^{Str}$ component of the $\tilde{S}^1$ norm the desired bound
 \begin{align*}
  \lesssim T_0 \sum_{l=1}^\infty \sum_{k \in \Z} \, \biggl\| P_k \nabla_{t,x} \bigl( \chi_l e^{-i\gamma_n^{(l)}} \tilde{\phi}_n^{(l)} \bigr) \biggr\|_{L^\infty_t L^2_x(\R\times\R^4)}^2.
 \end{align*}
 For large modulations $j > 0$ and large frequencies $k > 0$, the space-time supports of the cutoffs~$\chi_l$ are approximately preserved, up to exponential tails that can be treated easily. Denoting by $\tilde{\chi}_l$ slightly fattended versions of the cutoffs~$\chi_l$, we may therefore estimate schematically 
 \begin{align*}
  &\sum_{k > 0} \sup_{j>0} 2^j \biggl\| \sum_{l=1}^\infty P_k Q_j \nabla_{t,x} \bigl( \chi_l e^{-i\gamma_n^{(l)}} \tilde{\phi}_n^{(l)} \bigr) \biggr\|_{L^2_t L^2_x(\R\times\R^4)}^2 \\
  &\simeq \sum_{k > 0} \sup_{j > 0} 2^j \biggl\| \sum_{l=1}^\infty \tilde{\chi}_l P_k Q_j \nabla_{t,x} \bigl( \chi_l e^{-i\gamma_n^{(l)}} \tilde{\phi}_n^{(l)} \bigr) \biggr\|_{L^2_t L^2_x(\R\times\R^4)}^2 \\
  &\lesssim \sum_{l=1}^\infty \sum_{k > 0} \sup_{j > 0} 2^j \bigl\| P_k Q_j \nabla_{t,x} \bigl( \chi_l e^{-i\gamma_n^{(l)}} \tilde{\phi}_n^{(l)} \bigr) \bigr\|_{L^2_t L^2_x(\R\times\R^4)}^2 \\
  &\lesssim \sum_{l=1}^\infty \sum_{k > 0} \bigl\| P_k \nabla_{t,x} \bigl( \chi_l e^{-i\gamma_n^{(l)}} \tilde{\phi}_n^{(l)} \bigr) \bigr\|_{X^{0, \frac{1}{2}}_\infty}^2,
 \end{align*}
 which is of the desired form. It therefore remains to consider the case of large modulations $j > 0$ and small frequencies $k \leq 0$. Here we distinguish the cases $l > -k$ and $1 \leq l < -k$. For $l > -k$, the projection $P_k Q_j$ approximately preserves the space-time localization of $\chi_l$ and we immediately obtain the schematic estimate
 \begin{align*}
  \sum_{k \leq 0} \sup_{j > 0} 2^j \, \biggl\| \sum_{l > -k} P_k Q_j \nabla_{t,x} \bigl( \chi_l e^{-i\gamma_n^{(l)}} \tilde{\phi}_n^{(l)} \bigr) \biggr\|_{L^2_t L^2_x(\R\times\R^4)}^2 \lesssim \sum_{k \leq 0} \sum_{l > -k} \, \bigl\| P_k \nabla_{t,x} \bigl( \chi_l e^{-i\gamma_n^{(l)}} \tilde{\phi}_n^{(l)} \bigr) \bigr\|_{X_\infty^{0, \frac{1}{2}}(\R\times\R^4)}^2,
 \end{align*}
 which is of the desired form. Finally, for $1 \leq l < -k$ and large modulations $j > 0$, $Q_j$ approximately preserves the time localization of $\chi_l$. Thus, we obtain for slightly fattened cutoffs $\tilde{\chi}_l$ that
 \begin{align*}
  &\sum_{k \leq 0} \sup_{j>0} 2^j \biggl\| \sum_{l=1}^{-k} P_k Q_j \nabla_{t,x} \bigl( \chi_l e^{-i\gamma_n^{(l)}} \tilde{\phi}_n^{(l)} \bigr) \biggr\|_{L^2_t L^2_x(\R\times\R^4)}^2 \\
  &\simeq \sum_{k \leq 0} \sup_{j > 0} 2^j \biggl\| \sum_{l=1}^{-k} P_k \Bigl( \tilde{\chi}_l Q_j \nabla_{t,x} \bigl( \chi_l e^{-i\gamma_n^{(l)}} \tilde{\phi}_n^{(l)} \bigr) \Bigr) \biggr\|_{L^2_t L^2_x(\R\times\R^4)}^2.
 \end{align*}
 Since $\chi_l$ and $\tilde{\chi}_l$ approximately live at frequency $2^{-l}$, this is basically a high-high interaction term and we may write schematically 
 \begin{align*}
  &\simeq \sum_{k \leq 0} \sup_{j > 0} 2^j \biggl\| \sum_{l=1}^{-k} P_k \Bigl( P_{-l} \bigl( \tilde{\chi}_l \bigr) P_{-l} \bigl( Q_j \nabla_{t,x} \bigl( \chi_l e^{-i\gamma_n^{(l)}} \tilde{\phi}_n^{(l)} \bigr) \bigr) \Bigr) \biggr\|_{L^2_t L^2_x(\R\times\R^4)}^2 \\
  &\lesssim \sum_{k \leq 0} \sup_{j > 0} 2^j \Biggl( \sum_{l=1}^{-k} \, \Bigl\| P_k \Bigl( P_{-l} \bigl( \tilde{\chi}_l \bigr) P_{-l} \bigl( Q_j \nabla_{t,x} \bigl( \chi_l e^{-i\gamma_n^{(l)}} \tilde{\phi}_n^{(l)} \bigr) \bigr) \Bigr) \Bigr\|_{L^2_t L^2_x(\R\times\R^4)} \Biggr)^2.
 \end{align*}
 By Bernstein's estimate and H\"older's inequality we then find
 \begin{align*}
  &\lesssim \sum_{k \leq 0} \sup_{j > 0} 2^j \Biggl( \sum_{l=1}^{-k} \, 2^{2k} \bigl\| \tilde{\chi}_l \bigr\|_{L^\infty_t L^2_x(\R\times\R^4)}  \bigl\| P_{-l}  Q_j \nabla_{t,x} \bigl( \chi_l e^{-i\gamma_n^{(l)}} \tilde{\phi}_n^{(l)} \bigr) \bigr\|_{L^2_t L^2_x(\R\times\R^4)} \Biggr)^2 \\
  &\lesssim \sum_{k \leq 0} \, \Biggl( \sum_{l=1}^{-k} 2^{2k+2l} \bigl\| P_{-l} \nabla_{t,x} \bigl( \chi_l e^{-i\gamma_n^{(l)}} \tilde{\phi}_n^{(l)} \bigr) \bigr\|_{X_\infty^{0, \frac{1}{2}}(\R\times\R^4)} \Biggr)^2, 
 \end{align*}
 where in the last line we used that the spatial support of $\tilde{\chi}_l$ is at scale $2^l$. Using Young's inequality, we arrive at the desired bound
 \begin{align*}
  \lesssim \sum_{l=1}^\infty \, \bigl\| P_{-l} \nabla_{t,x} \bigl( \chi_l e^{-i\gamma_n^{(l)}} \tilde{\phi}_n^{(l)} \bigr) \bigr\|_{X_\infty^{0, \frac{1}{2}}(\R\times\R^4)}^2 \lesssim \sum_{l=1}^\infty \, \bigl\| \chi_l e^{-i\gamma_n^{(l)}} \tilde{\phi}_n^{(l)} \bigr\|_{\tilde{S}^1(\R\times\R^4)}^2.
 \end{align*}
 This finishes the proof of Lemma~\ref{lem:almost_orthogonality_phi}.
\end{proof}

It remains to prove Lemma \ref{lem:S_1_norm_bounds_for_chi_multiplier_on_A} and Lemma \ref{lem:S_1_norm_bounds_for_chi_multiplier_on_phi}. We only give the proof of Lemma \ref{lem:S_1_norm_bounds_for_chi_multiplier_on_phi}, the other one being similar.
\begin{proof}[Proof of Lemma \ref{lem:S_1_norm_bounds_for_chi_multiplier_on_phi}]
 We have to prove that for any $\psi \in \tilde{S}^1(\R\times\R^4)$,
 \[
  \| \chi \psi \|_{\tilde{S}^1(\R\times\R^4)} = \Biggl( \sum_{k \in \Z} \, \bigl\| P_k \nabla_{t,x} \bigl( \chi \psi \bigr) \bigr\|_{S_k^{Str}(\R\times\R^4)}^2 + \bigl\| P_k \nabla_{t,x} \bigl( \chi \psi \bigr) \bigr\|_{X_\infty^{0, \frac{1}{2}}(\R\times\R^4}^2 \Biggr)^{1/2} \leq C D \|\psi\|_{\tilde{S}^1(\R\times\R^4)}.
 \]
 To this end we will constantly invoke the assumed space-time bounds \eqref{equ:S_1_norm_bounds_for_chi_multiplier_on_phi_assumed_bounds} for the multiplier $\chi$.

 \medskip

 We first consider the $S_k^{Str}$ component of the $\tilde{S}_k$ norm. Here we denote by $(q,r)$ any wave-admissible exponent pair, i.e. satisfying $2 \leq q,r \leq \infty$ and $\frac{2}{q} + \frac{3}{r} \leq \frac{3}{2}$. For any $k \in \Z$ we have
 \begin{equation} \label{equ:multiplier_bounds_splitting_S_k_Str}
  \bigl\| P_k \nabla_{t,x} \bigl( \chi \psi \bigr) \bigr\|_{S_k^{Str}} \leq \bigl\| P_k \bigl( (\nabla_{t,x} \chi ) \psi \bigr) \bigr\|_{S_k^{Str}} + \bigl\| P_k \bigl( \chi \nabla_{t,x} \psi \bigr) \bigr\|_{S_k^{Str}}
 \end{equation}
 and begin with estimating the first term on the right hand side of \eqref{equ:multiplier_bounds_splitting_S_k_Str}. If $k \leq 0$, we obtain by the Bernstein and Sobolev inequalities uniformly for all $(q,r)$ that
 \begin{align*}
  2^{(\frac{1}{q} + \frac{4}{r} - 2) k} \bigl\| P_k \bigl( (\nabla_{t,x} \chi ) \psi \bigr) \bigr\|_{L^q_t L^r_x} &\lesssim 2^{\frac{1}{q} k} 2^{2k} \bigl\| \nabla_{t,x} \chi \bigr\|_{L^q_t L^{\frac{4}{3}}_x} \| \psi \|_{L^\infty_t L^4_x} \\
  &\lesssim 2^{2k} \bigl\| \nabla_{t,x} \chi \bigr\|_{L^q_t L^{\frac{4}{3}}_x} \| \nabla_x \psi \|_{L^\infty_t L^2_x} \\
  &\lesssim 2^{2k} \bigl\| \nabla_{t,x} \chi \bigr\|_{L^q_t L^{\frac{4}{3}}_x} \| \psi \|_{\tilde{S}^1}. 
 \end{align*}
 Here we used that
 \[
  \|\nabla_x \psi\|_{L^\infty_t L^2_x} \lesssim \Biggl( \sum_{k \in \Z} \| P_k \nabla_x \psi \|_{L^\infty_t L^2_x}^2 \Biggr)^{1/2} \lesssim \Biggl( \sum_{k \in \Z} \| P_k \nabla_{t,x} \psi \|_{S_k^{Str}}^2 \Biggr)^{1/2} \lesssim \|\psi\|_{\tilde{S}^1}.
 \]
 If $k > 0$, we have
 \begin{align*}
  &2^{(\frac{1}{q} + \frac{4}{r} - 2) k} \bigl\| P_k \bigl( (\nabla_{t,x} \chi) \psi \bigr) \bigr\|_{L^q_t L^r_x} \\
  &\lesssim 2^{(\frac{1}{q} + \frac{4}{r} - 2) k} \bigl\| P_k \bigl( \bigl( P_{>k-C} (\nabla_{t,x} \chi ) \bigr) \psi \bigr) \bigr\|_{L^q_t L^r_x} + 2^{(\frac{1}{q} + \frac{4}{r} - 2) k} \bigl\| P_k \bigl( \bigl(P_{\leq k-C} (\nabla_{t,x} \chi ) \bigr) P_{k + O(1)} \psi \bigr) \bigr\|_{L^q_t L^r_x} \\
  &\lesssim 2^{\frac{1}{q} k} \bigl\| P_{>k-C} \nabla_{t,x} \chi \bigr\|_{L^q_t L^4_x} \|\psi\|_{L^\infty_t L^4_x} + 2^{(\frac{1}{q} + \frac{4}{r} - 2) k} \bigl\| P_{\leq k-C} \nabla_{t,x} \chi  \bigr\|_{L^\infty_t L^\infty_x} \bigl\| P_{k+O(1)} \psi \bigr\|_{L^q_t L^r_x} \\
  &\lesssim 2^{\frac{1}{q} k} 2^{-k} \bigl\| \nabla_x \nabla_{t,x} \chi \bigr\|_{L^q_t L^4_x} \|\nabla_x \psi\|_{L^\infty_t L^2_x} + 2^{(\frac{1}{q} + \frac{4}{r} - 2) k} \bigl\| P_{\leq k-C} \nabla_{t,x} \chi \bigr\|_{L^\infty_t L^\infty_x} 2^{-k} \bigl\| P_{k+O(1)} \nabla_x \psi \bigr\|_{L^q_t L^r_x} \\
  &\lesssim 2^{-\frac{1}{2} k} \bigl\| \nabla_x^2 \partial_t \chi \bigr\|_{L^q_t L^4_x} \|\psi\|_{\tilde{S}^1} + \bigl\| \nabla_{t,x} \chi \bigr\|_{L^\infty_t L^\infty_x} 2^{(\frac{1}{q} + \frac{4}{r} - 2) k} \bigl\| P_{k + O(1)} \nabla_x \psi \bigr\|_{L^q_t L^r_x},
 \end{align*}
 where we used the reverse Bernstein inequality
 \[
  \bigl\| P_{>k-C} \nabla_{t,x} \chi \bigr\|_{L^q_t L^4_x} \lesssim 2^{-k} \bigl\| \nabla_x \nabla_{t,x} \chi \bigr\|_{L^q_t L^4_x}.
 \]
 Square-summing over $k \in \Z$ yields the desired bound. We continue with the second term on the right hand side of \eqref{equ:multiplier_bounds_splitting_S_k_Str}. If $k \leq 0$, we use Bernstein's inequality to bound
 \[
  2^{(\frac{1}{q} + \frac{4}{r} - 2) k} \bigl\| P_k \bigl( \chi \nabla_{t,x} \psi \bigr) \bigr\|_{L^q_t L^r_x} \lesssim 2^{\frac{1}{q} k} 2^{2k} \bigl\| \chi \bigr\|_{L^q_t L^2_x} \|\nabla_{t,x} \psi \|_{L^\infty_t L^2_x} \lesssim 2^{2k} \bigl\| \chi \bigr\|_{L^q_t L^2_x} \|\psi\|_{\tilde{S}^1}.
 \]
 For $k > 0$ we find
 \begin{align*}
  &2^{(\frac{1}{q} + \frac{4}{r} - 2) k} \bigl\| P_k \bigl( \chi \nabla_{t,x} \psi \bigr) \bigr\|_{L^q_t L^r_x} \\
  &\lesssim 2^{(\frac{1}{q} + \frac{4}{r} - 2) k} \bigl\| P_k \bigl( \bigl( P_{>k-C} \chi \bigr) \nabla_{t,x} \psi \bigr\|_{L^q_t L^r_x} + 2^{(\frac{1}{q} + \frac{4}{r} - 2) k} \bigl\| P_k \bigl( \bigl( P_{\leq k-C} \chi \bigr) P_{k + O(1)} \nabla_{t,x} \psi \bigr\|_{L^q_t L^r_x}  \\
  &\lesssim 2^{\frac{1}{q} k} \bigl\| P_{>k-C} \chi \bigr\|_{L^q_t L^\infty_x} \|\nabla_{t,x} \psi \|_{L^\infty_t L^2_x} + 2^{(\frac{1}{q} + \frac{4}{r} - 2) k} \bigl\| P_{\leq k-C} \chi \bigr\|_{L^\infty_t L^\infty_x} \bigl\| P_{k + O(1)} \nabla_{t,x} \psi \bigr\|_{L^q_t L^r_x}  \\
  &\lesssim 2^{-\frac{1}{2} k} \bigl\| \nabla_x \chi \bigr\|_{L^q_t L^\infty_x} \|\psi\|_{\tilde{S}^1} + \bigl\| \chi \bigr\|_{L^\infty_t L^\infty_x} 2^{(\frac{1}{q} + \frac{4}{r} - 2) k} \bigl\| P_{k+O(1)} \nabla_{t,x} \psi \bigr\|_{L^q_t L^r_x}.
 \end{align*}
 The desired bound again follows after square-summing over $k \in \Z$.

 \medskip

 Next we consider the $X_\infty^{0, \frac{1}{2}}$ component of the $\tilde{S}$ norm. For any $k \in \Z$ we have
 \begin{equation} \label{equ:multiplier_bounds_splitting_X_infty}
  \bigl\| P_k \nabla_{t,x} \bigl( \chi \psi \bigr) \bigr\|_{X_\infty^{0,\frac{1}{2}}} \leq \bigl\| P_k \bigl( ( \nabla_{t,x} \chi ) \psi \bigr) \bigr\|_{X_\infty^{0, \frac{1}{2}}} + \bigl\| P_k \bigl( \chi \nabla_{t,x} \psi \bigr) \bigr\|_{X_\infty^{0,\frac{1}{2}}}.
 \end{equation}
 We start with the first term on the right hand side of \eqref{equ:multiplier_bounds_splitting_X_infty}. If $k \leq 0$, we split into a small and a large modulation term
 \begin{equation} \label{equ:multiplier_bounds_X_infty_small_freq_split_modulation}
  P_k \bigl( ( \nabla_{t,x} \chi ) \psi \bigr) = P_k Q_{\leq 0} \bigl( (\nabla_{t,x} \chi ) \psi \bigr) + P_k Q_{> 0} \bigl( (\nabla_{t,x} \chi) \psi \bigr).
 \end{equation}
 We easily estimate the small modulation term using Bernstein's inequality,
 \begin{align*}
  \bigl\| P_k Q_{\leq 0} \bigl( ( \nabla_{t,x} \chi ) \psi \bigr) \bigr\|_{X^{0, \frac{1}{2}}_\infty} &\lesssim \bigl\| P_k \bigl( ( \nabla_{t,x} \chi ) \psi \bigr) \bigr\|_{L^2_t L^2_x} \lesssim 2^{2k} \bigl\| \nabla_{t,x} \chi \bigr\|_{L^2_t L^\frac{4}{3}_x} \|\psi\|_{L^\infty_t L^4_x} \\
  &\lesssim 2^{2k} \bigl\| \nabla_{t,x} \chi \bigr\|_{L^2_t L^\frac{4}{3}_x} \|\psi\|_{\tilde{S}^1}.
 \end{align*}
 To estimate the large modulation term we consider for any $j > 0$,
 \begin{equation} \label{equ:multiplier_bounds_X_infty_small_freq_large_modulation}
  \begin{split}
   2^{\frac{1}{2} j} \bigl\| P_k Q_j \bigl( ( \nabla_{t,x} \chi ) \psi \bigr) \bigr\|_{L^2_t L^2_x} &\lesssim 2^{\frac{1}{2} j} \bigl\| P_k Q_j \bigl( \bigl( P_{> j-C} (\nabla_{t,x} \chi ) \psi \bigr) \bigr\|_{L^2_t L^2_x} \\
   &\quad + 2^{\frac{1}{2} j} \bigl\| P_k Q_j \bigl( \bigl( P_{\leq j-C} Q_{> j-C} (\nabla_{t,x} \chi) \psi \bigr) \bigr\|_{L^2_t L^2_x} \\
   &\quad + 2^{\frac{1}{2} j} \bigl\| P_k Q_j \bigl( \bigl( P_{\leq j-C} Q_{\leq j-C} ( \nabla_{t,x} \chi ) Q_{j+O(1)} \psi \bigr) \bigr\|_{L^2_t L^2_x}.
  \end{split}
 \end{equation}
 We bound the first term using the reverse Bernstein inequality,
 \begin{align*}
  2^{\frac{1}{2} j} \bigl\| P_k Q_j \bigl( \bigl( P_{> j-C} (\nabla_{t,x} \chi) \bigr) \psi \bigr\|_{L^2_t L^2_x} &\lesssim 2^{2k} 2^{\frac{1}{2} j} \bigl\| \bigl( P_{> j-C} (\nabla_{t,x} \chi) \bigr) \psi \bigr\|_{L^2_t L^1_x} \\
  &\lesssim 2^{2k} 2^{- \frac{1}{2} j} \bigl\| \nabla_x \nabla_{t,x} \chi \bigr\|_{L^2_t L^{\frac{4}{3}}_x} \|\psi\|_{\tilde{S}^1}.
 \end{align*}
 For the second term on the right hand side of \eqref{equ:multiplier_bounds_X_infty_small_freq_large_modulation} we obtain from a reverse Bernstein estimate in time that
 \begin{align*}
  2^{\frac{1}{2} j} \bigl\| P_k Q_j \bigl( \bigl( P_{\leq j-C} Q_{> j-C} (\nabla_{t,x} \chi ) \bigr) \psi \bigr\|_{L^2_t L^2_x} &\lesssim 2^{2k} 2^{-\frac{1}{2} j} \bigl\| \nabla_{t,x} \partial_t \chi \bigr\|_{L^2_t L^{\frac{4}{3}}_x} \|\psi\|_{\tilde{S}^1}.
 \end{align*}
 The third term on the right hand side of \eqref{equ:multiplier_bounds_X_infty_small_freq_large_modulation} can be estimated via a Littlewood-Paley trichotomy
 \begin{equation}
  \begin{split}
   &2^{\frac{1}{2} j} \bigl\| P_k Q_j \bigl( \bigl( P_{\leq j-C} Q_{\leq j-C} ( \nabla_{t,x} \chi ) \bigr) Q_{j+O(1)} \psi \bigr\|_{L^2_t L^2_x} \\
   &\quad \quad \lesssim \sum_{l = k+C}^{j-C} 2^{\frac{1}{2} j}  \bigl\| \bigl( P_l Q_{\leq j-C} (\nabla_{t,x} \chi) ) \bigr) P_{l+O(1)} Q_{j+O(1)} \psi \bigr\|_{L^2_t L^2_x} \\
   &\quad \quad \quad + 2^{\frac{1}{2} j} \bigl\| \bigl( P_{k+O(1)} Q_{\leq j-C} (\nabla_{t,x} \chi) \bigr) P_{\leq k + O(1)} Q_{j+O(1)} \psi \bigr\|_{L^2_t L^2_x} \\
   &\quad \quad \quad + 2^{\frac{1}{2} j} \bigl\| \bigl( P_{\leq k-C} Q_{\leq j-C} (\nabla_{t,x} \chi) \bigr) P_{k+O(1)} Q_{j+O(1)} \psi \bigr\|_{L^2_t L^2_x}.
  \end{split}
 \end{equation}
 We bound the high-high case by
 \begin{align*}
  &\sum_{l=k+C}^{j-C} 2^{2k} 2^{\frac{1}{2} j} \bigl\| P_l Q_{\leq j-C} (\nabla_{t,x} \chi) \bigr\|_{L^\infty_t L^2_x} \bigl\| P_{l+O(1)} Q_{j+O(1)} \psi \bigr\|_{L^2_t L^2_x} \\
  &\lesssim 2^{2k} \sum_{l=k+C}^{j-C} \bigl\| \nabla_{t,x} \chi \bigr\|_{L^\infty_t L^2_x} 2^{-l} \bigl\|P_{l+O(1)} \nabla_x \psi \bigr\|_{X_\infty^{0,\frac{1}{2}}} \\
  &\lesssim 2^k \bigl\| \nabla_{t,x} \chi \bigr\|_{L^\infty_t L^2_x} \|\psi\|_{\tilde{S}^1}.
 \end{align*}
 The high-low case is estimated by
 \begin{align*}
  &\sum_{l \leq k+O(1)} 2^{\frac{1}{2} j} \bigl\| P_{k+O(1)} Q_{\leq j-C} (\nabla_{t,x} \chi) \bigr\|_{L^\infty_t L^2_x} \bigl\| P_l Q_{j+O(1)} \psi \bigr\|_{L^2_t L^\infty_x} \\
  &\lesssim \sum_{l \leq k+O(1)} \bigl\| \nabla_{t,x} \chi \bigr\|_{L^\infty_t L^2_x} 2^l \bigl\| P_l \nabla_x \psi \bigr\|_{X_\infty^{0,\frac{1}{2}}} \\
  &\lesssim 2^k \bigl\| \nabla_{t,x} \chi \bigr\|_{L^\infty_t L^2_x} \|\psi\|_{\tilde{S}^1} 
 \end{align*}
 and the low-high case by
 \begin{align*}
  2^{2k} 2^{\frac{1}{2} j} \bigl\| P_{\leq k-C} Q_{\leq j-C} (\nabla_{t,x} \chi) \bigr\|_{L^\infty_t L^2_x} \bigl\| P_{k+O(1)} Q_{j+O(1)} \psi \bigr\|_{L^2_t L^2_x} \lesssim 2^k \bigl\| \nabla_{t,x} \chi \bigr\|_{L^\infty_t L^2_x} \bigl\| P_{k+O(1)} \nabla_x \psi \bigr\|_{X_\infty^{0,\frac{1}{2}}}.
 \end{align*}
 Thus, we obtain the following estimate for the large modulation term in \eqref{equ:multiplier_bounds_X_infty_small_freq_split_modulation},
 \begin{align*}
  \bigl\|P_k Q_{>0} \bigl( ( \nabla_{t,x} \chi ) \psi \bigr) \bigr\|_{X_\infty^{0,\frac{1}{2}}} &\lesssim 2^{2k} \bigl\| \nabla_{t,x}^2 \chi \bigr\|_{L^2_t L^{\frac{4}{3}}_x} \|\psi\|_{\tilde{S}^1} + 2^k \bigl\| \nabla_{t,x} \chi \bigr\|_{L^\infty_t L^2_x} \|\psi\|_{\tilde{S}^1} \\
  &\quad + \bigl\| \nabla_{t,x} \chi \bigr\|_{L^\infty_t L^2_x} \bigl\| P_{k+O(1)} \nabla_x \psi \bigr\|_{X_\infty^{0,\frac{1}{2}}}.
 \end{align*}
 If $k > 0$ in the first term on the right hand side of \eqref{equ:multiplier_bounds_splitting_X_infty}, we again split into a small and a large modulation term,
 \begin{equation}
  P_k \bigl( (\nabla_{t,x} \chi ) \psi \bigr) = P_k Q_{\leq k} \bigl( (\nabla_{t,x} \chi) \psi \bigr) + P_k Q_{> k} \bigl( (\nabla_{t,x} \chi) \psi \bigr).
 \end{equation}
 We can immediately dispose of the small modulation term as follows
 \begin{align*}
  \bigl\| P_k Q_{\leq k} \bigl( (\nabla_{t,x} \chi) \psi \bigr) \bigr\|_{X_\infty^{0, \frac{1}{2}}} &\lesssim 2^{\frac{1}{2} k} \bigl\| P_k \bigl( (\nabla_{t,x} \chi ) \psi \bigr) \bigr\|_{L^2_t L^2_x} \\
  &\lesssim 2^{-\frac{1}{2} k} \Bigl( \bigl\|\nabla_x \nabla_{t,x} \chi \bigr\|_{L^2_t L^4_x} \|\psi\|_{L^\infty_t L^4_x} + \bigl\| \nabla_{t,x} \chi \bigr\|_{L^2_t L^\infty_x} \|\nabla_x \psi\|_{L^\infty_t L^2_x} \Bigr) \\
  &\lesssim 2^{-\frac{1}{2} k} \Bigl( \bigl\|\nabla_x \nabla_{t,x} \chi \bigr\|_{L^2_t L^4_x} + \bigl\| \nabla_{t,x} \chi \bigr\|_{L^2_t L^\infty_x} \Bigr) \|\psi\|_{\tilde{S}^1}.
 \end{align*}
 To treat the large modulation term, we find that for any $j > k$,
 \begin{equation} \label{equ:multiplier_bounds_X_infty_large_freq_large_modulation}
  \begin{split}
   2^{\frac{1}{2} j} \bigl\| P_k Q_j \bigl( (\nabla_{t,x} \chi) \psi \bigr) \bigr\|_{L^2_t L^2_x} &\lesssim 2^{\frac{1}{2} j} \bigl\| P_k Q_j \bigl( \bigl( P_{>j-C} (\nabla_{t,x} \chi ) \bigr) \psi \bigr) \bigr\|_{L^2_t L^2_x} \\
   &\quad + 2^{\frac{1}{2} j} \bigl\| P_k Q_j \bigl( \bigl( P_{\leq j-C} Q_{>j-C} (\nabla_{t,x} \chi) \bigr) \psi \bigr) \|_{L^2_t L^2_x} \\
   &\quad + 2^{\frac{1}{2} j} \bigl\| P_k Q_j \bigl( \bigl( P_{\leq j-C} Q_{\leq j-C} (\nabla_{t,x} \chi) \bigr) Q_{j+O(1)} \psi \bigr) \bigr\|_{L^2_t L^2_x}.
  \end{split}
 \end{equation}
 We estimate the first term by
 \begin{align*}
  2^{\frac{1}{2} j} \bigl\| P_{> j-C} \nabla_{t,x} \chi \bigr\|_{L^2_t L^4_x} \|\psi\|_{L^\infty_t L^4_x} &\lesssim 2^{-\frac{1}{2} j} \bigl\| \nabla_x \nabla_{t,x} \chi \bigr\|_{L^2_t L^4_x} \| \nabla_x \psi\|_{L^\infty_t L^2_x} \\
  &\lesssim 2^{-\frac{1}{2}k} \bigl\| \nabla_x \nabla_{t,x} \chi \bigr\|_{L^2_t L^4_x} \|\psi\|_{\tilde{S}^1}.
 \end{align*}
 The second term on the right hand side of \eqref{equ:multiplier_bounds_X_infty_large_freq_large_modulation} is bounded by
 \begin{align*}
  2^{\frac{1}{2} j} \bigl\| P_{\leq j-C} Q_{> j-C} \nabla_{t,x} \chi \bigr\|_{L^2_t L^4_x} \|\psi\|_{L^\infty_t L^4_x} &\lesssim 2^{-\frac{1}{2}j} \bigl\|\nabla_{t,x} \partial_t \chi \bigr\|_{L^2_t L^4_x} \|\nabla_x \psi\|_{L^\infty_t L^2_x} \\
  &\lesssim 2^{-\frac{1}{2}k} \bigl\| \nabla_{t,x} \partial_t \chi \bigr\|_{L^2_t L^4_x} \|\psi\|_{\tilde{S}^1}.
 \end{align*}
 Using a Littlewood-Paley trichotomy we obtain the following estimate of the third term in \eqref{equ:multiplier_bounds_X_infty_large_freq_large_modulation}
 \[
  2^{-k} \Bigl( \bigl\| \nabla_{t,x} \chi \bigr\|_{L^\infty_t L^\infty_x} + \bigl\| \nabla_x^2 \nabla_{t,x} \chi \bigr\|_{L^\infty_t L^2_x} \Bigr) \|\psi\|_{\tilde{S}^1} + \bigl\| \nabla_{t,x} \chi \bigr\|_{L^\infty_t L^\infty_x} \bigl\| P_{k+O(1)} \nabla_x \psi \bigr\|_{X_\infty^{0,\frac{1}{2}}}.
 \]
 Finally, square-summing over $k \in \Z$ gives the desired bound for the first term on the right hand side of \eqref{equ:multiplier_bounds_splitting_X_infty}. The second term on the right hand side of \eqref{equ:multiplier_bounds_splitting_X_infty} can be handled similarly.
\end{proof}

\subsection{Localizing in physical space} \label{subsec:localizing_in_physical_space}

 In this subsection we consider Coulomb data $(A, \phi)[0] \in \dot{H}^s_x~\times~\dot{H}^{s-1}_x$ for all $s \geq 1$. We show that there exists $T>0$ and a $C^\infty$ solution of the same regularity class on each time slice of the space-time slab $[-T, T]\times \R^4$ satisfying  the required $S^1$ norm bound 
 \[
  \|(A, \phi)\|_{S^1([-T,T]\times \R^4)} < \infty. 
 \]
 To this end we fix a large $R_0>1$. For each $R \geq R_0$, we consider a cutoff $\chi_R \in C_c^\infty(\R^4)$ that equals $1$ on the ball $B_R(0)$ and has support in a dilate of $B_R(0)$. In the previous Subsection \ref{subsec:proof_of_joint_time_interval_prop} we demonstrated that upon writing 
 \[
  \tilde{A}_R := \chi_R A - \nabla_x\gamma_R
 \]
 for the spatial components of a new connection form $\tilde{A}_R$, where 
 \[
 \gamma_R = \Delta^{-1} \partial_j \big( \chi_R A^j \big),
 \]
 there is a way to pick the remaining data $\partial_t \tilde{A}_R(0)$ and $\tilde{\phi}_R[0]$, so that the corresponding data are all of class $H^{1+}_x \times L^{0+}_x$ and of Coulomb class. Thus, we obtain local solutions with these data from the local well-posedness result \cite{Selberg}. We can also arrange that $\tilde{\phi}_R[0]$ is supported within the ball of radius $\frac{R}{10}$ centered at the origin. It is then also easy to verify that 
 \[
 (\tilde{A}_R, \tilde{\phi}_R)[0] \rightarrow (A, \phi)[0] \text{ as } R \rightarrow \infty
 \]
 with respect to the $\dot{H}^s_x \times \dot{H}^{s-1}_x$ topology for any $s\geq 1$. Moreover, the argument in the previous subsection together with Proposition~\ref{prop:breakdown_criterion_frequency_envelopes} implies that these solutions extend of class $H^{s}_x \times H^{s-1}_x$ to a space-time slab $[-T, T]\times \R^4$, where $T>0$ is independent of $R \geq R_0$. It then remains to check that the corresponding local solutions on $[-T, T]\times \R^4$, call them again $ (\tilde{A}_R, \tilde{\phi}_R)$, converge with respect to the $S^1$ norm. This will essentially follow from the perturbation theory developed later on in the key Step 3 of the proof of Proposition \ref{prop:bootstrap}. The following proposition can be proved.
 \begin{prop} \label{prop:morelocaldatastuff} 
  The sequence $\bigl\{ (\tilde{A}_R, \tilde{\phi}_R) \bigr\}_{R \geq R_0}$ converges in $S^1([-T,T]\times\R^4)$ as $R \rightarrow \infty$. The limit is also of class $\dot{H}^s_x \times \dot{H}^{s-1}_x$ for all $s \geq 1$ on each time slice of $[-T, T]\times\R^4$, hence of class $C^\infty$, and a smooth solution to the MKG-CG system on $[-T, T]\times\R^4$ with initial data $(A,\phi)[0]$.
 \end{prop}
 \begin{proof}
  A sketch of the proof is given in Subsection \ref{subsec:interlude}.
 \end{proof}

\section{How to arrive at the minimal energy blowup solution} \label{sec:how_to_arrive_at}

In this section we address another delicate issue arising due to the difficulties with the perturbation theory for the MKG-CG system. Assume that $(A_n, \phi_n)$ is an ``essentially singular sequence'' of admissible solutions to the MKG-CG system that converges at time $t=0$ in the energy topology to a Coulomb energy class data pair $(A, \phi)[0]$ with $E(A,\phi) = E_{crit}$, 
\[
 \lim_{n \rightarrow \infty} (A_n, \phi_n)[0] = (A, \phi)[0].
\]
Using the concept of MKG-CG evolution for energy class data from Definition \ref{defn:energy_class_solution}, we obtain an energy class solution $(A, \phi)$ with maximal lifespan $I$. We then want to infer that
\begin{equation} \label{eq:minblow1}
 \sup_{J \subset I, J closed} \|(A, \phi)\|_{S^1(J \times \R^4)} = \infty,
\end{equation}
while by construction it holds that $E(A, \phi) = E_{crit}$. In view of Lemma \ref{lem:characerization_maximal_lifespan}, it suffices to consider the case $I = \R$. The problem here is that while we have 
\[
\lim_{n\rightarrow\infty} \|(A_n, \phi_n)\|_{S^1(I_n\times \R^4)} = \infty,
 \]
where $I_n$ are suitably chosen time intervals, we cannot use an immediate perturbative argument to obtain \eqref{eq:minblow1} as is possible for wave maps in \cite{KS}. The reason comes from the fact that the $(A_n, \phi_n)$ may have non-negligible low-frequency components. Nevertheless, we obtain the following result.

\begin{prop} \label{prop:minenblowup} 
Let $(A_n, \phi_n)$ be an essentially singular sequence of admissible solutions to the MKG-CG system. Assume that 
\[
 \lim_{n \rightarrow \infty} (A_n, \phi_n)[0] = (A,\phi)[0]
\]
in the energy topology for some Coulomb energy class data pair $(A, \phi)[0]$. Let $I$ be the maximal lifespan of the MKG-CG evolution $(A, \phi)$ of this data pair given by Definition \ref{defn:energy_class_solution}. Then it holds that
\[
 \sup_{J \subset I, J closed} \|(A, \phi)\|_{S^1(J\times\R^4)} = \infty.
\]
\end{prop}
\begin{proof}
 A sketch of the proof can be found in Subsection \ref{subsec:interlude}.
\end{proof}
 
We shall later on need certain variations of the preceding proposition. 

\begin{cor} \label{cor:lifespancont} 
 Let $\{ (A_n, \phi_n)[0] \}_{n \in \N}$ and $(A, \phi)[0]$ be Coulomb energy class data such that
 \[
  \lim_{n \rightarrow \infty} (A_n, \phi_n)[0] = (A, \phi)[0]
 \]
 in the energy topology and let $I$ be the maximal lifespan of the MKG-CG evolution of $(A,\phi)[0]$. If $J \subset I$ is a compact time interval, then it holds that
 \[
  \limsup_{n \rightarrow \infty} \|(A_n, \phi_n)\|_{S^1(J \times \R^4)} < \infty.
 \]
\end{cor}

This entails the following important corollary.

\begin{cor} \label{cor:lifespancompact}
 Let $\{ (A_n, \phi_n)[0] \}_{n \in \N} \subset \dot{H}^1_x \times L^2_x$ be a compact subset of Coulomb energy class data. Then there exists an open interval $I_\ast$ centered at $t=0$ with the property that 
 \[
  I_\ast \subset I_n \text{ for all } n \in \N,
 \]
 where $I_n$ denotes the maximal lifespan of the MKG-CG evolutions of $(A_n, \phi_n)[0]$ given by Definition~\ref{defn:energy_class_solution}.
\end{cor}
\begin{proof} 
 We argue by contradiction. Assume that there exists a subsequence $\{ (A_{n_k}, \phi_{n_k})[0] \}_{k \in \N}$ for which at least one of the lifespan endpoints of the associated MKG-CG evolutions converges to $t = 0$. Passing to a further subsequence, which we again denote by $\{ (A_{n_k}, \phi_{n_k})[0]\}_{k \in \N}$, we may assume that
 \[
  \lim_{k \rightarrow \infty} (A_{n_k}, \phi_{n_k})[0] = (A, \phi)[0]
 \]
 in the energy topology for some Coulomb energy class data $(A, \phi)[0]$. The contradiction now follows from Corollary \ref{cor:lifespancont}.
\end{proof}

\section{Concentration compactness step} \label{sec:concentration_compactness_step}

\subsection{General considerations} \label{subsec:general_considerations}

We begin by sorting out the relationship between the conserved energy and the $\dot{H}^1_x \times L^2_x$-norm of solutions $(A, \phi)$ to the MKG-CG system. Recall that the conserved energy is given by the expression
\[
E(A, \phi) = \frac{1}{4} \sum_{\alpha, \beta} \int_{\R^4} ( \partial_{\alpha} A_{\beta} - \partial_{\beta} A_{\alpha} )^2 \,dx + \frac{1}{2} \sum_{\alpha} \int_{\R^4} | \partial_{\alpha} \phi + i A_{\alpha} \phi |^2 \, dx.
\]
Using the Coulomb gauge condition, this can be written as 
\begin{align*}
&E(A, \phi) = \sum_{i < j} \int_{\R^4}(\partial_i A_j)^2 \, dx + \frac{1}{2} \sum_{i} \int_{\R^4} (\partial_t A_i)^2 + (\partial_i A_{0})^2 \, dx + \frac{1}{2} \sum_{\alpha} \int_{\R^4} |\partial_{\alpha} \phi + i A_{\alpha} \phi |^2 \, dx,
\end{align*}
which immediately implies
\[
E(A, \phi) \lesssim \sum_{i} \|\nabla_{t,x}A_i\|_{L_x^2}^2 + \|\nabla_x A_0\|_{L_x^2}^2 + \|\nabla_{t,x}\phi\|_{L_x^2}^2 + \sum_\alpha \|\nabla_x A_\alpha\|_{L^2_x}^2 \|\nabla_x \phi\|_{L^2_x}^2.
\]
Conversely, in order to exploit the conserved energy, we also need to show that the expression 
\[
\sum_{\alpha}\|\nabla_{t,x}A_{\alpha}\|_{L_x^2}^2 +  \|\nabla_{t,x}\phi\|_{L_x^2}^2
\]
is bounded in terms of $E(A, \phi)$. Here the only issue comes from bounding the terms $\|\nabla_{t,x}\phi\|_{L_x^2}$ and $\|\partial_t A_0\|_{L^2_x}$. However, the diamagnetic inequality gives the pointwise estimate 
\[
 |\partial_{\alpha} |\phi| | \leq |(\partial_\alpha + i A_{\alpha}) \phi |
\]
and Sobolev's inequality then yields
\[
 \|\phi\|_{L_x^4} \lesssim \|\nabla_x |\phi| \|_{L^2_x} \lesssim \sum_j \|(\partial_j + i A_j) \phi \|_{L_x^2}.
\]
Thus, we find
\begin{align*}
\|\partial_{\alpha} \phi \|_{L_x^2} &\leq  \|(\partial_\alpha + i A_{\alpha})\phi\|_{L_x^2} + \|A_{\alpha} \phi\|_{L_x^2} \\
&\lesssim  \|(\partial_\alpha + i A_{\alpha})\phi\|_{L_x^2} + \|A_\alpha\|_{L^4_x}^2 + \|\phi\|_{L^4_x}^2 \\
&\lesssim E(A,\phi)^{\frac{1}{2}} + E(A,\phi).
\end{align*}
In order to bound the time derivative $\| \partial_t A_0\|_{L^2_x}$, we use the compatibility relation
\[
 \Delta \partial_t A_0 = - \sum_j \partial_j \Im \bigl( \phi \overline{D_j \phi})
\]
to obtain
\[
 \|\partial_t A_0\|_{L_x^2} \lesssim \|\nabla_x \partial_t A_0 \|_{L^{\frac{4}{3}}_x} \lesssim \sum_{j} \| \phi \overline{D_j\phi} \|_{L_x^{\frac{4}{3}}} \lesssim \sum_j \|\phi\|_{L^4_x} \|D_j\phi\|_{L_x^2} \lesssim E(A, \phi).
\]

We also recall that the notation $(A,\phi)[0]$ for initial data for the MKG-CG system only refers to the prescribed data $A_j[0]$, $j = 1, \ldots, 4$, for the evolution of the spatial components of the connection form $A$. The component $A_0$ is determined via the compatibility relations.

\subsection{Setting up the induction on frequency scales} \label{subsec:setting_up_induction_on_freq_scales}

Our final goal will be to show the following. 

\medskip

{\it Let $(A,\phi)[0]$ be admissible Coulomb data. Then the corresponding MKG-CG evolution exists globally in time and denoting its energy by
\[
 E = \frac{1}{4} \sum_{\alpha,\beta} \int_{\R^4} (\partial_\alpha A_\beta - \partial_\beta A_\alpha)^2 \, dx + \frac{1}{2} \sum_\alpha \int_{\R^4} |\partial_\alpha + i A_\alpha \phi |^2 \, dx,
\]
there exists an increasing function $K\colon \R^+ \to \R^+$ such that 
\[
 \|(A,\phi)\|_{S^1(\R\times\R^4)} \leq K(E).
\]
}
To prove this result we proceed by contradiction. By the small data global well-posedness result \cite{KST} we know that the assertion holds for sufficiently small energies. So assume that it does not hold for all energies $E > 0$. Then the set of exceptional energies has a positive infimum, which we denote by $E_{crit}$, and we can find a sequence of admissible data $\{(A^n, \phi^n)[0]\}_{n\in\N}$ with evolutions $\{(A^n, \phi^n)\}_{n\in\N}$ defined on $(-T_0^n, T_1^n) \times \R^4$ such that
\begin{align*}
 &\lim_{n\rightarrow\infty} E(A^n, \phi^n) = E_{crit},  \\
 &\lim_{n\rightarrow\infty} \|(A^n, \phi^n)\|_{S^1((-T_0^n, T_1^n)\times \R^4)} = +\infty.
\end{align*}
We call such a sequence of initial data {\it essentially singular}.

\medskip

We now implement a two step Bahouri-G\'erard type procedure. The first step consists in selecting frequency atoms. Here we largely follow the setup of Subsection 9.1 and Subsection 9.2 in \cite{KS}, which in turn is partially based on Section III.1 of \cite{Bahouri-Gerard}. We recall the following terminology from \cite{Bahouri-Gerard}. A \emph{scale} is a sequence of positive numbers $\{\lambda^n\}_{n\in\N}$. We say that two scales $\{\lambda^{na}\}_{n\in\N}$ and $\{\lambda^{nb}\}_{n\in\N}$ are \emph{orthogonal} if
\[
 \lim_{n\to\infty} \frac{\lambda^{na}}{\lambda^{nb}} + \frac{\lambda^{nb}}{\lambda^{na}} = +\infty.
\]
Let $\{ \lambda^n \}_{n \in \N}$ be a scale and let $\{ (f^n, g^n) \}_{n \in \N}$ be a bounded sequence of functions in $\dot{H}^1_x(\R^4) \times L^2_x(\R^4)$. Then we say that $\{ (f^n, g^n) \}_{n \in \N}$ is \emph{$\lambda^n$-oscillatory} if
\[
 \lim_{R \to + \infty} \limsup_{n\to\infty} \, \Biggl( \int_{ \{\lambda^n |\xi| \leq \frac{1}{R} \} } |\hat{\nabla_x f^n}(\xi)|^2 + |\hat{g^n}(\xi)|^2 \, d\xi + \int_{ \{\lambda^n |\xi| \geq R\} } |\hat{\nabla_x f^n}(\xi)|^2 + |\hat{g^n}(\xi)|^2 \, d\xi \Biggr) = 0
\]
and we say that $\{ (f^n, g^n) \}_{n \in \N}$ is \emph{$\lambda^n$-singular} if for all $b > a > 0$,
\[
 \lim_{n \to \infty} \int_{ \{ a \leq \lambda^n |\xi| \leq b \} } |\hat{\nabla_x f^n}(\xi)|^2 + |\hat{g^n}(\xi)|^2 \, d\xi = 0.
\]

\medskip

We obtain the following decomposition of the essentially singular sequence of data $\{ (A^n, \phi^n)[0] \}_{n \in \N}$ into frequency atoms.
\begin{prop} \label{prop:selecting_frequency_atoms}
 Let $\{ (A^n, \phi^n)[0] \}_{n\in\N}$ be a sequence of admissible data with energy bounded by~$E$. Up to passing to a subsequence the following holds. Given $\delta > 0$, there exists an integer $\Lambda = \Lambda(\delta, E) > 0$ and for every $n \in \N$ a decomposition
 \begin{align*}
  A^n[0] &= \sum_{a=1}^\Lambda A^{na}[0] + A^n_\Lambda[0], \\
  \phi^n[0] &= \sum_{a=1}^\Lambda \phi^{na}[0] + \phi_{\Lambda}^n[0].
 \end{align*}
 For $a = 1, \ldots, \Lambda$, the frequency atoms $(A^{na}, \phi^{na})[0]$ are $\lambda^{na}$-oscillatory for a family of pairwise orthogonal frequency scales $\{ \lambda^{na} \}_n$. The error $(A_\Lambda^n, \phi_\Lambda^n)[0]$ is $\lambda^{na}$-singular for every $1 \leq a \leq \Lambda$ and satisfies the smallness condition
 \begin{align*}
  \limsup_{n \rightarrow \infty} \, \bigl\| A_\Lambda^n[0] \bigr\|_{\dot{B}_{2,\infty}^1 \times \dot{B}_{2,\infty}^0} < \delta, \quad \limsup_{n \rightarrow \infty} \, \bigl\| \phi_\Lambda^n[0] \bigr\|_{\dot{B}_{2,\infty}^1 \times \dot{B}_{2,\infty}^0} < \delta.
 \end{align*}
 Moreover, for $a = 1, \ldots, \Lambda$, the frequency atoms $(A^{na}, \phi^{na})[0]$ have sharp frequency support in the frequency intervals $|\xi| \in [ (\lambda^{na})^{-1} R_n^{-1}, (\lambda^{na})^{-1} R_n]$ for a suitable sequence $R_n \to +\infty$. For different values of $a$, these frequency intervals $[ (\lambda^{na})^{-1} R_n^{-1}, (\lambda^{na})^{-1} R_n]$ are mutually disjoint for sufficiently large~$n$. Finally, we have asymptotic decoupling of the energy
 \begin{equation*}
  E(A^n, \phi^n) = \sum_{a = 1}^\Lambda E(A^{na}, \phi^{na}) + E(A_\Lambda^n, \phi_\Lambda^n) + o(1) \quad \text{ as } n \rightarrow \infty,
 \end{equation*}
 where the temporal components $A_0^{na}$ are determined via the compatibility relation
 \[
  \bigl( \Delta - |\phi^{na}|^2 \bigr) A_0^{na} = - \Im \bigl( \phi^{na} \overline{\partial_t \phi^{na}} \bigr),
 \]
 and similarly for $A^n_{\Lambda, 0}$.
\end{prop}
\begin{proof}
 We suppress the notation [0] in the proof. As in Section III.1 of \cite{Bahouri-Gerard}, we obtain a decomposition of the data $\{ (A^n, \phi^n) \}_{n\in\N}$ into frequency atoms
 \begin{align*}
  A^n = \sum_{a=1}^\Lambda \tilde{A}^{na} + \tilde{A}_\Lambda^{n}, \quad \phi^n = \sum_{a=1}^\Lambda \tilde{\phi}^{na} + \tilde{\phi}^n_\Lambda,
 \end{align*}
 where $(\tilde{A}^{na}, \tilde{\phi}^{na})$ are $\lambda^{na}$-oscillatory for a family of pairwise orthogonal scales $\{\lambda^{na}\}_{n\in\N}$ for $a = 1, \ldots, \Lambda$. The error $(\tilde{A}_\Lambda^n, \tilde{\phi}_\Lambda^n)$ is $\lambda^{na}$-singular for $a = 1, \ldots, \Lambda$ and satisfies the smallness condition
 \begin{align*}
  \limsup_{n \rightarrow \infty} \, \bigl\| \tilde{A}_\Lambda^n \bigr\|_{\dot{B}_{2,\infty}^1 \times \dot{B}_{2,\infty}^0} < \delta, \quad \limsup_{n \rightarrow \infty} \, \bigl\| \tilde{\phi}_\Lambda^n \bigr\|_{\dot{B}_{2,\infty}^1 \times \dot{B}_{2,\infty}^0} < \delta.
 \end{align*} 
 In order to get a clean separation of the frequency atoms in frequency space, we have to prepare them a bit more, because their decay from the scale $(\lambda^{na})^{-1}$ might be arbitrarily slow. To this end let $R_n \rightarrow \infty$ be a sequence growing sufficiently slowly such that the intervals $[ (\lambda^{na})^{-1} R_n^{-1}, (\lambda^{na})^{-1} R_n ]$ are mutually disjoint for $n$ large enough and for different values of $a$. Then we replace the error $\tilde{A}_\Lambda^{n}$ by
 \[
  A_\Lambda^n = P_{\cap_{a'=1}^\Lambda [\mu^{na'} - \log R_n, \mu^{na'} + \log R_n]^c} \tilde{A}_\Lambda^{n} + \sum_{a=1}^\Lambda P_{\cap_{a' = 1}^\Lambda [\mu^{na'} - \log R_n, \mu^{na'} + \log R_n]^c} \tilde{A}^{na},
 \]
 where $\mu^{na} = -\log \lambda^{na}$, and the frequency atoms $\tilde{A}^{na}$ by
 \[
  A^{na} =  P_{[\mu^{na} - \log R_n, \mu^{na} + \log R_n]} \tilde{A}_\Lambda^n + P_{[\mu^{na}-\log R_n, \mu^{na} + \log R_n]} \sum_{a'=1}^\Lambda \tilde{A}^{na'}
 \]
 for $a = 1, \ldots, \Lambda$. In order to remove the dependence on $\Lambda$ in the new profiles, we may replace $\Lambda$ by $\Lambda_n$ with $\Lambda_n \to \infty$ sufficiently slowly as $n \to \infty$. Analogously, we define $\phi^{na}$ and $\phi^n_\Lambda$. This new decomposition
 \begin{equation} \label{equ:frequency_atom_decomposition}
  A^n = \sum_{a=1}^\Lambda A^{na} + A^n_\Lambda, \quad \phi^n = \sum_{a=1}^\Lambda \phi^{na} + \phi^n_\Lambda,
 \end{equation}
 has the same properties as the original one, but that we have now arranged for a sharp separation of the frequency supports of the frequency atoms. 

 Finally, we turn to the asymptotic decoupling of the energy. Here we recall that the ``elliptic components'' $A_0^{na}$ associated with a frequency atom $(A^{na}, \phi^{na})$ are determined via the elliptic compatibility equations. It therefore suffices to show that the decomposition \eqref{equ:frequency_atom_decomposition} (which only refers to the spatial components of the connection form $A^n$) implies a similar frequency atom decomposition
 \begin{equation} \label{equ:frequency_atom_decomposition_temporal}
  A_0^n = \sum_{a=1}^\Lambda A_0^{na} + A^n_{\Lambda, 0} + o_{\dot{H}^1_x}(1) \quad \text{ as } n \rightarrow \infty,
 \end{equation}
 where $A_0^{na}$ is $\lambda^{na}$-oscillatory and $A_{\Lambda,0}^n$ is $\lambda^{na}$-singular for each $a = 1, \ldots, \Lambda$. Then the decoupling of the energy is an immediate consequence of the construction of the frequency atoms. For example, we have the limiting relations
 \[
  \lim_{n\rightarrow\infty} \int_{\R^4} \partial_{\alpha}\phi^{na} \overline{A_{\alpha}^{na'}\phi^{na''}} \, dx = 0,
 \]
 if not all of $a, a', a''$ are equal, as well as
 \[
  \lim_{n\rightarrow\infty}\int_{\R^4}A_{\alpha}^{na}\phi^{na'}\overline{A_{\alpha}^{na''}\phi^{na'''}}\,dx = 0,
 \]
 if not all of $a, a', a'', a'''$ are equal. It remains to prove the decomposition \eqref{equ:frequency_atom_decomposition_temporal}. To show this, we first observe that (at fixed time $t = 0$)
 \[
  -\Im\big(\phi^n\overline{\partial_t\phi^n}\big) = \sum_{a=1}^{\Lambda}-\Im\big(\phi^{na}\overline{\partial_t\phi^{na}}\big) - \Im\big(\phi^{n}_\Lambda \overline{\partial_t \phi^n_\Lambda}\big) + o_{L^{\frac{4}{3}}_x}(1) \quad \text{ as } n \rightarrow \infty.
 \]
 It is then enough to show that 
 \begin{align*}
  &\big( \Delta  - |\phi^{n}|^2 \big) A^{na}_0 = -\Im\big(\phi^{na}\overline{\partial_t\phi^{na}}\big) + o_{L^{\frac{4}{3}}_x}(1) \quad \text{ as } n \rightarrow \infty, \\
  &\big( \Delta - |\phi^{n}|^2 \big) A^n_{\Lambda,0} = -\Im\big(\phi^{n}_{\Lambda,0} \overline{\partial_t \phi^n_{\Lambda, 0}}\big) + o_{L_x^{\frac{4}{3}}}(1) \quad \text{ as } n \rightarrow \infty.
 \end{align*}
 This in turn will easily follow once we have shown that each $A^{na}_0$ is $\lambda^{na}$-oscillatory, while $A^n_{\Lambda,0}$ is $\lambda^{na}$-singular for $a = 1, \ldots, \Lambda$. We demonstrate this for $a = 1$, where we may assume by scaling invariance of these assertions that $\lambda^{n1} = 1$ throughout. We start from the compatibility relation
 \begin{align*}
  \Delta A^{n1}_0 - |\phi^{n1}|^2A^{n1}_0 = -\Im\big(\phi^{n1}\overline{\partial_t\phi^{n1}}\big) 
 \end{align*}
 and distinguish between small and large frequencies.

We begin with the small frequencies. For $R \ll -1$ we write
\[
 \Delta P_{\leq R} A^{n1}_0 - P_{\leq R}\big( |\phi^{n1}|^2 A^{n1}_0 \big) = - P_{\leq R} \big(\Im\big(\phi^{n1} \overline{\partial_t \phi^{n1}} \big) \big),
\]
where we have
\[
 \lim_{R\rightarrow -\infty} \limsup_{n\to\infty} \big\|P_{\leq R} \big( \Im \big(\phi^{n1} \overline{\partial_t\phi^{n1}} \big) \big) \big\|_{L^{\frac{4}{3}}_x} = 0.
\]
Next, we split
\[
 P_{\leq R} \big( |\phi^{n1}|^2 A^{n1}_0 \big) = P_{\leq R} \big( P_{\leq \frac{R}{2}} \big( |\phi^{n1}|^2 \big) A^{n1}_0 \big) + P_{\leq R} \big( P_{> \frac{R}{2}} (|\phi^{n1}|^2) A^{n1}_0 \big).
\]
Then we have 
\[
\lim_{R \rightarrow -\infty} \limsup_{n\to\infty} \big\|P_{\leq \frac{R}{2}} |\phi^{n1}|^2 \big\|_{L^2_x} = 0,
\]
whence 
\[
 \lim_{R\rightarrow -\infty} \limsup_{n\to\infty} \big\| P_{\leq R} \big( P_{\leq \frac{R}{2}} \big(|\phi^{n1}|^2) A^{n1}_0 \big) \big\|_{L^{\frac{4}{3}}_x} = 0,
\]
while for the second term above, we obtain from Bernstein's inequality that
\begin{align*}
\big\| P_{\leq R} \big( P_{> \frac{R}{2}} \big( |\phi^{n1}|^2 \bigr) A^{n1}_0 \big) \big\|_{L^{\frac{4}{3}}_x} &\leq \sum_{ \substack{ k_1 = k_2 + O(1), \\ k_1 > \frac{R}{2}} } \big\|P_{\leq R} \big( P_{k_1} (|\phi^{n1}|^2) P_{k_2} A^{n1}_0 \big) \big\|_{L^{\frac{4}{3}}_x} \\
&\lesssim 2^R \sum_{ \substack{ k_1 = k_2 + O(1), \\ k_1 > \frac{R}{2}} } \big\| P_{k_1} |\phi^{n1}|^2 \big\|_{L^2_x} \big\|P_{k_2} A^{n1}_0\big\|_{L^2_x} \\
&\lesssim 2^{\frac{R}{2}} \big\| \phi^{n1} \big\|_{L^4_x}^2 \big\|A^{n1}_0\big\|_{\dot{H}^1_x}.
\end{align*}
We immediately conclude that 
\[
 \lim_{R\rightarrow - \infty} \limsup_{n\to\infty} \big\| P_{\leq R} \big( P_{> \frac{R}{2}} (|\phi^{n1}|^2) A^{n1}_0 \big) \big\|_{L^{\frac{4}{3}}_x} = 0. 
\]
Using Sobolev's inequality, it then follows that 
\begin{align*}
 \limsup_{n\to\infty} \big\|P_{\leq R} A^{n1}_0 \big\|_{\dot{H}^1_x} \lesssim \limsup_{n\to\infty} \big\|P_{\leq R} \big( \Im\big(\phi^{n1} \overline{\partial_t\phi^{n1}} \big)\big)\big\|_{L^{\frac{4}{3}}_x} + \limsup_{n\to\infty} \big\|P_{\leq R} \big( |\phi^{n1}|^2A^{n1}_0 \big) \big\|_{L^{\frac{4}{3}}_x} \to 0
\end{align*}
as $R \to - \infty$. Next, we consider the large frequencies. For $R \gg 1$, we write 
\[
 \Delta P_{> R} A^{n1}_0 - P_{>R} \big( |\phi^{n1}|^2 A^{n1}_0 \big) = - P_{>R} \big( \Im \big(\phi^{n1} \overline{\partial_t\phi^{n1}} \big)\big),
\]
where we have 
\[
\lim_{R\rightarrow\infty} \limsup_{n\to\infty} \big\|P_{>R}\big( \Im\big(\phi^{n1}\overline{\partial_t\phi^{n1}}\big)\big) \big\|_{L^{\frac{4}{3}}} = 0.
\]
Then we split
\begin{align*}
 P_{>R} \big( |\phi^{n1}|^2 A^{n1}_0 \big) = P_{>R} \big( P_{> \frac{R}{2}} ( |\phi^{n1}|^2 ) A^{n1}_0 \big) + P_{>R} \big( P_{\leq \frac{R}{2}} (|\phi^{n1}|^2) A^{n1}_0 \big).
\end{align*}
By frequency localization we have 
\[
 \lim_{R\rightarrow\infty} \limsup_{n\to\infty} \big\|P_{>\frac{R}{2}} (|\phi^{n1}|^2) \big\|_{L^2_x} = 0,
\]
and thus,
\[
 \lim_{R\rightarrow\infty} \limsup_{n\to\infty} \big\|P_{>R} \big( P_{> \frac{R}{2}} (|\phi^{n1}|^2) A^{n1}_0 \big) \big\|_{L^{\frac{4}{3}}_x} = 0.
\]
On the other hand, we have 
\begin{align*}
 \big\|P_{>R} \big( P_{\leq \frac{R}{2}} (|\phi^{n1}|^2) A^{n1}_0 \big) \big\|_{L^{\frac{4}{3}}_x} &\lesssim \sum_{ \substack{ k_1 = k_2 + O(1), \\ k_1 > R }} \big\| P_{k_1} \big( P_{\leq \frac{R}{2}} (|\phi^{n1}|^2) P_{k_2}A^{n1}_0 \big) \big\|_{L^{\frac{4}{3}}_x} \\
 &\leq \sum_{ \substack{ k_1 = k_2 + O(1), \\ k_1 > R }} \big\| P_{\leq \frac{R}{2}} (|\phi^{n1}|^2) \big\|_{L^4_x} \big\| P_{k_2} A^{n1}_0\big\|_{L^2_x} \\
 &\lesssim \sum_{k_2>R} 2^{\frac{R}{2} - k_2} \bigl\| \phi^{n1} \bigr\|^2_{L^4_x} \big\|P_{k_2} A^{n1}_0 \big\|_{\dot{H}^1_x} \\
 &\lesssim 2^{-\frac{R}{2}} \bigl\| \phi^{n1} \bigr\|^2_{L^4_x} \big\|A^{n1}_0\big\|_{\dot{H}^1_x}
\end{align*}
and hence, 
\begin{align*}
 \lim_{R \to \infty} \limsup_{n\to\infty} \big\|P_{>R} \big( P_{\leq \frac{R}{2}} (|\phi^{n1}|^2) A^{n1}_0 \big) \big\|_{L^{\frac{4}{3}}_x} = 0.
\end{align*}
We then conclude from Sobolev's inequality that 
\[
\lim_{R\rightarrow \infty} \limsup_{n\to\infty} \big\|P_{>R} A^{n1}_0\big\|_{\dot{H}^1_x} = 0. 
\]
\end{proof}

Given an essentially singular sequence of initial data, by Proposition \ref{prop:selecting_frequency_atoms} for any $\delta > 0$ we obtain another essentially singular sequence $\{(A^n, \phi^n)[0]\}_{n \in \N}$ of the form
\begin{equation} \label{equ:decomposition_ess_singular_data}
 \begin{split}
  A^n[0] &= \sum_{a=1}^\Lambda A^{na}[0] + A^n_\Lambda[0], \\
  \phi^n[0] &= \sum_{a=1}^\Lambda \phi^{na}[0] + \phi^n_\Lambda[0]
 \end{split}
\end{equation}
with
\[
 \limsup_{n \rightarrow \infty} \, \bigl\|A_\Lambda^n[0]\bigr\|_{\dot{B}^1_{2,\infty} \times \dot{B}^0_{2,\infty}} < \delta, \quad \limsup_{n \rightarrow \infty} \, \bigl\|\phi_\Lambda^n[0]\bigr\|_{\dot{B}^1_{2,\infty} \times \dot{B}^0_{2,\infty}} < \delta.
\]
{\it Eventually, we will prove that necessarily only one frequency atom $(A^{na}, \phi^{na})[0]$ in the decomposition \eqref{equ:decomposition_ess_singular_data} is non-trivial and has to be asymptotically of energy $E_{crit}$. In fact, the subsequent considerations will show that if there are at least two frequency atoms $(A^{n1}, \phi^{n1})[0], (A^{n2}, \phi^{n2})[0]$ that both do not vanish asymptotically, or if there is only one frequency atom $(A^{n1}, \phi^{n1})[0]$ with the error satisfying
\[
 \limsup_{n\rightarrow\infty} \, \|(A^n_1, \phi^n_1)[0]\|_{\dot{H}^1_x \times L^2_x} > 0,
\]
then we get an a priori bound on the $S^1$ norm of the evolutions
\[
 \liminf_{n\rightarrow\infty} \, \|(A^n, \phi^n)\|_{S^1((-T_0^n, T_1^n)\times\R^4)} < \infty,
\]
contradicting the assumption that $\{(A^n, \phi^n)[0]\}_{n\in\N}$ is essentially singular. }

\medskip

We now introduce a smallness parameter $\varepsilon_0 > 0$ that will eventually be chosen sufficiently small depending only on $E_{crit}$. In particular, we assume that $\varepsilon_0$ is less than the small energy threshold of the small energy global well-posedness result \cite{KST}. 

By passing to a suitable subsequence and by renumbering the frequency atoms, if necessary, we may assume that for some integer $\Lambda_0 > 0$,
\[
 \sum_{a \geq \Lambda_0 + 1} \limsup_{n \rightarrow \infty} E(A^{na}, \phi^{na}) < \varepsilon_0.
\]
Moreover, we may assume that the frequency atoms $\{(A^{na}, \phi^{na})[0]\}_{n\in\N}$, $a = 1, \ldots, \Lambda_0$, have ``increasing frequency supports'' in the sense that $(\lambda^{na})^{-1}$ is growing in terms of $a$ (for each fixed $n$). The key idea now is as follows.
 
\medskip

{\it We approximate the initial data $(A^n, \phi^n)[0]$ by low frequency truncations, obtained by removing all or some of the atoms $(A^{na}, \phi^{na})[0]$, $a = 1, \ldots, \Lambda_0$, and inductively obtain bounds on the $S^1$ norm of the MKG-CG evolutions of the truncated data. As this induction stops after $\Lambda_0$ many steps, we will have obtained the desired contradiction, forcing eventually that there has to be exactly one frequency atom $(A^{n1}, \phi^{n1})[0]$ that is asymptotically of energy $E_{crit}$.}

\subsection{Evolving the ``non-atomic'' lowest frequency approximation} \label{subsec:evolving_non_atomic}

From now on we suppress the notation $[0]$ for the initial data. The errors $(A^n_{\Lambda_0}, \phi^n_{\Lambda_0})$ in the decomposition \eqref{equ:decomposition_ess_singular_data} are by construction supported away in frequency space from the frequency scales $(\lambda^{na})^{-1}$, $a = 1,2,\ldots, \Lambda_0$. It is then clear that the errors $\{(A^{n}_{\Lambda_0}, \phi^{n}_{\Lambda_0})\}_{n\in\N}$ can be written as the sum of $\Lambda_0 + 1$ pieces, which correspond to the $\Lambda_0+1$ shells that the frequency space gets split into by the frequency supports of the atoms $(A^{na}, \phi^{na})$. Thus, we can write 
\begin{equation} \label{equ:decomposition_error_term}
A^{n}_{\Lambda_0} = \sum_{j=1}^{\Lambda_0+1}A^{nj}_{\Lambda_0}, \quad \phi^{n}_{\Lambda_0} = \sum_{j=1}^{\Lambda_0+1}\phi^{nj}_{\Lambda_0},
\end{equation}
where the first pieces $(A^{n1}_{\Lambda_0}, \phi^{n1}_{\Lambda_0})$ have Fourier support in the region closest to the origin, i.e. in 
\[
|\xi| \leq (\lambda^{n1})^{-1} (R_n)^{-1}.
\]
In other words, one essentially obtains the ``lowest frequency approximations'' $(A_{\Lambda_0}^{n1}, \phi_{\Lambda_0}^{n1})$ by removing all the atoms $(A^{na}, \phi^{na})$, $a = 1, \ldots, \Lambda_0$, from the data.

\medskip
 
We then start our grand inductive procedure by showing the following proposition.
\begin{prop} \label{prop:lowfreq}
 The parameter $\varepsilon_0 > 0$ can be chosen sufficiently small depending only on the size of $E_{crit}$ such that the following holds. Constructing the lowest frequency approximations $\{ (A^{n1}_{\Lambda_0}, \phi^{n1}_{\Lambda_0}) \}_{n\in\N}$ as described in \eqref{equ:decomposition_error_term}, then there exists a constant $C(E_{crit}) > 0$ such that for all sufficiently large $n$, the data given by $(A^{n1}_{\Lambda_0}, \phi^{n1}_{\Lambda_0})$ can be evolved globally in time and the corresponding solution satisfies 
 \[
  \|(A^{n1}_{\Lambda_0}, \phi^{n1}_{\Lambda_0})\|_{S^1(\R\times\R^4)} \leq C(E_{crit}). 
 \]
\end{prop}
\begin{proof}
The idea is to use a finite number of further low frequency approximations of $\{ (A^{n1}_{\Lambda_0}, \phi^{n1}_{\Lambda_0}) \}_{n\in\N}$ and to inductively obtain bounds on the $S^1$ norms of their evolutions. Here it is essential that the number of these further approximations is bounded by $C_1(E_{crit}) > 0$ and thus independent of the choices already made. Picking these low frequency approximations requires a somewhat delicate construction involving a further frequency atom decomposition of $\{ (A^{n1}_{\Lambda_0}, \phi^{n1}_{\Lambda_0}) \}_{n\in\N}$. To begin with, for some sufficiently small $\delta_1 = \delta_1(E_{crit}) > 0$, in particular $\delta_1 \ll \varepsilon_0$, we use decompositions 
\begin{align*}
 A^{n1}_{\Lambda_0} &= \sum_{j=1}^{\Lambda_1(\delta_1)}A^{n1(j)}_{\Lambda_0} + A^{n1}_{\Lambda_0(\Lambda_1)}, \\ 
 \phi^{n1}_{\Lambda_0} &= \sum_{j=1}^{\Lambda_1(\delta_1)}\phi^{n1(j)}_{\Lambda_0} + \phi^{n1}_{\Lambda_0(\Lambda_1)}, 
\end{align*}
where the frequency atoms $(A^{n1(j)}_{\Lambda_0}, \phi^{n1(j)}_{\Lambda_0})$ have frequency support in mutually disjoint intervals 
 \[
 [ (\lambda^{n1(j)})^{-1} (R_{n1}^{(j)})^{-1}, (\lambda^{n1(j)})^{-1} R_{n1}^{(j)} ]
\]
with $R_{n1}^{(j)} \to \infty$ as $n \to \infty$, and furthermore, we have the bound 
\[
\limsup_{n\rightarrow\infty} \, \bigl\{ \|A^{n1}_{\Lambda_0(\Lambda_1)}\|_{\dot{B}^1_{2,\infty} \times \dot{B}^0_{2,\infty}} + \|\phi^{n1}_{\Lambda_0(\Lambda_1)} \|_{\dot{B}^1_{2,\infty} \times \dot{B}^0_{2,\infty}} \bigr\} < \delta_1.
\]
We may again assume that the atoms $(A^{n1(j)}_{\Lambda_0}, \phi^{n1(j)}_{\Lambda_0})$ have increasing frequency support as $j$ increases. The number of frequency atoms $\Lambda_1(\delta_1)$ here is potentially extremely large. It is crucial that the number of steps, i.e. the number of low frequency approximations of $\{ (A_{\Lambda_0}^{n1}, \phi_{\Lambda_0}^{n1}) \}_{n\in\N}$, required in the inductive procedure is in fact much smaller, of size $C_1 = C_1(E_{crit}) \ll \Lambda_1(\delta_1)$. As we shall see, $C_1$ can be chosen independently of $\delta_1$ and $\Lambda_1(\delta_1)$. We now pick the low frequency approximations of the data $\{ (A_{\Lambda_0}^{n1}, \phi_{\Lambda_0}^{n1}) \}_{n\in\N}$. For $\varepsilon_0$ fixed as before, we inductively construct $O( {\textstyle \frac{E_{crit}}{\varepsilon_0} } )$ closed frequency intervals $\tilde{J}_l$ for the variable $|\xi|$, disjoint up the the endpoints and increasing. The chosen intervals will also depend on $n$, but for notational ease we do not indicate this. So consider $n$ and $\Lambda_1$ fixed now. Having picked the intervals $\tilde{J}_1 = (-\infty, b_1]$, $\tilde{J}_l = [a_l, b_l]$ with $b_{l-1} = a_l$ for $l = 2, \ldots, L-1$, we pick an interval $[a_L, \tilde{b}_L]$ with $a_L = b_{L-1}$ as follows.  First, pick $\tilde{b}_{L}$ in such a fashion that 
\[
 E(P_{[a_L, \tilde{b}_{L}]} A^{n1}_{\Lambda_0}, P_{[a_L, \tilde{b}_L]} \phi^{n1}_{\Lambda_0}) = \varepsilon_0
\]
or else, if this is impossible, then pick $\tilde{b}_{L} = \log \, (\lambda^{n1})^{-1} - \log R_n$, i.e. pick the upper endpoint of the frequency interval containing the lowest frequency ``large atom'' $(A^{n1}, \phi^{n1})$. Now, in the former case assume that 
\[
  \tilde{b}_L \in [ \log \, ( \lambda^{n1(j)} )^{-1} - \log R_{n1}^{(j)}, \log \, ( \lambda^{n1(j)} )^{-1} + \log R_{n1}^{(j)} ]
\]
for some $1 \leq j \leq \Lambda_1(\delta_1)$, i.e. $\tilde{b}_L$ falls within the frequency support of one of the (finite number of) ``small frequency atoms'' $( A^{n1(j)}_{\Lambda_0}, \phi^{n1(j)}_{\Lambda_0} )$ constituting $(A^{n1}_{\Lambda_0}, \phi^{n1}_{\Lambda_0})$. Then we shift $\tilde{b}_{L}$ upwards to coincide with the upper limit, that is, we set
\[
 b_L = \log \, (\lambda^{n1(j)})^{-1} + \log R_{n1}^{(j)}. 
\]
Otherwise, we set
\[
 b_L = \tilde{b}_L.
\]
Then we define the interval $\tilde{J}_L = [a_L, b_L]$. Observe that for sufficiently large $n$, we have 
\[
 E(P_{\tilde{J}_L} A^{n1}_{\Lambda_0}, P_{\tilde{J}_L} \phi^{n1}_{\Lambda_0}) \lesssim \varepsilon_0.
\]
In particular, this implies that for sufficiently large $n$ the total number of intervals $\tilde{J}_l$ is $C_1 = O( {\textstyle \frac{E_{crit}}{\varepsilon_0} } )$. We now define the low frequency approximations of the data $(A_{\Lambda_0}^{n1}, \phi_{\Lambda_0}^{n1})$ by truncating the frequency support of $(A^{n1}_{\Lambda_0}, \phi^{n1}_{\Lambda_0})$ to the intervals
\[
 J_L := \cup_{l=1}^L \tilde{J}_l.
\]
More precisely, for $1 \leq L \leq C_1$ we define the $L$-th low frequency approximation of the data $(A^{n1}_{\Lambda_0}, \phi^{n1}_{\Lambda_0})$ by the expression 
\[
 (P_{J_L} A^{n1}_{\Lambda_0}, P_{J_L} \phi^{n1}_{\Lambda_0}), 
\]
where by construction $C_1 = C_1(E_{crit}) \lesssim \frac{E_{crit}}{\varepsilon_0}$. In particular, we have
\[
 (P_{J_{C_1}} A^{n1}_{\Lambda_0}, P_{J_{C_1}} \phi^{n1}_{\Lambda_0}) = (A^{n1}_{\Lambda_0}, \phi^{n1}_{\Lambda_0}).
\]
We also state the following key lemma, whose proof is a consequence of the preceding construction.
\begin{lem} \label{lem:asymptoevacuate}
 For $L = 1, \ldots, C_1$ and for any $R>0$, we have for all sufficiently large $n$ that
 \[
  \big\|P_{[a_L - R, a_L + R]} \nabla_{t,x} A^{n1}_{\Lambda_0} \big\|_{L^2_x} + \big\|P_{[a_L - R, a_L + R]} \nabla_{t,x} \phi^{n1}_{\Lambda_0} \big\|_{L^2_x} \lesssim R \delta_1.
 \]
\end{lem}

In order to prove Proposition \ref{prop:lowfreq}, we inductively show that for $L = 1, \ldots, C_1$ and for all sufficiently large $n$, the evolutions of the data 
\[
 (P_{J_{L}} A^{n1}_{\Lambda_0}, P_{J_{L}} \phi^{n1}_{\Lambda_0})
\]
exist globally and satisfy the desired global $S^1$ norm bounds, which of course get larger as $L$ grows. For $L = 1$ this is a direct consequence of the small energy theory. The main work now goes into proving the following proposition.

\begin{prop} \label{prop:bootstrap}
Let us assume that the evolution of the data 
\[
 \bigl( P_{J_{L-1}} A^{n1}_{\Lambda_0}, P_{J_{L-1}} \phi^{n1}_{\Lambda_0} \bigr)
\]
is globally defined for some $1 \leq L < C_1$. We denote this evolution by $(A^{n1, (L-1)}_{\Lambda_0}, \phi^{n1, (L-1)}_{\Lambda_0})$. Furthermore, assume that for all sufficiently large $n$, it holds that
\[
 \big\|\bigl( A^{n1, (L-1)}_{\Lambda_0}, \phi^{n1, (L-1)}_{\Lambda_0} \bigr) \big\|_{S^1(\R\times\R^4)} \leq C_2 < \infty.
\]
Provided $\delta_1^{-1} \gg C_2$ with $\delta_1 > 0$ as above, there exists $C_3 = C_3(C_2) < \infty$ such that for all sufficiently large $n$, the data 
\[
 \bigl(P_{J_L} A^{n1}_{\Lambda_0}, P_{J_L} \phi^{n1}_{\Lambda_0}\bigr)
\]
can be evolved globally and for the corresponding evolutions $\bigl( A^{n1, (L)}_{\Lambda_0}, \phi^{n1, (L)}_{\Lambda_0} \bigr)$, it holds that
\[
 \big\| \bigl( A^{n1, (L)}_{\Lambda_0}, \phi^{n1, (L)}_{\Lambda_0} \bigr) \big\|_{S^1(\R\times\R^4)} \leq C_3.
\]
\end{prop}

Proposition \ref{prop:lowfreq} is then an immediate consequence of applying Proposition \ref{prop:bootstrap} $C_1$ many times. We note that there exists $\delta_{11} > 0$ depending only on $E_{crit}$ such that choosing $\delta_1 < \delta_{11}$ in each step, Proposition \ref{prop:bootstrap} can be applied. Since $C_1 = C_1(E_{crit})$ this results in a bound
\[
 \big\| (A^{n1}_{\Lambda_0}, \phi^{n1}_{\Lambda_0}) \big\|_{S^1(\R\times\R^4)} \leq C(E_{crit}). 
\]
\end{proof}

\begin{proof}[Proof of Proposition \ref{prop:bootstrap}] We proceed in several steps. 
 
\medskip

\noindent {\bf Step 1.} {\it The assumed bound on $\big\|\bigl( A^{n1, (L-1)}_{\Lambda_0}, \phi^{n1, (L-1)}_{\Lambda_0} \bigr)\big\|_{S^1(\R\times\R^4)}$ implies an exponential decay for large frequencies,
\[
\big\|\bigl( P_k A^{n1, (L-1)}_{\Lambda_0}, P_k\phi^{n1, (L-1)}_{\Lambda_0} \bigr)\big\|_{S^1(\R\times\R^4)} \lesssim 2^{-\sigma(k - b_{L-1} )} \text{ for } k \geq b_{L-1}.
\]
}
This will follow once we can show that in fact 
\[
\big\| \bigl( P_k A^{n1, (L-1)}_{\Lambda_0}, P_k \phi^{n1, (L-1)}_{\Lambda_0} \bigr)\big\|_{S^1(\R\times\R^4)} \lesssim c_k^{(L-1)},
\]
where $\bigl\{ c_k^{(L-1)} \bigr\}_{k\in\Z}$ is a sufficiently flat frequency envelope covering the initial data $\bigl( A^{n1, (L-1)}_{\Lambda_0}, \phi^{n1, (L-1)}_{\Lambda_0} \bigr)[0]$ at time $t = 0$. This in turn is a consequence of Proposition \ref{prop:breakdown_criterion} whose proof will be given in Subsection~\ref{subsec:interlude}. 

\medskip

\noindent {\bf{Step 2.}} {\it Localizing $\bigl( A^{n1, (L-1)}_{\Lambda_0}, \phi^{n1, (L-1)}_{\Lambda_0} \bigr)$ to suitable space-time slices}. In order to ensure that we can induct on perturbations of size $\sim \varepsilon_0$ that are not ``too small'' (such as the $\delta_1$), we have to make sure that the $S^1$ norms of $\bigl( A^{n1, (L-1)}_{\Lambda_0}, \phi^{n1, (L-1)}_{\Lambda_0} \bigr)$ are not too large. To simplify the notation, we label these components by $(A, \phi)$ for the rest of this step. The idea is to localize to suitable space-time slices $I \times \R^4$, whose number may be very large (depending on $\|(A, \phi)\|_{S^1(\R\times\R^4)}$ and $E_{crit}$), but such that we have on each slice
\[
 \|(A, \phi)\|_{S^1(I\times\R^4)} \leq C(E_{crit}),
\] 
where the function $C(\cdot)$ grows at most polynomially. 

\begin{prop} \label{prop:decomposition}
 There exist $N = N\bigl(\|(A,\phi)\|_{S^1(\R\times\R^4)}, E_{crit} \bigr)$ many time intervals $I_1, \ldots, I_N$ partitioning the time axis $\R$ such that we have for $n = 1, \ldots, N$ a decomposition (referring to the spatial components of the connection form simply by $A$)
 \begin{equation}
  A|_{I_n} = A^{free, (I_n)} + A^{nonlin, (I_n)}, \quad \Box A^{free, (I_n)} = 0,
 \end{equation}
 where $A^{free, (I_n)}$ and $A^{nonlin, (I_n)}$ are in Coulomb gauge and satisfy
 \begin{align}
  \bigl\| \nabla_{t,x} A^{free, (I_n)} \bigr\|_{L^\infty_t L^2_x(\R \times\R^4)} &\lesssim E_{crit}^{1/2}, \label{equ:A_free_energy_bound}\\
  \bigl\| A^{nonlin, (I_n)} \bigr\|_{\ell^1 S^1(I_n \times \R^4)} &\ll 1. \label{equ:A_nonlin_S_1_norm_bound_short_interval}
 \end{align}
 Moreover, we have for $n = 1, \ldots, N$ that 
 \begin{equation} \label{equ:phi_S_1_norm_bound_short_interval}
  \|\phi\|_{S^1(I_n\times\R^4)} \lesssim C(E_{crit}),
 \end{equation}
 where $C(\cdot)$ grows at most polynomially.
\end{prop}
\begin{proof}
We first define precisely the decompositions $A = A^{free} + A^{nonlin}$ that we are using. The nonlinear structure inherent in $A^{nonlin}$ will be pivotal for controlling the equation for $\phi$. For a time interval $I \subset \R$, say of the form $I = [t_0, t_1]$ for some $t_0 < t_1$, we define for $i = 1, \ldots, 4$,
\begin{equation} \label{equ:definition_A_nonlin_I}
 A^{nonlin, (I)}_i := - \chi_I \sum_{k,j} \Box^{-1} P_k Q_j \Im {\mathcal P}_i \bigl( (\chi_I \phi) \cdot \nabla_x (\chi_I \overline{\phi}) -  \chi_I i A |\phi|^2 \bigr),
\end{equation}
where $\chi_I$ is a smooth cutoff to the interval $I$ and $\Box^{-1}$ denotes multiplication by the Fourier symbol. Then we define $A^{free, (I)}$ to be the free wave with initial data at time $t_0$ given by $A[t_0] - A^{nonlin, (I)}[t_0]$. By construction, we then have
\[
 A = A^{free, (I)} + A^{nonlin, (I)} \text{ on } I \times \R^4.
\]
We now describe how to partition the time axis into $N = N \bigl(\|(A, \phi)\|_{S^1}, E_{crit} \bigr)$ many suitable time intervals so that the bounds \eqref{equ:A_free_energy_bound} -- \eqref{equ:phi_S_1_norm_bound_short_interval} hold on each such interval. For this, we first need the following technical lemma. 
\begin{lem} \label{lem:nonlinear_structure_small_short_interval} 
 Given $\varepsilon > 0$, there exist $M = M \bigl(\|(A, \phi)\|_{S^1(\R\times\R^4)}, \varepsilon \bigr)$ many time intervals $I_1, \ldots, I_M$ partitioning the time axis $\R$ such that for $m = 1, \ldots, M$ and $i = 1, \ldots, 4$,
 \begin{equation} \label{equ:nonlinear_structure_ell_1_L_2_small_short_interval}
  \sum_k \, \Bigl\| \nabla_{t,x} \sum_j \Box^{-1} P_k Q_j {\mathcal P}_i \bigl( (\chi_{I_m} \phi) \cdot \nabla_x (\chi_{I_m} \overline{\phi}) - \chi_{I_m} i A |\phi|^2 \bigr) \Bigr\|_{L^\infty_t L^2_x(\R\times\R^4)} \lesssim \varepsilon
 \end{equation}
 and 
 \begin{equation} \label{equ:nonlinear_structure_ell_1_N_small_short_interval}
  \bigl\| {\mathcal P}_i \bigl( (\chi_{I_m} \phi) \cdot \nabla_x (\chi_{I_m} \overline{\phi}) - \chi_{I_m} i A |\phi|^2 \bigr) \bigr\|_{(\ell^1 N \cap \ell^1 L^2_t \dot{H}^{-\frac{1}{2}}_x)(\R\times\R^4)} \lesssim \varepsilon.
 \end{equation}
 In particular, it then holds that
 \begin{align}
  \bigl\| \nabla_{t,x} A^{free, (I_m)}_i \bigr\|_{L^\infty_t L^2_x(\R\times\R^4)} &\lesssim E_{crit}^{1/2} + \varepsilon, \label{equ:A_free_L_infty_t_L_2_x} \\
  \bigl\| \nabla_{t,x} A_i^{nonlin, (I_m)} \bigr\|_{L^\infty_t L^2_x(\R\times\R^4)} &\lesssim \varepsilon, \label{equ:A_nonlin_L_infty_t_L_2_x} \\
  \bigl\| A_i^{nonlin, (I_m)} \bigr\|_{\ell^1 S^1(I_m\times\R^4)} &\lesssim \varepsilon. \label{equ:A_nonlin_ell_1_S_1}
 \end{align}
\end{lem}
\begin{proof}
 We begin with the quadratic interaction term in \eqref{equ:nonlinear_structure_ell_1_L_2_small_short_interval} and show that the time axis $\R$ can be partitioned into $M_1 = M_1 \bigl(\|(A, \phi)\|_{S^1}, \varepsilon \bigr)$ many intervals so that on each such interval $I$, it holds that
 \begin{equation} \label{equ:nonlinear_structure_small_short_interval_quadratic_term}
  \sum_k \, \Bigl\| \nabla_{t,x} \sum_j \Box^{-1} P_k Q_j \mathcal{P}_i \bigl( (\chi_{I} \phi) \cdot \nabla_x (\chi_{I} \overline{\phi}) \bigr)  \Bigr\|_{L_t^\infty L_x^2(\R\times\R^4)} \lesssim \varepsilon.
 \end{equation}
 To this end we exploit that there is an inherent null form in the above expression
 \[
  {\mathcal P}_i \bigl( \phi \cdot \nabla_x \overline{\phi} \bigr) = \Delta^{-1} \nabla^r {\mathcal N}_{ir}( \phi, \overline{\phi} ),
 \]
 where 
 \[
  {\mathcal N}_{ir}(\phi, \psi) = (\partial_i \phi) (\partial_r \psi) - (\partial_r \phi) (\partial_i \psi).
 \] 
 We first prove that on suitable intervals $I$,
 \begin{equation} \label{equ:nonlinear_structure_small_short_interval_quadratic_term_H_k}
  \sum_k \sum_{j \leq k+C} \bigl\| \nabla_{t,x} \Box^{-1} P_k Q_j \Delta^{-1} \nabla^r {\mathcal N}_{ir} \bigl( Q_{\leq j-C} (\chi_{I} \phi), Q_{\leq j-C}(\chi_{I} \overline{\phi}) \bigr) \bigr\|_{L_t^\infty L_x^2(\R\times\R^4)} \lesssim \varepsilon.
 \end{equation}
 By a Littlewood-Paley trichotomy we may reduce to the case where both inputs are at frequency $\sim 2^k$. The singular operator $\Box^{-1}$ costs $2^{-j-k}$, so we need to recover the factor $2^{-j}$. From the null form we gain $2^{\frac{j-k}{2}}$, while the inclusion $ Q_j L^2_t L^2_x \hookrightarrow L_t^\infty L_x^2$ gains another $2^{\frac{j}{2}}$. Finally, we obtain a small power in $j-k$ from the improved Bernstein estimate $P_k Q_j L_t^2 L_x^{\frac{3}{2}} \hookrightarrow 2^{\frac{2}{3} k} 2^{\frac{1}{6}(j-k)} L^2_t L^2_x$ (by interpolating with the $X^{s,b}$ version of the Strichartz estimate $P_k Q_j L^2_t L^2_x \hookrightarrow 2^{\frac{4}{3} k} 2^{\frac{1}{2}(j-k)} L^2_t L^6_x$) and that $L^2_t L^6_x \cdot L^\infty_t L^2_x \hookrightarrow L^2_t L^{\frac{3}{2}}_x$. Thus, we find
 \begin{align*}
  &\sum_k \sum_{j \leq k+C} \bigl\| \nabla_{t,x} \Box^{-1} P_k Q_j \Delta^{-1} \nabla^r {\mathcal N}_{ir} \bigl( Q_{\leq j-C} (\chi_{I} \phi_k), Q_{\leq j-C}(\chi_{I} \overline{\phi}_k) \bigr) \bigr\|_{L_t^\infty L_x^2} \\
  &\quad \quad \lesssim \biggl( \sum_k \Bigl( 2^{-\frac{5}{6} k} \| \chi_I \nabla_x \phi_k \|_{L^2_t L^6_x} \Bigr)^2 \biggr)^{1/2} \|\phi\|_{S^1}
 \end{align*}
 and smallness follows from divisibility of the $L^2_t L^6_x(\R\times\R^4)$ norm. Next, we show that on suitable intervals $I$, it holds that
 \begin{equation} \label{equ:nonlinear_structure_small_short_interval_quadratic_term_large_modulation_output}
  \sum_k \sum_{j > k+C} \bigl\| \nabla_{t,x} \Box^{-1} P_k Q_j \Delta^{-1} \nabla^r {\mathcal N}_{ir} \bigl( \chi_I \phi, \chi_I \overline{\phi} \bigr) \bigr\|_{L^\infty_t L^2_x(\R\times\R^4)} \lesssim \varepsilon.
 \end{equation}
 By a Littlewood-Paley trichotomy we may again reduce to the case where both inputs are at frequency $\sim 2^k$. Then we obtain, using the Bernstein inequality both in time and space, that
 \begin{align*}
  &\sum_k \sum_{j > k+C} \bigl\| \nabla_{t,x} \Box^{-1} P_k Q_j \Delta^{-1} \nabla^r {\mathcal N}_{ir} \bigl( \chi_I \phi_k, \chi_I \overline{\phi}_k \bigr) \bigr\|_{L^\infty_t L^2_x} \lesssim \biggl( \sum_k \Bigl( 2^{-\frac{5}{6} k} \bigl\| \chi_I \nabla_x \phi_k \bigr\|_{L^2_t L^6_x} \Bigr)^2 \biggr)^{1/2} \|\phi\|_{S^1}
 \end{align*}
 and smallness follows from the divisibility of the $L^2_t L^6_x(\R\times\R^4)$ norm. In view of \eqref{equ:nonlinear_structure_small_short_interval_quadratic_term_H_k} and \eqref{equ:nonlinear_structure_small_short_interval_quadratic_term_large_modulation_output}, in order to finish the proof of \eqref{equ:nonlinear_structure_small_short_interval_quadratic_term} we may assume that one of the two inputs has the leading modulation. It therefore suffices to show that on suitable intervals $I$ we have bounds of the form
 \begin{equation}
  \sum_{k, j} \, \Bigl\| \nabla_{t,x} \Box^{-1} P_k Q_{\leq j-C} \Delta^{-1} \nabla^r {\mathcal N}_{ir} \bigl( Q_j (\chi_I \phi), Q_{\leq j-C} (\chi_I \overline{\phi}) \bigr) \Bigr\|_{L^\infty_t L^2_x(\R\times\R^4)} \lesssim \varepsilon,
 \end{equation}
 where we use the convention $\Box^{-1} P_k Q_{\leq j-C} = \sum_{l \leq j-C} \Box^{-1} P_k Q_{l}$. Using that
 \[
  \bigl( \Box^{-1}P_kQ_{<j-C} F \bigr)(t, \cdot) = - \int_t^\infty \frac{\sin((t-s)|\nabla|)}{|\nabla|} (P_k Q_{<j-C} F)(s, \cdot) \, ds,
 \]
 it is enough to show
 \[
  \sum_{k, j} \, \Bigl\| P_k Q_{\leq j-C} \Delta^{-1} \nabla^r {\mathcal N}_{ir} \bigl( Q_j (\chi_I \phi), Q_{\leq j-C} (\chi_I \overline{\phi}) \bigr) \Bigr\|_{L^1_t L^2_x(\R\times\R^4)} \lesssim \varepsilon.
 \]
 By estimate (143) in \cite{KST} we may reduce to the case where $j = k + O(1)$ and both inputs are at frequency $\sim 2^k$. Then we find
 \begin{align*}
  \sum_k \, \Bigl\| P_k Q_{\leq k-C} \Delta^{-1} \nabla^r {\mathcal N}_{ir} \bigl( Q_k (\chi_I \phi_k), Q_{\leq k-C} (\chi_I \overline{\phi}_k) \bigr) \Bigr\|_{L^1_t L^2_x} &\lesssim \sum_k \bigl\| \chi_I \nabla_x \phi_k \bigr\|_{X_\infty^{0, \frac{1}{2}}} 2^{-\frac{5}{6} k} \bigl\| \chi_I \nabla_x \phi_k \bigr\|_{L^2_t L^6_x} \\
  &\lesssim \|\phi\|_{S^1} \biggl( \sum_k \Bigl( 2^{-\frac{5}{6} k} \bigl\| \chi_I \nabla_x \phi_k \bigr\|_{L^2_t L^6_x} \Bigr)^2 \biggr)^{1/2}
 \end{align*}
 and smallness follows by divisibility. Next, we consider the cubic term in \eqref{equ:nonlinear_structure_ell_1_L_2_small_short_interval}. Here we have to prove that on suitable intervals $I$ it holds that
 \begin{equation} \label{equ:nonlinear_structure_small_short_interval_cubic_term}
  \sum_k \, \Bigl\| \nabla_{t,x} \sum_j \Box^{-1} P_k Q_j \mathcal{P}_i \bigl( \chi_I A |\phi|^2 \bigr)  \Bigr\|_{L_t^\infty L_x^2(\R\times\R^4)} \lesssim \varepsilon.
 \end{equation}
 By similar arguments as above, this reduces to showing
 \[
  \sum_k \, \bigl\| P_k \bigl( \chi_I A |\phi|^2 \bigr) \bigr\|_{L^1_t L^2_x(\R\times\R^4)} \lesssim \varepsilon,
 \]
 which follows from estimate (64) in \cite{KST} and a divisibility argument. We note that the bound \eqref{equ:nonlinear_structure_ell_1_L_2_small_short_interval} implies that the estimates \eqref{equ:A_free_L_infty_t_L_2_x} and \eqref{equ:A_nonlin_L_infty_t_L_2_x} hold on each such interval $I$.

 \medskip

 It remains to choose the intervals so that the bound \eqref{equ:nonlinear_structure_ell_1_N_small_short_interval} also holds. The energy estimate for the $S_k$ and $N_k$ spaces together with the bounds \eqref{equ:nonlinear_structure_ell_1_L_2_small_short_interval} and \eqref{equ:nonlinear_structure_ell_1_N_small_short_interval} then also imply the bound \eqref{equ:A_nonlin_ell_1_S_1}. We pick $M_1\bigl(\|(A, \phi)\|_{S^1}, \varepsilon \bigr)$ many time intervals $I_m$, $m = 1, \ldots, M_1,$ on which the bound \eqref{equ:nonlinear_structure_ell_1_L_2_small_short_interval} already holds. We show that, if necessary, each time interval $I_m$ can be subdivided into $M_2 = M_2 \bigl(\|(A, \phi)\|_{S^1}, \varepsilon \bigr)$ many intervals $I_{ma}$, $a = 1, \ldots, M_2$, such that we have 
 \begin{equation} \label{equ:nonlinear_structure_ell_1_N_small_short_interval_proof}
  \bigl\| {\mathcal P}_i \bigl( (\chi_{I_{ma}} \phi) \cdot \nabla_x (\chi_{I_{ma}} \overline{\phi}) - \chi_{I_{ma}} i A |\phi|^2 \bigr) \bigr\|_{(\ell^1 N \cap \ell^1 L^2_t \dot{H}^{-\frac{1}{2}}_x)(\R\times\R^4)} \lesssim \varepsilon.
 \end{equation}
 For the rest of the proof of \eqref{equ:nonlinear_structure_ell_1_N_small_short_interval_proof} we denote an interval $I_{ma}$ just by $I$ and say that it is of the form $I = [t_0, t_1]$ for some $t_0 < t_1$. We only outline how to make the left hand side of \eqref{equ:nonlinear_structure_ell_1_N_small_short_interval_proof} small in $\ell^1 N$ for suitable intervals $I$, the $\ell^1 L^2_t \dot{H}^{-\frac{1}{2}}_x$ component being easier. We first estimate the quadratic interaction term in \eqref{equ:nonlinear_structure_ell_1_N_small_short_interval_proof},
 \[
  \sum_k \, \bigl\| {\mathcal P}_i \bigl( (\chi_{I} \phi) \cdot \nabla_x (\chi_{I} \overline{\phi}) \bigr) \bigr\|_{N_k} = \sum_k \, \bigl\| P_k \Delta^{-1} \nabla^r {\mathcal N}_{ir} \bigl( \chi_I \phi, \chi_I \overline{\phi} \bigr) \bigr\|_{N_k}.
 \]
 By (131) in \cite{KST}, it suffices to consider the case where both inputs are at frequency $\sim 2^k$ and have angular separation $\sim 1$ in Fourier space,
 \begin{equation} \label{equ:nonlinear_structure_quadratic_term_reduced}
  \sum_k \, \bigl\| P_k \Delta^{-1} \nabla^r {\mathcal N}_{ir} \bigl( \chi_I \phi_k, \chi_I \overline{\phi}_k \bigr)' \bigr\|_{N_k}.
 \end{equation}
 Here, the prime indicates the angular separation. We split into high and low modulation output.
 \begin{align*}
  \sum_k \, \bigl\| P_k \Delta^{-1} \nabla^r {\mathcal N}_{ir} \bigl( \chi_I \phi_k, \chi_I \overline{\phi}_k \bigr)' \bigr\|_{N_k} &\leq \sum_k \, \bigl\| P_k Q_{> k-C} \Delta^{-1} \nabla^r {\mathcal N}_{ir} \bigl( \chi_I \phi_k, \chi_I \overline{\phi}_k \bigr)' \bigr\|_{N_k} \\
  &\quad + \sum_k \, \bigl\| P_k Q_{\leq k-C} \Delta^{-1} \nabla^r {\mathcal N}_{ir} \bigl( \chi_I \phi_k, \chi_I \overline{\phi}_k \bigr)' \bigr\|_{N_k}.
 \end{align*}
 The term with high modulation output is estimated by
 \begin{align*}
  \sum_k \bigl\| P_k Q_{>k-C} \Delta^{-1} \nabla^r {\mathcal N}_{ir} \bigl( \chi_I \phi_k, \chi_I \overline{\phi}_k \bigr)' \bigr\|_{N_k} &\lesssim \sum_k 2^{-\frac{3}{2} k} \bigl\| \chi_I \nabla_x \phi_k \bigr\|_{L^2_t L^\infty_x} \bigl\| \chi_I \nabla_x \phi_k \bigr\|_{L^\infty_t L^2_x} \\
  &\lesssim \biggl( \sum_k \Bigl( 2^{-\frac{3}{2} k} \bigl\| \chi_I \nabla_x \phi_k \bigr\|_{L^2_t L^\infty_x} \Bigr)^2 \biggr)^{1/2} \|\phi\|_{S^1}
 \end{align*}
 and can be made small on suitable intervals $I$ using the divisibility of the quantity
 \[
  \sum_k \Bigl( 2^{-\frac{3}{2} k} \bigl\| \nabla_x \phi_k \bigr\|_{L^2_t L^\infty_x(\R\times\R^4)} \Bigr)^2 \lesssim \|\phi\|_{S^1(\R\times\R^4)}^2.
 \]
 For the term with low modulation output we note that the angular separation of the inputs allows us to write schematically
 \begin{align*}
  P_k Q_{\leq k-C} \Delta^{-1} \nabla^r {\mathcal N}_{ir} \bigl( \chi_I \phi_k, \chi_I \overline{\phi}_k \bigr)' &= P_k Q_{\leq k-C} \Delta^{-1} \nabla^r {\mathcal N}_{ir} \bigl( Q_{> k-C} ( \chi_I \phi_k ), \chi_I \overline{\phi}_k \bigr)' \\
  &\quad + P_k Q_{\leq k-C} \Delta^{-1} \nabla^r {\mathcal N}_{ir} \bigl( Q_{\leq k-C} (\chi_I \phi_k), Q_{>k-C}(\chi_I \overline{\phi}_k) \bigr)'.
 \end{align*}
 Then we estimate
 \begin{align*}
  \sum_k \bigl\| P_k Q_{\leq k-C} \Delta^{-1} \nabla^r {\mathcal N}_{ir} \bigl( Q_{> k-C} (\chi_I \phi_k), \chi_I \overline{\phi}_k \bigr)' \bigr\|_{N_k} &\lesssim \sum_k 2^{-k} \bigl\| Q_{> k-C} \bigl( \chi_I \nabla_x \phi_k \bigr) \bigr\|_{L^2_t L^2_x} \bigl\| \chi_I \nabla_x \phi_k \bigr\|_{L^2_t L^\infty_x} \\
  &\lesssim \|\phi\|_{S^1} \biggl( \sum_{k} \Bigl( 2^{-\frac{3}{2} k} \bigl\| \chi_I \nabla_x \phi_k \bigr\|_{L^2_t L^\infty_x} \Bigr)^2 \biggr)^{1/2}
 \end{align*}
 and similarly for the other term. Smallness follows as before by divisibility. The cubic interaction term in \eqref{equ:nonlinear_structure_ell_1_N_small_short_interval_proof} is much simpler to treat, it can be made small on suitable intervals $I$ using estimate (64) from \cite{KST} and divisibility of the $L^2_t \dot{W}^{\frac{1}{6}, 6}_x$ norm.
 \end{proof}

 It remains to prove that the bound \eqref{equ:phi_S_1_norm_bound_short_interval} in the statement of Proposition \ref{prop:decomposition} holds. For $\varepsilon > 0$ to be fixed sufficiently small further below, depending only on the size of $\|(A,\phi)\|_{S^1}$ and $E_{crit}$, there exist $M\bigl(\|(A,\phi)\|_{S^1}, \varepsilon \bigr)$ many intervals $I_m$, $m = 1, \ldots, M$, partitioning the time axis $\R$ on which the conclusion of Lemma \ref{lem:nonlinear_structure_small_short_interval} holds. We pick such an interval $I_m$ and now show that, if necessary, it can be subdivided into $M_3 \bigl(\|(A, \phi)\|_{S^1}, E_{crit} \bigr)$ many intervals $I_{ma}$, $a = 1, \ldots, M_3,$ such that 
 \[
  \|\phi\|_{S^1(I_{ma}\times\R^4)} \leq C(E_{crit}),
 \]
 where $C(\cdot)$ grows at most polynomially. Upon renumbering the intervals $I_{ma}$, we will then have finished the proof of Proposition \ref{prop:decomposition}.

 \medskip

 For the remainder of the proof, we denote an interval $I_{ma}$ just by $I$ and assume that it is of the form $I = [t_0, t_1]$ for some $t_0 < t_1$. From the equation $\Box_A \phi = 0$ and the decomposition $A|_{I} = A^{free, (I)} + A^{nonlin, (I)}$ provided by Lemma \ref{lem:nonlinear_structure_small_short_interval}, we conclude that on $I \times \R^4$ it holds that
 \begin{equation} \label{equ:Box_A_on_short_interval}
  \begin{split}
   \Box_{A^{free, (I)}}^p \phi &= - 2i \sum_k \bigl(P_{>k-C} A_j^{free, (I)}\bigr) P_k \partial^j \phi - 2i A_j^{nonlin, (I)} \partial^j \phi + 2i A_0 \partial_t \phi + i (\partial_t A_0) \phi + A_\alpha A^\alpha \phi \\
   &\equiv {\mathcal M}_1 + {\mathcal M}_2,
  \end{split}
 \end{equation}
 where we use the notation
 \begin{align*}
  \Box_{A^{free, (I)}}^p \phi &= \Box \phi + 2i \sum_k \bigl( P_{\leq k-C} A_j^{free, (I)} \bigr) P_k \partial^j \phi, \\
  {\mathcal M}_1 &= - 2i \sum_k \bigl(P_{>k-C} A_j^{free, (I)}\bigr) P_k \partial^j \phi - 2i A_j^{nonlin, (I)} \partial^j \phi + 2i A_0 \partial_t \phi, \\
  {\mathcal M}_2 &= i (\partial_t A_0) \phi + A_\alpha A^{\alpha} \phi.
 \end{align*}
 We further split the term ${\mathcal M}_1$ into
 \[
  {\mathcal M}_1 \equiv \sum_k {\mathcal N} \bigl( P_{>k-C} A^{free, (I)}, P_k \phi \bigr) + {\mathcal N} \bigl( A^{nonlin, (I)}, \phi \bigr) + {\mathcal N}_0 \bigl( A_0, \phi \bigr).
 \]
 Since $A^{free, (I)}$ and $A^{nonlin, (I)}$ are in Coulomb gauge, we observe that the terms ${\mathcal N}\bigl( P_{>k-C} A^{free, (I)}, P_k \phi \bigr)$ and ${\mathcal N}\bigl( A^{nonlin, (I)}, \phi\bigr)$ exhibit a null structure,
 \begin{align*}
  {\mathcal N}\bigl( P_{>k-C} A^{free, (I)}, P_k \phi \bigr) &= - 2i \sum_{j,r} {\mathcal N}_{jr} \bigl( \Delta^{-1} \nabla_j P_{>k-C} A^{free, (I)}_r, P_k \phi \bigr), \\
  {\mathcal N}\bigl( A^{nonlin, (I)}, \phi \bigr) &= - 2i \sum_{j,r} {\mathcal N}_{jr} \bigl( \Delta^{-1} \nabla_j A_r^{nonlin, (I)}, \phi \bigr).
 \end{align*}
 We emphasize that the right hand side of \eqref{equ:Box_A_on_short_interval} is defined on the whole space-time, but which only coincides with $\Box_{A^{free, (I)}}^p \phi$ on $I \times \R^4$. Using the linear estimate \eqref{equ:linear_estimates_magnetic_wave_equation} for the  magnetic wave operator $\Box_{A^{free, (I)}}^p$ and working with suitable Schwartz extensions, we obtain that 
 \begin{align*}
  \|\phi\|_{S^1(I\times\R^4)} &\lesssim \|\nabla_{t,x} \phi(t_0)\|_{L^2_x} + \bigl\| \chi_I \bigl( {\mathcal M}_1 + {\mathcal M}_2 \bigr) \bigr\|_{N \cap \ell^1 L^2_t \dot{H}^{-\frac{1}{2}}_x(\R\times\R^4)} \\
  &\lesssim E_{crit} + \bigl\| \chi_I \bigl( {\mathcal M}_1 + {\mathcal M}_2 \bigr) \bigr\|_{N \cap \ell^1 L^2_t \dot{H}^{-\frac{1}{2}}_x(\R\times\R^4)}.
 \end{align*}
 We note that by Theorem \ref{thm:linear_estimates_magnetic_wave_equation}, the implicit constant in the above estimate for the magnetic wave operator depends polynomially on $\|\nabla_{t,x} A^{free, (I)}\|_{L^\infty_t L^2_x}$ and we have $\|\nabla_{t,x} A^{free, (I)}\|_{L^\infty_t L^2_x} \lesssim E_{crit}^{1/2}$ by Lemma \ref{lem:nonlinear_structure_small_short_interval}. In order to prove the bound \eqref{equ:phi_S_1_norm_bound_short_interval}, it therefore suffices to show that we can choose the intervals $I$ such that
 \[
  \bigl\| {\mathcal M}_1 + {\mathcal M}_2 \bigr\|_{N \cap \ell^1 L^2_t \dot{H}^{-\frac{1}{2}}_x(I\times\R^4)} \lesssim E_{crit}.
 \]
 Our general strategy to achieve this consists in first using the off-diagonal decay in the multilinear estimates from \cite{KST} to reduce to a situation in which a suitable divisibility argument works.

 \medskip

 We only outline how to obtain smallness of the term ${\mathcal M}_1$ in $N(I\times\R^4)$, the estimate of ${\mathcal M}_1$ in $\ell^1 L^2_t \dot{H}^{-\frac{1}{2}}_x$ and of ${\mathcal M}_2$ in $N \cap \ell^1 L^2_t \dot{H}^{-\frac{1}{2}}_x$ being easier. We begin with the first term in the definition of ${\mathcal M}_1$,
 \[
  \bigl\| \sum_k {\mathcal N} \bigl( P_{>k-C} A^{free, (I)}, P_k \phi \bigr) \bigr\|_{N(I \times \R^4)}.
 \]
 From the estimate (131) in \cite{KST}, we conclude that it suffices to bound the expression
 \begin{equation} \label{equ:decomposition_A_free_phi_term_divisibility}
  \sum_{k_1} \bigl\| P_{k_1} {\mathcal N} \bigl( P_{k_2} A^{free, (I)}, P_{k_3} \phi \bigr)' \bigr\|_{N_{k_1}(I\times\R^4)}^2,
 \end{equation}
 where $k_2 = k_3 = k_1 + O(1)$ and both inputs have angular separation $\sim 1$. Similarly to the estimate of \eqref{equ:nonlinear_structure_quadratic_term_reduced}, we bound this term by
 \[
  \biggl( \sum_k \Bigl( 2^{-\frac{3}{2} k} \bigl\| \chi_I P_k \nabla_x A^{free, (I)} \bigr\|_{L^2_t L^\infty_x} \Bigr)^2 \biggr) \, \|\phi\|_{S^1}^2
 \]
 and a divisibility argument then yields smallness. To deal with the other two terms in ${\mathcal M}_1$, we need to achieve
 \[
  \bigl\| {\mathcal N}\bigl(A^{nonlin, (I)}, \phi \bigr) + {\mathcal N}_0(A_0, \phi) \bigr\|_{N(I\times\R^4)} \lesssim E_{crit}
 \]
 on suitable intervals $I$. To this end we will make similar reductions as in Section 4 of \cite{KST}, peeling off the ``good parts'' of ${\mathcal N}\bigl(A^{nonlin, (I)}, \phi\bigr)$ and of ${\mathcal N}_0\bigl( A_0, \phi \bigr)$ until we are left with three quadrilinear null form bounds.

 We introduce the expressions
 \[
  {\mathcal N}^{lowhi}\bigl( A^{nonlin, (I)}, \phi \bigr) = \sum_k {\mathcal N}\bigl( P_{\leq k-C} A^{nonlin, (I)}, P_k \phi \bigr)
 \]
 and 
 \[
  {\mathcal H}^\ast {\mathcal N}^{lowhi}(A^{nonlin, (I)}, \phi) = \sum_k \sum_{k' \leq k-C} \sum_{j \leq k' + C} Q_{\leq j-C} {\mathcal N}\bigl( Q_j P_{k'} A^{nonlin, (I)}, Q_{\leq j-C} P_k \phi \bigr).
 \] 
 By estimate (53) in \cite{KST}, we have
 \begin{equation} \label{equ:N_minus_N_lowhi}
  \bigl\| {\mathcal N}\bigl( A^{nonlin, (I)}, \phi \bigr) - {\mathcal N}^{lowhi}\bigl( A^{nonlin, (I)}, \phi \bigr) \bigr\|_{N} \lesssim \bigl\| A^{nonlin, (I)} \bigr\|_{S^1} \|\phi\|_{S^1}
 \end{equation}
 and by estimate (54) in \cite{KST}, it holds that
 \begin{equation} \label{equ:N_lowhi_minus_H_ast_N_lowhi}
  \bigl\| {\mathcal N}^{lowhi}\bigl( A^{nonlin, (I)}, \phi \bigr) - {\mathcal H}^\ast {\mathcal N}^{lowhi}\bigl( A^{nonlin, (I)}, \phi \bigr) \bigr\|_{N} \lesssim \bigl\| A^{nonlin, (I)} \bigr\|_{\ell^1 S^1} \|\phi\|_{S^1}.
 \end{equation}
 Fixing $\varepsilon > 0$ sufficiently small, depending only on the size of $\|(A, \phi)\|_{S^1}$ and $E_{crit}$, Lemma \ref{lem:nonlinear_structure_small_short_interval} ensures that $\| A^{nonlin, (I)} \|_{\ell^1 S^1}$ is small enough so that the right hand sides of \eqref{equ:N_minus_N_lowhi} and \eqref{equ:N_lowhi_minus_H_ast_N_lowhi} are bounded by $E_{crit}$. We now define
 \[
  {\mathcal H} A_i^{nonlin, (I)} := - \chi_I \sum_{  \substack{k_1, k_2, k \\ k \leq \min\{k_1, k_2\}-C } } \sum_{j \leq k+C} \Box^{-1} P_k Q_j \Im {\mathcal P}_i \bigl( Q_{\leq j-C} (\chi_I \phi_{k_1}) \cdot \nabla_x Q_{\leq j-C} (\chi_I \overline{\phi}_{k_2}) \bigr).
 \]
 By estimate (55) in \cite{KST} it holds that
 \[
  \bigl\| {\mathcal H}^\ast {\mathcal N}^{lowhi}\bigl( A^{nonlin, (I)} - {\mathcal H} A^{nonlin, (I)}, \phi \bigr) \bigr\|_{N} \lesssim \bigl\| A^{nonlin, (I)} - {\mathcal H} A^{nonlin, (I)} \bigr\|_{Z} \|\phi\|_{S^1},
 \]
 so we have to make $\bigl\| A^{nonlin, (I)} - {\mathcal H} A^{nonlin, (I)} \bigr\|_{Z}$ small. We recall the definition of the $Z$ space,
 \[
   \|\phi\|_{Z} = \sum_k \|P_k \phi\|_{Z_k}, \quad \|\phi\|_{Z_k}^2 = \sup_{l < C} \sum_\omega 2^l \|P_l^\omega Q_{k+2l} \phi \|_{L^1_t L^\infty_x}^2.
 \]
 Using estimate (134) in \cite{KST} and that one obtains an extra gain for very negative $l$ when estimating in the $Z$ space, we are reduced to bounding 
 \[
  \sum_k \Box^{-1} P_k Q_{k+O(1)} \Delta^{-1} \nabla_x {\mathcal N} \bigl(\chi_I \phi_k, \chi_I \overline{\phi}_k \bigr).
 \]
 We easily find that
 \begin{align*}
  &\sum_k \, \bigl\| \Box^{-1} P_k Q_{k+O(1)} \Delta^{-1} \nabla_x {\mathcal N} \bigl(\chi_I \phi_k, \chi_I \overline{\phi}_k \bigr) \bigr\|_{L^1_t L^\infty_x} \lesssim \biggl( \sum_k \Bigl( 2^{-\frac{5}{6} k} \bigl\| \chi_I \nabla_x \phi_k \bigr\|_{L^2_t L^6_x} \Bigr)^2 \biggr)^{1/2} \|\phi\|_{S^1},
 \end{align*}
 which can be made small by a divisibility argument. We are thus left with the term
 \[
  {\mathcal H}^\ast {\mathcal N}^{lowhi}\bigl( {\mathcal H} A^{nonlin, (I)}, \phi \bigr).
 \]

 \medskip

 Carrying out similar reductions as in Section 4 of \cite{KST} for the ``elliptic term'' ${\mathcal N}_0(A_0,\phi)$, we arrive at the key remaining term
 \[
  {\mathcal H}^\ast {\mathcal N}_0^{lowhi} \bigl( {\mathcal H} A^{(I)}_0, \phi \bigr),
 \]
 where
 \[
  {\mathcal H}^\ast {\mathcal N}^{lowhi}_0 \bigl( {\mathcal H} A^{(I)}_0, \phi \bigr) = \sum_k \sum_{k' \leq k-C} \sum_{j \leq k' + C} Q_{\leq j-C} {\mathcal N}_0 \bigl( Q_j P_{k'} {\mathcal H} A^{(I)}_0, Q_{\leq j-C} P_k \phi \bigr)
 \]
 and
 \[
  {\mathcal H} A^{(I)}_0 := - \chi_I \sum_{\substack{k_1, k_2, k \\ k \leq \min \{k_1, k_2\} -C}} \sum_{j \leq k+C} \Delta^{-1} P_k Q_j \Im \bigl( Q_{\leq j-C} (\chi_I \phi_{k_1}) \cdot Q_{\leq j-C} \partial_t (\chi_I \overline{\phi}_{k_2}) \bigr).
 \]
 As in \cite{KST}, we combine the ``hyperbolic term'' ${\mathcal H}^\ast {\mathcal N}^{lowhi}\bigl( {\mathcal H} A^{nonlin, (I)}, \phi \bigr)$ and the preceding ``elliptic term'' ${\mathcal H}^\ast {\mathcal N}_0^{lowhi} \bigl( {\mathcal H} A^{(I)}_0, \phi \bigr)$ and wind up with the null forms (61) -- (63) in \cite{KST}. We formulate these as quadrilinear expressions as in \cite{KST} and then prove that smallness can be achieved for each of these. 

 \medskip

 \noindent {\it First null form ((61) in \cite{KST}).} By estimate (148) in \cite{KST}, it suffices to consider the following two cases. First, we show that
 \begin{align*}
  &\sum_k \sum_{ k_1 = k_2 + O(1) } \bigl| \bigl\langle \Box^{-1} P_k Q_j \bigl( Q_{\leq j-C} (\chi_I \phi_{k_1}) \cdot \partial_\alpha Q_{\leq j-C} (\chi_I \phi_{k_2}) \bigr), P_k Q_j \bigl( \partial^\alpha Q_{\leq j-C} \phi_{k_3} \cdot Q_{\leq j-C} \psi_{k_4} \bigr) \bigr\rangle \bigr| \\
  &\quad \quad \ll \| \psi \|_{N^\ast},
 \end{align*}
 where $k_1 > k + C$, $j = k + O(1)$ and $k_3 = k_4 + O(1) = k + O(1)$. Second, we prove that
 \begin{align*}
  &\sum_k \sum_{ k_3 = k_4 + O(1) } \bigl| \bigl\langle \Box^{-1} P_k Q_j \bigl( Q_{\leq j-C} (\chi_I \phi_{k_1}) \cdot \partial_\alpha Q_{\leq j-C} (\chi_I \phi_{k_2}) \bigr), P_k Q_j \bigl( \partial^\alpha Q_{\leq j-C} \phi_{k_3} \cdot Q_{\leq j-C} \psi_{k_4} \bigr) \bigr\rangle \bigr| \\
  &\quad \quad \ll \| \psi \|_{N^\ast},
 \end{align*}
 where $k_3 > k + C$, $j = k + O(1)$ and $k_1 = k_2 + O(1) = k + O(1)$. 

 We begin with the first case. Here, the inputs $Q_{\leq j-C} (\chi_I \phi_{k_1})$ and $\partial_\alpha Q_{\leq j-C} (\chi_I \phi_{k_2})$ have Fourier supports in identical (or opposite) angular sectors $\omega$ of size $\sim 2^{k-k_1}$. Then we bound
 \begin{align*}
  &\sum_k \sum_{ k_1 > k + O(1)} \big| \big\langle \Box^{-1} P_k Q_{k + O(1)} \bigl( Q_{\leq k -C} (\chi_I \phi_{k_1}) \cdot \partial_\alpha Q_{\leq k-C} (\chi_I \phi_{k_1 + O(1)}) \bigr), \\
  &\qquad \qquad \qquad \qquad \qquad   P_k Q_{k+O(1)} \bigl( \partial^\alpha Q_{\leq k-C} \phi_{k+O(1)} \cdot Q_{\leq k-C} \psi_{k + O(1)} \big) \big\rangle \big| \\
  &\lesssim \sum_k \sum_{k_1 > k+O(1)} 2^{\frac{1}{6}(k-k_1)} \biggl( \sum_{\omega}  2^{\frac{1}{3} k_1} \bigl\| P^\omega Q_{\leq k-C} (\chi_I \phi_{k_1}) \bigr\|_{L^2_t L^6_x}^2 \biggr)^{\frac{1}{2}} \biggl( \sum_\omega \, \bigl\| P^\omega Q_{\leq k-C} \nabla_{t,x} (\chi_I \phi_{k_2}) \bigr\|_{L^\infty_t L^2_x}^2 \biggr)^{\frac{1}{2}} \times \\
  &\qquad \qquad \qquad \qquad \qquad \times 2^{-\frac{5}{6} k} \bigl\| \nabla_{t,x} \phi_{k+O(1)} \bigr\|_{L^2_t L^6_x} \|\psi_{k+O(1)}\|_{L^\infty_t L^2_x} \\
  &\lesssim \biggl( \sum_{k_1} \sup_{l < -C} \sum_\omega 2^{\frac{1}{3} k_1} \bigl\|P^\omega_l Q_{\leq k_1 + l -C} (\chi_I \phi_{k_1}) \bigr\|_{L^2_t L^6_x}^2 \biggr)^{\frac{1}{2}} \|\phi\|_{S^1}^2 \|\psi\|_{N^\ast}.
 \end{align*}
 The desired smallness comes from the divisibility of the quantity
 \[
  \biggl( \sum_{k_1} \sup_{l < -C} \sum_\omega 2^{\frac{1}{3} k_1} \bigl\|P^\omega_l Q_{\leq k_1 + l -C} (\chi_I \phi_{k_1}) \bigr\|_{L^2_t L^6_x}^2 \biggr)^{\frac{1}{2}}.
 \]
 To see the divisibility, we write
 \begin{equation} \label{equ:smallness_first_null_form_splitting_for_divisibility}
  P^\omega_l Q_{\leq k_1 + l -C} (\chi_I \phi_{k_1}) = P^\omega_{\frac{1}{2} l} Q_{\leq k_1 + l -C} (\chi_I P_l^\omega Q_{\leq k_1 + l + M} \phi_{k_1}) + P^\omega_l Q_{\leq k_1 + l -C} (\chi_I Q_{> k_1 + l + M} \phi_{k_1})
 \end{equation}
 for some $M > 0$ to be chosen sufficiently large. By disposability of the operator $P_{\frac{1}{2} l}^\omega Q_{\leq k_1 + l -C}$, we estimate the first term on the right hand side of \eqref{equ:smallness_first_null_form_splitting_for_divisibility} by
 \[
  \biggl( \sum_{k_1} \sup_{l < -C} \sum_\omega 2^{\frac{1}{3} k_1} \bigl\|P^\omega_{\frac{1}{2} l} Q_{\leq k_1 + l -C} (\chi_I P_l^\omega Q_{\leq k_1 + l + M} \phi_{k_1}) \bigr\|_{L^2_t L^6_x}^2 \biggr)^{\frac{1}{2}} \lesssim \biggl( \sum_{k_1} \sup_{l < -C} \sum_\omega 2^{\frac{1}{3} k_1} \bigl\|\chi_I P^\omega_l Q_{\leq k_1 + l + M} \phi_{k_1} \bigr\|_{L^2_t L^6_x}^2 \biggr)^{\frac{1}{2}}
 \]
 and smallness can be forced by divisibility of the quantity
 \[
  \biggl( \sum_{k_1} \sup_{l < -C} \sum_\omega 2^{\frac{1}{3} k_1} \bigl\| P^\omega_l Q_{\leq k_1 + l + M} \phi_{k_1} \bigr\|_{L^2_t L^6_x}^2 \biggr)^{\frac{1}{2}} \lesssim \|\phi\|_{S^1}. 
 \]
 For the second term on the right hand side of \eqref{equ:smallness_first_null_form_splitting_for_divisibility}, we use
 \[
  P^\omega_l Q_{\leq k_1 + l -C} (\chi_I Q_{> k_1 + l + M} \phi_{k_1}) = P^\omega_l Q_{\leq k_1 + l - C} \bigl( Q_{> k_1 + l + \frac{M}{2}}(\chi_I) Q_{> k_1 + l + M} \phi_{k_1} \bigr).
 \]
 By Bernstein's inequality in space and in time, we then have
 \begin{align*}
  &\bigl\| P^\omega_l Q_{\leq k_1 + l - C} \bigl( Q_{> k_1 + l + \frac{M}{2}}(\chi_I) Q_{> k_1 + l + M} \phi_{k_1} \bigr) \bigr\|_{L^2_t L^6_x} \\
  &\lesssim 2^{\frac{1}{2} (k_1 + l)} 2^{\frac{4}{3} k_1} 2^l \bigl\| Q_{> k_1 + l + \frac{M}{2}} \chi_I \bigr\|_{L^2_t} \bigl\| P_l^\omega Q_{>k_1 + l + M} \phi_{k_1} \bigr\|_{L^2_t L^2_x} \\
  &\lesssim 2^{\frac{1}{2} (k_1 + l)} 2^{\frac{4}{3} k_1} 2^l 2^{-\frac{1}{2} (k_1 + l + \frac{M}{2})} \bigl\| P_l^\omega Q_{>k_1 + l + M} \phi_{k_1} \bigr\|_{L^2_t L^2_x} \\
  &\lesssim 2^{\frac{4}{3} k_1} 2^l 2^{- \frac{M}{4}} \bigl\| P_l^\omega Q_{>k_1 + l + M} \phi_{k_1} \bigr\|_{L^2_t L^2_x}.
 \end{align*}
 Thus, we obtain
 \begin{align*}
  &\biggl( \sum_{k_1} \sup_{l < -C}  \sum_\omega 2^{\frac{1}{3} k_1} \bigl\|P^\omega_l Q_{\leq k_1 + l -C} \bigl( Q_{> k_1 + l + \frac{M}{2}}(\chi_I) Q_{> k_1 + l + M} \phi_{k_1} \bigr) \bigr\|_{L^2_t L^6_x}^2 \biggr)^{\frac{1}{2}} \\
  &\lesssim \biggl( \sum_{k_1} \sup_{l < -C} 2^l 2^{-\frac{3}{2} M} \|\nabla_x \phi_{k_1}\|_{X^{0, \frac{1}{2}}_\infty}^2 \biggr)^{\frac{1}{2}} \\
  &\lesssim 2^{-\frac{3}{4} M} \|\phi\|_{S^1}
 \end{align*}
 and a smallness factor follows for sufficiently large $M > 0$. 

 The second case is easier to treat. Here we estimate
 \begin{align*}
  &\sum_k \sum_{k_3 > k+O(1)} \big| \big\langle \Box^{-1} P_k Q_{k+O(1)} \bigl( Q_{\leq k-C} (\chi_I \phi_{k+O(1)}) \cdot \partial_\alpha Q_{\leq k-C} (\chi_I \phi_{k+O(1)}) \bigr), \\
  &\qquad \qquad \qquad \qquad \qquad \qquad \qquad P_k Q_{k+O(1)} \bigl( \partial^\alpha Q_{\leq k-C} \phi_{k_3} \cdot Q_{\leq k-C} \psi_{k_3 + O(1)} \bigr) \big\rangle \big| \\
  &\lesssim \sum_k \sum_{k_3 > k+O(1)} 2^{-\frac{3}{2} k} \bigl\| \chi_I \nabla_x \phi_{k+O(1)} \bigr\|_{L^2_t L^\infty_x} 2^{-\frac{3}{2} k} \bigl\| \nabla_{t,x} (\chi_I \phi_{k+O(1)}) \bigr\|_{L^2_t L^\infty_x} \times \\
  &\qquad \qquad \qquad \qquad \qquad \qquad \qquad \times \bigl\| Q_{\leq k-C} \nabla_{t,x} \phi_{k_3} \bigr\|_{L^\infty_t L^2_x} \bigl\| Q_{\leq k-C} \psi_{k_3 + O(1)} \bigr\|_{L^\infty_t L^2_x} \\
  &\lesssim \biggl( \sum_k \Bigl( 2^{-\frac{3}{2} k} \bigl\| \chi_I \nabla_x \phi_{k+O(1)} \bigr\|_{L^2_t L^\infty_x} \Bigr)^2 \biggr)^{\frac{1}{2}} \|\phi\|_{S^1}^2 \|\psi\|_{N^\ast}
 \end{align*}
 and immediately obtain smallness by divisibility.
 
 \medskip

 \noindent {\it Second null form ((62) in \cite{KST})}. By the estimates (149) and (150) in \cite{KST}, we only have to show that
 \begin{align*}
  &\sum_k \sum_{k_3 = k_4 + O(1)} \bigl| \bigl\langle (\Box \Delta)^{-1} P_k Q_j \partial_t \partial_\alpha \bigl( Q_{\leq j-C} (\chi_I \phi_{k_1}) \cdot \partial^\alpha Q_{\leq j-C} (\chi_I \phi_{k_2}) \bigr), P_k Q_j \bigl( \partial_t Q_{\leq j-C} \phi_{k_3} \cdot Q_{\leq j-C} \psi_{k_4} \bigr) \bigr\rangle \bigr| \\
  &\quad \quad \ll \|\psi\|_{N^\ast},
 \end{align*}
 where $j = k + O(1)$, $k_1 = k_2 + O(1) = k + O(1)$ and $k_3 > k + C$. Then we estimate
 \begin{align*}
  &\sum_k \sum_{k_3 > k + C} \big| \big\langle (\Box \Delta)^{-1} P_k Q_{k+O(1)} \partial_t \partial_\alpha \bigl( Q_{\leq k-C} (\chi_I \phi_{k+O(1)}) \cdot \partial^\alpha Q_{\leq k-C} (\chi_I \phi_{k+O(1)}) \bigr), \\
  &\qquad \qquad \qquad \qquad \qquad \qquad \qquad \qquad \qquad P_k Q_{k+O(1)} \bigl( \partial_t Q_{\leq k-C} \phi_{k_3} \cdot Q_{\leq k-C} \psi_{k_3 + O(1)} \bigr) \big\rangle \big| \\
  &\lesssim \sum_k \sum_{k_3 > k+C} \bigl\| (\Box \Delta)^{-1} P_k Q_{k+O(1)} \partial_t \partial_\alpha \bigl( Q_{\leq k-C} (\chi_I \phi_{k+O(1)}) \cdot \partial^\alpha Q_{\leq k-C} (\chi_I \phi_{k+O(1)}) \bigr) \bigr\|_{L^1_t L^\infty_x} \times \\
  &\qquad \qquad \qquad \qquad \qquad \qquad \qquad \qquad \qquad \times \bigl\| P_k Q_{k+O(1)} \bigl( \partial_t Q_{\leq k-C} \phi_{k_3} \cdot Q_{\leq k-C} \psi_{k_3 + O(1)} \bigr) \bigr\|_{L^\infty_t L^1_x} \\
  &\lesssim \sum_k \sum_{k_3 > k+C} 2^{-\frac{3}{2} k} \bigl\| \nabla_x (\chi_I \phi_{k+O(1)}) \bigr\|_{L^2_t L^\infty_x} 2^{-\frac{3}{2} k} \bigl\| \nabla_x (\chi_I \phi_{k+O(1)}) \bigr\|_{L^2_t L^\infty_x} \|\partial_t \phi_{k_3}\|_{L^\infty_t L^2_x} \|\psi_{k_3 + O(1)}\|_{L^\infty_t L^2_x} \\
  &\lesssim \biggl( \sum_k \Bigl( 2^{-\frac{3}{2} k} \bigl\| \nabla_x (\chi_I \phi_{k+O(1)}) \bigr\|_{L^2_t L^\infty_x} \Bigr)^2 \biggr)^{1/2} \|\phi\|^2_{S^1} \|\psi\|_{N^\ast}
 \end{align*}
 and smallness follows by divisibility.

 \medskip

 \noindent {\it Third null form ((63) in \cite{KST}).} By the estimates (152) -- (154) in \cite{KST}, it suffices to consider the following two cases. First, we show that
 \begin{align*}
  &\sum_k \sum_{ k_1 = k_2 + O(1) } \bigl| \bigl\langle (\Box \Delta)^{-1} P_k Q_j \partial^i \bigl( Q_{\leq j-C} (\chi_I \phi_{k_1}) \cdot \partial_i Q_{\leq j-C} (\chi_I \phi_{k_2}) \bigr), P_k Q_j \partial_\alpha \bigl( \partial^\alpha Q_{\leq j-C} \phi_{k_3} \cdot Q_{\leq j-C} \psi_{k_4} \bigr) \bigr\rangle \bigr| \\
  &\quad \quad \ll \| \psi \|_{N^\ast},
 \end{align*}
 where $k_1 > k + C$, $j = k + O(1)$ and $k_3 = k_4 + O(1) = k + O(1)$. Second, we prove that
 \begin{align*}
  &\sum_k \sum_{ k_3 = k_4 + O(1) } \bigl| \bigl\langle (\Box \Delta)^{-1} P_k Q_j \partial^i \bigl( Q_{\leq j-C} (\chi_I \phi_{k_1}) \cdot \partial_i Q_{\leq j-C} (\chi_I \phi_{k_2}) \bigr), P_k Q_j \partial_\alpha \bigl( \partial^\alpha Q_{\leq j-C} \phi_{k_3} \cdot Q_{\leq j-C} \psi_{k_4} \bigr) \bigr\rangle \bigr| \\
  &\quad \quad \ll \| \psi \|_{N^\ast},
 \end{align*}
 where $k_3 > k+C$, $j = k+O(1)$ and $k_1 = k_2 + O(1) = k + O(1)$. 

 In the first case we note that the first two inputs have Fourier supports in identical (or opposite) angular sectors $\omega$ of size $\sim 2^{k-k_1}$. Using Bernstein's inequality, we then place the first input in $L^2_t L^6_x$, the second one in $L^\infty_t L^2_x$, the third one in $L^2_t L^\infty_x$ and the fourth one in $L^\infty_t L^2_x$. As in the first case of the first null form we obtain the desired smallness by divisibility of the quantity
 \[
  \biggl( \sum_{k_1} \sup_{l < -C} \sum_\omega 2^{\frac{1}{3} k_1} \bigl\|P^\omega_l Q_{\leq k_1 + l -C} (\chi_I \phi_{k_1}) \bigr\|_{L^2_t L^6_x}^2 \biggr)^{\frac{1}{2}}.
 \]
 The second case is easier to deal with and we omit the details.
\end{proof}

\noindent {\bf Step 3.} {\it Solution of perturbative problems on suitable space-time slices.} This is the crucial technical step. We write
\[
 \bigl( A^{n1, (L)}_{\Lambda_0}, \phi^{n1, (L)}_{\Lambda_0} \bigr) = \bigl( A^{n1, (L-1)}_{\Lambda_0}, \phi^{n1, (L-1)}_{\Lambda_0} \bigr) + \bigl( \delta A^{(L)}, \delta \phi^{(L)} \bigr).
\]
Then we obtain the following system of equations for the perturbations $\bigl( \delta A^{(L)}, \delta \phi^{(L)} \bigr)$,
\begin{equation} \label{equ:perturbation_equation_phi}
 \Box_{A_{\Lambda_0}^{n1, (L-1)} + \delta A^{(L)}} \bigl( \phi_{\Lambda_0}^{n1, (L-1)} + \delta \phi^{(L)} \bigr) - \Box_{A_{\Lambda_0}^{n1, (L-1)}} \phi_{\Lambda_0}^{n1, (L-1)} = 0,
\end{equation}
\begin{equation} \label{equ:perturbation_equation_A}
 \begin{split} 
  \Box \delta A^{(L)} &= - \Im {\mathcal P} \Bigl( \phi^{n1, (L-1)}_{\Lambda_0} \cdot \nabla_x \overline{\delta \phi}^{(L-1)} + \delta \phi^{(L-1)} \cdot \nabla_x \overline{\phi}^{n1, (L-1)}_{\Lambda_0} + \delta \phi^{(L-1)} \cdot \nabla_x \overline{\delta \phi}^{(L-1)} \Bigr) \\
  &\quad \quad + \Im {\mathcal P} \Bigl( \bigl(A^{n1, (L-1)}_{\Lambda_0} + \delta A^{(L)}\bigr) \bigl| \phi^{n1, (L-1)}_{\Lambda_0} + \delta \phi^{(L)} \bigr|^2 - A^{n1, (L-1)}_{\Lambda_0} \bigl| \phi^{n1, (L-1)}_{\Lambda_0} \bigr|^2 \Bigr).
 \end{split}
\end{equation}
We have to show that if the initial data $(\delta A^{(L)}, \delta \phi^{(L)})[0]$ are less than the absolute constant $\varepsilon_0$ in the energy sense, then we can prove frequency localized $S^1$ norm bounds via bootstrap on any space-time slice on which certain ``divisible'' norms of $\bigl(A^{n1, (L-1)}_{\Lambda_0}, \phi^{n1, (L-1)}_{\Lambda_0} \bigr)$ are small. Furthermore, the number of such space-time slices needed to fill all of space-time depends on the a priori assumed $S^1$ norm bounds for the components $\bigl( A^{n1, (L-1)}_{\Lambda_0}, \phi^{n1, (L-1)}_{\Lambda_0} \bigr)$.

One technical difficulty is the formulation of the correct frequency localized $S^1$ norm bound for the propagation of $\delta \phi^{(L)}$, because there is a contribution from low frequencies of $\phi^{n1, (L-1)}_{\Lambda_0}$, and similarly for $\delta A^{(L)}$. However, this low frequency contribution can be made arbitrarily small by picking $n$ large and $\delta_1$ small enough. 

We note that while $\bigl(A^{n1, (L-1)}_{\Lambda_0}, \phi^{n1, (L-1)}_{\Lambda_0} \bigr)$ exists globally in time, $\bigl( \delta A^{(L)}, \delta \phi^{(L)} \bigr)$ only exists locally in time and we will have to prove global existence and $S^1$ norm bounds for it. For now, any statement we make about $\bigl( \delta A^{(L)}, \delta \phi^{(L)} \bigr)$ is meant locally in time on some interval $I_0$ around $t = 0$. Proposition \ref{prop:decomposition} yields a partition of the time axis $\R$ into $N = N \bigl( \| \bigl( A^{n1, (L-1)}_{\Lambda_0}, \phi^{n1, (L-1)}_{\Lambda_0} \bigr) \|_{S^1(\R\times\R^4)} \bigr)$ many time intervals $\{I_j\}_{j=1}^N$, on which the smallness conclusions (in terms of $E_{crit}$) of Proposition \ref{prop:decomposition} hold. We tacitly assume that these intervals are intersected with $I_0$ and now fix the interval $I_1$, which we assume to contain $t = 0$. All the arguments in this step can be carried out for any of the later intervals $I_2, \ldots, I_N$.

\medskip

\noindent {\bf Bootstrap assumptions }: Suppose that there exist decompositions
\begin{align*}
 \delta A^{(L)} = \delta A^{(L)}_{\text{\ding{172}}} + \delta A^{(L)}_{\text{\ding{173}}}, \quad \delta \phi^{(L)} = \delta \phi^{(L)}_{\text{\ding{172}}} + \delta \phi^{(L)}_{\text{\ding{173}}}
\end{align*}
satisfying the following bounds.

(i) Let $\bigl\{ c^{(L)}_{\delta A, k} \bigr\}_{k \in \Z}$ be a frequency envelope controlling the data $P_k \delta A^{(L)}[0]$ at time $t = 0$ and let $\bigl\{ d^{(L)}_{\delta A, k} \bigr\}_{k \in \Z}$ be a frequency envelope that decays exponentially for $k > b_L$ but is otherwise not localized and satisfies the smallness condition 
\[
 \sum_k d^{(L)}_{\delta A, k} \leq \delta_2 = \delta_2(\delta_1).
\]
Then we assume that for all $k \in \Z$,
\begin{align*}
 \bigl\| P_k \delta A^{(L)}_{\text{\ding{172}}} \bigr\|_{S^1(I_1\times\R^4)} &\leq C c^{(L)}_{\delta A, k}, \\
 \bigl\| P_k \delta A^{(L)}_{\text{\ding{173}}} \bigr\|_{S^1(I_1\times\R^4)} &\leq C d^{(L)}_{\delta A, k},
\end{align*}
where $C \equiv C\bigl(E_{crit}\bigr)$ is sufficiently large. 

(ii) Let $\bigl\{ c^{(L)}_{\delta \phi, k} \bigr\}_{k \in \Z}$ be a frequency envelope controlling the data $P_k \delta \phi^{(L)}[0]$ at time $t = 0$ and let $\bigl\{ d^{(L)}_{\delta \phi, k} \bigr\}$ be a frequency envelope that decays exponentially for $k > b_L$, but is otherwise not localized and satisfies the smallness condition
\[
 \Bigl( \sum_k \bigl( d^{(L)}_{\delta \phi, k} \bigr)^2 \Bigr)^{\frac{1}{2}} \leq \delta_3 = \delta_3(\delta_1).
\]
Then we assume that for all $k \in \Z$,
\begin{align*}
 \bigl\| P_k \delta\phi^{(L)}_{\text{\ding{172}}} \bigr\|_{S^1_k(I_1\times\R^4)} \leq C c^{(L)}_{\delta \phi, k}, \\
 \bigl\| P_k \delta\phi^{(L)}_{\text{\ding{173}}} \bigr\|_{S^1_k(I_1\times\R^4)} \leq C d^{(L)}_{\delta \phi, k},
\end{align*}
where $C \equiv C\bigl(E_{crit}\bigr)$ is sufficiently large.

\medskip

We now show that we can improve this to a similar decomposition with 
\begin{equation} \label{equ:improved_bootstrap_bounds}
 \begin{split}
  \bigl\| P_k \delta A^{(L)}_{\text{\ding{172}}} \bigr\|_{S^1_k(I_1\times\R^4)} &\leq \frac{C}{2} c^{(L)}_{\delta A, k}, \\
  \bigl\| P_k \delta A^{(L)}_{\text{\ding{173}}} \bigr\|_{S^1_k(I_1\times\R^4)} &\leq \frac{C}{2} d^{(L)}_{\delta A, k}, \\
  \bigl\| P_k \delta\phi^{(L)}_{\text{\ding{172}}} \bigr\|_{S^1_k(I_1\times\R^4)} &\leq \frac{C}{2} c^{(L)}_{\delta \phi, k}, \\
  \bigl\| P_k \delta\phi^{(L)}_{\text{\ding{173}}} \bigr\|_{S^1_k(I_1\times\R^4)} &\leq \frac{C}{2} d^{(L)}_{\delta \phi, k},
 \end{split}
\end{equation}
provided we make the additional assumption
\[
 \delta_2 \ll \delta_3
\]
with implied constant depending only on $E_{crit}$. 

\medskip

Observe that we have
\[
 \sum_k \bigl( c_{\delta A, k}^{(L)} \bigr)^2 + \bigl( c_{\delta \phi, k}^{(L)} \bigr)^2 \lesssim \varepsilon_0
\]
and that our smallness parameters satisfy
\[
 \delta_1 \ll \delta_2 \ll \delta_3 \ll \varepsilon_0.
\]
For the remainder of this step we simply write $I \equiv I_1$ and $\phi \equiv \phi^{n1, (L-1)}_{\Lambda_0}$, $\delta \phi \equiv \delta \phi^{(L)}$, $A \equiv A^{n1, (L-1)}_{\Lambda_0}$, $\delta A \equiv \delta A^{(L)}$.

\medskip

\noindent {\bf Step 3a.} {\it Reorganizing the key equation \eqref{equ:perturbation_equation_phi}.} We introduce the connection form $(A+\delta A)^{nonlin, (I)}$ analogously to \eqref{equ:definition_A_nonlin_I} by setting for $i = 1, \ldots, 4$,
\begin{equation} \label{eq:AplusdeltaA}
 (A+\delta A)^{nonlin, (I)}_i := - \chi_I \sum_{k, j} \Box^{-1} P_k Q_j \mathcal{P}_i \bigl( \chi_I (\phi + \delta\phi) \cdot \nabla_x \bigl( \chi_I (\overline{\phi} + \overline{\delta \phi}) \bigr) - \chi_I i (A + \delta A) |\phi + \delta \phi|^2 \bigr),
 \end{equation}
and define $(A+\delta A)^{free, (I)}$ as the free wave with initial data at time $t = 0$ given by 
\[
 (A + \delta A)^{free, (I)}[0] = (A+\delta A)[0] -  (A+\delta A)^{nonlin, (I)}[0].
\]
Then we have
\[
 (A + \delta A)|_I = (A + \delta A)^{free, (I)} + (A + \delta A)^{nonlin, (I)}.
\]
On $I \times \R^4$ we may rewrite the equation \eqref{equ:perturbation_equation_phi} for $\delta \phi$ into the following frequency localized form
\begin{equation}\label{equ:perturbation_equation_phi_frequency_localized}
 \begin{split}
  \Box_{(A + \delta A)^{free, (I)}}^p \bigl( P_0 \delta \phi \bigr) &= - \bigl[ P_0, \Box_{(A+\delta A)^{free, (I)}}^p \bigr] \delta \phi \\
  &\quad - P_0 \Bigl( 2i  \sum_k P_{> k-C} (A + \delta A)_j^{free, (I)} P_k \partial^j \delta \phi \Bigr) \\
  &\quad - P_0 \Bigl( 2i (A + \delta A)^{nonlin, (I)}_j \partial^j \delta \phi - 2i (A + \delta A)_0 \partial_t \delta \phi \Bigr) \\
  &\quad - P_0 \Bigl( 2i (\delta A)_j \partial^j \phi - 2i (\delta A)_0 \partial_t \phi \Bigr) \\
  &\quad + P_0 \Bigl( i (\partial_t A_0 + \partial_t \delta A_0) (\phi + \delta \phi) - i (\partial_t A_0) \phi \Bigr) \\
  &\quad + P_0 \Bigl( (A + \delta A)_\alpha (A + \delta A)^\alpha \phi - A_\alpha A^\alpha \phi \Bigr).
 \end{split}
\end{equation}
We immediately see that compared to \eqref{equ:Box_A_on_short_interval}, a qualitatively new feature in \eqref{equ:perturbation_equation_phi_frequency_localized} is the interaction term 
\begin{equation} \label{equ:new_interaction_term}
 P_0 \Bigl( (\delta A)_j \partial^j \phi - (\delta A)_0 \partial_t \phi \Bigr).
\end{equation}

\medskip

\noindent {\bf Step 3b.} {\it Improving the bounds for $\delta \phi$ using \eqref{equ:perturbation_equation_phi_frequency_localized}.} In order to obtain bounds on the $S^1(I\times\R^4)$ norm of $P_0 \delta \phi$ by bootstrap, we work with suitable Schwartz extensions and use the linear estimate \eqref{equ:linear_estimates_magnetic_wave_equation} for the magnetic wave operator $\Box_{(A + \delta A)^{free, (I)}}^p$. We first consider the new interaction term \eqref{equ:new_interaction_term}. As usual the main difficulty comes from the low-high interactions, so we begin with this case, i.e. the term 
\[
 P_{<0} (\delta A)_j P_0 \partial^j \phi - P_{<0} (\delta A)_0 P_0 \partial_t \phi.
\]
For the spatial components of $\delta A$, we define
\[
 (\delta A)^{free, (I)} := (A + \delta A)^{free, (I)} - A^{free, (I)}, \quad (\delta A)^{nonlin, (I)} := (A + \delta A)^{nonlin, (I)} - A^{nonlin, (I)}
\]
and correspondingly have on $I \times \R^4$ that
\[
 (\delta A)|_I = (\delta A)^{free, (I)} + (\delta A)^{nonlin, (I)}.
\]
We can therefore split on $I \times \R^4$,
\begin{align*}
 P_{<0} (\delta A)_{j} P_0 \partial^j \phi = P_{<0} (\delta A)_j^{free, (I)} P_0 \partial^j \phi + P_{<0} (\delta A)_j^{nonlin, (I)} P_0 \partial^j \phi.
\end{align*}
The first term on the right hand side can in turn be split into two contributions
\begin{equation} \label{equ:new_interaction_term_free_part_split}
 P_{<0} (\delta A)_j^{free, (I)} P_0 \partial^j \phi = P_{<0} (\delta A)_j^{free_1, (I)} P_0 \partial^j \phi + P_{<0} (\delta A)_j^{free_2, (I)} P_0 \partial^j \phi,
\end{equation}
where $(\delta A)^{free_1, (I)}$ is the free evolution of the data $(\delta A)^{(L)}[0]$, while $(\delta A)^{free_2, (I)}$ is the free wave with data 
\begin{align*}
 &\biggl( \sum_{k, j} \Box^{-1} P_k Q_j \mathcal{P} \bigl( \chi_I (\phi + \delta\phi) \cdot \nabla_x \bigl( \chi_I (\overline{\phi} + \overline{\delta \phi}) \bigr) - \chi_I i (A + \delta A) |\phi + \delta \phi|^2 \bigr) \biggr)[0], \\
 &\quad - \biggl( \sum_{k, j} \Box^{-1} P_k Q_j \mathcal{P} \bigl( (\chi_I \phi) \cdot \nabla_x (\chi_I \overline{\phi}) - \chi_I i A |\phi|^2 \bigr) \biggr)[0].
\end{align*}
In order to estimate the terms on the right hand side of \eqref{equ:new_interaction_term_free_part_split}, we will invoke the following estimate from \cite{KST} for a free wave $A^{free}$ in Coulomb gauge for $k_1 \leq k_2 - C$,
\begin{equation} \label{equ:low_high_N_estimate_for_free_wave}
 \bigl\| P_{k_1} A^{free}_j P_{k_2} \partial^j \phi \bigr\|_{N} \lesssim \bigl\| P_{k_1} A^{free} \bigr\|_{S^1} \bigl\| P_{k_2} \phi \bigr\|_{S^1} \lesssim \bigl\| P_{k_1} A^{free}[0] \bigr\|_{\dot{H}^1_x \times L^2_x} \bigl\| P_{k_2} \phi \bigr\|_{S^1}.
\end{equation}
We begin with the first term on the right hand side of \eqref{equ:new_interaction_term_free_part_split},
\[
 P_{<0} (\delta A)_j^{free_1, (I)} P_0 \partial^j \phi.
\]
Here we have to take advantage of the properties of the Fourier support of the data $P_{<0} (\delta A)_j^{free_1, (I)}[0]$. It suffices to assume that 
\[
 \bigl\| P_k (\delta A)_j^{free_1}[0] \bigr\|_{\dot{H}^1_x \times L^2_x} \leq C \bigl( c_{\delta A, k} + d_{\delta A, k} \bigr)
\]
for $C \equiv C(E_{crit})$ sufficiently large. This is an assumption that will hold inductively at later initial times (for the intervals $I_2, \ldots, I_N$). We observe that the frequency envelope $\bigl\{ c_{\delta A, k} \bigr\}_{k \in \Z}$ is ``sharply localized'' to the dyadic frequency interval $[a_L, b_L]$ in the sense that it is exponentially decaying for $k < a_L$ and $k > b_L$. By \eqref{equ:low_high_N_estimate_for_free_wave} we then have
\begin{equation} \label{equ:new_interaction_term_free_1_bound}
 \bigl\| P_{<0} (\delta A)_j^{free_1, (I)} P_0 \partial^j \phi \bigr\|_{N_0(I\times \R^4)} \lesssim \sum_{k<0} c_{\delta A, k} \bigl\| P_0 \phi \bigr\|_{S^1(I\times \R^4)} + \sum_{k<0} d_{\delta A, k} \bigl\| P_0 \phi \bigr\|_{S^1(I\times \R^4)}.
\end{equation}
We begin to estimate the first term on the right hand side of \eqref{equ:new_interaction_term_free_1_bound}, where we only consider the case when $a_L < 0$. For $R > 0$ to be chosen sufficiently large later on, we split
\begin{equation} \label{equ:bootstrap_new_interaction_term_low_high_free_wave}
 \begin{split}
  \sum_{k<0} c_{\delta A, k} \bigl\| P_0 \phi \bigr\|_{S^1(I\times \R^4)} &= \sum_{k \leq a_L - R} c_{\delta A, k} \bigl\| P_0 \phi \bigr\|_{S^1(I\times \R^4)} + \sum_{a_L - R < k \leq a_L + R} c_{\delta A, k} \bigl\| P_0 \phi \bigr\|_{S^1(I\times \R^4)} \\
  &\quad \quad + \sum_{a_L + R < k < 0} c_{\delta A, k} \bigl\| P_0 \phi \bigr\|_{S^1(I\times \R^4)}.
 \end{split}
\end{equation}
To make the first term on the right hand side of \eqref{equ:bootstrap_new_interaction_term_low_high_free_wave} small, we use the exponential decay of the frequency envelope $\bigl\{ c_{\delta A, k} \bigr\}_{k \in \Z}$ to bound
\[
 \sum_{k \leq a_L - R} c_{\delta A, k} \bigl\| P_0 \phi \bigr\|_{S^1(I\times \R^4)} \lesssim \sum_{k \leq a_L - R} 2^{-\sigma (a_L - k)} \|\delta A[0]\|_{\dot{H}^1_x \times L^2_x}  \|P_0 \phi\|_{S^1(I\times\R^4)} \lesssim_{E_{crit}} 2^{-\sigma R} \|P_0 \phi\|_{S^1(I\times\R^4)}.
\]
Upon replacing the output frequency $0$ by general $l \in \Z$, square summing over $l$ and choosing $R > 0$ sufficiently large, we bound the preceding by $\ll_{E_{crit}} \delta_3$, as desired. In order to make the third term on the right hand side of \eqref{equ:bootstrap_new_interaction_term_low_high_free_wave} small, we exploit that by Step 1 the $S^1$ norms of $P_l \phi$ are exponentially decaying beyond the scale $l > a_L$. We have
\begin{align*}
 \sum_{a_L + R < k < 0} c_{\delta A, k} \bigl\| P_0 \phi \bigr\|_{S^1(I\times \R^4)} \lesssim_{E_{crit}} ( |a_L| - R ) c_0,
\end{align*}
where $\{ c_l \}_{l \in \Z}$ is a sufficiently flat frequency envelope covering the initial data $\bigl( A^{n1, (L-1)}_{\Lambda_0}, \phi^{n1, (L-1)}_{\Lambda_0} \bigr)[0]$ as in Step 1. Then replacing the frequency $0$ in $P_0 \phi$ by a general dyadic frequency $l > a_L + R$, square summing over $l$ and choosing $R > 0$ sufficiently large, we find
\begin{align*}
 \sum_{l > a_L + R} \Bigl| \sum_{a_L + R < k < l} c_{\delta A, k} \bigl\| P_l \phi \bigr\|_{S^1(I\times\R^4)} \Bigr|^2 &\lesssim_{E_{crit}} \sum_{l > a_L + R} (l - a_L - R)^2 c_l^2 \lesssim_{E_{crit}} 2^{-\sigma R} \ll_{E_{crit}} \delta_3,
\end{align*}
which is acceptable. It remains to make the second term on the right hand side of \eqref{equ:bootstrap_new_interaction_term_low_high_free_wave} small. To this end we exploit the frequency evacuation property of the data $\bigl( A_{\Lambda_0}^{n1}, \phi_{\Lambda_0}^{n1} \bigr)$ at the edges of the frequency intervals $[a_L, b_L]$ that we established in Lemma \ref{lem:asymptoevacuate}. For sufficiently small $\delta_1 > 0$ and all sufficiently large $n$, we then have
\[
 \sum_{a_L - R < k \leq a_L + R} c_{\delta A, k} \bigl\| P_0 \phi \bigr\|_{S^1(I\times \R^4)} \lesssim R \delta_1 \bigl\| P_0 \phi \bigr\|_{S^1(I\times \R^4)} \ll \delta_3 \bigl\| P_0 \phi \bigr\|_{S^1(I\times\R^4)}.
\]
Upon replacing the frequency $0$ in $P_0 \phi$ by an arbitrary dyadic frequency $l \in \Z$ and square summing, we obtain the desired smallness $\ll_{E_{crit}} \delta_3$ for the last estimate. 

The contribution of the second term on the right hand side of \eqref{equ:new_interaction_term_free_1_bound} is acceptable, because, upon replacing the output frequency $0$ by $l \in \Z$, square summing and using the bootstrap assumptions on the interval $I$, we obtain the bound
\[
 \biggl( \sum_l \Bigl| \sum_{k<l} d_{\delta A, k} \bigl\| P_l \phi \bigr\|_{S^1(I\times \R^4)} \Bigr|^2 \biggr)^{\frac{1}{2}} \lesssim_{E_{crit}} \delta_2 \ll \delta_3,
\]
where the implied constant in $\lesssim_{E_{crit}}$ depends at most polynomially on $E_{crit}$. 

\medskip

Next, we estimate the second term on the right hand side of \eqref{equ:new_interaction_term_free_part_split},
\[
 P_{<0} (\delta A)^{free_2, (I)}_j P_0 \partial^j \phi.
\]
By \eqref{equ:low_high_N_estimate_for_free_wave} we have
\begin{equation}
 \bigl\| P_{<0} (\delta A)^{free_2, (I)}_j P_0 \partial^j \phi \bigr\|_{N_0(I\times\R^4)} \lesssim \bigl\| P_{<0} (\delta A)^{free_2, (I)} \bigr\|_{\ell^1 S^1(I\times\R^4)} \bigl\| P_0 \phi \bigr\|_{S^1(I\times\R^4)}.
\end{equation}
We illustrate how to obtain the desired smallness in this case by assuming for simplicity that $P_{<0} (\delta A)^{free_2, (I)}$ is just the free evolution of the data
\[
 \sum_{k<0} \sum_j \Box^{-1} P_k Q_j {\mathcal P} \bigl( \chi_I \phi \cdot \nabla_x (\chi_I \overline{\delta \phi}) \bigr)[0].
\]
If $\delta \phi = (\delta \phi)_{\text{\ding{172}}}$, we obtain by similar estimates as in the proof of Lemma \ref{lem:nonlinear_structure_small_short_interval} that
\[
 \bigl\| P_{<0} (\delta A)^{free_2, (I)} \bigr\|_{\ell^1 S^1(I\times\R^4)} \lesssim \sum_{k<0} \, \bigl\|\chi_I \phi\bigr\|_{S^1} c_{\delta \phi, k}.
\]
Then we achieve the desired smallness for 
\[
 \bigl\| P_{<0} (\delta A)^{free_2, (I)}_j P_0 \partial^j \phi \bigr\|_{N_0(I\times\R^4)} \lesssim_{E_{crit}} \sum_{k<0} c_{\delta \phi, k} \bigl\| P_0 \phi \bigr\|_{S^1(I\times\R^4)}
\]
by proceeding exactly as with the term \eqref{equ:bootstrap_new_interaction_term_low_high_free_wave}. If instead $\delta \phi = (\delta \phi)_{\text{\ding{173}}}$, we find
\[
 \bigl\| P_{<0} (\delta A)^{free_2, (I)}_j P_0 \partial^j \phi \bigr\|_{N_0(I\times\R^4)} \lesssim \bigl\| \chi_I \phi \bigr\|_{S^1} \biggl( \sum_k d_{\delta \phi, k}^2 \biggr)^{\frac{1}{2}} \bigl\| P_0 \phi \bigr\|_{S^1(I\times\R^4)} \lesssim \delta_3 \bigl\| \chi_I \phi \bigr\|_{S^1} \bigl\| P_0 \phi \bigr\|_{S^1(I\times\R^4)}.
\]
Upon replacing the output frequency $0$ by $l \in \Z$, square summing and using that $\|\phi\|_{S^1(I\times\R^4)} \lesssim C(E_{crit})$ by the choice of the interval $I$, we obtain the bound $\lesssim_{E_{crit}} \delta_3$. This is unfortunately not yet enough to close the bootstrap. To gain the extra smallness we partition the interval $I$ further and use divisibility arguments as in the proof of Lemma \ref{lem:nonlinear_structure_small_short_interval}. However, the number of intervals needed for this partition depends only on $E_{crit}$ (and not on the stage of the induction), which is acceptable.

\medskip

This finishes the estimate of the contribution of 
\[
 P_{<0} (\delta A)_j^{free, (I)} P_0 \partial^j \phi
\]
and we now have to bound
\[
 \bigl\| P_{<0} (\delta A)_j^{nonlin, (I)} P_0 \partial^j \phi - P_{<0} (\delta A)_0 P_0 \partial_t \phi \bigr\|_{N_0(I\times\R^4)}.
\]
At this point we can proceed by analogy to the treatment of the $\phi$ equation in the proof of Proposition~\ref{prop:decomposition}. After a further partitioning of the interval $I$ and similar divisibility arguments, we can replace the output frequency $0$ by $l \in \Z$ and upon square summing, we obtain a bound of the desired form $\ll_{E_{crit}} \delta_3$. 

\medskip

The remaining frequency interactions in the estimate of the term \eqref{equ:new_interaction_term} as well as all other terms on the right hand side of \eqref{equ:perturbation_equation_phi_frequency_localized} are easier to control. We omit the details.
 
\medskip

\noindent {\bf Step 3c.} {\it Improving the bounds for $\delta A$ using \eqref{equ:perturbation_equation_A}.} In order to deduce $S^1(I\times\R^4)$ norm bounds on $P_0 \delta A$ from the perturbation equation \eqref{equ:perturbation_equation_A} by bootstrap, we perform the same kind of divisibility arguments as in the proof of estimate \eqref{equ:nonlinear_structure_ell_1_N_small_short_interval} in Lemma \ref{lem:nonlinear_structure_small_short_interval} for the terms linear in $\delta \phi$ or $\delta A$.

\medskip

\noindent {\bf Step 4.} {\it Repetition of the bootstrap on suitable space-time slices; proof that the energy of perturbation remains small.} 
In this final step we show that the crucial assumption on the energy of the perturbation 
\[
 E \bigl(\delta A^{(L)}, \delta \phi^{(L)} \bigr)(0) < \varepsilon_0
\]
remains in tact along the evolution up to a very small correction. We recall that
\[
 \delta A^{(L)} = A^{n1, (L)}_{\Lambda_0} - A^{n1, (L-1)}_{\Lambda_0}, \quad \delta \phi^{(L)} = \phi^{n1, (L)}_{\Lambda_0} - \phi^{n1, (L-1)}_{\Lambda_0}.
\]

\begin{lem} \label{lem:approximate_energy_conservation_for_perturbation}
 Assuming the bounds \eqref{equ:improved_bootstrap_bounds} on the evolution of $\bigl( \delta A^{(L)}, \delta \phi^{(L)} \bigr)$ on $I_1 \times \R^4$, we have for sufficiently small $\delta_1 > 0$ and all sufficiently large $n$ that
 \[
  E \bigl( \delta A^{(L)}, \delta \phi^{(L)} \bigr)(t) < \varepsilon_0 \quad \text{ for } t \in I_1.
 \]
\end{lem}
\begin{proof}
 By energy conservation for the evolutions of $\bigl( A^{n1, (L)}_{\Lambda_0}, \phi^{n1, (L)}_{\Lambda_0} \bigr)$ and of $\bigl( A^{n1, (L-1)}_{\Lambda_0}, \phi^{n1, (L-1)}_{\Lambda_0}\bigr)$, it suffices to show that 
 \[
  \Bigl| E \bigl( A^{n1, (L)}_{\Lambda_0}, \phi^{n1, (L)}_{\Lambda_0} \bigr)(t) - E \bigl( A^{n1, (L-1)}_{\Lambda_0}, \phi^{n1, (L-1)}_{\Lambda_0}\bigr)(t) - E \bigl( \delta A^{(L)}, \delta \phi^{(L)} \bigr)(t) \Bigr|
 \]
 can be made arbitrarily small uniformly for all $t \in I_1$ by choosing $\delta_1 > 0$ sufficiently small and $n$ sufficiently large. This reduces to bounding the following expression evaluated at any time $t \in I_1$,
 \begin{align*}
  &\biggl| 2 \sum_{i < j} \int_{\R^4} \bigl( \partial_i A_{\Lambda_0, j}^{n1, (L-1)} \bigr) \bigl( \partial_i \delta A_j^{(L)} \bigr) \, dx + \sum_i \int_{\R^4} \bigl( \partial_t A_{\Lambda_0, i}^{n1, (L-1)} \bigr) \bigl( \partial_t \delta A_i^{(L)} \bigr) + \bigl( \partial_i A_{\Lambda_0, 0}^{n1, (L)} \bigr) \bigl( \partial_i \delta A_0^{(L-1)} \bigr) \, dx \\
  &\, + \frac{1}{2} \sum_\alpha \int_{\R^4} \bigl| \bigl( A_{\Lambda_0, \alpha}^{n1, (L-1)} \bigr) \bigl( \delta \phi^{(L)} \bigr) + \bigl( \delta A^{(L)}_\alpha \bigr) \bigl( \phi_{\Lambda_0}^{n1, (L-1)} \bigr)\bigr|^2 \, dx \\
  &\, + \sum_\alpha \Re \int_{\R^4} \bigl( \partial_\alpha \phi_{\Lambda_0}^{n1, (L-1)} + i A_{\Lambda_0, \alpha}^{n1, (L-1)} \phi_{\Lambda_0}^{n1, (L-1)} \bigr) \bigl( \overline{ \partial_\alpha \delta \phi^{(L)} + i \delta A_\alpha^{(L)} \delta \phi^{(L)} } \bigr) \\
  &\quad \quad \quad \quad \quad \quad + \bigl( \partial_\alpha \delta \phi^{(L)} + i \delta A^{(L)}_\alpha \delta \phi^{(L)} \bigr) \bigl( \overline{ i A_{\Lambda_0, \alpha}^{n1, (L-1)} \delta \phi^{(L)} + i \delta A^{(L)}_\alpha \phi_{\Lambda_0}^{n1, (L-1)} } \bigr) \\
  &\quad \quad \quad \quad \quad \quad \quad \quad + \bigl( \partial_\alpha \phi_{\Lambda_0}^{n1, (L-1)} + i A_{\Lambda_0, \alpha}^{n1, (L-1)} \phi_{\Lambda_0}^{n1, (L-1)} \bigr) \bigl( \overline{ i A_{\Lambda_0, \alpha}^{n1, (L-1)} \delta \phi^{(L)} + i \delta A^{(L)}_\alpha \phi_{\Lambda_0}^{n1, (L-1)} } \bigr) \, dx \biggr|.
 \end{align*}
 We note that in this expression at least one term of the form $A_{\Lambda_0}^{n1, (L-1)}$ or $\phi_{\Lambda_0}^{n1, (L-1)}$ is paired against at least one term of the form $\delta A^{(L)}$ or $\delta \phi^{(L)}$. By Plancherel's theorem (and a Littlewood-Paley trichotomy to deal with the multilinear interactions), we reduce to estimating a sum of the form
 \begin{align*}
  \biggl| 2 \sum_{k \in \Z} \sum_{i < j} \int_{\R^4} P_k \bigl( \partial_i A_{\Lambda_0, j}^{n1, (L-1)} \bigr) \, P_k \bigl( \partial_i \delta A_j^{(L)} \bigr) \, dx + \ldots \biggr|. 
 \end{align*}
 By the bounds \eqref{equ:improved_bootstrap_bounds} and Step 1, we estimate this by
 \begin{align*}
  &\lesssim \sum_{k \in \Z} \bigl\| P_k \nabla_x A_{\Lambda_0}^{n1, (L-1)}(t) \bigr\|_{L^2_x} \bigl\| P_k \nabla_x \delta A^{(L)}(t) \bigr\|_{L^2_x} + \ldots \\
  &\lesssim \sum_{k \in \Z} c_{k}^{(L-1)} \bigl( c_{\delta A, k}^{(L)} + d_{\delta A, k}^{(L)} \bigr) + \ldots  
 \end{align*}
 To see that this expression can be made arbitrarily small, we split
 \[
  \sum_{k \in \Z} c_{k}^{(L-1)} c_{\delta A, k}^{(L)} = \sum_{k \leq a_L - R} c_{k}^{(L-1)} c_{\delta A, k}^{(L)} + \sum_{a_L - R < k \leq a_L + R} c_{k}^{(L-1)} c_{\delta A, k}^{(L)} + \sum_{k > a_L + R} c_{k}^{(L-1)} c_{\delta A, k}^{(L)}.
 \]
 The first term can be made arbitrarily small for sufficiently large $R > 0$ by the exponential decay of the frequency envelope $\bigl\{ c_{\delta A, k}^{(L)} \bigr\}_{k\in\Z}$ beyond $[a_L, b_L]$. Similarly, we achieve smallness for the third term for sufficiently large $R > 0$ by the exponential decay of $\bigl\{ c_{k}^{(L-1)} \bigr\}_{k\in\Z}$ for $k > a_L$ established in Step 1. Finally, we gain smallness for the second term for all sufficiently large $n$ from the frequency evacuation property in Lemma \ref{lem:asymptoevacuate}. Moreover, we have by \eqref{equ:improved_bootstrap_bounds} that
 \[
  \sum_{k \in \Z} c_{k}^{(L-1)} d_{\delta A, k}^{(L)} \lesssim_{E_{crit}} \delta_2(\delta_1) \ll 1.
 \]
\end{proof}

We now summarize how the previous steps yield the proof of Proposition~\ref{prop:bootstrap}. In order to derive $S^1$ norm bounds on the evolutions $\bigl( A_{\Lambda_0}^{n1, (L)}, \phi_{\Lambda_0}^{n1, (L)} \bigr)$, we first use Proposition~\ref{prop:decomposition} from Step~2 to partition the time axis $\R$ into $N = N \Bigl( \bigl\| \bigl( A_{\Lambda_0}^{n1, (L-1)}, \phi_{\Lambda_0}^{n1, (L-1)} \bigr) \bigr\|_{S^1(\R\times\R^4)} \Bigr)$ many time intervals $I_1, \ldots, I_N$, on which certain ``divisible norms'' of $\bigl( A^{n1, (L-1)}_{\Lambda_0}, \phi_{\Lambda_0}^{n1, (L-1)} \bigr)$ are small in terms of $E_{crit}$. Let $I_1$ be the time interval containing $t=0$. By construction of the frequency intervals $J_L$, $1 \leq L \leq C_1$, the energy of the perturbation 
\[
 \bigl( \delta A^{(L)}, \delta \phi^{(L)} \bigr)[0] = \bigl( A^{n1, (L)}_{\Lambda_0} - A^{n1, (L-1)}_{\Lambda_0}, \phi^{n1, (L)}_{\Lambda_0} - \phi^{n1, (L-1)}_{\Lambda_0} \bigr)[0]
\]
at time $t=0$ is less than the absolute constant $\varepsilon_0$. Thus, we can prove frequency localized $S^1$ norm bounds for $\bigl( \delta A^{(L)}, \delta \phi^{(L)} \bigr)$ on $I_1 \times \R^4$ by bootstrap as in Step~3. Crucially, Lemma~\ref{lem:approximate_energy_conservation_for_perturbation} from Step~4 ensures that the energy of the perturbation $\bigl( \delta A^{(L)}, \delta \phi^{(L)} \bigr)[t]$, $t \in I_1$, is approximately conserved on the time interval $I_1$, up to a very small error term that is controlled by the size of $\delta_1$. Hence, we can ensure that at the starting points of all later (or earlier) time intervals $I_2, \ldots, I_N$, the energy of the perturbation is still less than the absolute constant $\varepsilon_0$ by choosing $\delta_1$ sufficiently small depending on the number $N$ of ``divisibility intervals'', which is bounded by the size of 
\[
 \bigl\| \bigl( A_{\Lambda_0}^{n1, (L-1)}, \phi_{\Lambda_0}^{n1, (L-1)} \bigr) \bigr\|_{S^1(\R\times\R^4)} \leq C_2.
\]
 This allows us to repeat the same bootstrap argument from Step~3 on all other time intervals $I_2, \ldots, I_N$. Putting all estimates together, we then obtain the desired $S^1$ norm bounds 
\[
 \bigl\| \bigl( A_{\Lambda_0}^{n1, (L)}, \phi_{\Lambda_0}^{n1, (L)} \bigr) \bigr\|_{S^1(\R\times\R^4)} \leq C_3(C_2),
\]
where the bound $C_3$ only depends on the size of $C_2$.
\end{proof}

\subsection{Interlude: Proofs of Proposition~\ref{prop:breakdown_criterion_frequency_envelopes}, Proposition~\ref{prop:perturbation}, Proposition~\ref{prop:morelocaldatastuff} and Proposition~\ref{prop:minenblowup}} \label{subsec:interlude}

\begin{proof}[Proof of Proposition \ref{prop:breakdown_criterion_frequency_envelopes}] Let $E$ denote the conserved energy of the admissible solution $(A, \phi)$. Analogously to the proof of Proposition \ref{prop:decomposition}, we can partition the time interval $(-T_0, T_1)$ into $N = N \bigl( \|(A, \phi)\|_{S^1((-T_0, T_1)\times\R^4)} \bigr)$ many intervals $I$ such that
 \[
  A|_I = A^{free, (I)} + A^{nonlin, (I)}, \quad \Box A^{free, (I)} = 0
 \]
 and 
 \begin{align*}
  \bigl\| \nabla_{t,x} A^{free, (I)} \bigr\|_{L^\infty_t L^2_x(\R\times\R^4)} &\lesssim E^{1/2}, \\
  \bigl\| A^{nonlin, (I)} \bigr\|_{\ell^1 S^1(I\times\R^4)} &\ll 1, \\
  \|\phi\|_{S^1(I\times\R^4)} &\lesssim C(E),
 \end{align*}
 where $C(\cdot)$ grows at most polynomially. For each such interval $I$, say of the form $I = [t_0, t_1]$ for some $t_0 < t_1$, we let $\{ c_k \}_{k\in\Z}$ be a sufficiently flat frequency envelope covering the data $(A, \phi)[t_0]$ at time $t_0$. Then we show that the bootstrap assumption
 \[
  \bigl\| P_k A \bigr\|_{S^1_k(I\times\R^4)} + \bigl\| P_k \phi \bigr\|_{S^1_k(I\times\R^4)} \leq D c_k
 \]
 for $D = D(E)$ sufficiently large, implies the improved bound
 \[
  \bigl\| P_k A \bigr\|_{S^1_k(I\times\R^4)} + \bigl\| P_k \phi \bigr\|_{S^1_k(I\times\R^4)} \leq \frac{D}{2} c_k.
 \]
 We only discuss the equation for $\phi$, because the equation for $A$ is easier. It suffices to consider the case $k = 0$. On $I \times \R^4$ we may rewrite the equation for $\phi$ into the following frequency localized form
 \begin{equation} \label{equ:breakdown_criterion_phi_equation_freq_localized}
  \begin{split}
   \Box_{A^{free, (I)}}^p \bigl( P_0 \phi \bigr) &= - \bigl[ P_0, \Box_{A^{free, (I)}}^p \bigr] \phi - 2i P_0 \biggl( \sum_k P_{>k-C} A^{free, (I)}_j P_k \partial^j \phi \biggr) \\
   &\quad - 2i P_0 \bigl( A_j^{nonlin, (I)} \partial^j \phi - A_0 \partial_t \phi \bigr) + P_0 \bigl( i (\partial_t A_0) \phi + A^\alpha A_\alpha \phi \bigr).
  \end{split}
 \end{equation}
 In order to close the bootstrap argument we now translate the estimates in the proof of Proposition \ref{prop:decomposition} into the language of frequency envelopes. For example, to bound the high-high interactions in the term
 \[
  P_0 \biggl( \sum_k P_{>k-C} A^{free, (I)}_j P_k \partial^j \phi \biggr),
 \]
 we use estimate (131) from \cite{KST} to obtain
 \begin{align*}
  \sum_{\substack{ k_1 = k_2 + O(1) \\ k_1 > O(1) }} \bigl\| P_0 \bigl( P_{k_1} A^{free, (I)}_j P_{k_2} \partial^j \phi \bigr) \bigr\|_{N_0(I\times\R^4)} &\lesssim \sum_{\substack{ k_1 = k_2 + O(1) \\ k_1 > O(1) }} 2^{-\delta k_1} \bigl\| P_{k_1} A^{free, (I)} \bigr\|_{S^1(I\times\R^4)} \bigl\| P_{k_2} \phi \bigr\|_{S^1(I\times\R^4)} \\
  &\lesssim \sum_{\substack{ k_1 = k_2 + O(1) \\ k_1 > O(1) }} 2^{-\delta k_1} \bigl\| P_{k_1} A^{free, (I)} \bigr\|_{S^1(I\times\R^4)} c_{k_2}.
 \end{align*}
 Summing over all sufficiently large $k_1 \gg 1$ and using the properties of frequency envelopes, we can bound this by 
 \[
  \leq \frac{D}{2} c_0.
 \]
 This allows us to reduce to the case $k_1 = k_2 + O(1) = O(1)$. Here we gain the necessary smallness by further partitioning the interval $I$ (where the total number of subintervals depends only on the size of $E$), using exactly the same divisibility argument as for the term \eqref{equ:decomposition_A_free_phi_term_divisibility} in the proof of Proposition \ref{prop:decomposition}. All other terms on the right hand side of \eqref{equ:breakdown_criterion_phi_equation_freq_localized} can be treated analogously to the above argument.
\end{proof}

\medskip

\begin{proof}[Proof of Proposition \ref{prop:perturbation}]
Here we are in the situation of Step 3 of the proof of Proposition~\ref{prop:bootstrap}. We obtain the bound 
\[
 \big\| (\delta A, \delta\phi) \big\|_{S^1([-T_0, T_1]\times\R^4)} \lesssim_L \delta_0
\]
for sufficiently small $\delta_0$ by means of a bootstrap argument performed on a finite number of space-time slices, whose number depends on $L$. We select these space-time slices as in Proposition~\ref{prop:decomposition}. The main difficulty arises from the equation for $\phi$. As in Step 3a of the proof of Proposition~\ref{prop:bootstrap}, we localize the equation for $\phi$ to frequency $0$ on a suitable space-time slice $I \times \R^4$ with $0 \in I$. Then, as in Step 3b there, the main difficulty comes from the new low-high interaction term 
\[
 P_{<0} (\delta A)_j P_0 \partial^j \phi - P_{<0} (\delta A)_0 P_0 \partial_t \phi.
\]
Using notation from the proof of Proposition \ref{prop:bootstrap}, the worst contribution comes from
\begin{equation} \label{equ:perturbation_result_low_high_A_free_interaction_term}
 P_{<0} \delta A_j^{free_1, (I)} P_0 \partial^j \phi,
\end{equation}
where we recall that $\delta A^{free_1, (I)}$ is the free evolution of the data $\delta A[0]$. We observe that for $\delta A^{free_1, (I)}$ the interaction term \eqref{equ:perturbation_result_low_high_A_free_interaction_term} vanishes by assumption on the frequency support of $\delta A[0]$ unless $K \leq 0$. By our assumption on the frequency supports of the data $(A, \phi)[0]$ and by Proposition~\ref{prop:breakdown_criterion_frequency_envelopes}, we obtain 
\[
 \bigl\| P_0 \phi \bigr\|_{S^1([-T_0, T_1]\times\R^4)} \lesssim_L 2^{\sigma K}.
\]
More generally, replacing the frequency $0$ by $l \in \Z$ with $l\geq K$, we have
\[
 \bigl\| P_l \phi \bigr\|_{S^1([-T_0, T_1]\times\R^4)} \lesssim_L 2^{- \sigma (l - K)}.
\]
By estimate \eqref{equ:low_high_N_estimate_for_free_wave}, we then find for $l \geq K$ that
\begin{align*}
 \bigl\| P_{<l} \delta A_j^{free_1, (I)} P_l \partial^j \phi \bigr\|_{N_l(I\times\R^4)} \lesssim_L \delta_0 |l-K| 2^{- \sigma (l-K)},
\end{align*}
where the extra factor $ |l-K|$ arises due to the $\ell^1$ summation over the frequencies of $P_{<l} \delta A^{free_1, (I)}$. But then we get the bound 
\[
 \biggl( \sum_{l \geq K} \, \bigl\| P_{<l} \delta A_j^{free_1, (I)} P_l \partial^j \phi \bigr\|_{N_l(I\times\R^4)}^2 \biggr)^{\frac{1}{2}} \lesssim_{L} \delta_0,
\]
which gives the required smallness for this term. Then the argument proceeds as for Proposition~\ref{prop:bootstrap}. 
\end{proof}

\medskip

\begin{proof}[Proof of Proposition~\ref{prop:morelocaldatastuff}]
We write for large $R_2 \geq R_1 \geq R_0$
\[
 \bigl( \tilde{A}_{R_2}, \tilde{\phi}_{R_2} \bigr) = \bigl( \tilde{A}_{R_1} + \delta A, \tilde{\phi}_{R_1} + \delta \phi \bigr).
\]
Then we analyze the equations for $(\delta A, \delta\phi)$. In fact, the only new feature occurs for the $\delta\phi$ equation and so we explain this here. We obtain the equation 
\[
 \Box_{\tilde{A}_{R_1}+\delta A}(\delta\phi) + \big(\Box_{\tilde{A}_{R_1}+\delta A} - \Box_{\tilde{A}_{R_1}}\big)\tilde{\phi}_{R_1} = 0. 
\]
Here we only retain the key difficult term that cannot be treated via a perturbative argument, using suitable divisibility properties as for example done in great detail in Step 3 of the proof of Proposition~\ref{prop:bootstrap}. This term is given by 
\[
\sum_{k\in \Z} P_{<k} (\delta A^{free})_j P_k \partial^j \tilde{\phi}_{R_1}.
\]
However, since we localize to a small time interval $[-T, T]$ around $t = 0$, it will be possible to obtain good $N$ norm bounds. Note that on account of the estimates in Subsection~\ref{subsec:proof_of_joint_time_interval_prop}, we may assume that 
\[
 \limsup_{R\rightarrow\infty} \big\|(\tilde{A}_R, \tilde{\phi}_R)\big\|_{S^1([-T, T]\times \R^4)} < \infty
\]
for all $R$ sufficiently large, provided that $T$ is sufficiently small. We shall assume, as we may that $T<1$. Then write 
\begin{equation} \label{equ:localize_in_physical_space_difficult_term_split}
 \begin{split}
  \sum_k P_{<k} (\delta A^{free})_j P_k \partial^j \tilde{\phi}_{R_1} = &\sum_k P_{< \, \min\{-\frac{1}{2}\log R_1, k\} } (\delta A^{free})_j P_k \partial^j \tilde{\phi}_{R_1} \\
  &\quad + \sum_{k} P_{[-\frac{1}{2}\log R_1, k]} (\delta A^{free})_j P_k \partial^j \tilde{\phi}_{R_1}.
 \end{split}
\end{equation}
The last term will be estimated by taking advantage of Huygens' principle as well as our particular choice of initial data, namely that $\tilde{\phi}_{R_1}[0]$ is supported on the set $\bigl\{ |x|\leq \frac{R_1}{10} \bigr\}$, while $\delta A[0]$ is supported  on $\bigl\{ |x| \geq R_1 \bigr\}$ up to tails that essentially decay exponentially fast. We now estimate both terms on the right hand side of \eqref{equ:localize_in_physical_space_difficult_term_split}. For the first term we find
\begin{align*}
&\big\| \sum_k P_{< \, \min\{-\frac{1}{2}\log R_1, k\} } (\delta A^{free})_j P_k \partial^j \tilde{\phi}_{R_1} \big\|_{N([-T, T] \times \R^4)} \\
&\lesssim \biggl( \sum_k \big\| P_{< \, \min\{-\frac{1}{2}\log R_1, k\} } (\delta A^{free})_j P_k \partial^j \tilde{\phi}_{R_1} \big\|_{L_t^1 L_x^2([-T, T]\times \R^4)}^2 \biggr)^{\frac{1}{2}} \\
&\lesssim \sup_{l} \big\| P_{< \, \min\{-\frac{1}{2}\log R_1, l\} } \delta A^{free} \big\|_{L^\infty_t L^\infty_x} \biggl( \sum_k \big\|P_k \nabla_x \tilde{\phi}_{R_1} \big\|_{L_t^\infty L_x^2([-T, T]\times \R^4)}^2 \biggr)^{\frac{1}{2}} \\
&\lesssim R_1^{-\frac{1}{2}} \big\|\delta A[0]\big\|_{\dot{H}^1_x\times L^2_x}\big\|\tilde{\phi}_{R_1}\big\|_{S^1([-T, T]\times \R^4)}
\end{align*}
and so this converges to $0$ as $R_1\rightarrow +\infty$. For the second term we have
\begin{align*}
&\big\| \sum_{k} P_{[-\frac{1}{2}\log R_1, k]} (\delta A^{free})_j P_k \partial^j \tilde{\phi}_{R_1} \big\|_{N([-T, T]\times \R^4)}\\
&\lesssim  \biggl( \sum_{k} \, \bigl\| P_{[-\frac{1}{2}\log R_1, k]} (\delta A^{free})_j P_k \partial^j \tilde{\phi}_{R_1} \big\|_{L_t^1 L_x^2([-T, T]\times \R^4)}^2 \biggr)^{\frac{1}{2}}\\
&\lesssim \biggl( \sum_{k} \big\|\chi_{ \{ |x|<\frac{R_1}{2} \} } P_{[-\frac{1}{2}\log R_1, k]} (\delta A^{free})_j P_k \partial^j \tilde{\phi}_{R_1} \big\|_{L_t^1 L_x^2([-T, T]\times \R^4)}^2 \biggr)^{\frac{1}{2}}\\
&\quad+ \biggl( \sum_{k} \big\| \chi_{ \{ |x|\geq \frac{R_1}{2} \} } P_{[-\frac{1}{2}\log R_1, k]} (\delta A^{free})_j P_k \partial^j \tilde{\phi}_{R_1} \big\|_{L_t^1 L_x^2([-T, T]\times \R^4)}^2 \biggr)^{\frac{1}{2}}.
\end{align*}
For the last term but one we use the localization properties of $\delta A^{free}$ to conclude
\begin{align*}
& \biggl( \sum_{k} \big\| \chi_{ \{ |x|<\frac{R_1}{2} \} } P_{[-\frac{1}{2}\log R_1, k]} (\delta A^{free})_j P_k \partial^j \tilde{\phi}_{R_1} \big\|_{L_t^1 L_x^2([-T, T]\times \R^4)}^2 \biggr)^{\frac{1}{2}} \\
&\lesssim \sup_{l > -\frac{1}{2} \log(R_1)} \big\| \chi_{ \{ |x|<\frac{R_1}{2} \} } P_{[-\frac{1}{2}\log R_1, l]} (\delta A^{free}) \big\|_{L^\infty_t L^\infty_x([-T, T]\times \R^4)} \biggl( \sum_{k} \big\|P_k \nabla_x \tilde{\phi}_{R_1}\big\|_{L_t^\infty L_x^2([-T, T]\times \R^4)}^2 \biggr)^{\frac{1}{2}} \\
&\lesssim R_1^{-M} \big\| \tilde{\phi}_{R_1} \big\|_{S^1([-T, T]\times \R^4)} \lesssim R_1^{-M},
\end{align*}
while for the last term, we get 
\begin{align*}
&\biggl( \sum_{k} \big\| \chi_{ \{ |x|\geq \frac{R_1}{2} \} } P_{[-\frac{1}{2}\log R_1, k]} (\delta A^{free})_j P_k \partial^j \tilde{\phi}_{R_1} \big\|_{L_t^1 L_x^2([-T, T]\times \R^4)}^2 \biggr)^{\frac{1}{2}} \\
&\lesssim \biggl( \sum_{k} \big\| P_{[-\frac{1}{2}\log R_1, k]} (\delta A^{free}) \big\|_{L^\infty_t L^\infty_x}^2 \big\| \chi_{|x|\geq \frac{R_1}{2}} P_k \nabla_x \tilde{\phi}_{R_1} \big\|_{L_t^\infty L_x^2([-T, T]\times \R^4)}^2 \biggr)^{\frac{1}{2}}.
\end{align*}
Here we use the localization properties of $\tilde{\phi}_{R_1}$ to bound the second factor by 
\[
 \big\| \chi_{|x|\geq \frac{R_1}{2}} P_k \nabla_x \tilde{\phi}_{R_1} \big\|_{L_t^\infty L_x^2([-T, T]\times \R^4)}^2 \lesssim \Bigl( 2^{\max\{k,0\}} R_1 \Bigr)^{-M}
\]
as long as $k > -\frac{1}{2} \log R_1$ as we may assume and also we have the crude bound 
\[
 \big\| P_{[-\frac{1}{2}\log R_1, k]} (\delta A^{free}) \big\|_{L^\infty_t L^\infty_x} \lesssim 2^k,
\]
whence we finally obtain the bound 
\begin{align*}
\biggl( \sum_{k} \big\| P_{[-\frac{1}{2}\log R_1, k]} (\delta A^{free}) \big\|_{L^\infty_t L^\infty_x}^2 \big\| \chi_{|x|\geq \frac{R_1}{2}} P_k \nabla_x \tilde{\phi}_{R_1} \big\|_{L_t^\infty L_x^2([-T, T]\times \R^4)}^2 \biggr)^{\frac{1}{2}} \lesssim R_1^{-M}.
\end{align*}
Letting $R_1 \rightarrow +\infty$ then again gives the required smallness. 
\end{proof}

\medskip

\begin{proof}[Proof of Proposition \ref{prop:minenblowup}] 
In view of Lemma~\ref{lem:characerization_maximal_lifespan}, it suffices to consider the case $I = \R$. We argue by contradiction. Assume that we have 
\begin{equation} \label{eq:nominblow1}
 \|(A, \phi)\|_{S^1(\R \times \R^4)} < \infty.
\end{equation}
Then the idea is that using this ingredient as well as a correct perturbative ansatz for the evolutions $(A_n, \phi_n)$ for $n$ large enough, we can show that the corresponding $S^1$ norms of $(A_n, \phi_n)$ must stay finite, contradicting the assumption. We introduce the perturbative term $\delta A_n$ for the magnetic potential by
\[
 A_n = A + \delta A_n
\]
and the perturbative term $\delta \phi_n$ by means of 
\[
 \phi_n = \chi_{I_1} \phi + \chi_{I_2} \tilde{\phi}_{A, \delta A_n^{free}} + \delta \phi_n,
\]
where $I_1$ is a very large time interval centered around $t = 0$ and $I_2$ represents the complement. The function $\tilde{\phi}_{A, \delta A_n^{free}}$ solves the wave equation 
\begin{align*}
 \tilde{\Box}_{A + \delta A_n^{free}}(\tilde{\phi}_{A, \delta A_n^{free}}) &= 0, \\
 \tilde{\phi}_{A, \delta A_n^{free}}[0] &= \phi[0], 
\end{align*}
where
\[
 \tilde{\Box}_{A + \delta A_n^{free}} = \Box + 2 i (A + \delta A_n^{free})_\nu \partial^\nu
\]
and in this context we let $\delta A_n^{free}$ be the actual free evolution of the data $\delta A_n[0]$ (as usual only the spatial components). We let $\chi_{I_1}, \chi_{I_2}$ be a smooth partition of unity subordinate to dilates of the intervals $I_1$, $I_2$. 
We note that in this argument one has to in fact replace the energy class solution $(A,\phi)$ by the evolution of a low frequency approximation of the energy class data very close to it and then show that this implies $S^1$ norm bounds for $(A_n, \phi_n)$ uniformly for all sufficiently close low frequency approximations.

To begin with, observe that we can show by a variant of the proof of Lemma~\ref{lem:uniform_dispersive_bounds}, proved later independently, that given any $\gamma > 0$ and choosing $I_1$ suitably large (depending on $A, \phi, \gamma$), we can arrange that 
\[
 \tilde{\phi}_{A, \delta A_n^{free}} = \big( \tilde{\phi}_{A, \delta A_n^{free}} \big)_1 + \big( \tilde{\phi}_{A, \delta A_n^{free}} \big)_2
\]
with
\[
\big\| \big(\tilde{\phi}_{A, \delta A_n^{free}}\big)_1 \big\|_{S^1} < \gamma, \quad \big\|\chi_{I_2} \big(\tilde{\phi}_{A, \delta A_n^{free}}\big)_2 \big\|_{L^\infty_t L^\infty_x} < \gamma. 
\]
Now the equation for $\delta\phi_n$ becomes the following 
\begin{align*}
 \Box_{A + \delta A_n} \delta \phi_n &= -\chi_{I_1} \Box_{A+\delta A_n} \phi - \chi_{I_2} \Box_{A+\delta A_n} \tilde{\phi}_{A, \delta A_n^{free}} + (\partial_t^2 \chi_{I_1}) \bigl( \phi - \tilde{\phi}_{A, \delta A_n^{free}} \bigr) \\
 &\quad \quad + 2 (\partial_t \chi_{I_1}) \bigl( \partial_t \phi - \partial_t \tilde{\phi}_{A, \delta A_n^{free}} \bigr) + 2i (\partial_t \chi_{I_1}) (A + \delta A_n)_0 \bigl( \phi - \tilde{\phi}_{A, \delta A_n^{free}} \bigr).
\end{align*}
The error term $\partial_t^2(\chi_{I_1})(\phi - \tilde{\phi}_{A, \delta A_n^{free}})$ is potentially problematic, because we cannot place the factor 
\[
 \bigl( \phi - \tilde{\phi}_{A, \delta A_n^{free}} \bigr)
\]
into $L_t^{\infty} L_x^2$. In fact, the latter is only possible provided we have compact spatial support (precisely, in a ball of radius $R$ with  $1 \ll R \leq |I_1|$) according to the Huygens principle, because then the extra factor $|I_1|^{-1}$ stemming from $ \partial_t^2(\chi_{I_1})$ will counterbalance the factor $|I_1|$ in 
\[
 \big\|\phi - \tilde{\phi}_{A, \delta A_n^{free}}\big\|_{L^\infty_t L_x^2(I_1\times\R^4)} \lesssim |I_1| \big\|\nabla_x \bigl( \phi - \tilde{\phi}_{A, \delta A_n^{free}} \bigr)\big\|_{L^\infty_t L_x^2(I_1\times\R^4)}.
\]
Here it is natural to truncate the data $\phi[0]$ in physical space to force this spatial localization later in time via Huygens' principle, but one needs to ensure that this does not destroy the good $S^1$ norm bounds for $(A, \phi)$. In fact, since we use the same $A$ data, the argument for Proposition~\ref{prop:perturbation} applies to yield a global $S^1$ norm bound for the new $(A, \phi)$. We then incorporate the error due to truncating the data $\phi[0]$ into $\delta\phi_n$ (while $\delta A_n[0]$ remains unchanged!), and hence infer the desired bound 
\[
 \big\|\partial_t^2(\chi_{I_1}) \bigl( \phi - \tilde{\phi}_{A, \delta A_n^{free}} \bigr) \big\|_{N} \lesssim \big\|\nabla_x \bigl( \phi - \tilde{\phi}_{A, \delta A_n^{free}} \bigr) \big\|_{L_t^\infty L_x^2(I_1\times \R^4)}.
\]
This gives the required smallness provided we can make $ \big\|\nabla_x \bigl(\phi - \tilde{\phi}_{A, \delta A_n^{free}} \bigr) \big\|_{L_t^\infty L_x^2(I_1\times \R^4)}$ small. For this observe that (we omit the cubic interaction terms) 
\[
\Box_A \bigl(\phi - \tilde{\phi}_{A, \delta A_n^{free}} \bigr) = 2i (\delta A_n^{free})_j \partial^j \tilde{\phi}_{A, \delta A_n^{free}} + \ldots, 
\]
and further
\[
\big\|\chi_{I_1} 2i (\delta A_n^{free})_j \partial^j\tilde{\phi}_{A, \delta A_n^{free}}\big\|_{N} \leq C(I_1, \phi, A) \big\|\delta A_n[0]\big\|_{\dot{H}^1_x \times L^2_x}.
\]
This implies
\[
 \big\|\nabla_{t,x} \bigl( \phi - \tilde{\phi}_{A, \delta A_n^{free}}) \big\|_{L^\infty_t L_x^2(I_1\times\R^4)} \leq C(I_1, \phi, A) \big\|\delta A_n[0]\big\|_{\dot{H}^1_x \times L^2_x}
\]
and so we conclude that 
\begin{align*}
\big\|\partial_t(\chi_{I_1}) \bigl( \partial_t\phi - \partial_t\tilde{\phi}_{A, \delta A_n^{free}} \bigr)\big\|_{N} + \big\|\partial_t^2(\chi_{I_1}) \bigl(\phi - \tilde{\phi}_{A, \delta A_n^{free}} \bigr)\big\|_{N} \lesssim C(I_1, \phi, A) \big\|\delta A_n[0]\big\|_{\dot{H}^1_x\times L^2_x}.
\end{align*}
Furthermore, we can write
\[
-\chi_{I_1} \Box_{A+\delta A_n} \phi = -\chi_{I_1} \Box_{A + \delta A_n^{nonlin}} \phi - \text{error},
\]
where we use the decomposition 
\[
 \delta A_n = \delta A_n^{free} + \delta A_n^{nonlin}
\]
with the first term on the right hand side the free propagation of $\delta A_n[0]$. For the error term we get
\[
 \big\|\text{error}\big\|_{N} \leq C(|I_1|, A) \big\|\delta A_n[0]\big\|_{\dot{H}^1_x\times L^2_x}.
\]
Furthermore, by using a divisibility argument and subdividing time axis into $N \bigl(\|(A, \phi)\|_{S^1} \bigr)$ many time intervals $J$, using the argument for Proposition~\ref{prop:bootstrap}, we can force (for each such $J$)
\[
 \big\|\chi_{I_1} \Box_{A + \delta A_n^{nonlin}} \phi \big\|_{N(J\times \R^4)} \ll \|\delta \phi_n\|_{S^1} + \|\delta A_n\|_{S^1}.
\]
Similarly, we have 
\begin{align*}
 \big\|\chi_{I_2} \Box_{A + \delta A_n} \tilde{\phi}_{A, \delta A_n^{free}} \big\|_{N(J\times \R^4)} \ll \|\delta\phi_n\|_{S^1} + \|\delta A_n\|_{S^1},
\end{align*}
which then suffices for the bootstrap for $\delta\phi_n$. 

\medskip

Next, we consider the equation for $\delta A_n$, which is of the schematic form
\begin{align*}
 \Box \delta A_n &= \phi \cdot \nabla_x \phi - \bigl( \chi_{I_1} \phi \bigr) \cdot \nabla_x \bigl(\chi_{I_1} \phi \bigr) \\
 &\quad - \bigl( \chi_{I_1} \phi \bigr) \cdot \nabla_x \bigl( \chi_{I_2} \tilde{\phi}_{A, \delta A_n^{free}} + \delta \phi_n \bigr) \\
 &\quad - \bigl( \chi_{I_2} \tilde{\phi}_{A, \delta A_n^{free}} + \delta \phi_n \bigr) \cdot \nabla_x \bigl( \chi_{I_1} \phi + \chi_{I_2} \tilde{\phi}_{A, \delta A_n^{free}} + \delta \phi_n \bigr) \\
 &\quad + (A + \delta A_n) \bigl| \chi_{I_1} \phi + \chi_{I_2} \tilde{\phi}_{A, \delta A_n^{free}} + \delta \phi_n \bigr|^2 - A |\phi|^2.
\end{align*}
Then we make the following observations. The first line on the right hand side satisfies
\[
 \big\| \phi \cdot \nabla_x \phi - \bigl( \chi_{I_1} \phi \bigr) \cdot \nabla_x \bigl(\chi_{I_1} \phi \bigr) \big\|_{N} \leq \nu_1
\]
for any prescribed $\nu_1>0$, provided we pick $I_1$ sufficiently large. The reason for this is that this term is supported around the endpoints of $I_1$ (which is centered around $t = 0$), and since $\Box_A\phi = 0$, we obtain similarly to the proof of Lemma~\ref{lem:uniform_dispersive_bounds} the dispersive decay for $\phi$ at large times, which easily gives the desired smallness for the $N$ norm. For the second and third line on the right, we find
\begin{align*}
&\big\| \bigl( \chi_{I_1} \phi \bigr) \cdot \nabla_x \bigl( \chi_{I_2} \tilde{\phi}_{A, \delta A_n^{free}} + \delta \phi_n \bigr) \big\|_{N(J\times\R^4)} + \big\| \bigl( \chi_{I_2} \tilde{\phi}_{A, \delta A_n^{free}} + \delta \phi_n \bigr) \cdot \nabla_x \bigl( \chi_{I_1} \phi + \chi_{I_2} \tilde{\phi}_{A, \delta A_n^{free}} + \delta \phi_n \bigr) \big\|_{N(J\times \R^4)} \\ 
&\lesssim \nu_1 + M^{-1} \| \delta\phi_n \|_{S^1(J\times \R^4)} + C \| \delta\phi_n \|_{S^1(J\times \R^4)}^2,
\end{align*}
where $J$ is a member of a suitable partition of the time axis into $N\big( \|(A,\phi)\|_{S^1}, M\big)$ many intervals and $C$ is a universal constant. Here we exploit the uniform dispersive decay of $\tilde{\phi}_{A, \delta A_n^{free}}$. The last line is handled similarly,
\begin{align*}
&\big\| (A + \delta A_n) \bigl| \chi_{I_1} \phi + \chi_{I_2} \tilde{\phi}_{A, \delta A_n^{free}} + \delta \phi_n \bigr|^2 - A |\phi|^2 \big\|_{N(J\times \R^4)} \\
&\lesssim \nu_1 + C_1 M^{-1} \bigl( \|\delta\phi_n \|_{S^1(J\times \R^4)} + \| \delta A_n \|_{S^1(J\times \R^4)} \bigr) \\
&\quad + C_2 \bigl( \|\delta\phi_n\|_{S^1(J\times \R^4)}^2 + \|\delta A_n\|_{S^1(J\times \R^4)}^2 \bigr) + \|\delta A_n\|_{S^1(J\times\R^4)} \|\delta \phi_n\|_{S^1(J\times\R^4)}^2.
\end{align*}
Combining these bounds, we then finally infer for the interval $J$ containing $t = 0$
\begin{align*}
 \| \delta A_n \|_{S^1(J\times \R^4)} &\lesssim \|\delta A_n[0]\|_{\dot{H}^1_x\times L^2_x} + \nu_1 + C_1 M^{-1} \bigl( \|\delta\phi_n\|_{S^1(J\times \R^4)} + \|\delta A_n\|_{S^1(J\times \R^4)} \bigr) \\
 &\quad +  C_2 \bigl( \|\delta\phi_n\|_{S^1(J\times \R^4)}^2 + \|\delta A_n\|_{S^1(J\times \R^4)}^2 \bigr) + \|\delta A_n\|_{S^1(J\times\R^4)} \|\delta \phi_n\|_{S^1(J\times\R^4)}^2,
\end{align*}
which suffices to bootstrap the bound for $\|\delta A_n\|_{S^1}$ on $J$. The bootstrap on the remaining intervals follows by induction (and choosing $\nu_1$ and $\|\delta A_n[0]\|_{\dot{H}^1_x \times L^2_x}$ sufficiently small depending on $M$ and $\|(A, \phi)\|_{S^1}$). Finally, we observe that the $S^1$ norm bounds on $\phi$, $\tilde{\phi}_{A, \delta A_n^{free}}$, and $\delta\phi_n$ are ``inherited'' by the expression 
\[
 \chi_{I_1} \phi + \chi_{I_2} \tilde{\phi}_{A, \delta A_n^{free}} + \delta\phi_n
\]
on account of the support properties of the functions $\phi$, $\tilde{\phi}_{A, \delta A_n^{free}}$. 
\end{proof}

\subsection{Selecting concentration profiles and adding the first large frequency atom}  \label{subsec:adding_first_atom}

We recall that we decomposed the essentially singular sequence of data $\bigl\{ (A^n, \phi^n)[0] \bigr\}_{n \in \N}$ into frequency atoms
\begin{align*}
 A^n[0] &= \sum_{a=1}^{\Lambda_0} A^{na}[0] + A^n_{\Lambda_0}[0], \\
 \phi^n[0] &= \sum_{a=1}^{\Lambda_0} \phi^{na}[0] + \phi^n_{\Lambda_0}[0],
\end{align*}
where $\Lambda_0$ was chosen such that
\[
 \sum_{a \geq \Lambda_0 + 1} \limsup_{n \to \infty} E(A^{na}, \phi^{na}) < \varepsilon_0.
\]
Moreover, we remind the reader that the frequency atoms split the errors $\big( A^n_{\Lambda_0}, \phi^n_{\Lambda_0} \big)[0]$ into $\Lambda_0 + 1$ pieces $\big( A^{nj}_{\Lambda_0}, \phi^{nj}_{\Lambda_0} \big)[0]$, $1 \leq j \leq \Lambda_0 + 1$, ordered by the size of $|\xi|$ in their Fourier supports. 

Having established control over the evolution of the data $\big(A^{n1}_{\Lambda_0}, \phi^{n1}_{\Lambda_0}\big)[0]$ in the preceding subsections, we now add the components $\big(A^{n1}, \phi^{n1}\big)[0]$, i.e. we pass to the initial data 
\begin{equation} \label{equ:addfirstatom}
 \big(A^{n1}_{\Lambda_0} + A^{n1}, \phi^{n1}_{\Lambda_0} + \phi^{n1}\big)[0].
\end{equation}
Here we first have to understand the lack of compactness of the ``large'' added term $\big( A^{n1}, \phi^{n1} \big)[0]$. To this end we carry out a careful profile decomposition in physical space of the added data $\big( A^{n1}, \phi^{n1} \big)[0]$. To obtain a profile decomposition for the magnetic potential components $A^{n1}_j[0]$, $j = 1, \ldots, 4$, we just use the standard Bahouri-G\'erard method \cite{Bahouri-Gerard} to extract the profiles via the free wave evolution. However, for the $\phi$ field, we mimic \cite{KS} and select the concentration profiles by evolving the data $\phi^{n1}[0]$ using the following ``covariant'' wave operator
\begin{equation} \label{equ:definition_covariant_wave_operator}
 \widetilde{\Box}_{A^{n1}} := \Box + 2i \bigl( A^{n1}_{\Lambda_0, \nu} + A^{n1, free}_\nu \bigr) \partial^\nu.
\end{equation}
Here, the functions $A^{n1, free}_\nu$ are defined as the solutions to the free wave equation
\begin{align*}
 \left\{ \begin{aligned}
  \Box A^{n1, free}_j &= 0, \\
  A^{n1, free}_j[0] &= A^{n1}_j[0]
 \end{aligned} \right.
\end{align*}
for $j = 1, \ldots, 4$, while we simply put
\[
 A_0^{n1, free} \equiv 0.
\]
It follows from standard results that the solution $u$ to
\[
 \widetilde{\Box}_{A^{n1}} u = 0
\]
with initial data $u[0] \in \dot{H}^1_x(\R^4) \times L^2_x(\R^4)$ exists globally in time. Moreover, the parametrix construction from Section~\ref{sec:magnetic_wave_equation} together with suitable divisibility arguments yields that this solution satisfies the global $S^1$ norm bound
\begin{equation} \label{equ:covariant_wave_operator_S1_norm_bound}
 \|u\|_{S^1(\R\times\R^4)} \lesssim_{E_{crit}} \|\nabla_{t,x} u(0)\|_{L^2_x}.
\end{equation}

\medskip

At this point we emphasize that both the influence of the evolution of the low frequency magnetic potential $A^{n1}_{\Lambda_0}$ and the influence of the free wave evolution of the data $A^{n1}[0]$ are built into the ``covariant'' wave operator $\widetilde{\Box}_{A^{n1}}$. This is different from the situation for critical wave maps in~\cite{KS}, where only the corresponding low frequency components are built into the ``covariant'' wave operator there, see Definition~9.18 in \cite{KS}. The reason for this is that the interaction term
\[
 A^{n1}_\nu \partial^\nu \phi^{n1}
\]
where \emph{both factors} are essentially supported at frequency $\sim 1$, cannot be bounded due to the contribution of the free term $A^{n1, free}_\nu$. Thus, the $\phi$ field experiences not only an ``asymptotic'' twisting due to the contribution of the extremely low frequency components $A^{n1}_{\Lambda_0}$ (as is the case for critical wave maps), but also from the frequency $\sim 1$ field $A^{n1, free}_\nu$. This needs to be reflected by our choice of concentration profiles.

\medskip

An important fact about the wave operator $\widetilde{\Box}_{A^{n1}}$ is that solutions to $\widetilde{\Box}_{A^{n1}} u = 0$ preserve the free energy in a certain asymptotic sense as $n \to \infty$. By rescaling we may assume that $\lambda^{n1} = 1$, which means that the frequency support of $\big( A^{n1}, \phi^{n1} \big)[0]$ is uniformly concentrated around $|\xi| \sim 1$. Precisely, we then have
\begin{lem} \label{lem:asymptotic_energy_conservation}
 Assume that the Schwartz data $u[0]$ is essentially supported at frequency $|\xi| \sim 1$ with $\|u[0]\|_{\dot{H}^1_x \times L^2_x} \lesssim 1$. Moreover, assume that $A^{n1}$ is $1$-oscillatory and that $A^{n1}_{\Lambda_0}$ satisfies a uniform $S^1$ norm bound
 \[
  \limsup_{n \to \infty} \big\|A^{n1}_{\Lambda_0}\big\|_{S^1} < \infty
 \]
 as well as $\sup_n \big\|A^{n1}_j[0]\big\|_{\dot{H}^1_x \times L^2_x} < \infty$ for $j = 1, \ldots, 4$. Let $\widetilde{\Box}_{A^{n1}}$ be defined as in \eqref{equ:definition_covariant_wave_operator}. Then the solutions $u(t, x)$ of the linear problem (with implicit $n$ dependence suppressed)
 \[
  \widetilde{\Box}_{A^{n1}} u = 0
 \]
 with fixed initial data $u[0]$ satisfy
 \begin{equation} \label{equ:asymptotic_energy_conservation_single}
  \lim_{R \to +\infty} \lim_{n \to \infty} \sup_{t \in \R_{+}} \big| \|\nabla_{t,x} u(R +t, \cdot)\|_{L_x^2}^2 - \|\nabla_{t,x} u(R, \cdot)\|_{L_x^2}^2 \big| = 0.
 \end{equation}
 The same holds even when replacing $+\infty$ by $-\infty$ and $\R_+$ by $\R_-$. Furthermore, assume that $u_k$ is a sequence of solutions to (again suppressing the $n$ dependence)
 \[
  \widetilde{\Box}_{A^{n1}} u_k = 0, 
 \]
 supported at frequency $|\xi|\sim 1$ (in the sense of $1$-oscillatory), and satisfying $S^1$ norm bounds uniform in $k$, while $u$ is as above (with fixed data $u[0]$). Then we have 
 \begin{equation} \label{equ:asymptotic_energy_conservation_mixed}
  \lim_{R \to +\infty} \lim_{n \to \infty} \sup_{t \in \R_{+}} \bigg| \int_{\R^4} \Big( \nabla_{t,x} u(t+R, x) \cdot \nabla_{t,x} u_k(t, x) - \nabla_{t,x} u(R, x) \cdot \nabla_{t,x} u_k(0, x) \Big) \, dx \bigg| = 0
 \end{equation}
 uniformly in $k$, and the same holds when replacing $+\infty$ by $-\infty$, and $\R_+$ by $\R_-$.
\end{lem}

In the proof of Lemma~\ref{lem:asymptotic_energy_conservation} we shall need the following uniform dispersive type bounds. These will also play a crucial role to control the interactions of the concentration profiles to be discussed below. Note that this is an analogue of Proposition~9.20 in \cite{KS} and is proved in an analogous fashion. 
\begin{lem} \label{lem:uniform_dispersive_bounds} 
 Let $u[0] \in \dot{H}^1_x(\R^4)\times L^2_x(\R^4)$ be fixed initial data and consider the solution $u(t, x)$ of the linear problem 
 \[
  \widetilde{\Box}_{A^{n1}} u = 0
 \]
 with given data $u[0]$ at time $t = 0$. Then for any $\gamma > 0$, there exists a decomposition 
 \[
  u = u_1 + u_2
 \]
 such that
 \[
  \|u_2\|_{S^1} < \gamma
 \]
 and there exists a time $t_0 = t_0\big(u[0], \gamma\big)$ such that for any $|t|>t_0$, 
 \[
  \|u_1(t, \cdot)\|_{L_x^\infty} < \gamma.
 \] 
\end{lem}
\begin{proof}
 We first prove the dispersive type bounds for solutions to the microlocalized magnetic wave equation
 \begin{align*}
  \Box_A^p u \equiv \Box u + 2i \sum_{k \in \Z} P_{\leq k-C} A_j^{free} P_k \partial^j u &= 0
 \end{align*}
 with initial data $u[0] = (f, g) \in \dot{H}^1_x(\R^4) \times L^2_x(\R^4)$. Here, the spatial components $A_j^{free}$ of the magnetic potential are in Coulomb gauge and are solutions to the free wave equation. We recall that the magnetic wave operator $\Box_A^p$ was treated in detail in Section~\ref{sec:magnetic_wave_equation}. The asserted dispersive type bounds for solutions to $\widetilde{\Box}_{A^{n1}} u = 0$ then follow by suitable divisibility arguments, which allow to iterate away the additional error terms.

 The main difference over the argument for wave maps in \cite[Proposition 9.20]{KS} is that we need to use a nested double iteration, on account of the fact that our parametrix for $\Box_A^p u = 0$ is only approximate. To begin with, we recall that the phase correction $\psi_{\pm}(t, x, \xi)$ defined in \eqref{equ:definition_phase} for the construction of the parametrix for, say, the frequency $0$ mode is truncated to low frequencies $k \leq -\frac{C_1}{\sigma}$. This generates the additional error terms 
 \[
  2i \sum_{-\frac{C_1}{\sigma} \leq k < 0} P_k A_j^{free} P_0 \partial^{j} u.
 \]
 These can only be iterated away by using divisibility, i.e. by restricting to a finite number of suitable time intervals. In fact, due to the summation over $k\in [-\frac{C_1}{\sigma}, 0]$, it is seen that this number of intervals needs to be proportional to $C_1$ (and also depends on the energy and $\sigma$, of course). Now we formally denote the (exact) Duhamel propagator for the equation $\Box_A^p u = F$ by 
 \[
  u(t, \cdot) = \int_0^t\tilde{U}(t-s)F(s)\,ds. 
 \]
 Moreover, we denote by $J_0, J_1, \ldots, J_N$ the partition of the forward time axis $[0,\infty)$ into consecutive time intervals on which the error terms 
 \[
  N^{lh}(u) := 2i \sum_{m \in \Z} \sum_{-\frac{C_1}{\sigma} + m \leq k < m} P_k A_j^{free} P_m \partial^{j} u
 \]
 as well as the remaining errors generated by the parametrix $\tilde{U}$ need to be handled by divisibility. As observed before, their number depends linearly on $C_1$ and implicitly on the energy and $\sigma$. We write $J_i = [t_i, t_{i+1}]$ for $0 \leq i \leq N-1$ with $t_0 = 0$ and $J_N = [t_N, \infty)$. Then, proceeding by exact analogy to the proof of Proposition 9.20 in \cite{KS}, we can write for $u^{(i)} := u|_{J_i}$, 
 \begin{align*}
  u^{(i)} &= \sum_{l=0}^{\infty} u^{(J_i, l)}, \quad u^{(J_i, 0)}(t) = \tilde{S}(t-t_i) u^{(i-1)}[t_i], \\
  u^{(J_i, l)}(t) &= -\int_{t_i}^t\tilde{U}(t-s)N^{lh}(u^{(J_i, l-1)})(s)\,ds,
 \end{align*}
 where $\tilde{S}$ is the homogeneous data propagator for $\Box_A^p$, while $\tilde{U}$ is the homogeneous propagator for data of the special form $(0, g)$.  Then the inductive nature of the construction is revealed by the relation (see (9.74) in \cite{KS})
 \[
  u^{(J_i, 0)}(t) = \tilde{S}(t)(f, g) - \sum_{k=0}^{i-1} \sum_{l=1}^{\infty}\int_{t_k}^{t_{k+1}}\tilde{U}(t-s)N^{lh}(u^{(J_k, l)})(s)\,ds. 
 \]
 The new aspect of our setting is that the propagators $\tilde{U}, \tilde{S}$ themselves are only obtained as infinite convergent sums of further terms, which need to be analyzed. Our strategy is to reduce precisely to the situation treated in \cite{KS}, by using the error analysis in \cite{KST}. Thus, denote the approximate inhomogeneous Duhamel parametrix by 
 \[
  \int_0^t\tilde{U}^{(app)}(t-s)F(s)\,ds.
 \]
 Note that due to Proposition 7 in \cite{KST}, the parametrix $\tilde{U}^{(app)}(t-s)$ is given by an integral kernel that satisfies the same decay estimates as the standard d'Alembertian propagator, independent of the precise potential $A^{free}$ used (but with implicit constants depending on its energy, of course). Then recall from the proof of Theorem 4  in \cite{KST} that we may write 
\[
 \int_0^t\tilde{U}(t-s)F(s)\,ds = \sum_{j=0}^{\infty} \int_0^t\tilde{U}^{(app)}(t-s)F_j(s)\,ds, 
 \]
 where we have $F_0 = F$ and writing inductively
 \[
  B_j :=  \int_0^t\tilde{U}^{(app)}(t-s)F_j(s)\,ds,
 \]
 we have for $j \geq 1$,
 \[
  F_j = F_j^1 + F_j^2 + F_j^3 + F_j^4
 \]
 with (schematic notation following \cite{KST})
 \begin{align*}
  P_0 F_j^1 &= \Big( \Box_{A_{<0}}^pe_{<0}^{-i\psi_{\pm}}(t, x D) - e_{<0}^{-i\psi_{\pm}}(t, x D)\Box \Big) P_0 B_{j-1}, \\
  P_0 F_j^2 &= \frac{1}{2} \Big( e_{<0}^{-i\psi_{\pm}}(t, x, D)e_{<0}^{i\psi_{\pm}}(D, y, t) - 1 \Big) P_0 F_{j-1}, \\
  P_0 F_j^3 &= \frac{1}{2} \Big( e_{<0}^{-i\psi_{\pm}}(t, x, D)|D|^{-1}e_{<0}^{i\psi_{\pm}}(D, y, t) - |D|^{-1} \Big) P_0 \partial_t F_{j-1}, \\
  P_0F_j^4 &= \frac{1}{2} \Big( e_{<0}^{-i\psi_{\pm}}(t, x, D)|D|^{-1} \partial_t e_{<0}^{i\psi_{\pm}}(D, y, t) - |D|^{-1} \Big) P_0 F_{j-1}.
 \end{align*}
 Here the first term, which is treated in Section 10.2 in \cite{KST}, gains a smallness factor of the form $2^{-\sigma C_1}$, which of course overwhelms any losses polynomial in $C_1$ for $C_1\gg 1$. However, the remaining three terms do not gain smallness from $C_1$, but rather by divisibility, and so we have to be more careful to force smallness for them (we cannot make the number of intervals depend on the prescribed smallness threshold $\gamma$). Here we exploit the fact that due to Proposition 6 in \cite{KST}, the kernels of the operators
 \begin{align*}
  &\frac{1}{2} \Big( e_{<0}^{-i\psi_{\pm}}(t, x, D)e_{<0}^{i\psi_{\pm}}(D, y, t) - 1 \Big), \quad \frac{1}{2} \Big( e_{<0}^{-i\psi_{\pm}}(t, x, D)|D|^{-1}e_{<0}^{i\psi_{\pm}}(D, y, t) - |D|^{-1} \Big), \\
  &\frac{1}{2} \Big( e_{<0}^{-i\psi_{\pm}}(t, x, D)|D|^{-1}\partial_t e_{<0}^{i\psi_{\pm}}(D, y, t) - |D|^{-1} \Big)
 \end{align*}
 are rapidly decaying away from the diagonal $x = y$. This means that up to small errors (which may be incorporated into the small energy part of $u$), we may think of these operators as local ones, and then the estimates in the proof of  Proposition 9.20 in \cite{KS} which rely on the inductive bound (9.81) there, go through for the error terms $F_j^r, r = 2,3,4$, as long as $F_{j-1}^r, r = 2,3,4$, satisfy these bounds. This means that the inductive argument in \cite{KS} goes through here as well. 
\end{proof}

\medskip

We are now in a position to prove the asymptotic energy conservation for solutions to $\widetilde{\Box}_{A^{n1}} u = 0$ as stated in Lemma~\ref{lem:asymptotic_energy_conservation}.
\begin{proof}[Proof of Lemma~\ref{lem:asymptotic_energy_conservation}]
 We consider the natural energy functional 
 \[
  E_{A^{n1}}(u)(t) = \int_{\R^4} \frac{1}{2} \sum_{\alpha = 0}^4 \bigl| \bigl( \partial_\alpha u + i \bigl( A_{\Lambda_0, \alpha}^{n1} + A_\alpha^{n1,free} \bigr) u \bigr)(t,x) \bigr|^2 \, dx,
 \]
 where it is to be kept in mind that the potential $A$ is in Coulomb gauge. Differentiating this energy functional with respect to $t$ and using that $\widetilde{\Box}_{A^{n1}} u = 0$, we infer the following relation
 \begin{equation} \label{equ:energyalmostconserve}
  \begin{split}
   &E_{A^{n1}}(u)(R+T) - E_{A^{n1}}(u)(R) \\
   &= \Re \int_R^{R+T} \int_{\R^4} \big( \partial_t A^{n1}_{\Lambda_0, 0} \bigr) u \, \overline{ \big( \partial_t + iA^{n1}_{\Lambda_0, 0} \big) u}  \, dx \, dt \\
   &\quad + \Re \int_R^{R+T} \int_{\R^4} \Big( - \big( A^{n1}_{\Lambda_0,0} \big)^2 + \sum_j \big( A_{\Lambda_0,j}^{n1} + A_j^{n1, free} \big)^2 \Big) u \, \overline{ \big( \partial_t + i A^{n1}_{\Lambda_0,0} \big) u } \, dx \, dt \\
   &\quad + \sum_j \Re \int_R^{R+T} \int_{\R^4} \big( \partial_j + i \big( A^{n1}_{\Lambda_0,j} + A^{n1, free}_j \big) \big) u \, \overline{ i \big( \partial_t A^{n1}_{\Lambda_0, j} + \partial_t A^{n1, free}_j - \partial_j A_{\Lambda_0,0}^{n1} \big) u } \, dx \, dt.
 \end{split}
\end{equation}
We now show that uniformly in $T \geq 0$, the terms on the right hand side converge to zero as $n\rightarrow\infty$ and then $R \to +\infty$.

\medskip

The quartic and quintic terms are all expected to be straightforward and so we focus on the more difficult cubic interaction terms. Here we note that the cubic interaction terms
\[
 \int_R^{R+T} \int_{\R^4} \big( \partial_t A^{n1}_{\Lambda_0,0} \big) u \overline{\partial_t u} \, dx \, dt 
\]
and 
\[
 \sum_j \int_R^{R+T} \int_{\R^4} \partial_j u \, \overline{ i \big( \partial_j A^{n1}_{\Lambda_0,0} \big) u } \, dx \, dt
\]
are also easier to treat due to the inherent quadratic nonlinear struture of the temporal components $A^{n1}_{\Lambda_0, 0}$ as solutions to the elliptic compatibility equation of MKG-CG. 

\medskip

So we now consider the delicate cubic interaction terms
\begin{align}
 \sum_j \Re \int_R^{R+T} \int_{\R^4} \partial_j u \, \overline{i \big( \partial_t A^{n1}_{\Lambda_0, j} \big) u} \,dx \, dt &=  - \sum_j \int_R^{R+T} \int_{\R^4} \Im \big( \partial_j u \overline{u} \big) \, \big( \partial_t A^{n1}_{\Lambda_0, j} \big) \, dx \, dt, \label{equ:asymptotic_energy_conservation_delicate1} \\
 \sum_j \Re \int_R^{R+T} \int_{\R^4} \partial_j u \, \overline{i \big( \partial_t A^{n1, free}_j \big) u} \,dx \, dt &=  - \sum_j \int_R^{R+T} \int_{\R^4} \Im \big( \partial_j u \overline{u} \big) \, \big( \partial_t A^{n1, free}_j \big) \, dx \, dt. \label{equ:asymptotic_energy_conservation_delicate2}
\end{align}
We begin with the first term \eqref{equ:asymptotic_energy_conservation_delicate1}. The Coulomb condition satisfied by $\partial_t A^{n1}_{\Lambda_0, j}$ allows us to project the term $\Im \big(\partial_j u \overline{u} \big)$ onto its divergence-free part, which means that we can replace this by a null form of the schematic type
\[
 \Im \big(\partial_j u \overline{u} \big) \longrightarrow \Delta^{-1} \partial^i \mathcal{N}_{ij} \big(u, \overline{u} \big).
\]
Thus we reduce to bounding uniformly the following schematic integral
\[
 \sum_j \int_R^{R+T} \int_{\R^4} \Delta^{-1} \partial^i \mathcal{N}_{ij} \big( u, \overline{u} \big) \, \big( \partial_t A^{n1}_{\Lambda_0, j} \big) \, dx \, dt.
\]
Now we claim the microlocalized bound 
\begin{align*}
 &\bigg| \int_R^{R+T} \int_{\R^4} \Delta^{-1} \partial^i \mathcal{N}_{ij} \big( P_{k_1} u, \overline{P_{k_2} u} \big) \, P_{k_3} \big( \partial_t A^{n1}_{\Lambda_0, j} \big) \, dx \, dt \bigg| \\
 &\quad \quad \lesssim 2^{\sigma (\min\{k_1, k_2, k_3\} - \max\{k_1, k_2, k_3\})} \big\|P_{k_1} u \big\|_{S^1} \big\| P_{k_2} u \big\|_{S^1} \big\|P_{k_3} A^{n1}_{\Lambda_0}\big\|_{S^1}
\end{align*}
for suitable $\sigma>0$. Since there are at least two comparable frequencies in the above, this is enough to give the desired result in view of the frequency localizations of $u$ and $A^{n1}_{\Lambda_0}$. In order to prove this, we localize the above expression further and also omit the localization to the time interval $[R, R+T]$, as we may get rid of it via a suitable cutoff (which is compatible with the $S^1$ norms), 
\begin{align*}
 &\int_{\R^{1+4}} \Delta^{-1} \partial^i \mathcal{N}_{ij} \big( P_{k_1} u, \overline{P_{k_2} u} \big) \, P_{k_3} \big( \partial_t A^{n1}_{\Lambda_0, j} \big) \, dx \, dt \\
 &\quad = \int_{\R^{1+4}} \Delta^{-1} \partial^i \mathcal{N}_{ij} \big( P_{k_1} u, \overline{P_{k_2} u} \big) \, P_{k_3} Q_{> k_3}  \big( \partial_t A^{n1}_{\Lambda_0, j} \big) \, dx \, dt \\
 &\quad \quad + \int_{\R^{1+ 4}} \Delta^{-1} \partial^i \mathcal{N}_{ij} \big(P_{k_1} u, \overline{P_{k_2} u} \big) \, P_{k_3} Q_{\leq k_3} \big(\partial_t A^{n1}_{\Lambda_0, j} \big) \, dx \, dt.
\end{align*}
Here we only estimate the more difficult second term on the right hand side. We write this term as 
\begin{align*}
 &\int_{\R^{1+4}} \Delta^{-1} \partial^i \mathcal{N}_{ij} \big( P_{k_1} u, \overline{P_{k_2} u} \big) \, P_{k_3} Q_{\leq k_3} \big( \partial_t A^{n1}_{\Lambda_0, j} \big) \, dx \, dt \\
 &\quad = \sum_{l \leq k_3} \int_{\R^{1+4}} \Delta^{-1} \partial^i \mathcal{N}_{ij} \big( P_{k_1} u, \overline{P_{k_2} u} \big) \, P_{k_3} Q_l \big( \partial_t A^{n1}_{\Lambda_0, j} \big) \, dx \, dt.
\end{align*}
By symmetry we may assume $k_2\leq k_1$. Then we distinguish the following cases.

\medskip

\noindent {\it{Case 1: $k_1 = k_2+O(1) > k_3+O(1)$}}. Since the $Q_l$ transfers to the null form $\mathcal{N}_{ij}$, we save 
\[
 2^{k_3 - k_2 + \frac{1}{2} (l-k_3)}.
\]
Thus, we obtain
\begin{align*}
&\bigg| \int_{\R^{1+4}} \Delta^{-1} \partial^i \mathcal{N}_{ij} \big( P_{k_1} u, \overline{P_{k_2} u} \big) \, P_{k_3} Q_{l} \big(\partial_t A^{n1}_{\Lambda_0, j}\big) \, dx \, dt \bigg| \\
&\quad \lesssim 2^{k_3 - k_2 + \frac{1}{2} (l-k_3)}  2^{-k_3} \big\| P_{k_1} \nabla_x u\big\|_{L_t^4 L_x^3} \big\| P_{k_2} \nabla_x u\big\|_{L_t^4 L_x^3} \big\| P_{k_3} Q_l \big( \partial_t A^{n1}_{\Lambda_0, j} \big) \big\|_{L_t^2 L_x^3},
\end{align*}
where we observe that the exponent pair $(4,3)$ is Strichartz admissible in four space dimensions. Then we use the improved Sobolev type bound 
\[
 \big\| P_{k_3} Q_l \big( \partial_t A^{n1}_{\Lambda_0, j} \big) \big\|_{L_t^2 L_x^3} \lesssim 2^{\frac{2}{3} k_3} 2^{\gamma (l-k_3)} \big\| P_{k_3} Q_l \big( \partial_t A^{n1}_{\Lambda_0, j} \big) \big\|_{L_t^2 L_x^2}
\]
for suitable $\gamma>0$ to infer 
\begin{align*}
&\bigg| \int_{\R^{1+4}} \Delta^{-1} \partial^i \mathcal{N}_{ij} \big( P_{k_1} u, \overline{P_{k_2} u} \big) P_{k_3} Q_l \big( \partial_t A^{n1}_{\Lambda_0, j} \big) \, dx \, dt \bigg| \\
&\quad \lesssim  2^{k_3 - k_2 + \frac{1}{2} (l-k_3)} 2^{-k_3}  2^{\frac{5}{12} k_1}  2^{\frac{5}{12} k_2} 2^{-\frac{1}{2} l} 2^{\frac{2}{3} k_3} 2^{\gamma (l-k_3)} \big\|P_{k_1} u\big\|_{S^1} \big\|P_{k_2} u\big\|_{S^1} \big\|P_{k_3} A^{n1}_{\Lambda_0} \big\|_{S^1},
\end{align*}
which in turn can be bounded by 
\begin{align*}
\lesssim 2^{\frac{1}{6} (k_3 - k_2)} 2^{\gamma (l-k_3)}  \big\| P_{k_1} u \big\|_{S^1} \big\| P_{k_2} u\big\|_{S^1} \big\|P_{k_3} A^{n1}_{\Lambda_0}\big\|_{S^1}.
\end{align*}
Summing over $l \leq k_3$ yields the desired bound in this case. 

\medskip

\noindent {\it Case 2: $k_1 = k_3+O(1)>k_2$}. We distinguish between the cases $l \leq k_2$ and $l > k_2$. 

\medskip

\noindent {\it{Case 2a: $l \leq k_2$.}} Here we estimate 
\begin{align*}
 &\bigg| \int_{\R^{1+4}} \Delta^{-1} \partial^i \mathcal{N}_{ij} \big( P_{k_1} u, \overline{P_{k_2} u} \big) \, P_{k_3} Q_l \big( \partial_t A^{n1}_{\Lambda_0, j} \big) \, dx \, dt \bigg| \\
 &\lesssim 2^{-k_1} 2^{\frac{1}{2} (l - k_2)} \big\| \nabla_{t,x} P_{k_1} u \big\|_{L_t^6 L_x^{\frac{18}{7}}} \big\| \nabla_{t,x} P_{k_2} u \big\|_{L_t^3 L_x^{9}} \big\|P_{k_3} Q_{l} \big( \partial_t A^{n1}_{\Lambda_0} \big) \big\|_{L^2_t L^2_x} \\
 &\lesssim 2^{-\frac{13}{18} k_1} 2^{\frac{22}{18} k_2}2^{\frac{1}{2} (l-k_2)} 2^{-\frac{1}{2} l} \big\|P_{k_1} u\big\|_{S^1} \big\|P_{k_2} u\big\|_{S^1} \big\|P_{k_3} A^{n1}_{\Lambda_0} \big\|_{S^1}.
\end{align*}
To get summability over $l$ one can replace the norm $\|\cdot\|_{L^2_t L^2_x}$ by $\|\cdot\|_{L_{t}^2 L_x^{2+}}$ and then use $\|\cdot\|_{L_t^3 L_x^{9-}}$ instead for the second factor. 

\medskip

\noindent {\it{Case 2b: $l > k_2$.}} Here we simply get the bound 
\begin{align*}
\bigg|\int_{\R^{1+4}} \Delta^{-1} \partial^i \mathcal{N}_{ij} \, \big( P_{k_1} u, \overline{P_{k_2} u} \big) P_{k_3} Q_l \big( \partial_t A^{n1}_{\Lambda_0, j} \big) \, dx \, dt \bigg| \lesssim 2^{\frac{13}{18}(k_2 - k_1)} \big\| P_{k_1} u\big\|_{S^1} \big\| P_{k_2} u\big\|_{S^1} \big\|P_{k_3} A^{n1}_{\Lambda_0}\big\|_{S^1},
\end{align*}
which can then be summed over $k_1 > l > k_2$ to give the desired bound. This in essence finishes the estimate of the cubic interaction term \eqref{equ:asymptotic_energy_conservation_delicate1}. 
 
\medskip

Next, we consider the other delicate cubic interaction term \eqref{equ:asymptotic_energy_conservation_delicate2}. Using that $\partial_t A^{n1, free}_j$ also satisfies the Coulomb condition, we reduce as before to bounding uniformly the expression 
\[
 \sum_j \int_R^{R+T} \int_{\R^4} \Delta^{-1}\partial^i \mathcal{N}_{ij}\big(P_{k_1} u, \overline{P_{k_2} u} \big) \, P_{k_3} \big( \partial_t A^{n1, free}_{j} \big) \, dx \, dt.
\]
Compared with the treatment of the previous cubic term, the issue here is how to deal with the interactions of $u$ and $A^{n1, free}$, which are now both $1$-oscillatory. We may assume that all frequencies $2^{k_{1,2,3}} \sim 1$, otherwise smallness follows from the treatment of the previous cubic interaction term \eqref{equ:asymptotic_energy_conservation_delicate1}. Choosing $R > 0$ sufficiently large, we obtain from the dispersive decay from Lemma~\ref{lem:uniform_dispersive_bounds} and interpolation with the endpoint Strichartz estimate that 
\[
 \big\|P_{k_{1,2}} u\big\|_{L_t^4 L_x^{3+}([R, R+T]\times \R^4)} \ll 1
\]
uniformly for all $T \geq 0$ and $n$ (recalling that the implicit dependence of $u$ on $n$ is suppressed). On the other hand, for the factor $P_{k_3} \big( \partial_t A^{n1, free}_{j} \big)$, we can use $L_t^2 L_x^{3-}$ instead. 

\medskip

The last statement of the lemma follows similarly, by expressing the inner product in terms of the energies of $u$ and $u_k$, and reducing to bounding expressions such as 
\[
 \sum_j \int_R^{R+T} \int_{\R^4} \Delta^{-1} \partial^i \mathcal{N}_{ij} \big( P_{k_1} u, \overline{P_{k_2} u_k} \big) \, P_{k_3} \big( \partial_t A^{n1, free}_{j} \big) \, dx \, dt.
\] 
\end{proof}

We now begin to quantify the lack of compactness for the functions $\big\{ (A^{n1}, \phi^{n1})[0] \big\}_{n \in \N}$. To clarify the notation and make it adapted to the ensuing induction procedure, we replace the superscript $1$ in $(A^{n1}, \phi^{n1})[0]$ by $a$ to indicate the frequency level of the large frequency atom, although we are only considering $a = 1$ in this subsection. We first consider the functions $\big\{ \phi^{na}[0] \big\}_{n \in \N}$. We evolve each of these using the flow of the covariant wave operator $\widetilde{\Box}_{A^{na}}$ and extract concentration profiles. The method for this follows along the lines of the modified Bahouri-G\'erard profile extraction procedure of Lemma~9.23 in \cite{KS}. However, we have to use the asymptotic energy conservation from Lemma~\ref{lem:asymptotic_energy_conservation} instead of the stronger asymptotic energy conservation in \cite[Lemma 9.19]{KS}, which forces us to modify the asymptotic orthogonality relation for the free energies of the profiles. We first introduce the following terminology.
\begin{defn}
 Given initial data $u[0] \in \dot{H}^1_x(\R^4) \times L^2_x(\R^4)$, we denote by
 \[
  S_{A^{na}} \big( u[0] \big)
 \]
 the solution to the initial value problem $\widetilde{\Box}_{A^{na}} u = 0$ with data $u[0]$ at time $t=0$.
\end{defn}
Following \cite{KS}, which in turn mimics \cite{Bahouri-Gerard}, we introduce the set $\mathcal{U}_{A^{na}}(\phi^{na}[0])$, which consists of all functions that can be extracted as weak limits in the following fashion
\begin{align*}
 \mathcal{U}_{A^{na}}(\phi^{na}[0]) &= \Big\{ V \in L_{t,loc}^2 H^1_x \cap C^1 L_x^2 \, : \,\exists \, \big\{(t_n, x_n)\big\}_{n\geq 1} \subset \R\times\R^4 \,\, \text{s.t.} \\ &\hspace{5cm} S_{A^{na}}\big(\phi^{na}[0]\big)(t+t_n, x+x_n)\rightharpoonup V(t, x) \Big\}.
\end{align*}
Here the weak limit is in the sense of $L^2_{t,loc} H^1_x$. We emphasize that the sequences $\big\{(t_n, x_n)\big\}_{n\geq 1} \subset \R \times \R^4$ are completely arbitrary. We observe that for a non-trivial profile $V \in \mathcal{U}_{A^{na}}(\phi^{na}[0])$ with associated sequence of space-time translations $\big\{ (t_n^{ab}, x_n^{ab}) \big\}_{n \geq 1}$, by passing to a further subsequence, we may assume that either $S\big( A^{na}[0] \big)(t+t_n^{ab}, x+x_n^{ab}) \rightharpoonup 0$ or else $S\big( A^{na}[0] \big)(t+t_n^{ab}, x+x_n^{ab}) \rightharpoonup A^{ab}(t,x)$. Here, $S\big(\cdot\big)(t,x)$ denotes the free wave propagator and $A^{ab}$ are free waves, see Proposition~\ref{prop:A_concentration_profiles} below. Noting that the contribution of $A^{n1}_{\Lambda_0}$ in the definition of $\widetilde{\Box}_{A^{na}}$ vanishes in the limit, then in the former case we have $\Box V = 0$, i.e. $V$ is actually a weak solution to the free wave equation, while in the latter situation, $V$ solves the linear magnetic wave equation $\big( \Box + 2i A^{ab}_j \partial^j \bigr) V = 0$. 

Moreover, we set
\[
 \eta_{A^{na}} \big( \phi^{na}[0] \big) := \sup \Big\{ E_0(V) \, : \, V \in \mathcal{U}_{A^{na}}\big( \phi^{na}[0] \big) \Big\} < \infty,
\]
where $E_0$ refers to the functional
\[
 E_0(V) = \int_{\R^4} \big| \nabla_{t,x} V(0,x) \big|^2 \, dx.
\]
Observe that for temporally unbounded sequences, i.e. $|t_n| \to \infty$, the energy $E_0(V)$ is identical to the ``asymptotic free energy'' associated with solutions to $\big( \Box + 2i A^{ab}_j \partial^j \bigr) V = 0$ in light of Lemma~\ref{lem:asymptotic_energy_conservation}. 
In the next proposition we establish the crucial linear profile decomposition for the sequence $\big\{ \phi^{na}[0] \big\}_{n \in \N}$, which is at the core of the second stage of the modified Bahouri-G\'erard procedure for MKG-CG. Recall that we consider $a = 1$ here. 
\begin{prop} \label{prop:phi_concentration_profiles}
There exists a collection of sequences $\{(t_n^{ab}, x_n^{ab})\}_{n \in \N} \subset \R \times \R^4$, $b\geq 1$, as well as a corresponding family of concentration profiles 
\[
 \phi^{ab}[0] \in \dot{H}^1_x(\R^4)\times L_x^2(\R^4), \quad b \geq 1,
\]
with the following properties: Introducing the space-time translated gauge potentials 
\[
 \tilde{A}^{nab}_{\nu}(t, x) := A^{na}_{\Lambda_0, \nu}(t + t_n^{ab}, x+x_n^{ab}) + A^{na, free}_{\nu}(t+t_n^{ab}, x+x_n^{ab}), \quad \nu = 0, 1, \ldots, 4,
\]
we have 
\begin{itemize}[leftmargin=*]
\item For any $B \geq 1$, there exists a decomposition 
\begin{equation} \label{equ:phi_concentration_profiles_decomposition}
 S_{A^{na}}\big( \phi^{na}[0] \big)(t, x) = \sum_{b=1}^B S_{\tilde{A}^{nab}}\big( \phi^{ab}[0] \big)(t - t^{ab}_n, x-x^{ab}_n) + \phi^{naB}(t, x),
\end{equation}
where each of the functions 
\[
 \tilde{\phi}^{nab}(t,x) := S_{\tilde{A}^{nab}}\big( \phi^{ab}[0] \big)(t - t^{ab}_n, x-x^{ab}_n), \quad \phi^{naB}(t, x)
\]
solves the covariant wave equation 
\[
 \widetilde{\Box}_{A^{na}} u = 0.
\]
Moreover, the error satisfies the crucial asymptotic vanishing condition 
\begin{equation} 
 \lim_{B \to \infty} \eta_{A^{na}} \big( \phi^{naB}[0] \big) = 0.
\end{equation}
\item The sequences are mutually divergent, by which we mean that for $b\neq b'$, 
\begin{equation} \label{equ:phi_concentration_profiles_mutual_divergence}
 \lim_{n\rightarrow\infty} \big( |t_n^{ab} - t_{n}^{ab'}| + |x_n^{ab} - x_n^{ab'}| \big) = \infty.
\end{equation}
\item There is asymptotic energy partition
\begin{equation} \label{equ:phi_concentration_profiles_asymptotic_energy_partition}
 E_0(\phi^{na}[0]) = \sum_{b=1}^B E_0(\tilde{\phi}^{nab}[0]) + E_0(\phi^{naB}[0]) + o(1),
\end{equation}
where the meaning of $o(1)$ here is $\limsup_{n\rightarrow\infty} o(1) = 0$. 
\item All profiles $\phi^{ab}[0]$ as well as all errors $\phi^{naB}[0]$ are $1$-oscillatory.   
\end{itemize}
\end{prop}
Before we begin with the proof of Proposition~\ref{prop:phi_concentration_profiles}, we introduce the following important distinction between two possible types of profiles.
\begin{itemize}[leftmargin=*]
\item {\bf Temporally unbounded profiles:} Those profiles for which
\[
 \lim_{n \to \infty} |t_n^{ab}| = \infty.
\]
\item {\bf Temporally bounded profiles:} Those profiles for which 
\[
 \liminf_{n \to \infty} |t_{n}^{ab}| < \infty.
\]
By passing to a subsequence we may then as well assume that for all $n \in \N$,
\[
 t_n^{ab} = 0. 
\]
For two distinct such profiles corresponding to $b \neq b'$, we must have 
\[
 \lim_{n \to \infty} |x_n^{ab} - x_n^{ab'}| = \infty.
\]
\end{itemize}
\begin{proof}[Proof of Proposition~\ref{prop:phi_concentration_profiles}]
 There is nothing to do if 
 \[
  \eta_{A^{na}} \big( \phi^{na}[0] \big) = 0.
 \]
 Let us therefore assume that this quantity is strictly greater than $0$. Then we pick a profile
 \[
  \phi^{a1} \in L^2_{t,loc} H^1_x \cap C^1 L^2_x
 \]
 and an associated sequence $\big\{ (t_n^{a1}, x_n^{a1}) \big\}_{n \in \N} \subset \R \times \R^4$ such that
 \begin{equation} \label{equ:weak_convergence_first_profile_phi}
  S_{A^{na}}\big( \phi^{na}[0] \big)(t + t_n^{a1}, x + x_n^{a1}) \rightharpoonup \phi^{a1}(t,x)
 \end{equation}
 with
 \[
  E_0(\phi^{a1}) \geq \frac{1}{2} \eta_{A^{na}} \big( \phi^{na}[0] \big).
 \]
 Then we have
 \begin{align*}
  &S_{A^{na}} \big( \phi^{na}[0] \big)(t+t_n^{a1}, x+x_n^{a1}) - S_{A^{na}}\Big( S_{\tilde{A}^{na1}}\big(\phi^{a1}[0] \big)(0 - t_n^{a1}, \cdot - x_n^{a1}) \Big)(t+t_n^{a1}, x+x_n^{a1}) \\
  &\quad = S_{A^{na}} \big( \phi^{na}[0] \big)(t+t_n^{a1}, x+x_n^{a1}) - S_{\tilde{A}^{na1}}\big( \phi^{a1}[0] \big)(t,x) \rightharpoonup 0
 \end{align*}
 as $n \to \infty$ by the construction. Furthermore, it holds that
 \begin{align} \label{equ:phi_concentration_profiles_energy_partition_proof}
 \begin{aligned}
  E_0 \big( \phi^{na}[0] \big) &= E_0 \Big( \tilde{\phi}^{na1}[0] \Big) + E_0 \Big( \phi^{na}[0] - \tilde{\phi}^{na1}[0] \Big) \\
  &\qquad + 2 \int_{\R^4} \nabla_{t,x} \Big( S_{\tilde{A}^{na1}}\big( \phi^{a1}[0] \big)(0-t_n^{a1}, x - x_n^{a1}) \Big) \cdot \\
  &\qquad \qquad \qquad \cdot \nabla_{t,x} \Big( \phi^{na}(0,x) - S_{\tilde{A}^{na1}} \big( \phi^{a1}[0] \big)(0 - t_n^{a1}, x - x_n^{a1}) \Big) \, dx,
 \end{aligned}
 \end{align}
 where in the last term we ignored that the $\phi$ field is complex-valued. If $\phi^{a1}[0]$ is a temporally unbounded profile, we may without loss of generality assume that $t_n^{a1} \to + \infty$. In view of \eqref{equ:asymptotic_energy_conservation_mixed} from Lemma~\ref{lem:asymptotic_energy_conservation} the last term on the right hand side of \eqref{equ:phi_concentration_profiles_energy_partition_proof} can be arbitrarily well approximated by
 \begin{align} \label{equ:phi_concentration_profiles_energy_partition_proof_last_term}
  \begin{aligned}
   &2 \int_0^1 \int_{\R^4} \nabla_{t,x} S_{A^{na}} \Big( S_{\tilde{A}^{na1}}\big( \phi^{a1}[0] \big)(0-t_n^{a1}, \cdot - x_n^{a1}) \Big)(t - R + t_n^{a1}, x+x_n^{a1}) \cdot \\
   &\qquad \qquad \cdot \nabla_{t,x} S_{A^{na}} \Big( \phi^{na}[0] - S_{\tilde{A}^{na1}} \big( \phi^{a1}[0] \big)(0 - t_n^{a1}, \cdot - x_n^{a1}) \Big)(t - R +t_n^{a1}, x+x_n^{a1}) \, dx \, dt
  \end{aligned}
 \end{align}
 as $n \to \infty$ by choosing $R > 0$ sufficiently large. Then we observe that the first factor in the integrand in \eqref{equ:phi_concentration_profiles_energy_partition_proof_last_term} satisfies
 \[
  \nabla_{t,x} S_{A^{na}} \Big( S_{\tilde{A}^{na1}}\big( \phi^{a1}[0] \big)(0-t_n^{a1}, \cdot - x_n^{a1}) \Big)(t - R + t_n^{a1}, x+x_n^{a1}) = \nabla_{t,x} \phi^{a1}(t-R,x) + o_{L^2_x}(1)
 \]
 as $n \to \infty$, while by construction 
 \[
  S_{A^{na}} \Big( \phi^{na}[0] - S_{\tilde{A}^{na1}} \big( \phi^{a1}[0] \big)(0 - t_n^{a1}, \cdot - x_n^{a1}) \Big)(\cdot+t_n^{a1}, \cdot+x_n^{a1}) \rightharpoonup 0 
 \]
 weakly in $L^2_{t,loc} H^1_x$ as $n \to \infty$. Thus, we conclude that
 \[
  E_0 \big( \phi^{na}[0] \big) = E_0 \Big( \tilde{\phi}^{na1}[0] \Big) + E_0 \Big( \phi^{na}[0] - \tilde{\phi}^{na1}[0] \Big) + o(1)
 \]
 as $n \to \infty$. If instead $\phi^{a1}[0]$ is a temporally bounded profile, we may and shall have $t_n^{a1} = 0$ for all $n \in \N$. Then the last term on the right hand side of \eqref{equ:phi_concentration_profiles_energy_partition_proof} is given by
 \[
  2 \int_{\R^4} \nabla_{t,x} \phi^{a1}(0, x-x_n^{a1}) \cdot \nabla_{t,x} \bigl( \phi^{na}(0,x) - \phi^{a1}(0, x-x_n^{a1}) \bigr) \, dx,
 \]
 which vanishes as $n \to \infty$ by the weak convergence \eqref{equ:weak_convergence_first_profile_phi} and therefore yields the desired asymptotic energy partition \eqref{equ:phi_concentration_profiles_asymptotic_energy_partition}.

 \medskip

 Now we repeat this procedure, but replace $\phi^{na}[0]$ by
 \[
  \phi^{na}[0] - \tilde{\phi}^{na1}[0].
 \]
 Thus, if $\eta_{A^{na}}\bigl( \phi^{na}[0] - \tilde{\phi}^{na1}[0] \bigr) > 0$, we select a sequence $\bigl\{ (t_n^{a2}, x_n^{a2}) \bigr\}_{n\in\N}$ and a concentration profile $\phi^{a2}(t,x)$ such that 
 \[
  E_0(\phi^{a2}) \geq \frac{1}{2} \eta_{A^{na}} \bigl( \phi^{na}[0] - \tilde{\phi}^{na1}[0] \bigr) 
 \]
 and 
 \[
  S_{A^{na}} \bigl( \phi^{na}[0] - \tilde{\phi}^{na1}[0] \bigr)(t+t_n^{a2}, x+x_n^{a2}) \rightharpoonup \phi^{a2}(t,x).
 \]
 We observe that we must necessarily have
 \[
  \lim_{n\to\infty} \bigl( |t_n^{a1} - t_n^{a2}| + |x_n^{a1} - x_n^{a2}| \bigr) = \infty.
 \]
 Iterating this process yields the decomposition \eqref{equ:phi_concentration_profiles_decomposition} together with \eqref{equ:phi_concentration_profiles_mutual_divergence} and \eqref{equ:phi_concentration_profiles_asymptotic_energy_partition}. 

 \medskip

 Finally, we turn to proving the crucial asymptotic vanishing condition 
 \[
  \lim_{B\to\infty} \eta_{A^{na}} \bigl( \phi^{naB}[0] \bigr) = 0.
 \]
 Here we observe that the fixed profiles $\phi^{ab}[0]$ satisfy
 \[
  \phi^{ab}(0,x) = S_{A^{na}} \big( \tilde{\phi}^{nab}[0] \big)(0+t_n^{ab}, x + x_n^{ab}).
 \]
 Then the global $S^1$ norm bounds \eqref{equ:covariant_wave_operator_S1_norm_bound} for solutions to the covariant wave equation $\widetilde{\Box}_{A^{na}} u = 0$ imply in particular that
 \[
  \|\nabla_{t,x} \phi^{ab}(0, \cdot)\|^2_{L^2_x} \lesssim \big\| S_{A^{na}} \big( \tilde{\phi}^{nab}[0] \big) \big\|^2_{S^1} \lesssim_{E_{crit}} \|\nabla_{t,x} \tilde{\phi}^{nab}(0, \cdot)\|^2_{L^2_x} \lesssim_{E_{crit}} E_0 (\tilde{\phi}^{nab}[0]),
 \]
 where the implied constant is independent of $n$. From the asymptotic energy partition \eqref{equ:phi_concentration_profiles_asymptotic_energy_partition} we conclude that for any $B \geq 1$, by passing to a subsequence in $n$, if necessary, we have
 \[
  \sum_{b=1}^B \limsup_{n\to\infty} E_0(\tilde{\phi}^{nab}[0]) \leq \limsup_{n\to\infty} E_0(\phi^{na}[0]) \lesssim_{E_{crit}} 1.
 \]
 Thus, we have that uniformly in $B$,
 \[
  \sum_{b=1}^B E_0(\phi^{ab}[0]) \lesssim_{E_{crit}} 1.
 \]
 By construction, the error $\eta_{A^{na}}(\phi^{naB})$ must therefore vanish as $B \to \infty$. This finishes the proof of Proposition~\ref{prop:phi_concentration_profiles}.
\end{proof}

We emphasize that in the preceding linear profile decomposition for the $\phi^{na}$ fields, the asymptotic energy partition \eqref{equ:phi_concentration_profiles_asymptotic_energy_partition} does not yield a sharp energy bound for the actual profiles $\phi^{ab}[0]$ of temporally unbounded character, which is in contrast to the standard Bahouri-G\'erard profile decomposition \cite{Bahouri-Gerard} and the modified Bahouri-G\'erard profile decomposition in the context of critical wave maps \cite[Lemma 9.23]{KS}. Fortunately, this will not doom the construction of the nonlinear concentration profiles, because there is a kind of ``asymptotic orthogonality statement'', see Lemma~\ref{lem:MKG_energy_almost_orthogonality_of_profiles}, in particular \eqref{equ:MKG_energy_asymptotic_temporally_unbounded}. This will allow us to circumvent the problem. 

\medskip

Having selected the linear concentration profiles for the $\phi^{na}$ fields, it remains to pick corresponding profiles for the magnetic potential components $A_j^{na}$ for $j = 1, \ldots, 4$. In fact, for the latter, we simply use the standard Bahouri-G\'erard method \cite{Bahouri-Gerard} to extract the profiles via the free wave evolution. By passing to suitable subsequences, one obtains an intertwined linear profile decomposition for $(A^{na}, \phi^{na})[0]$. Thus, the same sequences of space-time shifts $\{(t_n^{ab}, x_n^{ab})\}_{n\geq 1}, b \geq 1,$ are being used for the linear concentration profiles for $A^{na}[0]$ and for $\phi^{na}[0]$. This will be crucial later on when we construct the associated nonlinear profiles, as the truly nonlinear behavior of \emph{both} $(A, \phi)$ will be exhibited in space-time boxes centered around the points $(t_n^{ab}, x_n^{ab})$, see Step 1 in the proof of Theorem~\ref{thm:nonlinear_profiles_S1_bound}. We quote 
\begin{prop} \label{prop:A_concentration_profiles} 
There exists a collection of sequences $\{(t_n^{ab}, x_n^{ab})\}_{n \in \N} \subset \R\times\R^4$, $b\geq 1$, as well as a corresponding family of concentration profiles 
\[
A_j^{ab}[0]\in \dot{H}^1_x(\R^4)\times L_x^2(\R^4), \quad b\geq 1
\]
for $j = 1, \ldots, 4$ with the following properties:
\begin{itemize}[leftmargin=*]
\item For any $B\geq 1$, we have a decomposition 
\[
 S\big( A^{na}_j[0] \big)(t, x) = \sum_{b=1}^B S\big(A_j^{ab}[0]\big)(t - t^{ab}_n, x-x^{ab}_n) + A_j^{naB}(t, x),
\]
where $S(\cdot)(t,x)$ denotes the free wave propagator. Then each of the functions 
\[
 S \big( A^{ab}_j[0] \big)(t - t^{ab}_n, x-x^{ab}_n), \quad A_j^{naB}(t, x)
\]
solves the linear wave equation 
\[
 \Box u = 0.
\]
Moreover, the error satisfies the crucial asymptotic vanishing condition 
\begin{equation} \label{equ:Avanishing}
\lim_{B\rightarrow\infty} \eta \big( A_j^{naB}[0] \big) = 0.
\end{equation}
\item The sequences are mutually divergent, by which we mean that for $b \neq b'$,
\[
\lim_{n\rightarrow\infty} \big( |t_n^{ab} - t_{n}^{ab'}| + |x_n^{ab} - x_n^{ab'}| \big) = \infty.
\]
\item There is asymptotic energy partition 
\[
 E_0(A_j^{na}[0]) = \sum_{b=1}^B E_0(A_j^{ab}[0]) + E_0(A_j^{naB}[0]) + o(1),
\]
where the meaning of $o(1)$ is $\limsup_{n\rightarrow\infty}o(1) = 0$. 
\item All profiles $A_j^{ab}[0]$ as well as all errors $A_j^{naB}[0]$ are $1$-oscillatory. Moreover, they all satisfy the Coulomb condition.  
\end{itemize}
\end{prop}

In the preceding propositions on the linear profile decompositions for the $\phi^{na}$ fields and for the spatial components $A^{na}_j$ of the connection form, we established an asymptotic orthogonality of the profiles with respect to the standard free energy functional. However, for our induction on energy procedure, we have to use the energy functional of the Maxwell-Klein-Gordon system, which involves nonlinear interactions between the $\phi$ field and the connection form $A$. In the next proposition we carefully analyze the asymptotic orthogonality relations of the linear profiles with respect to this proper energy functional. 
\begin{lem} \label{lem:MKG_energy_almost_orthogonality_of_profiles}
 Given any $\delta_4 > 0$, there exists $B_0 = B_0(\delta_4)$ such that\footnote{The $B_0$ also depends on the sequence of linear concentration profiles, but we omit this dependency here.}
 \begin{equation} 
  \limsup_{n \to \infty} \bigg| E(A^{na}, \phi^{na})(0) - \sum_{b=1}^{B_0} E(\tilde{A}^{nab}, \tilde{\phi}^{nab})(0) - E(A^{naB_0}, \phi^{naB_0})(0) \bigg| < \delta_4, 
 \end{equation}
 where $E$ refers to the energy functional of the Maxwell-Klein-Gordon system. Here we denote 
 \begin{align*}
  \tilde{\phi}^{nab} &:= S_{\tilde{A}^{nab}} \big( \phi^{ab}[0] \big)(0 - t^{ab}_n, x-x^{ab}_n), \\
  \tilde{A}^{nab}_j &:=  S \big( A_j^{ab}[0] \big)(0 - t^{ab}_n, x - x^{ab}_n), \quad j = 1, \ldots, 4, 
 \end{align*}
 and the temporal components $\tilde{A}^{nab}_0(0)$ are determined in terms of $\tilde{\phi}^{nab}[0]$ via the elliptic compatibility equation, and similarly for $A^{naB_0}_0(0)$. In particular, if there are at least two non-zero concentration profiles $(A^{ab}, \phi^{ab})[0]$ (corresponding to two distinct values of $b$), then there exists $\delta > 0$ such that for all $b$,
 \[
  \limsup_{n\rightarrow\infty}E(\tilde{A}^{nab}, \tilde{\phi}^{nab})(0) < E_{crit} - \delta.
 \]
 Moreover, for a temporally unbounded profile $(\tilde{A}^{nab}, \tilde{\phi}^{nab})$ with, say,  $t^{ab}_n\rightarrow +\infty$ as $n\rightarrow \infty$, we have 
 \begin{equation} \label{equ:MKG_energy_asymptotic_temporally_unbounded}
  E(\tilde{A}^{nab}, \tilde{\phi}^{nab})(0) = E(\tilde{A}^{nab}, \tilde{\phi}^{nab})(t_n^{ab} - R_b) + \kappa_{ab}(n, R_b),
 \end{equation}
 where
 \[
  \lim_{R_b\rightarrow +\infty}\limsup_{n\rightarrow\infty}\kappa_{ab}(n, R_b) = 0.
 \]
\end{lem}
\begin{proof} We check the various interaction terms and show that they become small when choosing $B_0$ as well as $n$ sufficiently large. 

\medskip

\noindent {\it{(1) Two temporally bounded profiles.}} This is straightforward since $\lim_{n\rightarrow\infty} |x_n^{ab} - x_n^{ab'}| = \infty$. In fact, we immediately infer that schematically
\begin{align*}
&\lim_{n\rightarrow\infty} \sum_{\text{temporally bounded profiles},\,b\neq b'} \Big|\int_{\R^4} \Re\big((\partial_{\alpha}\tilde{\phi}^{nab} + i \tilde{A}_{\alpha}^{nab}\tilde{\phi}^{nab}) \cdot\overline{(\partial_{\alpha}\tilde{\phi}^{nab'} + i\tilde{A}_{\alpha}^{nab'}\tilde{\phi}^{nab'})} \big) \,dx\Big|\\
& + \lim_{n\rightarrow\infty}\sum_{\text{temporally bounded profiles},\,b\neq b'} \sum_{j=1}^4 \Big|\int_{\R^4}\nabla_{t,x}\tilde{A}^{nab}_j\cdot \nabla_{t,x}\tilde{A}^{nab'}_j\,dx \Big| \\
& +  \lim_{n\rightarrow\infty}\sum_{\text{temporally bounded profiles},\,b\neq b'}\sum_{j=1}^4 \Big|\int_{\R^4}\nabla_{x}\tilde{A}^{nab}_0\cdot \nabla_{x}\tilde{A}^{nab'}_0\,dx\Big|\\
& = 0.
\end{align*}

\medskip

\noindent {\it{(2) One temporally bounded and one temporally unbounded profile.}} Here we exploit that the amplitude of the temporally unbounded profile vanishes asymptotically (at time $t = 0$) as $n\rightarrow\infty$, while the temporally bounded profile has bounded support. We conclude that schematically
\begin{align*}
&\lim_{n\rightarrow\infty} \sum_{\substack{b\,\text{temporally bounded}\\ b'\,\text{temporally unbounded}}} \Big|\int_{\R^4}\Re\big((\partial_{\alpha}\tilde{\phi}^{nab} + i \tilde{A}_{\alpha}^{nab}\tilde{\phi}^{nab}) \cdot\overline{(\partial_{\alpha}\tilde{\phi}^{nab'} + i\tilde{A}_{\alpha}^{nab'}\tilde{\phi}^{nab'})}\big) \,dx\Big|\\
& + \lim_{n\rightarrow\infty}\sum_{\substack{b\,\text{temporally bounded}\\ b'\,\text{temporally unbounded}}} \sum_{j=1}^4 \Big|\int_{\R^4}\nabla_{t,x}\tilde{A}^{nab}_j\cdot \nabla_{t,x}\tilde{A}^{nab'}_j\,dx \Big| \\
& +  \lim_{n\rightarrow\infty}\sum_{\substack{b\,\text{temporally bounded}\\ b'\,\text{temporally unbounded}}}\sum_{j=1}^4 \Big|\int_{\R^4}\nabla_{x}\tilde{A}^{nab}_0\cdot \nabla_{x}\tilde{A}^{nab'}_0\,dx\Big|\\
& = 0.
\end{align*}

\medskip

\noindent {\it{(3) Two temporally unbounded profiles.}} Here we exploit the asymptotic energy conservation and that the functions 
\[
 \phi^{ab}[0], \quad S_{\tilde{A}^{nab'}}\big(\phi^{ab'}[0]\big)(t_n^{ab} - t_{n}^{ab'}, x-x_n^{ab'})
\]
are asymptotically orthogonal. Similarly, we argue for the interaction terms between the components of the profiles $\tilde{A}^{nab}$ and $\tilde{A}^{nab'}$.

\medskip

\noindent {\it{(4) Weakly small error $\phi^{naB_0}$ and profiles.}} This is handled like the interaction of a temporally bounded and a temporally unbounded profile. One uses the fact that we get 
 \[
 \phi^{naB_0} =  \phi_1^{naB_0} +  \phi_2^{naB_0},
\]
where we have the bounds
\[
\big\| \phi_1^{naB_0}\big\|_{L^\infty_t L^\infty_x}<\delta_4,\,\big\|\nabla_{t,x}\phi_2^{naB_0}\big\|_{L_t^\infty L_x^2}<\delta_4,
\]
provided $B_0$ is sufficiently large. Of course, choosing $B_0$ large means that more and more interactions have to be controlled, and we can no longer simply use the choice of extremely large $n$ to ``asymptotically kill'' all such interactions as in the preceding cases. Thus, one has to argue carefully as follows: Given $\delta_4>0$, we pick $\tilde{B}_0$ sufficiently large such that for any $B\geq \tilde{B}_0$, we have 
\[
\limsup_{n\rightarrow\infty} \sum_{b = \tilde{B}_0}^{B} \Bigl( E_0(\tilde{\phi}^{nab})+ E_0(\tilde{A}^{nab}) \Bigr) \ll \delta_4,
\]
where $E_0$ indicates the standard free energy. Then, passing to the interaction terms in the Maxwell-Klein-Gordon energy functional corresponding to $\phi^{naB_0}$ and $A^{naB_0}$ with the sum 
\[
\sum_{b = \tilde{B}_0}^{B_0}\tilde{\phi}^{nab},\,\sum_{b = \tilde{B}_0}^{B_0}\tilde{A}^{nab}
\]
leads to terms bounded by $\ll\delta_4$ for any $B_0\geq \tilde{B}_0$, provided $n$ is chosen sufficiently large (depending on $B_0$). But then picking $B_0$ large enough, we can also ensure that the sum of all the interactions in $E(A, \phi)$ generated by the profiles $\tilde{\phi}^{nab}, \tilde{A}^{nab}$, $1\leq b\leq \tilde{B}_0$ are small, since $B_0 \geq \tilde{B}_0$ can be chosen independently.

\medskip

The last assertion \eqref{equ:MKG_energy_asymptotic_temporally_unbounded} is again a consequence of the asymptotic energy conservation from Lemma~\ref{lem:asymptotic_energy_conservation} and the asymptotic vanishing of the amplitude of a temporally unbounded profile at $t=0$ as $n \to \infty$.
\end{proof}

We now begin with the construction of the nonlinear concentration profiles. In what follows, we assume that the linear concentration profiles $\big( A^{ab}, \phi^{ab} \big)[0]$, $b \geq 1$, have been chosen, as well as the parameter sequences $\big\{ (t^{ab}_n, x^{ab}_n) \big\}_{n \geq 1}$. We recall that when the profile is temporally bounded, i.e. 
\[
 \limsup_{n\rightarrow\infty} |t^{ab}_n| < \infty,
\]
we may and shall have $t^{ab}_n = 0$ identically. We also recall the notation
\begin{align*}
 \tilde{\phi}^{nab}(t,x) &:= S_{\tilde{A}^{nab}} \big( \phi^{ab}[0] \big)(t - t^{ab}_n, x - x^{ab}_n), \\
 \tilde{A}^{nab}_j(t,x) &:=  S \big( A_j^{ab}[0] \big)(t - t^{ab}_n, x - x^{ab}_n), \quad j = 1, \ldots, 4.
\end{align*}
Thus, if the profile is temporally bounded, it holds that
\[
 \tilde{A}^{nab}[0] = A^{ab}[0], \quad \tilde{\phi}^{nab}[0] = \phi^{ab}[0].
\]
We can now state the key result of this subsection.
\begin{thm} \label{thm:nonlinear_profiles_S1_bound}  
 Let $a = 1$. Assume that there exist at least two non-zero profiles $\big(A^{ab}, \phi^{ab}\big)[0]$, or all such profiles are zero, or else there exists only one such profile but with 
 \[
  \liminf_{n\rightarrow\infty} E(\tilde{A}^{nab}, \tilde{\phi}^{nab})(0) < E_{crit}.
 \]
 Then the initial data $\big(A^{n1}_{\Lambda_0} + A^{n1}, \phi^{n1}_{\Lambda_0} + \phi^{n1}\big)[0]$ can be evolved globally in time, resulting in a solution with finite $S^1$ norm bounds uniformly for all sufficiently large $n$.
\end{thm}
\begin{proof}
We proceed in several steps.

\medskip

\noindent {\bf Step 1:} {\it Construction of the nonlinear concentration profiles.} We distinguish between temporally bounded and unbounded $(\tilde{A}^{nab}, \tilde{\phi}^{nab})$. In what follows we shall use the notation 
\begin{align*}
 A^{na, low} := A^{n1}_{\Lambda_0}, \quad \phi^{na, low} := \phi^{n1}_{\Lambda_0}.
\end{align*}

\medskip

\noindent {\it{Temporally bounded case:}} Here we have $(\tilde{A}^{nab}[0], \tilde{\phi}^{nab}[0]) = (A^{ab}[0], \phi^{ab}[0])$ with $A^{ab}$ as usual in the Coulomb gauge. Then we define the nonlinear concentration profile
\[
 \big(\mathcal{A}^{nab}, \Phi^{nab}\big)
\]
as follows. Pick a large time $T_b>0$ whose size will be fixed later on. On $[-T_b, T_b] \times \R^4$, we define the profiles to be the solutions $(\cA^{ab}, \Phi^{ab})$ to the MKG-CG system with data $(A^{ab}, \phi^{ab})[0]$ at time $t=0$, which exist globally in time by Lemma~\ref{lem:MKG_energy_almost_orthogonality_of_profiles} and the assumption of the theorem with a global finite $S^1$ norm bound
\[
 \big\|\big(\mathcal{A}^{ab}, \Phi^{ab}\big)\big\|_{S^1}<\infty.
\]
Here the profiles do not depend on $n$, but we include this superscript since the profiles on the rest of space-time will be $n$-dependent. On the complement $[-T_b, T_b]^c\times \R^4$, we define the profiles as follows. On $[T_b, \infty)\times \R^4$, we let 
\[
 \Box \mathcal{A}^{nab} = 0
\]
with data $\mathcal{A}^{ab}[T_b]$ given by the profile constructed on $[-T_b, T_b]\times \R^4$, and we proceed analogously on $(-\infty, -T_b]\times\R^4$. 
As for the $\Phi$-field, we postulate on $[-T_b, T_b]^c\times \R^4$ the linear equation 
\[
 \Box_{A^{na, low} + \sum_{b'=1}^B\mathcal{A}^{nab'} + A^{naB}}\Phi^{nab} = 0
\]
with data given at time $T_b$, respectively $-T_b$, by the profile on $[-T_b, T_b] \times \R^4$. Note that in order for this to make sense, we also need to know the definition of the temporally unbounded $\mathcal{A}^{nab'}$, which is, of course, accomplished below without knowing the temporally bounded $\Phi^{nab}$ to avoid circularity. 

\medskip

\noindent {\it{Temporally unbounded case:}} Assume, for example, that $\lim_{n\rightarrow\infty}t^{ab}_n = +\infty$. Using Lemma~\ref{lem:MKG_energy_almost_orthogonality_of_profiles} and the assumption of the theorem, we can pick $R_b>0$ sufficiently large such that 
\[
 \tilde{\phi}^{nab}(t^{ab}_n - R_b, \cdot) = S_{\tilde{A}^{nab}}(\phi^{ab}[0])(-R_b, \cdot-x^{ab}_n) 
\]
satisfies 
\[
 E \Big( S(A^{ab}[0])(-R_b, \cdot - x_n^{ab}), \, S_{\tilde{A}^{nab}}(\phi^{ab}[0])(-R_b, \cdot-x^{ab}_n) \Big) < E_{crit}.
\]
Then we use the data 
\[
 \Big( S(A^{ab}[0])[-R_b](\cdot - x_n^{ab}),\,S_{\tilde{A}^{nab}}(\phi^{ab}[0])[-R_b](\cdot-x^{ab}_n) \Big)
\]
at time $t = t^{ab}_n - R_b$, and evolve them {\it{forward in time}} using the MKG-CG system up to time $t^{ab}_n +R_b$, say, resulting  in the nonlinear profiles 
\[
\big(\mathcal{A}^{nab}, \Phi^{nab}\big)
\]
on $[t^{ab}_n - R_b, t^{ab}_n + R_b] \times \R^4$. Observe that this construction does not require knowledge of the other profiles $\big(\mathcal{A}^{nab'}, \Phi^{nab'}\big)$. Finally, on the complement $[t^{ab}_n - R_b, t^{ab}_n + R_b]^c \times \R^4$, we evolve $\mathcal{A}^{nab}$ via the free equation $\Box \mathcal{A}^{nab} = 0$, and $\Phi^{nab}$ via the linear evolution
\[
\Box_{A^{na, low} + \sum_{b'=1}^B\mathcal{A}^{nab'} + A^{naB}}\Phi^{nab} = 0,
\]
with data given at time $t_n^{ab} - R_b$, respectively $t_n^{ab} + R_b$, by the profiles constructed on $[t_n^{ab}-R_b, t_n^{ab}+R_b] \times \R^4$.

\medskip

\noindent {\bf Step 2:} {\it Making an ansatz for the evolution $(\mathcal{A}^n, \Phi^n)$ of the full data $\big(A^{n1}_{\Lambda_0} + A^{n1}, \phi^{n1}_{\Lambda_0} + \phi^{n1}\big)[0]$.} 
We now assemble the pieces that we have constructed. We shall write
\begin{equation} \label{eq:ansatzA}
 \mathcal{A}^n: = A^{na, low} + \sum_{b=1}^B\mathcal{A}^{nab} + \mathcal{A}^{naB} + \delta_A^n,
\end{equation}
where $\mathcal{A}^{naB}$ is actually simply given by $A^{naB}$ from Proposition~\ref{prop:A_concentration_profiles}. We immediately observe the crucial fact that 
\[
\delta_A^n[0] = 0,
\]
i.e. the choice of profiles matches the data. We proceed analogously for $\Phi^n$, writing 
\begin{equation}\label{eq:ansatzPhi}
\Phi^n: = \phi^{na, low} + \sum_{b=1}^B\Phi^{nab} + \Phi^{naB} + \delta_{\Phi}^n,
\end{equation}
where $\Phi^{naB}$ is actually simply given by $\phi^{naB}$ from Proposition~\ref{prop:phi_concentration_profiles}. We finally observe that by truncating the frequency support of the data of the $\Phi^{nab}$ to a set $\{|\xi|\leq K\}$ for some very large $K$ and incorporating the error into $\delta_{\Phi}^n$, we may assume that the $\Phi^{nab}$ have frequency support in $|\xi|\leq K$ up to (slowly) exponentially decaying tails. This will be of use later on when controlling the errors. 

\medskip

\noindent {\bf Step 3:} {\it Showing accuracy of the ansatz.} Here we finally prove the following key proposition.
\begin{prop} \label{prop:keydecomp} 
Assuming the conditions of Theorem~\ref{thm:nonlinear_profiles_S1_bound} and given any $\delta_5>0$, there exists $B$ sufficiently large (depending on the bounds on $(A^{na, low}, \phi^{na, low})$, the actual concentration profiles and on $\delta_5$) such that for all sufficiently large $n$,
\[
 \big\|\delta_\Phi^n\big\|_{S^1} + \big\|\delta^n_A\big\|_{\ell^1 S^1} <\delta_5. 
\]
\end{prop}
In light of the immediately verified facts that 
\[
\limsup_{n\rightarrow\infty}\sum_{b=1}^B\big\|\Phi^{nab}\big\|_{S^1} + \limsup_{n\rightarrow\infty}\sum_{b=1}^B\big\|\mathcal{A}^{nab}\big\|_{S^1}<\infty
\]
and
\[
 \limsup_{n\rightarrow\infty}\big\|\Phi^{naB}\big\|_{S^1} + \limsup_{n\rightarrow\infty}\big\|\mathcal{A}^{naB}\big\|_{S^1}<\infty,
\]
this proposition then implies Theorem~\ref{thm:nonlinear_profiles_S1_bound}.
\end{proof}

\begin{proof}[Proof of Proposition~\ref{prop:keydecomp}]
For the most part, this consists in checking that the (very large number of) interaction terms sum up to something negligible upon correct choice of $B$ and $n$. We start with the equation for $\delta_\Phi^n$. To begin with, we note that $\delta_\Phi^n[0]$ is not necessarily  $0$, since the asymptotic evolution of the profiles $\Phi^{nab}$ given by 
\[
\Box_{A^{na, low} + \sum_{b'=1}^B\mathcal{A}^{nab'} + A^{naB}}\Phi^{nab} = 0
\]
is different than the one used to extract the concentration profiles, i.e. $\widetilde{\Box}_{A^{na}} u = 0$. But we also observe that each profile $\mathcal{A}^{nab'}$ differs from the corresponding linear component in Proposition~\ref{prop:A_concentration_profiles} given by 
\[
  S(A^{ab'}[0])(t - t^{ab'}_n, x-x^{ab'}_n)
\]
by a possibly large term, which however lives in a better space 
\[
\big\|\mathcal{A}^{nab'}(t,x) - S(A^{ab'}[0])(t - t^{ab'}_n, x-x^{ab'}_n) \big\|_{\ell^1 S^1} < \infty.
\]
Denote this difference by $\mathcal{B}^{nab'}(t,x)$. Then it suffices to show 
\begin{lem} \label{lem:inprop1} 
For any temporally unbounded profile $\Phi^{nab}$ we have 
\[
 \lim_{n\rightarrow\infty} \sum_{ \substack{ 1 \leq b' \leq B, \\ b' \neq b} } \big\|2i \mathcal{B}_{\nu}^{nab'} \partial^{\nu} \Phi^{nab}\big\|_{N} = 0.
\]
\end{lem}
\begin{proof}
We proceed as in the proof of Proposition~\ref{prop:decomposition}, expressing the difference $\mathcal{B}^{nab'}$ in the schematic form 
\[
 \sum_{k, j} \Box^{-1} P_k Q_j \mathcal{P} \big(\phi\cdot\nabla_x\phi + A |\phi|^2 ),
\]
or else as a free wave satisfying a Besov $\ell^1$-bound for the data instead of the weaker energy bound. Using the multilinear estimates from \cite{KST} and that all factors as well as $\Phi^{nab}$ are $1$-oscillatory, we reduce to a diagonal situation, where the frequency of all factors as well as the output modulation are essentially restricted to $\sim 1$, and have generic position, i.e. the Fourier supports do not have angular alignment. Then, using that the profiles  $\mathcal{A}^{nab'}$ disperse away from $t_n^{ab'}$ uniformly in $n$  by Lemma~\ref{lem:uniform_dispersive_bounds}, we easily infer the claim. 

To be more precise, we first consider the case when $A_j^n$, $j = 1,2,3,4$, are free waves, which are $1$-oscillatory, obey the Coulomb condition, and satisfy 
\[
 \big\|A_j^n\big\|_{\ell^1 S^1}<\infty.
\]
Moreover, assume that $\Phi^{n}$ is $1$-oscillatory and satisfies
\[
 \sup_{n}\big\|\Phi^{n}\big\|_{S^1}<\infty
\]
and in view of the dispersive bounds from Lemma~\ref{lem:uniform_dispersive_bounds} also
\[
 \lim_{n\rightarrow\infty}\big\|\Phi^{n}\big\|_{L^\infty_t L^\infty_x} = 0.
\]
We now prove that 
\[
 \lim_{n\rightarrow\infty} \big\|2i A_j^n \partial^j\Phi^n\big\|_{N} = 0.
\]
By the $1$-oscillatory character of the inputs and the $\ell^1$-Besov bound for $A^n$, one may restrict to frequencies $\sim 1$ in both factors, and assume the output to be at modulation $\sim 1$ (else the null structure gives smallness). Then one uses the Strichartz exponents $\big(\frac{10}{3}, \frac{10}{3}\big)$ for the first factor, and an interpolate of $(\infty, \infty)$ with that same space for the second factor to place the output into $L^2_t L^2_x$. 

Next, consider the case where $A_j^n$ is of the schematic form 
\[
\sum_{k, j} \Box^{-1} P_k Q_j \mathcal{P}_i\big(\phi\cdot\nabla\phi + A |\phi|^2).
\]
We only consider the most difficult case, where the space-time frequency localizations have been implemented and the null form structure revealed as in \cite[Theorem 12.1]{KST}. For example, consider an expression 
\[
 \Box^{-1} P_k Q_j \big( Q_{\leq j-C} P_{k_1} \phi^n \partial_{\alpha} Q_{\leq j-C} P_{k_2} \phi^n \big) \partial^{\alpha} Q_{\leq j-C} P_{k_3} \Phi^n,
\]
where the $k_j$ indicate frequency localizations, all inputs are $1$-oscillatory, and satisfy uniform $S^1$ norm bounds, and $\Phi^n$ satisfies the same vanishing relation as above. Also, from \cite{KST} we have the alignments $k_1 = k_2+O(1)$, $k_3\geq k+O(1)$, $j\leq k+O(1)$. One may then in fact assume $j = k+O(1)$, since else one gets smallness, and the $1$-oscillatory character allows us to assume $k_{1,2,3} = k+O(1) = O(1)$. Then one places the output into $L_t^1 L_x^2$ by using the Strichartz exponents $\big(\frac{10}{3}, \frac{10}{3}\big)$ for the first two factors, and an interpolate of $\big(\frac{5}{2}-, \frac{30}{7}+\big)$ with $\big(\infty, \infty\big)$ for the last factor. The remaining null forms (see (62) and (63) in \cite{KST}) are handled similarly.  
\end{proof}

From the preceding lemma, we infer that we can force 
\[
\big\|\delta^n_\Phi[0]\big\|_{\dot{H}^1_x \times L^2_x} \ll \delta_5,
\]
provided we pick $n$ sufficiently large. The equation for $\delta^n_\Phi$ is given by 
\begin{equation} \label{eqn:deltaphi1}
\Box_{A^{na,low} + \sum_{b' = 1}^B\mathcal{A}^{nab'} + \mathcal{A}^{naB}+\delta_A^n} \Big( \phi^{na, low} + \sum_{b=1}^B\Phi^{nab} + \Phi^{naB} + \delta_{\Phi}^n \Big) = 0.
\end{equation}
We rewrite this in the following form 
\begin{equation} \label{eqn:deltaphi2}
 \Box_{A^{na,low} + \sum_{b' = 1}^B\mathcal{A}^{nab'} + \mathcal{A}^{naB}+\delta_A^n}\delta_{\Phi}^n = - I - II - III,
\end{equation}
where we put 
\begin{align*}
I &:= \Box_{A^{na,low} + \sum_{b' = 1}^B\mathcal{A}^{nab'} + \mathcal{A}^{naB}+\delta_A^n} \Big( \phi^{na, low} \Big), \\
II &:= \Box_{A^{na,low} + \sum_{b' = 1}^B\mathcal{A}^{nab'} + \mathcal{A}^{naB}+\delta_A^n} \Big(\sum_{b=1}^B\Phi^{nab}\Big), \\
III &:= \Box_{A^{na,low} + \sum_{b' = 1}^B\mathcal{A}^{nab'} + \mathcal{A}^{naB}+\delta_A^n} \Big(\Phi^{naB}\Big).
\end{align*}
Now the idea is to show smallness of all these terms (in the $N$ norm sense) provided $B$ and then $n$ are chosen sufficiently large. Of course, one needs to be careful with the fact that increasing $B$ also leads to more and more terms in the sums 
\[
\sum_{b' = 1}^B\mathcal{A}^{nab'}, \quad \sum_{b=1}^B\Phi^{nab}.
\]
To deal with this, we use 
\begin{lem} \label{lem:tail estimate} 
Given $\delta_6>0$, there is a $B_1>0$ such that for all $B \geq B_1$ and all sufficiently large $n$ (depending on $B$), it holds that
\[
 \Big\| \sum_{b=B_1}^B \mathcal{A}^{nab} \Big\|_{S^1} < \delta_6, \quad \Big\|\sum_{b=B_1}^B\Phi^{nab}\Big\|_{S^1} < \delta_6.  
\]
\end{lem}
\begin{proof}
By construction we have 
\[
\Box \mathcal{A}^{nab} = - \chi_{I_b^n} {\mathcal P} \Im \big(\Phi^{nab}\overline{D_x \Phi^{nab}}\big),
\]
where $I_b^n = [-T_b, T_b]$ for temporally bounded profiles and $I_b^n = [t_{n}^{ab}-R_b, t_n^{ab}+R_b]$ for temporally unbounded ones. By picking $B_1$ sufficiently large, so that 
\[
E(\mathcal{A}^{nab}, \Phi^{nab})\ll 1,\,b\geq B_1, 
\]
we get 
\begin{align*}
 \sum_{b=B_1}^B \Big\|\chi_{I_b^n} {\mathcal P} \Im\big(\Phi^{nab}\overline{D\Phi^{nab}}\big)\Big\|_{N} \lesssim \sum_{b=B_1}^B E(\tilde{A}^{nab}, \tilde{\phi}^{nab}),
\end{align*}
where we recall the notation 
\begin{align*}
 \tilde{\phi}^{nab} &= S_{\tilde{A}^{nab}}\big(\phi^{ab}[0]\big)(0 - t^{ab}_n, x-x^{ab}_n), \\
 \tilde{A}^{nab}    &= S\big(A_j^{ab}[0]\big)(0 - t^{ab}_n, x-x^{ab}_n), \quad j = 1, \ldots, 4.
\end{align*}
But then since 
\[
\limsup_{n\rightarrow\infty} \Big\|\sum_{b=B_1}^B \mathcal{A}^{nab}[0]\Big\|_{\dot{H}^1_x\times L^2_x}<\delta_6, \quad \limsup_{n\rightarrow\infty} \sum_{b=B_1}^B E(\tilde{A}^{nab}, \tilde{\phi}^{nab})<\delta_6,
\]
upon choosing $B_1$ large enough, the first bound of the lemma follows. To get the second bound, one uses that for $B$ and $n$ large enough, as well as making some small additional assumption on the data $\phi^{ab}[0]$ (see Remark~\ref{rem:trickycutoff} below),  
\begin{equation} \label{eq:ohboy}
 \lim_{n\rightarrow\infty} \Big\|\Box_{A^{na, low} +\sum_{b'=1}^B \mathcal{A}^{nab'} + \mathcal{A}^{naB}} \biggl( \sum_{b=B_1}^B \chi_{I_b^n} \Phi^{nab} + (1- \chi_{I_b^n}) \tilde{\Phi}^{nab} \biggr) \Big\|_{N} \ll \delta_7,
\end{equation}
where now $\chi_{I_b^n}$ are suitable smooth time cutoffs and $\Phi^{nab}$ is as in Step 1 with $I_b^n = [-T_b, T_b]$ or $I_b^n = [t^{ab}_n+R_b, t^{ab}_n - R_b]$, while $\tilde{\Phi}^{nab}$ is as in Step 1 but on the complement $(I_b^n)^c$.
Again 
\[
\limsup_{n\rightarrow\infty} \Big\|\sum_{b=B_1}^B \Phi^{nab}[0]\Big\|_{\dot{H}^1_x \times L^2_x} < \delta_6,
\]
provided $B_1$ is chosen sufficiently large. We then infer 
\[
\Big\|\sum_{b=B_1}^B\Phi^{nab}\Big\|_{S^1}\ll \delta_6, 
\]
provided $\delta_7$ is sufficiently small. 
Note that there are small error terms due to the cutoff, which however are harmless and can be made arbitrarily small by picking the cutoff suitably, see \cite{KS}. In fact, we make the 
\begin{rem} \label{rem:trickycutoff} 
 To ensure smallness of the errors generated by the cutoffs $\chi_{I^n_b}$ and $1-\chi_{I^n_b}$, it suffices to localize each $\phi^{ab}[0]$ in physical space to a ball of radius $10|I^n_b|$, and each $A^{ab}[0]$ to a ball of radius $100|I^n_b|$, say.  The errors committed thereby may be included in $\Phi^{naB}$, respectively $\mathcal{A}^{naB}$. 
\end{rem}
Observe that for the term $\Box_{A^{na, low} +\sum_{b'=1}^B \mathcal{A}^{nab'} + \mathcal{A}^{naB}}\sum_{b=B_1}^B\chi_{I_b^n}\Phi^{nab}$, one generates errors of the schematic form 
\[
\chi_{I_b^n}''\Phi^{nab}, \quad \chi_{I_b^n}'(A^{na, low} +\mathcal{A}^{nab'}+\mathcal{A}^{naB})\Phi^{nab}, \quad \chi_{I_b^n} \bigl( \Box_{A^{na, low} +\sum_{b'=1}^B \mathcal{A}^{nab'} + \mathcal{A}^{naB}} - \Box_{\mathcal{A}^{nab}} \bigr) \Phi^{nab}.
\]
Then by using the crude bound
\[
\big\|\chi_{I_b^n}\mathcal{A}^{nab'}\nabla_{t,x}\Phi^{nab}\big\|_{L_t^1 L_x^2}\lesssim |I_b^n| \bigl\|\chi_{C_{bn}}\mathcal{A}^{nab'}\bigr\|_{L_t^\infty L_x^{4}} \bigl\|\nabla_{t,x}\Phi^{nab}\bigr\|_{L_t^\infty L_x^2},
\]
where $C_{bn}$ is a suitable space-time cube of width $\sim |I_b^n|$ centered around $(t^{ab}_n, x^{ab}_n)$, and the implied constant depends on the frequency support cutoff for the $\Phi^{nab}$ (see the end of Step 2), we see that in light of the decay properties of the $\mathcal{A}^{nab'}$ for $b'\neq b$, the norm converges to zero as $n\rightarrow\infty$. One argues similarly for 
\[
\big\|\chi_{I_b^n}\mathcal{A}^{nab'}\mathcal{A}^{nab''}\Phi^{nab}\big\|_{L_t^1 L_x^2},\,b'\neq b,
\]
as well as those terms generated when we replace $\mathcal{A}^{nab'}$ by $A^{na, low}$ or $\mathcal{A}^{naB}$, which then takes care of the third expression 
\[
\chi_{I_b^n} \bigl( \Box_{A^{na, low} + \sum_{b'=1}^B \mathcal{A}^{nab'}+ \mathcal{A}^{naB}} - \Box_{\mathcal{A}^{nab}} \bigr) \Phi^{nab}.
\]
Note that since we can force things to be arbitrarily small here if we simply choose $n$ large enough, we can also sum over $b\in [B_1, B]$, while maintaining smallness. The terms 
\[
 \chi_{I_b^n}''\Phi^{nab}, \quad \chi_{I_b^n}'(A^{na, low} +\mathcal{A}^{nab'}+\mathcal{A}^{naB})\Phi^{nab}
\]
almost cancel the corresponding ones generated by 
\[
\Box_{A^{na, low} + \sum_{b' = 1}^B \mathcal{A}^{nab'} + \mathcal{A}^{naB}} \Big(\sum_{b=B_1}^B(1-\chi_{I^n_b})\tilde{\Phi}^{nab} \Big),
\]
except the $\tilde{\Phi}^{nab}$ used in the latter differs from $\Phi^{nab}$ by a term $\delta\Phi^{nab}$ whose energy is bounded by 
\[
 \big\|\nabla_{t,x} \delta\Phi^{nab}\big\|_{L_t^\infty L_x^2(I_b^n\times \R^4)} \lesssim \big\|\chi_{I_b^n} \bigl( \Box_{A^{na, low} + \sum_{b' = 1}^B \mathcal{A}^{nab'} + \mathcal{A}^{naB}}-\Box_{\mathcal{A}^{nab}} \bigr) \Phi^{nab}\big\|_{N},
\]
and the expression on the right here, even when summed over $b\in [B_1, B]$, is $\ll\delta_7$ if we pick $n$ sufficiently large. This then suffices to bound 
\[
 \sum_{b=B_1}^B \big\|\chi_{I_b^n}'' \delta\Phi^{nab}\big\|_{L_t^1 L_x^2} + \sum_{b=B_1}^B\big\|\chi_{I_b^n}'(A^{na, low} +\mathcal{A}^{nab'}+ \mathcal{A}^{naB})\delta\Phi^{nab}\big\|_{L_t^1L_x^2}\ll \delta_7
\]
for sufficiently large $n$, where we take advantage of the spatial support properties that we assume about the data for $\Phi^{nab}$. 
\end{proof}

Next, we show that each of the terms I -- III can be made arbitrarily small up to certain error terms by picking $B$ and then $n$ sufficiently large.

\medskip

\noindent {\it The contribution of $I$.} One writes schematically 
\begin{align*}
\Box_{A^{na,low} + \sum_{b' = 1}^B\mathcal{A}^{nab'} + \mathcal{A}^{naB}+\delta_A^n}\phi^{na, low} &= \Box_{A^{na,low} + \sum_{b' = 1}^B\mathcal{A}^{nab'} + \mathcal{A}^{naB}}\phi^{na, low} + 2i(\delta^n_A)_{\nu} \partial^{\nu}\phi^{na, low}\\
&\quad \quad + \Big( A^{na,low} + \sum_{b' = 1}^B\mathcal{A}^{nab'} + \mathcal{A}^{naB}+\delta_A^n \Big) \delta_A^n \phi^{na, low}.
\end{align*}
Then one has for any $B$,
\begin{align*}
\lim_{n\rightarrow\infty} \big\|\Box_{A^{na,low} + \sum_{b' = 1}^B\mathcal{A}^{nab'} + \mathcal{A}^{naB}}\phi^{na, low}\big\|_{N} = 0
\end{align*}
due to the frequency localizations (up to exponential tails) of the inputs $\sum_{b' = 1}^B\mathcal{A}^{nab'}, \mathcal{A}^{naB}$ and $\phi^{na, low}$ as well as due to the fact that by construction we have 
\[
\Box_{A^{na,low}}\phi^{na, low} = 0. 
\]
More precisely, one uses an argument as in the proof of Proposition~\ref{prop:bootstrap}. We then still have the error terms 
\begin{equation}\label{eq:errorI}
 2i(\delta^n_A)_{\nu} \partial^{\nu}\phi^{na, low}
\end{equation}
and
\begin{equation*}
 \Big( A^{na,low} + \sum_{b' = 1}^B\mathcal{A}^{nab'} + \mathcal{A}^{naB}+\delta_A^n \Big) \delta^n_A \phi^{na, low}. 
\end{equation*}
The second term here shall be straightforward to treat by means of a simple divisibility argument, while the first will require the equation satisfied by $\delta^n_A$ in conjunction with a divisibility argument. 

\medskip

\noindent {\it{The contribution of $II$.}} We write schematically
\begin{align*}
\Box_{A^{na,low} + \sum_{b' = 1}^B\mathcal{A}^{nab'} + \mathcal{A}^{naB}+\delta_A^n} \Big(\sum_{b=1}^B\Phi^{nab}\Big) &= \sum_{b=1}^B \chi_{I_b^n} \Big( \Box_{A^{na,low} + \sum_{b' = 1}^B\mathcal{A}^{nab'} + \mathcal{A}^{naB}} - \Box_{\mathcal{A}^{nab}} \Big) \Phi^{nab} \\
 &\quad + \sum_{b=1}^B 2i (\delta_A^n)_{\nu} \partial^{\nu} \Phi^{nab} \\
 &\quad + \sum_{b=1}^B \Big( A^{na,low} + \sum_{b' = 1}^B \mathcal{A}^{nab'} + \mathcal{A}^{naB} + \delta_A^n \Big) \delta_A^n \Phi^{nab}.
\end{align*}
Here the time intervals $I_b^n$ correspond to $[-T_b, T_b]$ for the temporally bounded profiles and to $[t_n^{ab} - R_b, t_{n}^{ab}+R_b]$ for the temporally unbounded ones. We shall henceforth make the following {\bf{additional assumption}} that 
\[
 |I_b^n| = M \quad \forall b
\]
chosen very large (eventually depending on $\delta_5$ and the profiles). Then we observe that given any $\delta_5>0$, we can pick $B$ large enough such that for any sufficiently large $n$, we have 
\[
 \Big\| \sum_{b=1}^B \chi_{I_b^n} \big( \Box_{A^{na,low} + \sum_{b' = 1}^B\mathcal{A}^{nab'} + \mathcal{A}^{naB}} - \Box_{\mathcal{A}^{nab}}\big) \Phi^{nab} \Big\|_{N} \ll \delta_5. 
\]
To show this, we need 
\begin{align*}
\Big\| \sum_{b=1}^B\chi_{I_b^n} 2i \Big( A^{na,low} + \sum_{\substack{b' = 1\\ b'\neq b}}^B\mathcal{A}^{nab'} + \mathcal{A}^{naB} \Big)_{\nu} \partial^{\nu} \Phi^{nab} \Big\|_{N} &\ll \delta_5, \\
\Big\|\sum_{b=1}^B\chi_{I_b^n} \Big( (A^{na,low} + \sum_{b' = 1}^B\mathcal{A}^{nab'} + \mathcal{A}^{naB})^2 - (\mathcal{A}^{nab})^2 \Big) \Phi^{nab}\Big\|_{N} &\ll \delta_5.
\end{align*}
For the first expression, observe that the interactions of $\mathcal{A}^{nab'}$, $b'\neq b$, with $\Phi^{nab}$ are easily seen to vanish as $n\rightarrow\infty$, using crude bounds, due to the time localization from $\chi_{I_b^n}$, and the diverging supports of these profiles or their dispersive decay. Similarly, the interaction of $A^{na,low} $ with $\Phi^{nab}$ is seen to vanish asymptotically as $n\rightarrow\infty$, due to the divergent frequency supports and again taking advantage of the extra cutoff $\chi_{I_b^n}$. Note that at this point we have not yet used the parameter $B$. Finally, we also need to bound 
\[
 \Big\| \sum_{b=1}^B \chi_{I_b^n} \mathcal{A}^{naB}_{\nu}\partial^{\nu} \Phi^{nab} \Big\|_{N},
\]
and it is here that we shall take advantage of the size of $B$. Precisely, we divide the above term into two. First, pick $B_1$ very large, depending on the parameter $M$ (which controls the $I_b^n$ via $|I_b^n|\leq M$), such that we have for any $B\geq B_1$,
\[
\limsup_{n\rightarrow\infty} \Big\| \sum_{b=B_1}^B\chi_{I_b^n}\mathcal{A}^{naB}_{\nu}\partial^{\nu} \Phi^{nab} \Big\|_{N} \ll \delta_5. 
\]
That this is possible follows from Lemma~\ref{lem:tail estimate}. Then, with $B_1$ chosen, pick $B\geq B_1$ sufficiently large such that 
\[
\limsup_{n\rightarrow\infty} \Big\| \sum_{b=1}^{B_1}\chi_{I_b^n} \mathcal{A}^{naB}_{\nu} \partial^{\nu} \Phi^{nab} \Big\|_{N}\ll \delta_5. 
\]
Here we take advantage of the fact that we essentially have 
\[
\limsup_{n\rightarrow\infty}\big\|\mathcal{A}^{naB}_{\nu}\big\|_{L^\infty_t L^\infty_x +L_t^\infty \dot{H}^1_x}\rightarrow 0
\]
as $B\rightarrow\infty$. In fact, we have to be a bit careful here, because in Remark~\ref{rem:trickycutoff} we assume that we have incorporated some extra errors into the tail terms $\mathcal{A}^{naB}$ and $\Phi^{naB}$, which do not vanish as $B\rightarrow\infty$. However, considering the term corresponding to a fixed $b\in [1, B_1]$, we have that the extra contribution to $\mathcal{A}^{naB}$ (coming from truncating $\mathcal{A}^{nab}[0]$) interacts weakly with $\Phi^{nab}$ (in the sense that it vanishes as $|I_b|\rightarrow\infty$), see e.g. the proof of Proposition~\ref{prop:morelocaldatastuff}. The remaining contributions from truncating $\mathcal{A}^{nab'}[0]$ are easily seen to result in interactions vanishing as $n\rightarrow\infty$. The cubic term 
\[
 \Big\|\sum_{b=1}^B \chi_{I_b^n} \Big( (A^{na,low} + \sum_{b' = 1}^B\mathcal{A}^{nab'} + \mathcal{A}^{naB})^2 - (\mathcal{A}^{nab})^2 \Big) \Phi^{nab} \Big\|_{N}
\]
is actually simpler, because the temporal cutoffs $\chi_{I_b^n}$ are not even necessary to get the desired bound. This completes the estimate for $II$ except for the error terms
\begin{equation} \label{eq:errorII}
 \Box_{A^{na,low} + \sum_{b' = 1}^B\mathcal{A}^{nab'} + \mathcal{A}^{naB}+\delta_A^n} \Big( \sum_{b=1}^B \Phi^{nab} \Big)
 - \Box_{A^{na,low} + \sum_{b' = 1}^B\mathcal{A}^{nab'} + \mathcal{A}^{naB}} \Big( \sum_{b=1}^B \Phi^{nab} \Big).
\end{equation}
 
\medskip

\noindent {\it The contribution of $III$.} Here we take advantage of the fact that $\Phi^{naB}$ satisfies the equation $\widetilde{\Box}_{A^{na}} u = 0$ to first show that we can pick $B$ large enough such that for all sufficiently large $n$, 
\[
 \biggl\| \Box_{A^{na, low} + \sum_{b'=1}^B {\mathcal A}^{nab'} + {\mathcal A}^{naB}} \Phi^{naB} \biggr\|_N \ll \delta_5.
\]
Of course, the profiles $\mathcal{A}^{nab'}$ are not free waves, but they differ from free waves by terms that are negligible as far as interactions with $\Phi^{naB}$ are concerned. In fact, we recall that 
\[
\big\|\mathcal{A}^{nab'}(t,x) - S(A^{ab'}[0])(t - t^{ab'}_n, x-x^{ab'}_n) \big\|_{\ell^1 S^1} < \infty.
\]
Using Lemma~\ref{lem:tail estimate}, we can refine this to a tail estimate as follows. There exists $B_1$ sufficiently large such that denoting 
\[
 \mathcal{B}^{nab'} := \mathcal{A}^{nab'}(t,x) - S(A^{ab'}[0])(t - t^{ab'}_n, x-x^{ab'}_n),
\]
we have for any $B\geq B_1$,
\begin{align*}
\limsup_{n\rightarrow\infty} \sum_{b' = B_1}^B \big\|2i\mathcal{B}^{nab'}_{\nu}\partial^{\nu}\Phi^{naB}\big\|_{N}\ll \delta_5.
\end{align*}
On the other hand, with this $B_1$ fixed, we can use the argument for Lemma~\ref{lem:inprop1} to conclude that there exists $B\geq B_1$ such that we have 
\begin{align*}
\limsup_{n\rightarrow\infty}\sum_{b' = 1}^{B_1}\big\|2i\mathcal{B}^{nab'}_{\nu}\partial^{\nu}\Phi^{naB}\big\|_{N}\ll \delta_5.
\end{align*}
Finally, we are still left with the error terms 
\begin{equation}\label{eq:errorIII}
 \Box_{A^{na,low} + \sum_{b' = 1}^B\mathcal{A}^{nab'} + \mathcal{A}^{naB}+\delta_A^n}\Phi^{naB} - \Box_{A^{na,low} + \sum_{b' = 1}^B\mathcal{A}^{nab'} + \mathcal{A}^{naB}}\Phi^{naB}.
\end{equation}

\medskip

We have now shown smallness of the terms $I$ -- $III$ up to errors that are at least linear in $\delta_A^n$ given by \eqref{eq:errorI} -- \eqref{eq:errorIII}. 

\medskip

Having dealt with the equation for $\delta_\Phi^n$, we now come to the equation for $\delta_A^n$ given by
\begin{equation} 
 \begin{split} 
  &\Box \Big(A^{na,low}_j + \sum_{b' = 1}^B\mathcal{A}^{nab'}_j + \mathcal{A}^{naB}_j+(\delta_A^n)_j \Big) \\
  &\quad = - {\mathcal P}_j \Im \bigg( \Big( \phi^{na, low} + \sum_{b=1}^B\Phi^{nab} + \Phi^{naB} + \delta_{\Phi}^n \Big) \overline{D_x \Big( \phi^{na, low} + \sum_{b=1}^B\Phi^{nab} + \Phi^{naB} + \delta_{\Phi}^n \Big)} \bigg),
 \end{split}
\end{equation}
where the covariant derivative $D_x$ uses the underlying connection form 
\[
 A^{na,low} + \sum_{b' = 1}^B \mathcal{A}^{nab'} + \mathcal{A}^{naB} + (\delta_A^n).
\]
We rewrite this in the form
\begin{equation} 
 \begin{split} 
  \Box (\delta_A^n) = - IV - V,
  \end{split}
\end{equation}
where we put schematically
\begin{align*}
 IV &:= \Im \bigg( \Big( \phi^{na, low} + \sum_{b=1}^B\Phi^{nab} + \Phi^{naB} + \delta_{\Phi}^n \Big) \overline{D_x \Big( \phi^{na, low} + \sum_{b=1}^B\Phi^{nab} + \Phi^{naB} + \delta_{\Phi}^n \Big)} \bigg) \\
 &\quad \quad - \Im \bigg( \Big( \sum_{b=1}^B\Phi^{nab} + \Phi^{naB} + \delta_{\Phi}^n \Big) \overline{D_x \Big( \sum_{b=1}^B\Phi^{nab} + \Phi^{naB} + \delta_{\Phi}^n \Big)} \bigg) -  \Im \Big( \phi^{na, low} \overline{D_x  \phi^{na, low}} \Big), \\
 V &:= \Im \bigg( \Big( \sum_{b=1}^B\Phi^{nab} + \Phi^{naB} + \delta_{\Phi}^n \Big) \overline{D_x \Big(\sum_{b=1}^B\Phi^{nab} + \Phi^{naB} + \delta_{\Phi}^n \Big)} \bigg) - \sum_{b=1}^B\Box \mathcal{A}^{nab}.
\end{align*}
The term $IV$ can be written in terms of null forms as well as cubic terms involving at least one low frequency factor $ \phi^{na, low}$ as well as at least one high frequency term from 
\[
 \sum_{b=1}^B\Phi^{nab} + \Phi^{naB},
 \]
or else error terms involving at least one factor $\delta_{\Phi}^n$. The former type of interaction is easily seen to converge to zero with respect to $\|\cdot\|_{N}$ as $n\rightarrow\infty$, and so only the latter type of error term needs to be kept. 
As for term $V$, again ignoring the terms involving at least one factor $\delta_{\Phi}^n$, we reduce this to 
\begin{align*}
 \Im \bigg( \Big( \sum_{b=1}^B \Phi^{nab} + \Phi^{naB} \Big) \overline{D_x \Big( \sum_{b=1}^B\Phi^{nab} + \Phi^{naB} \Big)} \bigg) - \sum_{b=1}^B \Box \mathcal{A}^{nab}.
\end{align*}
Then from the definition of the profiles $\Phi^{nab}$, we can write this for some large $B_2$ and $B \geq B_2$ as
\begin{align*}
&\sum_{b=1}^{B_2}\chi_{(I_b^n)^c} \Im \Big( \Phi^{nab} \overline{D_x \Phi^{nab}} \Big) \\
&+\Im \bigg( \Big( \sum_{b=1}^{B_2}\Phi^{nab} \Big) \overline{D_x \Big( \sum_{b=1}^{B_2}\Phi^{nab} \Big) } \bigg) - \sum_{b=1}^{B_2} \Im \Big( \Phi^{nab} \overline{D_x \Phi^{nab} } \Big) \\
& + \Im \bigg( \Big( \sum_{b=B_2}^B\Phi^{nab} \Big) \overline{D_x \Big(\sum_{b=1}^B\Phi^{nab} + \Phi^{naB} \Big)} \bigg) + \Im \bigg( \Phi^{naB} \overline{D_x \Big( \sum_{b=B_2}^B \Phi^{nab} \Big) } \bigg) \\
& - \sum_{b=B_2}^B \chi_{I_b^n} \Im \Big( \Phi^{nab} \overline{D_x \Phi^{nab}} \Big)  \\
&+ \Im \Big( \Phi^{naB} \overline{D_x \Phi^{naB}} \Big) \\
&\equiv (V)_1 + (V)_2 + (V)_3 - (V)_4 + (V)_5. 
\end{align*}
Then given a $\delta_5>0$ arbitrarily small, we first pick $B_2$ sufficiently large such that for all sufficiently large $n$ we have 
\[
 \big\|(V)_3\big\|_{N} + \big\|(V)_4\big\|_{N} \ll \delta_5,
\]
using Lemma~\ref{lem:tail estimate}. Then one picks $n$ large enough such that 
\[
\big\|(V)_2\big\|_{N} \ll \delta_5. 
\]
Further, with $B_2$ fixed, pick $B\geq B_2$ sufficiently large such that 
\[
\big\|(V)_5\big\|_{N} \ll \delta_5. 
\]
Finally, with $B_2$ fixed, we choose $M = |I_b^n|$ large enough (depending on the profiles $\Phi^{nab}, b = 1,\ldots, B_2$, where these of course depend on the $n$-independent $\phi^{ab}[0]$), such that 
\[
\big\|(V)_1\big\|_{N} \ll \delta_5. 
\]
This is then the $M$ that needs to be used in the analysis of the $\delta_{\Phi}^n$ equation in the ``additional assumption'' there. 
\end{proof}

\subsection{Conclusion of the induction on frequency process} \label{subsec:conclusion}

In the preceding subsection we obtained global $S^1$ norm bounds for the MKG-CG evolution of the data
\[
 \bigl( A^{n1}_{\Lambda_0} + A^{n1}, \phi^{n1}_{\Lambda_0} + \phi^{n1} \bigr)[0]
\]
under the assumption that $\big( A^{n1}, \phi^{n1} \big)[0]$ has at least two non-zero concentration profiles, or all such profiles are zero, or else there exists only one such profile $(\tilde{A}^{n1b}, \tilde{\phi}^{n1b})$ but with 
\[
 \liminf_{n\to\infty} E (\tilde{A}^{n1b}, \tilde{\phi}^{n1b})(0) < E_{crit}.
\]
We now make this assumption and continue the process by considering the data
\begin{equation} \label{equ:adding_second_non_atomic_component_data}
 \bigl( A^{n1}_{\Lambda_0} + A^{n1} + A^{n2}_{\Lambda_0}, \phi^{n1}_{\Lambda_0} + \phi^{n1} + \phi^{n2}_{\Lambda_0} \bigr)[0]
\end{equation}
at time $t=0$. Proceeding almost identically to Subsection~\ref{subsec:evolving_non_atomic}, we prove that the MKG-CG evolution of this data exists globally and satisfies a priori $S^1$ norm bounds. These bounds depend on $E_{crit}$ and the a priori bounds on the evolution of the data $\bigl( A^{n1}_{\Lambda_0} + A^{n1}, \phi^{n1}_{\Lambda_0} + \phi^{n1} \bigr)[0]$. The only difference here is that in the decompositions (see Subsection~\ref{subsec:evolving_non_atomic})
\begin{align*}
 A^{n2}_{\Lambda_0}[0] &= \sum_{j=1}^{\Lambda_1(\delta_1)} A^{n2(j)}_{\Lambda_0}[0] + A^{n2}_{\Lambda_0(\Lambda_1)}[0], \\
 \phi^{n2}_{\Lambda_0}[0] &= \sum_{j=1}^{\Lambda_1(\delta_1)} \phi^{n2(j)}_{\Lambda_0}[0] + \phi^{n2}_{\Lambda_0(\Lambda_1)}[0],
\end{align*}
we now have to make sure that 
\[
 \bigl\| A^{n2}_{\Lambda_0(\Lambda_1)}[0] \bigr\|_{\dot{B}^1_{2,\infty} \times \dot{B}^0_{2,\infty}} + \bigl\| \phi^{n2}_{\Lambda_0(\Lambda_1)}[0] \bigr\|_{\dot{B}^1_{2,\infty} \times \dot{B}^0_{2,\infty}} 
\]
is small enough depending both on $E_{crit}$ and the a priori bounds for the MKG-CG evolution of the data $\bigl( A^{n1}_{\Lambda_0} + A^{n1}, \phi^{n1}_{\Lambda_0} + \phi^{n1} \bigr)[0]$. Then we continue by adding the second frequency atom $\bigl( A^{n2}, \phi^{n2} \bigr)[0]$ to the data at time $t=0$ and by repeating the procedure in Subsection~\ref{subsec:adding_first_atom}, but now using the ``covariant'' wave operator
\[
 \widetilde{\Box}_{A^{n2}} := \Box + 2i \bigl( A^{n1}_{\Lambda_0, \nu} + A^{n1}_\nu + A^{n2}_{\Lambda_0,\nu} + A^{n2, free}_\nu \bigr) \partial^\nu,
\]
where $A^{n1}_{\Lambda_0, \nu} + A^{n1}_\nu + A^{n2}_{\Lambda_0,\nu}$ is given by the global MKG-CG evolution of the data \eqref{equ:adding_second_non_atomic_component_data}.

\medskip

All in all we may carry out this process $\Lambda_0$ many times in order to finally conclude that if either there are at least two frequency atoms, or else there is only one frequency atom but with 
\[
 \liminf_{n\rightarrow\infty} E(A^{n1}, \phi^{n1})< E_{crit},
\]
or if we do have 
\[
 \lim_{n\rightarrow\infty}E(A^{n1}, \phi^{n1}) = E_{crit},
\]
but such that there are at least two concentration profiles, or finally if there is only one frequency atom of asymptotic energy $E_{crit}$ and only one concentration profile $(\tilde{A}^{n1b}, \tilde{\phi}^{n1b})$ with 
\[
\liminf_{n\rightarrow\infty}E(\tilde{A}^{n1b}, \tilde{\phi}^{n1b})(0)<E_{crit},
\]
then the sequence $(A^n, \phi^n)$ cannot possibly have been essentially singular, resulting in a contradiction to our assumption. We can then formulate the following 
\begin{cor} \label{cor:final}
Assume that $(A^n, \phi^n)$ is an essentially singular sequence. Then by re-scaling we may assume that the sequence of data $(A^n, \phi^n)[0]$ is $1$-oscillatory, and that there exist sequences 
\[
 \{(t_n, x_n)\}_{n \in \N} \subset \R \times \R^4
\]
and fixed profiles 
\[
 (A, \phi)[0] \in (\dot{H}^1_x \times L^2_x)^4 \times (\dot{H}^1_x \times L^2_x)
\]
with $A$ satisfying the Coulomb condition, such that we have for $j = 1, \ldots, 4$,
\[
 A^n_j[0] = S \big( A_j[0] \big)(\cdot - t_n, \cdot - x_n)[0] + o_{\dot{H}^1_x \times L^2_x}(1) \quad \text{ as } n \to \infty.
\]
Here, $S(\cdot)(t,x)$ denotes the standard free wave propagator. Furthermore, define for $j = 1, \ldots, 4$,
\[
 \tilde{A}_j(t,x) = S\big( A_j[0] \big)(t,x)
\]
and denote by $S_{\tilde{A}}\big( u[0] \big)(t,x)$ the solution to
\[
 \bigl( \Box + 2 i \tilde{A}_j \partial^j \bigr) u = 0
\]
with data $u[0] \in \dot{H}^1_x \times L^2_x$ at time $t=0$. Then we have
\[
 \phi^n[0] = S_{\tilde{A}}\big( \phi[0] \big)(\cdot - t_n, \cdot - x_n)[0] + o_{\dot{H}^1_x \times L^2_x}(1) \quad \text{ as } n \to \infty.
\]
\end{cor}
If the sequence $(t_n)_{n \in \N}$ admits a subsequence that is bounded, then by passing to this subsequence, we may as well replace $t_n$ by $t_n = 0$ for all $n$, and correspondingly obtain up to rescaling and spatial translations that
\[
 (A^n, \phi^n)[0] = (A, \phi)[0] + o_{\dot{H}^1_x \times L^2_x}(1).
\]
Then Proposition~\ref{prop:minenblowup} implies that the evolution $(\mathcal{A}^\infty, \Phi^\infty)$ of $(A, \phi)[0]$ is in fact a minimal energy blowup solution. In the case that $t_n \to + \infty$ (or $t_n \to - \infty$), we need to introduce the concept of a minimum regularity MKG-CG evolution associated with scattering data, or ``a solution at infinity''. Here we have the following 
\begin{prop} \label{prop:scatdata} 
 Let $(A, \phi)[0]$ be Coulomb energy class data and let $\{(t_n, x_n)\}_{n \in \N} \subset \R \times \R^4$ with $t_n \to +\infty$. We introduce the scattering data
 \begin{align} \label{eq:purescatter}
  \begin{aligned}   
   \mathcal{A}^n_j[0] &= S \big( A_j[0] \big)(\cdot - t_n, \cdot - x_n)[0], \quad j = 1, \ldots, 4, \\
   \Phi^n[0] &= S_{\tilde{A}} \big( \phi[0] \big)(\cdot - t_n, \cdot - x_n)[0],
  \end{aligned}
 \end{align}
where we use the notation $\tilde{A}_j(t,x) = S \big( A_j[0] \big)(t, x)$ for $j = 1, \ldots, 4$. Moreover, we denote by $(\mathcal{A}^n, \Phi^n)(t,x)$ the MKG-CG evolution (in the sense of Section~\ref{sec:concept_of_weak_evolution}) of the Coulomb data $(\mathcal{A}^n, \Phi^n)[0]$. Then there exists a sufficiently large $C \in \R_+$ such that there exists an energy class solution $\big(\mathcal{A}^\infty, \Phi^\infty\big)$ to MKG-CG on $(-\infty, -C) \times \R^4$, which is the limit of admissible solutions as in Section~\ref{sec:concept_of_weak_evolution} with 
\[
\big\|\big(\mathcal{A}^\infty, \Phi^\infty\big)\big\|_{S^1((-\infty, -C_0]\times \R^4)}<\infty \quad \forall C_0>C, 
\]
and such that for any $t\in (-\infty, -C)$ we have in the energy topology
\[
\lim_{n\rightarrow\infty} \big(\mathcal{A}^n, \Phi^n \big)(t+t_n, x+x_n) = \big(\mathcal{A}^\infty, \Phi^\infty\big)(t,x).
\]
In particular, the expressions on the left are well-defined (in the sense of Section~\ref{sec:concept_of_weak_evolution}) for $n$ sufficiently large. 
\end{prop}
\begin{proof}
This is a perturbative argument, which exploits the dispersive behaviour as evidenced by amplitude decay of the functions $\mathcal{A}^n[0]$ and $\Phi^n[0]$. We write 
\begin{align*}
 \mathcal{A}^n(t, x) &= \mathcal{A}^{1n}(t,x) + \delta \mathcal{A}^n(t,x), \\ 
 \Phi^n(t,x) &= \Phi^{1n}(t,x) + \delta\Phi^n(t,x),
\end{align*}
where we use the notation
\begin{align*}
 \mathcal{A}^{1n}_j(t,x) &= S \big( A_j[0] \big)(t-t_n, x-x_n), \quad j = 1, \ldots, 4, \\
 \Phi^{1n}(t,x) &= S_{\tilde{A}} \big( \phi[0] \big)(t-t_n, x-x_n).
\end{align*}
Also, keep in mind that $(\mathcal{A}^n, \Phi^n)(t,x)$ denotes the MKG-CG evolution (in the sense of Section~\ref{sec:concept_of_weak_evolution}) of the data $(\mathcal{A}^n, \Phi^n)[0]$. Then we show that $(\delta\mathcal{A}^n, \delta\Phi^n)$ satisfy good $S^1$-bounds on $(-\infty, t_n - C)\times \R^4$ for some $C>0$ sufficiently large, and all $n$ large enough. This means that the evolutions $(\mathcal{A}^n, \Phi^n)$ are well-defined on $ (-\infty, t_n - C)\times \R^4$. Furthermore, assuming as we may that $t_n$ is monotonously increasing, we will show that for $n'>n$, we have 
 \[
 \lim_{\substack{n, n'\rightarrow\infty\\n'>n}}\big\|\mathcal{A}^{n'}[t_{n'} - t_n] - \mathcal{A}^n[0]\big\|_{\ell^1\dot{H}^1_x \times \ell^1 L^2_x} +  \lim_{\substack{n, n'\rightarrow\infty\\n'>n}}\big\|\Phi^{n'}[t_{n'} - t_n] -  \Phi^n[0]\big\|_{\dot{H}^1_x \times L^2_x} = 0, 
 \]
 which together with standard perturbation theory then results in the fact that 
 \[
 \lim_{n\rightarrow\infty}\big(\mathcal{A}^n, \Phi^n\big)(t+t_n, x+x_n) = \big(\mathcal{A}^\infty, \Phi^\infty\big)(t,x),
 \]
provided $t\in (-\infty, -C)$, and the right hand side is a solution to MKG-CG in the sense of Section~\ref{sec:concept_of_weak_evolution}. 

\medskip

To get the desired bounds on $(\delta\mathcal{A}^n, \delta\Phi^n)$, we record the schematic system of equations that they satisfy 
\begin{align} 
 0 &= \Box_{\mathcal{A}^{1n}}\Phi^{1n} + \big(\Box_{\mathcal{A}^{1n} + \delta\mathcal{A}^n} - \Box_{\mathcal{A}^{1n}}\big)\Phi^{1n} + \Box_{\mathcal{A}^{1n} + \delta\mathcal{A}^n}\delta\Phi^n, \label{eq:deltaphiagain} \\
 \Box(\delta\mathcal{A}^n_{j}) &= \mathcal{P}_j\big(\Phi^{1n}\overline{\mathcal{D}_x\Phi^{1n}}\big) +  \mathcal{P}_j\big(\delta\Phi^{n}\overline{\mathcal{D}_x\Phi^{1n}}\big) + \mathcal{P}_j\big(\Phi^{n}\overline{\mathcal{D}_x\delta\Phi^{n}}\big) + \mathcal{P}_j\big(\delta\Phi^{n}\overline{\mathcal{D}_x\delta\Phi^{n}}\big). \label{eq:deltaAagain} 
\end{align}
We then show that given $\delta>0$, there exists $C = C(\delta,  A[0], \phi[0])$ such that we have 
\[
\big\|\delta\mathcal{A}^n\big\|_{\ell^1 S^1((-\infty, t_n - C]\times \R^4)} + \big\|\delta\Phi^n\big\|_{S^1((-\infty, t_n - C]\times \R^4)}<\delta.
\]
This follows as usual via a bootstrap argument. We show here how to obtain smallness of the non-perturbative source terms on the right hand side, i.e. the terms 
\[
 \Box_{\mathcal{A}^{1n}}\Phi^{1n}, \quad \mathcal{P}_i\big(\Phi^{1n}\overline{\mathcal{D}_x\Phi^{1n}}\big),
\]
while the remaining terms are handled either via the smallness of $\delta$ (provided they are quadratic in $\delta\mathcal{A}^n, \delta\Phi^n$), or else via a standard divisibility argument, just as in the proof of Proposition~\ref{prop:bootstrap}. 
Now the first term on the right is in effect equal to 
\[
 \mathcal{A}^{1n}_{\nu}\mathcal{A}^{1n, \nu}\Phi^{1n}.
\]
To treat it, we note that we may reduce all inputs as well as the output to frequency $\sim 1$, since else we gain smallness for the $L_t^1 L_x^2$-norm of the output by using standard Strichartz norms. Then we estimate the remainder by 
\begin{align*}
&\big\|P_{O(1)}\mathcal{A}^{1n}_{\nu}P_{O(1)}\mathcal{A}^{1n, \nu} P_{O(1)}\Phi^{1n}\big\|_{L_t^1 L_x^2((-\infty, t_n-C]\times \R^4)}\\
&\lesssim \big\|P_{O(1)}\mathcal{A}^{1n}_{\nu}\big\|_{L_t^{\frac{10}{3}} L_x^{\frac{10}{3}}((-\infty, t_n-C]\times \R^4)} \big\|P_{O(1)}\mathcal{A}^{1n, \nu}\big\|_{L_t^{5} L^5_x((-\infty, t_n-C]\times \R^4)} \big\| P_{O(1)}\Phi^{1n}\big\|_{L_t^2 L_x^\infty((-\infty, t_n-C]\times \R^4)}.
\end{align*}
Then by exploiting the $L^\infty_x$ decay and interpolation, for example, we get 
\[
\big\|P_{O(1)}\mathcal{A}^{1n, \nu}\big\|_{L^5_t L^5_x((-\infty, t_n-C]\times \R^4)}\ll \delta
\]
for $C$ sufficiently large, uniformly in $n$, and this suffices to get the necessary smallness on account of the fact that uniformly in $n$, 
\begin{align*}
&\big\|P_{O(1)}\mathcal{A}^{1n}_{\nu}\big\|_{L_t^{\frac{10}{3}}L_x^{\frac{10}{3}}((-\infty, t_n-C]\times \R^4)}+\big\| P_{O(1)}\Phi^{1n}\big\|_{L_t^2 L_x^\infty((-\infty, t_n-C]\times \R^4)} \lesssim \big\|(A[0], \phi[0])\big\|_{\dot{H}^1_x \times L^2_x}.
\end{align*}
As for the quadratic term 
\[
\mathcal{P}_j \big(\Phi^{1n}\overline{\mathcal{D}_x\Phi^{1n}}\big),
\]
its inherent null structure allows to reduce to the case of frequencies $\sim 1$ and inputs with an angular separation between their Fourier supports so that the output is at modulation $\sim 1$.  In that situation we have 
\[
\big\|\mathcal{P}_j\big(\Phi^{1n}\overline{\mathcal{D}_x\Phi^{1n}}\big)\big\|_{N((-\infty, t_n-C]\times \R^4)}\lesssim \big\|\mathcal{P}_j \big(\Phi^{1n}\overline{\mathcal{D}_x\Phi^{1n}}\big)\big\|_{L^2_t L^2_x((-\infty, t_n-C]\times \R^4)},
\]
which can be estimated by placing one input into $L_t^{\frac{10}{3}}L_x^{\frac{10}{3}}((-\infty, t_n-C]\times \R^4)$ and the other one into $L^5_t L^5_x((-\infty, t_n-C]\times \R^4)$. The latter norm is small uniformly in $n$ for $C$ sufficiently large on account of (a variant of) Lemma~\ref{lem:uniform_dispersive_bounds}. 
\end{proof}

\begin{rem} \label{rem:prop7.15}
The preceding proof implies in particular that if $(A, \phi)[0]$ is Coulomb energy class data, then there exists $t_0>0$ sufficiently large such that the initial data 
\[
 \Big( S \big( A[0] \big)(\cdot - t_0, \cdot), S_{\tilde{A}}\big( \phi[0] \big)(\cdot - t_0, \cdot) \Big)[0],
\]
where for $j = 1, \ldots, 4$,
\[
 \tilde{A}_j(t, x) = S \big( A_j[0] \big)(t,x),
\]
can be evolved in the sense of Section~\ref{sec:concept_of_weak_evolution} on $(-\infty, 0]\times \R^4$ and satisfies a global $S^1$-bound there. This is the analogue of Proposition 7.15 in \cite{KS}. 
\end{rem}

Extracting a minimal blowup solution in the case of one temporally unbounded profile is still not a direct consequence of the preceding proposition on account of the somewhat delicate perturbation theory, but follows by a slightly indirect argument. Here we state
\begin{prop}\label{prop:scatblowupsol} 
 Assume that the essentially singular sequence $(A^n, \phi^n)$ satisfies 
 \begin{align*}
  A^n[0] &= S \big( A[0] \big)(\cdot - t_n, \cdot - x_n)[0] + o_{\dot{H}^1_x \times L^2_x}(1), \\
  \phi^n[0] &= S_{\tilde{A}}\big( \phi[0] \big)(\cdot - t_n, \cdot - x_n)[0] + o_{\dot{H}^1_x \times L^2_x}(1),
 \end{align*}
 where $t_n\rightarrow +\infty$, say, and we use the same notation as in the preceding Corollary~\ref{cor:final}. Then denoting the corresponding MKG-CG evolution of these data by $(A^n, \phi^n)(t,x)$, its lifespan comprises $(-\infty, t_n-C)$ for $C$ sufficiently large, uniformly in $n$. Also, the sequence 
 \[
  \big\{(A^n, \phi^n)[t_n - 2C] \big\}_{n \in \N}
 \]
 forms a pre-compact set in the energy topology. Denoting a limit point (any such satisfies the Coulomb condition) by $(\mathcal{A}^\infty, \Phi^\infty)[0]$, we have $E(\mathcal{A}^\infty, \Phi^\infty) = E_{crit}$, and moreover, denoting the lifespan of its MKG-CG evolution by $I$, we get 
 \[
  \sup_{J\subset I}\big\|(\mathcal{A}^\infty, \Phi^\infty)\big\|_{S^1(J\times \R^4)} = \infty.
 \]
\end{prop}
\begin{proof} 
The fact that the evolution of $(A^n, \phi^n)(t,x)$ is defined and has finite $S^1$-bounds on $(-\infty, t_n-C)$ follows by exactly the same method as in the proof of the preceding proposition. We set
\begin{align*}
 A^n(t,x) &= A^{1n}(t,x) + \delta A^n(t,x), \\
 \phi^n(t,x) &= S_{\tilde{A}} \big( \phi[0] \big)(t-t_n, x-x_n) + \delta\phi^n(t,x),
\end{align*}
where we let $A^{1n}$ be the free wave evolution of $A^n[0]$, i.e. for $j = 1, \ldots, 4$,
\[
 A^{1n}_j(t,x) =  S \big( A_j[0] \big)(t-t_n, x-x_n) + o_{\dot{H}^1_x \times L^2_x}(1),
\]
and $\tilde{A}_j(t,x) = S \big( A_j[0] \big)(t,x)$ for $j = 1, \ldots, 4$. Also, note that $\delta A^n[0] = 0$. Then choosing $C$ large enough, we infer the bounds 
\[
\big\|(\delta A^n, \delta\phi^n)\big\|_{(\ell^1S^1 \times S^1)(-\infty, t_n - C)\times\R^4} \ll 1
\]
via bootstrap. Since $(A^n, \phi^n)$ is essentially singular, we know by the preceding results that the data 
\[
 (A^n, \phi^n)[t_n - 2C]
\]
are concentrated at fixed frequency $\sim 1$ and consist of exactly one concentration profile, which is necessarily temporally bounded. But this implies that the sequence 
\[
 \big\{(A^n, \phi^n)[t_n - 2C]\big\}_{n\geq 1}
\]
is pre-compact in the energy topology. Extracting a limiting profile $(\mathcal{A}^\infty, \Phi^\infty)[0]$, the last statement of the proposition follows directly from Proposition~\ref{prop:minenblowup}. 
\end{proof}

To conclude this section, we finally state the following crucial {\bf compactness property} of the minimal blowup solution $(\mathcal{A}^\infty, \Phi^\infty)$ extracted in the preceding.

\begin{thm} \label{thm:compact_orbit} 
Denote the lifespan of $(\mathcal{A}^\infty, \Phi^\infty)$ by $I$. There exist continuous functions $\bar{x}: I \rightarrow \R^4$, $\lambda: I\rightarrow \R_+$, so that each of the family of functions 
  \[
   \Bigg\{ \bigg( \frac{1}{\lambda(t)} {\mathcal A}^\infty_j \bigg(t, \frac{\cdot-\bar{x}(t)}{\lambda(t)} \bigg), \frac{1}{\lambda(t)^2} \partial_t {\mathcal A}^\infty_j \bigg(t, \frac{\cdot-\bar{x}(t)}{\lambda(t)} \bigg) \bigg) \, \colon \, t \in I \Bigg\}
  \]
  for $j = 1, \ldots, 4$ and 
  \[
   \Bigg\{ \bigg( \frac{1}{\lambda(t)} \Phi^\infty \bigg(t, \frac{\cdot-\bar{x}(t)}{\lambda(t)} \bigg), \frac{1}{\lambda(t)^2} \partial_t \Phi^\infty \bigg(t, \frac{\cdot-\bar{x}(t)}{\lambda(t)} \bigg) \bigg) \, \colon \, t \in I \Bigg\}
  \]
  is pre-compact in $\dot{H}^1_x(\R^4) \times L^2_x(\R^4)$. 
\end{thm}
The proof of this follows exactly as for Corollary 9.36 in \cite{KS}, using the preceding Remark~\ref{rem:prop7.15}.

\section{Rigidity argument} \label{sec:rigidity_argument}

 In this final section we rule out the existence of a minimal blowup solution $(\cA^\infty, \Phi^\infty)$ with the compactness property from Theorem~\ref{thm:compact_orbit}. To this end we largely follow the scheme of the rigidity argument by Kenig-Merle \cite{KM}. 

 In Subsection~\ref{subsec:rigidity0} we derive several energy and virial identities for energy class solutions to MKG-CG. Then we prove some preliminary properties of the minimal blowup solution $(\cA^\infty, \Phi^\infty)$, in particular that its momentum must vanish. Denoting by $I$ the lifespan of $(\cA^\infty, \Phi^\infty)$, we distinguish between $I^+ := I \cap [0,\infty)$ being a finite or an infinite time interval. In the next Subsection~\ref{subsec:rigidity1}, we exclude the existence of a minimal blowup solution $(\cA^\infty, \Phi^\infty)$ with infinite time interval $I^+$ using the virial identities, the fact that the momentum of $(\cA^\infty, \Phi^\infty)$ must vanish and an additional Vitali covering argument introduced in \cite{KS}. Moreover, we reduce the case of finite lifespan $I^+$ to a self-similar blowup scenario. In the last Subsection~\ref{subsec:rigidity2}, we then derive a suitable Lyapunov functional for the Maxwell-Klein-Gordon system in self-similar variables, which will finally enable us to also rule out the self-similar case.

\subsection{Preliminary properties of minimal blowup solutions with the compactness property} \label{subsec:rigidity0}

 We will sometimes use the following notation for the covariant derivatives
 \[
  \cD_\alpha = \partial_\alpha + i \cA_\alpha^\infty
 \]
 and the curvature components 
 \[
  \cF_{\alpha \beta}^\infty = \partial_\alpha \cA_\beta^\infty - \partial_\beta \cA_\alpha^\infty
 \]
 associated with the minimal blowup solution $(\cA^\infty, \Phi^\infty)$.

 \begin{lem} \label{lem:energy_on_slice_outside_cone}
  Let $(A,\phi)$ be an energy class solution to MKG-CG in the sense of Definition~\ref{defn:energy_class_solution} with lifespan $I$ containing $0$. For given $\varepsilon > 0$, let $R > 0$ be such that
  \[
   \int_{\{|x| \geq R\}} \biggl( \frac{1}{4} \sum_{\alpha,\beta} F_{\alpha \beta}(0,x)^2 + \frac{1}{2} \sum_\alpha |D_\alpha \phi(0,x)|^2 \biggr) \, dx \leq \varepsilon.
  \]
  Then we have for any $t \in I^+$ that
  \[
   \int_{\{|x| \geq R + t\}} \biggl( \frac{1}{4} \sum_{\alpha, \beta} F_{\alpha \beta}(t,x)^2 + \frac{1}{2} \sum_\alpha |D_\alpha \phi(t,x)|^2 \biggr) \, dx \leq \varepsilon.
  \]
 \end{lem}
 \begin{proof}
  Let $(A,\phi)$ be an admissible solution to MKG-CG with lifespan $I$ containing $0$. For $R > 0$ and $t \in I^+$, we define
  \[
   E_R(t) = \int_{\{ |x| \geq R  + t \}} \biggl( \frac{1}{4} \sum_{\alpha, \beta} F_{\alpha \beta}(t,x)^2 + \frac{1}{2} \sum_\alpha |D_\alpha \phi(t,x)|^2 \biggr) \, dx.
  \]
  Using that the energy-momentum tensor for the Maxwell-Klein-Gordon system
  \[
   T_{\alpha \beta} = F_\alpha{}^\gamma F_{\beta \gamma} - \frac{1}{4} m_{\alpha \beta} F_{\gamma \delta} F^{\gamma \delta} + \Re \, \bigl( D_\alpha \phi \overline{D_\beta \phi} \bigr) - \frac{1}{2} m_{\alpha \beta} D^\gamma \phi \overline{D_\gamma \phi},
  \]
  with $m_{\alpha \beta}$ denoting the Minkowski metric, is divergence free
  \[
   \partial^\alpha T_{\alpha \beta} = 0,
  \]
  we easily obtain from the divergence theorem that for any $t_0, t_1 \in I^+$ with $t_0 < t_1$,
  \begin{equation}
   E_R(t_0) = E_R(t_1) + \int_{M_{t_0}^{t_1}} \biggl( T_{00}(t,x) + \frac{x^j}{|x|} T_{j0}(t,x) \biggr) \, d\sigma(t,x).
  \end{equation}
  Here, $M_{t_0}^{t_1}$ denotes the part of the mantle of the forwards light cone $\{ (t, x) \in I^+ \times \R^4 : |x| \leq R + t \}$ enclosed by the time slices $\{t_0\} \times \R^4$ and $\{t_1\} \times \R^4$, and $d\sigma$ denotes the standard surface measure. One easily verifies that the flux
  \[
   T_{00}(t,x) + \frac{x^j}{|x|} T_{j0}(t,x)
  \]
  is non-negative using the general identity
  \[
   \sum_{j,k} \bigl( \omega_j r_k - \omega_k r_j \bigr)^2 = 2 \bigl( r^2 - (r \cdot \omega)^2 \bigr) \leq 2 r^2
  \]
  for $r, \omega \in \R^4$ with $|\omega| = 1$. We conclude that
  \begin{equation} \label{equ:energy_monotonicity_outside_forwards_light_cone}
   E_R(t_1) \leq E_R(t_0).
  \end{equation}
  Since an energy class solution to MKG-CG in the sense of Definition~\ref{defn:energy_class_solution} is a locally uniform limit of admissible solutions, the corresponding inequality \eqref{equ:energy_monotonicity_outside_forwards_light_cone} follows by passing to the limit. This implies the claim. 
 \end{proof}

 Next, we prove the following energy and virial identities for energy class solutions to MKG-CG.

 \begin{prop} \label{prop:virial_identities}
  Let $(A,\phi)$ be an energy class solution to MKG-CG in the sense of Definition~\ref{defn:energy_class_solution}. Then the following identities hold.
  \begin{itemize}[leftmargin=*]
   \item Energy conservation
    \begin{equation} \label{equ:conserved_energy}
     \frac{d}{dt} \int_{\R^4} \biggl( \frac{1}{4} \sum_{\alpha, \beta} F_{\alpha \beta}^2 + \frac{1}{2} \sum_\alpha |D_\alpha \phi|^2 \biggr) \, dx = 0.
    \end{equation}
   \item Momentum conservation
    \begin{equation} \label{equ:conserved_momentum}
     \frac{d}{dt} \int_{\R^4} \Bigl( F_{0j} F_{k}{}^j + \Re \, \bigl( D_0 \phi \overline{D_k \phi} \bigr) \Bigr) \, dx = 0
    \end{equation}
    for $k = 1, \ldots, 4$.
   \item Weighted energy
    \begin{equation} \label{equ:weighted_energy}
     \frac{d}{dt} \int_{\R^4} x_k \varphi_R \biggl( \frac{1}{4} \sum_{\alpha, \beta} F_{\alpha \beta}^2 + \frac{1}{2} \sum_\alpha |D_\alpha \phi|^2 \biggr) \, dx = - \int_{\R^4} \Bigl( F_{0j} F_k {}^j + \Re \, \bigl( D_0 \phi \overline{D_k \phi} \bigr) \Bigr) \, dx + O(r(R))
    \end{equation}
    for $k = 1, \ldots, 4$.
   \item Weighted momentum monotonicity 
    \begin{equation} \label{equ:weighted_momentum_monotonicity}
     \begin{split}
      &\frac{d}{dt} \int_{\R^4} x_k \varphi_R \Bigl( F_{0j} F^{kj} + \Re\, \bigl( D_0 \phi \overline{D^k \phi} \bigr) \Bigr) \, dx + \frac{d}{dt} \int_{\R^4} \varphi_R \Re \, \bigl( \phi \overline{D_0 \phi} \bigr) \, dx \\
      &\quad = - \int_{\R^4} \biggl( \sum_k F_{0k}^2 + |D_0 \phi|^2 \biggr) \, dx + O(r(R)).
     \end{split}
    \end{equation}
  \end{itemize}
  Here, $\varphi \in C_c^\infty(\R^4)$ is a smooth cutoff with $\varphi(x) = 1$ for $|x| \leq 1$ and $\varphi(x) = 0$ for $|x| \geq 2$. Moreover, for $R > 0$ we define $\varphi_R(x) = \varphi\bigl( \frac{x}{R} \bigr)$ and 
  \begin{equation} \label{equ:virial_identities_remainder}
   r(R) := \int_{ \{ |x| \geq R \} } \biggl( \sum_{\alpha, \beta} F_{\alpha \beta}^2 + \sum_\alpha |D_\alpha \phi|^2 + \frac{|\phi|^2}{|x|^2} \biggr) \, dx.
  \end{equation}
 \end{prop}
 \begin{proof}
  It suffices to verify these identities for admissible solutions to MKG-CG. Since energy class solutions in the sense of Definition~\ref{defn:energy_class_solution} are locally uniform limits of admissible solutions, the corresponding identities follow by passing to the limit in an integrated formulation. 

  So let $(A, \phi)$ be an admissible solution to MKG-CG. Then the energy conservation \eqref{equ:conserved_energy} and momentum conservation \eqref{equ:conserved_momentum} identities follow immediately from the divergence theorem and the fact that the energy-momentum tensor of the Maxwell-Klein-Gordon system
  \[
   T_{\alpha \beta} =  F_\alpha{}^\gamma F_{\beta \gamma} - \frac{1}{4} m_{\alpha \beta} F_{\gamma \delta} F^{\gamma \delta} + \Re \, \bigl( D_\alpha \phi \overline{D_\beta \phi} \bigr) - \frac{1}{2} m_{\alpha \beta} D^\gamma \phi \overline{D_\gamma \phi}
  \]
  for $\alpha, \beta \in \{0, 1, \ldots, 4\}$ is divergence free
  \begin{equation} \label{equ:em_tensor_div_free}
   \partial^\alpha T_{\alpha \beta} = 0.
  \end{equation}
  To prove the weighted energy identity \eqref{equ:weighted_energy}, we also use the divergence-free property \eqref{equ:em_tensor_div_free} of $T_{\alpha \beta}$ and compute for $k = 1, \ldots, 4$ that
  \begin{align*}
   \frac{d}{dt} \int_{\R^4} x_k \varphi_R(x) T_{00} \, dx &= \int_{\R^4} x_k \varphi_R(x) \partial^j T_{0j} \, dx \\
   &= - \int_{\R^4} \varphi_R(x) T_{0k} \, dx - \int_{\R^4} {\textstyle \frac{x_k}{R} } (\partial^j \varphi)({\textstyle \frac{x}{R} }) T_{0j} \, dx \\
   &= - \int_{\R^4} T_{0k} \, dx + O(r(R)),
  \end{align*}
  where we integrated by parts in the second to last step. This yields \eqref{equ:weighted_energy}. Finally, to show the weighted momentum monotonicity identity \eqref{equ:weighted_momentum_monotonicity}, we compute
  \begin{align} \label{equ:weighted_momentum_monotonicity_derive1}
   \begin{aligned}
    \frac{d}{dt} \int_{\R^4} x_k \varphi_R(x) T_0 {}^k \, dx &= \int_{\R^4} x_k \varphi_R(x) \partial_j T^{jk} \, dx \\
    &= - \int_{\R^4} \varphi_R(x) \biggl( \sum_{k=1}^4 T_{kk} \biggr) \, dx - \int_{\R^4} {\textstyle \frac{x_k}{R} } (\partial_j \varphi)({\textstyle \frac{x}{R} }) T^{jk} \, dx \\
    &= - \int_{\R^4} \biggl( \sum_{k=1}^4 F_{0k}^2 + 2 |D_0 \phi|^2 - \sum_{k=1}^4 |D_k \phi|^2 \biggr) \, dx + O(r(R)).
   \end{aligned}
  \end{align}
  Since the right hand side of \eqref{equ:weighted_momentum_monotonicity_derive1} does not yet exhibit the desired monotonicity, we also consider
  \begin{align*}
   \frac{d}{dt} \int_{\R^4} \varphi_R(x) \Re \, \bigl( \phi \overline{D_0 \phi} \bigr) \, dx &= \int_{\R^4} \varphi_R(x) \Re \, \bigl( \partial_t \phi \overline{D_0 \phi} \bigr) \, dx + \int_{\R^4} \varphi_R(x) \Re \, \bigl( \phi \overline{\partial_t D_0 \phi} \bigr) \, dx \\
   &= \int_{\R^4} \varphi_R(x) |D_0 \phi|^2 \, dx + \int_{\R^4} \varphi_R(x) \Re \, \Bigl( \phi \overline{D_0^2 \phi} \Bigr) \, dx. 
  \end{align*}
  Inserting the equation for $\phi$ and integrating by parts leads to
  \begin{align} \label{equ:weighted_momentum_monotonicity_derive2}
   \begin{aligned}
    \frac{d}{dt} \int_{\R^4} \varphi_R(x) \Re \, \bigl( \phi \overline{D_0 \phi} \bigr) \, dx &= \int_{\R^4} \varphi_R(x) |D_0 \phi|^2 \, dx - \sum_{k=1}^4 \int_{\R^4} \varphi_R(x) |D_k \phi|^2 \, dx \\
    &\quad \quad - \int_{\R^4} {\textstyle \frac{1}{R} } (\partial^k \varphi)({\textstyle \frac{x}{R}}) \Re \, \Bigl( \varphi \overline{D_k \phi} \Bigr) \, dx \\
    &= \int_{\R^4} \biggl( |D_0 \phi|^2 - \sum_{k=1}^4 |D_k \phi|^2 \biggr) \, dx + O(r(R)).
   \end{aligned}
  \end{align}
  Putting together \eqref{equ:weighted_momentum_monotonicity_derive1} and \eqref{equ:weighted_momentum_monotonicity_derive2}, we obtain \eqref{equ:weighted_momentum_monotonicity}.
 \end{proof}

 If $I^+$ is a finite time interval, we obtain a lower bound on $\lambda(t)$ from Theorem~\ref{thm:compact_orbit}. 
 \begin{lem} \label{lem:finite_lifespan_lower_bound_lambda}
  Assume that $I^+$ is finite and after re-scaling that $I^+ = [0,1)$. Let $\lambda: I^+ \rightarrow \R_+$ be as in Theorem~\ref{thm:compact_orbit}. Then there exists a constant $C_0(K) > 0$ such that 
  \[
   0<\frac{C_0(K)}{1-t}\leq \lambda(t)
  \]
  for all $0 \leq t < 1$.
 \end{lem}
 \begin{proof}
  The proof follows exactly as in \cite[Lemma 10.4]{KS} by combining Corollary~\ref{cor:lifespancompact} and Theorem~\ref{thm:compact_orbit}.
 \end{proof}

 Moreover, when $I^+$ is a finite time interval, we conclude the following sharp support properties of $\Phi^\infty$ and the curvature components $\cF^\infty_{\alpha \beta}$.

 \begin{lem} \label{lem:finite_lifespan_support_in_ball}
  Under the same assumptions as in Lemma~\ref{lem:finite_lifespan_lower_bound_lambda} there exists $x_0 \in \R^4$ such that 
  \[
   \text{supp} \, \Big( \cF_{\alpha \beta}^\infty(t, \cdot), \Phi^\infty(t, \cdot) \Big) \subset \overline{B}(x_0, 1-t) 
  \]
  for all $0 \leq t < 1$ and all $\alpha, \beta \in \{ 0, 1, \ldots,4 \}$.
 \end{lem}
 \begin{proof}
  We follow the proof of Lemma 4.8 in \cite{KM}. Consider a sequence $\{t_n\}_n \subset [0,1)$ with $t_n \to 1$ as $n \to \infty$. From the preceding Lemma~\ref{lem:finite_lifespan_lower_bound_lambda} we know that $\lambda(t_n) \to \infty$ as $n \to \infty$. Together with the compactness property expressed in Theorem~\ref{thm:compact_orbit}, we obtain for every $R > 0$ and $\varepsilon_0 > 0$ that for all sufficiently large $n$, it holds that
  \[
   \int_{\big\{ |x+\frac{\bar{x}(t_n)}{\lambda(t_n)}| \geq R \big\}} \biggl( \sum_\alpha \bigl|\nabla_{t,x} {\mathcal A}_{\alpha}^\infty(t_n, x) \bigr|^2 + \bigl|\nabla_{t,x} \Phi^\infty(t_n, x) \bigr|^2 \biggr) \, dx \leq \frac{\varepsilon_0}{100}
  \]  
  and 
  \[
   \int_{\big\{ |x+\frac{\bar{x}(t_n)}{\lambda(t_n)}| \geq R \big\}} \biggl( \sum_\alpha \bigl|{\mathcal A}_{\alpha}^\infty(t_n, x)\bigr|^4 + \bigl|\Phi^\infty(t_n, x) \bigr|^4 \biggr) \, dx \leq \frac{\varepsilon_0}{100}.
  \]
  Applying Lemma~\ref{lem:energy_on_slice_outside_cone} backwards in time, we conclude for every $R > 0$, $\varepsilon_0 > 0$, and $s \in [0,1)$ that we have for all sufficiently large $n$,
  \begin{equation} \label{equ:finite_lifespan_energy_outside_ball}
   \int_{\big\{ |x+\frac{\bar{x}(t_n)}{\lambda(t_n)}| \geq R + t_n - s \big\}} \biggl( \frac{1}{4}\sum_{\alpha, \beta} \cF_{\alpha \beta}^\infty(s, x)^2 + \frac{1}{2} \sum_{\alpha} \bigl|{\mathcal D}_{\alpha} \Phi^\infty(s, x) \bigr|^2 \biggr) \, dx  \leq \varepsilon_0.
  \end{equation}

  Next, we show that there exists $M > 0$ such that $\left| \frac{\overline{x}(t)}{\lambda(t)} \right| \leq M$ for all $0 \leq t < 1$. Suppose not. Then it suffices to consider a sequence $t_n \to 1$ with $\left| \frac{\overline{x}(t_n)}{\lambda(t_n)} \right| \to \infty$. For all $R > 0$, we have for sufficiently large $n$ that
  \[
   \bigl\{ x : |x| \leq R \bigr\} \subset \biggl\{ x : \Bigl| x + \frac{\bar{x}(t_n)}{\lambda(t_n)} \Bigr| \geq R + t_n \biggr\}.
  \]
  But then we obtain from \eqref{equ:finite_lifespan_energy_outside_ball} with $s = 0$ that for all $R > 0$,
  \[
   \int_{\{ |x| \leq R \}} \biggl( \frac{1}{4}\sum_{\alpha, \beta} \cF_{\alpha \beta}^\infty(0, x)^2 + \frac{1}{2} \sum_{\alpha} \bigl|{\mathcal D}_{\alpha} \Phi^\infty(0, x) \bigr|^2 \biggr) \, dx  \leq \varepsilon_0.
  \]
  Since $\varepsilon_0 > 0$ was arbitrary, this is a contradiction.

  Thus, we may pick a sequence $t_n \to 1$ such that
  \[
   \frac{\bar{x}(t_n)}{\lambda(t_n)} \to - x_0 \in \R^4.
  \]
  Now observe that for every $\eta_0 > 0$ and $s \in [0,1)$, we have for all sufficiently large $n$ that
  \[
   \bigl\{ x : |x-x_0| \geq \eta_0 + 1 - s \bigr\} \subset \biggl\{ x : \Bigl| x + \frac{\bar{x}(t_n)}{\lambda(t_n)} \Bigr| \geq \frac{1}{2} \eta_0 + t_n - s \biggr\}.
  \]
  Hence, we obtain from \eqref{equ:finite_lifespan_energy_outside_ball} that for every $\varepsilon_0 > 0$, $\eta_0 > 0$ and $s \in [0,1)$,
  \[
   \int_{\{ |x - x_0| \geq \eta_0 + 1 - s \}} \biggl( \frac{1}{4} \sum_{\alpha, \beta} \cF_{\alpha \beta}^\infty(s, x)^2 + \frac{1}{2} \sum_{\alpha} \bigl|{\mathcal D}_{\alpha} \Phi^\infty(s, x) \bigr|^2 \biggr) \, dx  \leq \varepsilon_0.
  \]
  We conclude that
  \[
   \text{supp} \, \Big( \cF_{\alpha \beta}^\infty(t, \cdot), \big( {\mathcal D}_\alpha \Phi^\infty \big)(t, \cdot) \Big) \subset \overline{B}(x_0, 1-t)
  \]
  for all $0 \leq t < 1$ and all $\alpha, \beta = 0, 1, \ldots, 4$. The claim then follows from the diamagnetic inequality.
 \end{proof}

 In the next key proposition we prove that the momentum of the minimal blowup solution $(\cA^\infty, \Phi^\infty)$ must vanish. This will later allow us to control the movement of the ``center of mass'', or more precisely a weighted energy of $(\cA^\infty, \Phi^\infty)$. For technical reasons we have to distinguish between the case of finite and infinite lifespan.

 \begin{prop} \label{prop:vanishing_momentum_finite_time}
  Let $({\mathcal A}^\infty, \Phi^\infty)$ be as above. Assume that $I^+$ is a finite interval. Then we have for $k = 1, \ldots, 4$ and all $t \in I^+$ that
  \begin{equation}
   \int_{\R^4} \biggl( \sum_{j=1}^4 \cF_{0j}^\infty \cF_{kj}^\infty + \Re \, \bigl( \mathcal{D}_0 \Phi^\infty \overline{ \mathcal{D}_k \Phi^\infty } \bigr) \biggr)(t,x) \, dx = 0.
  \end{equation}
 \end{prop}
 As for the critical focusing nonlinear wave equation \cite{KM} and for critical wave maps \cite{KS}, the Lorentz invariance of the Maxwell-Klein-Gordon system and transformational properties of the energy under Lorentz transformations are essential ingredients in the proof of Proposition~\ref{prop:vanishing_momentum_finite_time}. We begin by considering the relativistic invariance properties of our system. Assume that 
 \[
  L:\R^{1+4} \to \R^{1+4}
 \]
 is a Lorentz transformation, acting on column vectors via multiplication with the matrix $L$. Then $\phi$ transforms according to 
 \begin{equation} \label{equ:Lorentz_transformation_of_phi}
  \phi \mapsto \phi^L := \phi \bigl( L (t,x) \bigr),
 \end{equation}
 which results in 
 \[
  \nabla_{t,x} \phi \mapsto L^t \nabla_{t,x} \phi \bigl( L (t,x) \bigr).
 \]
 Then the potential $A_{\alpha}$ needs to transform accordingly, i.e. writing this as a column vector indexed by $\alpha$, we transform 
 \begin{equation} \label{equ:Lorentz_transformation_of_A}
  A \mapsto A^L := L^t A \bigl( L (t,x) \bigr).
 \end{equation}
 Then the expression $\partial^{\beta}F_{\alpha\beta}$, when interpreted as a column vector in $\alpha$, also transforms according to multiplication with $L^t$, as does the expression 
 \[
  \Im\big(\phi\overline{D_{\alpha}\phi}\big).
 \]
 Under these transformations, the Maxwell-Klein-Gordon system is then invariant. However, the conserved energy does not remain invariant under general Lorentz transformations, and our first step is to quantify this. In the sequel we only consider very specific Lorentz transformations of the form
 \begin{equation} \label{equ:specific_Lorentz_transform}
  L = \left( \begin{array}{ccccc}\frac{1}{\sqrt{1-d^2}}&\frac{-d}{\sqrt{1-d^2}}&0&0&0\\\frac{-d}{\sqrt{1-d^2}}&\frac{1}{\sqrt{1-d^2}}&0&0&0 \\
   0&0&1&0&0 \\
   0&0&0&1&0 \\
   0&0&0&0&1 \end{array} \right)
 \end{equation}
 for small $d \in \R$.
 \begin{lem} \label{lem:energy_transformation_behavior_under_lorentz}
  Let $(A, \phi)$ be an admissible global solution to MKG-CG and let $L: \R^{1+4} \to \R^{1+4}$ be a Lorentz transformation of the form \eqref{equ:specific_Lorentz_transform} for some $d \in \R$. Then we have for all $t \in \R$ that
  \begin{align} \label{equ:energy_transformation_behavior_under_lorentz}
   \begin{aligned}   
    E\bigl(A^L, \phi^L\bigr)(t) &= \int_{\R^4} \biggl( \frac{1}{4} \sum_{\alpha, \beta} F_{\alpha \beta}^2 + \frac{1}{2} \sum_\alpha |D_\alpha \phi|^2 \biggr)\bigl(L(t,x)\bigr) \, dx \\
    &\quad + \frac{d^2}{1-d^2} \int_{\R^4} \biggl( \sum_{j=2}^4 (F_{0j}^2 + F_{1j}^2) + \sum_{\alpha = 0}^1 |D_\alpha \phi|^2 \biggr)\bigl(L(t,x)\bigr) \, dx \\
    &\quad - \frac{2d}{1-d^2} \int_{\R^4} \biggl( \sum_{j=1}^4 F_{0j} F_{1j} + \Re \, \bigl( D_0 \phi \overline{D_1 \phi} \bigr) \biggr)\bigl(L(t,x)\bigr) \, dx.
   \end{aligned}
  \end{align}
 \end{lem}
 \begin{proof}
  The potential $A$ is transformed into $A^L$ as follows
  \begin{align*}
   A_0^L(t,x) &= \frac{1}{\sqrt{1-d^2}} A_0(L(t,x)) - \frac{d}{\sqrt{1-d^2}} A_1(L(t,x)), \\
   A_1^L(t,x) &= - \frac{d}{\sqrt{1-d^2}} A_0(L(t,x)) + \frac{1}{\sqrt{1-d^2}} A_1(L(t,x)), \\
   A_j^L(t,x) &= A_j(L(t,x)), \quad j = 2, 3, 4.
  \end{align*}
  Then we compute the corresponding curvature components
  \begin{align*}
   F_{01}^L &= \partial_t A_1^L - \partial_1 A_0^L \\
   &= - \frac{d}{\sqrt{1-d^2}} \biggl( \frac{1}{\sqrt{1-d^2}} \partial_t A_0 - \frac{d}{\sqrt{1-d^2}} \partial_1 A_0 \biggr) + \frac{1}{\sqrt{1-d^2}} \biggl( \frac{1}{\sqrt{1-d^2}} \partial_t A_1 - \frac{d}{\sqrt{1-d^2}} \partial_1 A_1 \biggr) \\
   &\quad - \frac{1}{\sqrt{1-d^2}} \biggl( - \frac{d}{\sqrt{1-d^2}} \partial_t A_0 + \frac{1}{\sqrt{1-d^2}} \partial_1 A_0 \biggr) + \frac{d}{\sqrt{1-d^2}} \biggl( - \frac{d}{\sqrt{1-d^2}} \partial_t A_1 + \frac{1}{\sqrt{1-d^2}} \partial_1 A_1 \biggr) \\
   &= F_{01}.
  \end{align*}
  Here the right hand side has to be evaluated at $L (t,x)$. We use this convention for the remainder of the proof. Further, we obtain
  \begin{align*}
   F_{02}^L &= \frac{1}{\sqrt{1-d^2}} \partial_t A_2 - \frac{d}{\sqrt{1-d^2}} \partial_1 A_2 - \partial_2 \biggl( \frac{1}{\sqrt{1-d^2}} A_0 - \frac{d}{\sqrt{1-d^2}} A_1 \biggr) \\
   &= \frac{1}{\sqrt{1-d^2}} F_{02} - \frac{d}{\sqrt{1-d^2}} F_{12}
  \end{align*}
  as well as
  \begin{align*}
   F_{03}^L = \frac{1}{\sqrt{1-d^2}} F_{03} - \frac{d}{\sqrt{1-d^2}} F_{13}, \quad F_{04}^L = \frac{1}{\sqrt{1-d^2}} F_{04} - \frac{d}{\sqrt{1-d^2}} F_{14}.
  \end{align*}
  Similarly, we compute 
  \begin{align*}
   F_{12}^L &= - \frac{d}{\sqrt{1-d^2}} \partial_t A_2 + \frac{1}{\sqrt{1-d^2}} \partial_1 A_2 + \frac{d}{\sqrt{1-d^2}} \partial_2 A_0 - \frac{1}{\sqrt{1-d^2}} \partial_2 A_1 \\
   &= - \frac{d}{\sqrt{1-d^2}} F_{02} + \frac{1}{\sqrt{1-d^2}} F_{12}
  \end{align*}
  and 
  \begin{align*}
   F_{13}^L = - \frac{d}{\sqrt{1-d^2}} F_{03} + \frac{1}{\sqrt{1-d^2}} F_{13}, \quad F_{14}^L = - \frac{d}{\sqrt{1-d^2}} F_{04} + \frac{1}{\sqrt{1-d^2}} F_{14}.
  \end{align*}
  Finally, we have for $i, j \geq 2$ that
  \begin{align*}
   F_{ij}^L = F_{ij}.
  \end{align*}
  In summary, we have found that
  \begin{align} \label{equ:energy_transformation_behavior_under_lorentz_A}
   \sum_{\alpha, \beta} \bigl( F_{\alpha \beta}^L \bigr)^2 &= \sum_{\alpha, \beta} F_{\alpha \beta}^2 + \frac{4 d^2}{1-d^2} \sum_{j=2}^4 \bigl( F_{0j}^2 + F_{1j}^2 \bigr) - \frac{8d}{1-d^2} \sum_{j=1}^4 F_{0j} F_{1j}.
  \end{align}
  We have to carry out the analogous computations for the part of the energy associated with the scalar field $\phi$. Here we have
  \begin{align*}
   &\bigl| \bigl( \partial_t + i A_0^L \bigr) \phi^L \bigr|^2 + \bigl| \bigl( \partial_1 + i A_1^L \bigr) \phi^L \bigr|^2 \\
   &= \biggl| \frac{1}{\sqrt{1-d^2}} \partial_t \phi - \frac{d}{\sqrt{1-d^2}} \partial_1 \phi + i \biggl( \frac{1}{\sqrt{1-d^2}} A_0 - \frac{d}{\sqrt{1-d^2}} A_1 \biggr) \phi \biggr|^2 \\
   &\quad + \biggl| - \frac{d}{\sqrt{1-d^2}} \partial_t \phi + \frac{1}{\sqrt{1-d^2}} \partial_1 \phi + i \biggl( - \frac{d}{\sqrt{1-d^2}} A_0 + \frac{1}{\sqrt{1-d^2}} A_1 \biggr) \phi \biggr|^2 \\
   &= \frac{1+d^2}{1-d^2} \Bigl( |D_0 \phi|^2 + |D_1 \phi|^2 \Bigr) - \frac{4d}{1-d^2} \Re \, \bigl( D_0 \phi \overline{D_1 \phi} \bigr)
  \end{align*}
  and for $j = 2, 3, 4$,
  \[
   \bigl| \bigl( \partial_j + i A_j^L \bigr) \phi^L \bigr|^2 = |D_j \phi|^2. 
  \]
  Thus, we obtain that
  \begin{align} \label{equ:energy_transformation_behavior_under_lorentz_phi}
   \sum_\alpha \bigl| \bigl( \partial_\alpha + i A_\alpha^L \bigr) \phi^L \bigr|^2 = \sum_\alpha |D_\alpha \phi|^2 + \frac{2 d^2}{1-d^2} \bigl( |D_0 \phi|^2 + |D_1 \phi|^2 \bigr) - \frac{4d}{1-d^2} \Re \, \bigl( D_0 \phi \overline{D_1 \phi} \bigr).
  \end{align}
  The assertion now follows from \eqref{equ:energy_transformation_behavior_under_lorentz_A} and \eqref{equ:energy_transformation_behavior_under_lorentz_phi}.
 \end{proof}

 The identity \eqref{equ:energy_transformation_behavior_under_lorentz} strongly suggests that if it is impossible to lower the energy by means of a Lorentz transform of the form \eqref{equ:specific_Lorentz_transform} for very small $d$ with a suitable sign, then the momentum must vanish. To make this observation rigorous, we also need to establish a relation between the $S^1$ norm of an admissible global solution $(A, \phi)$ to MKG-CG and the $S^1$ norm of a suitable evolution of the data $(A^L, \phi^L)[0]$ obtained from the Lorentz transformed solution $(A^L, \phi^L)$. Here we first observe that for an admissible global solution $(A, \phi)$ to MKG-CG, the Lorentz transformed solution $(A^L, \phi^L)$ is actually globally defined. We can therefore consider the data pair $(A^L, \phi^L)[0]$ and note that $(A^L, \phi^L)[0]$ is $C^\infty$-smooth, but not in Coulomb gauge. Moreover, if $(t,x) \in \R^{1+4}$ are restricted to a space-like hyperplane containing the origin, then we have
 \[
  \bigl| F_{jk}(t,x) \bigr| \lesssim \bigl( 1 + |t| + |x| \bigr)^{-N}
 \]
 for $j,k \in \{1, \ldots, 4\}$ and any $N \geq 1$. From the equation satisfied by $F_{\alpha \beta}$ we obtain after integration in time that
 \[
  \bigl| F_{0k}(t,x) \bigr| \lesssim \bigl( 1 + |t| + |x| \bigr)^{-3}
 \]
 for $k = 1, \ldots, 4$. Thus, the curvature components of $(A^L, \phi^L)[0]$ decay like $\langle x \rangle^{-3}$ as $|x| \to \infty$, which ensures $L^2_x$-integrability, and the components $\nabla_{t,x} \phi^L$ decay rapidly with respect to $x$. In particular, upon transforming $(A^L, \phi^L)[0]$ into Coulomb gauge, it is meaningful to consider its MKG-CG evolution and its $S^1$ norm. Then we prove the following technical 
 \begin{prop} \label{prop:S_norm_Lorentz}
  Let $(A, \phi)$ be an admissible global solution to MKG-CG and let $L: \R^{1+4} \to \R^{1+4}$ be a Lorentz transformation of the form \eqref{equ:specific_Lorentz_transform} for sufficiently small $|d|$. Let $(A^L, \phi^L)[0]$ be the data pair obtained from the Lorentz transformed solution $(A^L, \phi^L)$. Assume that $(A^L, \phi^L)[0]$, when transformed into the Coulomb gauge, results in a smooth global solution $\bigl( \tilde{A}^L, \tilde{\phi}^L \bigr)$ to MKG-CG satisfying
  \begin{equation*}
   \bigl\| \bigl( \tilde{A}^L, \tilde{\phi}^L \bigr) \bigr\|_{S^1} < \infty.
  \end{equation*}
  Then we have for the original evolution $(A, \phi)$ that
  \begin{equation*}
   \bigl\| (A, \phi) \bigr\|_{S^1} \leq C \Bigl( \bigl\| \bigl( \tilde{A}^L, \tilde{\phi}^L \bigr) \bigr\|_{S^1}, L \Bigr).
  \end{equation*}
 \end{prop}
 We defer the technical proof of Proposition~\ref{prop:S_norm_Lorentz} to the end of this subsection and first prove Proposition~\ref{prop:vanishing_momentum_finite_time} by combining Lemma~\ref{lem:energy_transformation_behavior_under_lorentz} and Proposition~\ref{prop:S_norm_Lorentz}.
 \begin{proof}[Proof of Proposition~\ref{prop:vanishing_momentum_finite_time}]
  In order to be able to apply Proposition~\ref{prop:S_norm_Lorentz}, we have to use smooth solutions that are globally defined, because otherwise we cannot meaningfully apply a Lorentz transformation. In fact, we may exploit that by the preceding Lemma~\ref{lem:finite_lifespan_support_in_ball}, the function $\Phi^\infty$ is compactly supported, which means that its Fourier transform cannot also be compactly supported (we may of course assume $\Phi^\infty[0]$ to be non-vanishing, since otherwise, the solution extends trivially in a global fashion and cannot be singular). But then, truncating the data $(\cA^\infty, \Phi^\infty)[0]$ in Fourier space as in Proposition~\ref{prop:perturbation} and the discussion following it, we may construct a sequence of smooth Coulomb data $(A_n, \phi_n)[0]$ converging to $(\mathcal{A}^\infty, \Phi^\infty)[0]$, and if necessary, multiplying the $\phi_n[0]$ in the resulting $(A_n, \phi_n)[0]$ by a small scalar $\lambda_n \in [0,1]$ with $\lambda_n \rightarrow 1$ as $n \to \infty$, we may force that for all $n \geq 1$,
  \begin{equation} \label{equ:S_norm_prop_energy_inequality}
   E(A_n, \phi_n) < E_{crit}.
  \end{equation}
  Note that then the perturbation theory developed in Proposition~\ref{prop:perturbation} still applies in relation to $(\mathcal{A}^\infty, \Phi^\infty)$, since we have not changed the data for $A_n$. This means that the data $(A_n, \phi_n)[0]$ do admit a global MKG-CG evolution by definition of $E_{crit}$, and can thus be Lorentz transformed. In order to justify various conservation laws for the Lorentz transformed $(A_n, \phi_n)$, we observe that we may also localize the data $(A_n, \phi_n)[0]$ in physical space to a sufficiently large ball, using the argument in Subsection~\ref{subsec:localizing_in_physical_space} as well as \cite{KMPT} such that the Lorentz transformed solution also has compact support on bounded time slices, and we still have the above inequality \eqref{equ:S_norm_prop_energy_inequality} for the energy. 

  \medskip

  We make the hypothesis that the momentum of $(\cA^\infty, \Phi^\infty)$ does not vanish. Then without loss of generality, there exists $\gamma > 0$ such that for all sufficiently large $n$, we have
  \begin{equation} \label{equ:momentum_lower_and_upper_bound}
   \int_{\R^4} \biggl( \sum_{j=1}^4 F_{n, 0j} F_{n, 1j} + \Re \, \bigl( (\partial_t  + i A_{n,0}) \phi_n \overline{ (\partial_1 + i A_{n,1} ) \phi_n } \bigr) \biggr)(t,x) \, dx \geq \gamma,
  \end{equation}
  where $F_{n, \alpha \beta}$ denote the curvature components of $(A_n, \phi_n)$. It suffices to show that a suitable Lorentz transformation $L$ of the form \eqref{equ:specific_Lorentz_transform} exists such that the transformed solutions $(A_n^L, \phi_n^L)$ to the Maxwell-Klein-Gordon system have energies
  \begin{equation} \label{equ:vanishing_momentum_uniform_upper_bound_energies}
   E(A_n^L, \phi_n^L) \leq E_{crit} - \kappa(\gamma, \mathcal{A}^\infty, \Phi^\infty) 
  \end{equation}
  uniformly in $n$ for some $\kappa(\gamma, \mathcal{A}^\infty, \Phi^\infty) > 0$. Then, upon transforming $(A_n^L, \phi_n^L)[0]$ into the Coulomb gauge, we obtain a global solution to MKG-CG with a finite global $S^1$ norm bound, and using Proposition~\ref{prop:S_norm_Lorentz}, we can infer a global $S^1$ norm bound for $(A_n, \phi_n)$ uniformly in $n$, which contradicts that $(\mathcal{A}^\infty, \Phi^\infty)$ is a singular solution. To implement this strategy, we combine the argument for Proposition 4.10 in \cite{KM} with Lemma~\ref{lem:energy_transformation_behavior_under_lorentz}. 

  By energy conservation for $(A_n^L, \phi_n^L)$, we have the relation
  \begin{equation*}
   \frac{1}{4} E \bigl( A_n^L, \phi_n^L \bigr)(0) = \int_{0}^{\frac{1}{4}} E\bigl( A_n^L, \phi_n^L \bigr)(t) \, dt,
  \end{equation*}
  where we recall that for a solution $(A, \phi)$ to the Maxwell-Klein-Gordon system the energy at time $t \in \R$ is given by
  \[
   E\bigl(A, \phi\bigr)(t) = \int_{\R^4} \biggl( \frac{1}{4} \sum_{\alpha, \beta} F_{\alpha \beta}^2 + \frac{1}{2} \sum_{\alpha} |D_{\alpha} \phi|^2 \biggr)(t,x) \, dx.
  \]
  According to Lemma~\ref{lem:energy_transformation_behavior_under_lorentz}, we can write 
  \[
   \frac{1}{4} E \bigl( A_n^L, \phi_n^L \bigr)(0) = I_1 + I_2,
  \]
  where
  \begin{align*}
   I_1 &= \int_{0}^{\frac{1}{4}} \int_{\R^4} \biggl( \frac{1}{4} \sum_{\alpha, \beta} F_{n, \alpha \beta}^2 + \frac{1}{2} \sum_\alpha \bigl|(\partial_\alpha + i A_{n,\alpha}) \phi_n \bigr|^2 \biggr)\bigl(L(t,x)\bigr) \, dx \, dt \\
   &\quad + \frac{d^2}{1-d^2} \int_{0}^{\frac{1}{4}} \int_{\R^4} \biggl( \sum_{j=2}^4 \bigl( F_{n,0j}^2 + F_{n,1j}^2 \bigr) + \sum_{\alpha = 0}^1 \bigl| (\partial_\alpha + i A_{n, \alpha}) \phi_n \bigr|^2 \biggr)\bigl(L(t,x)\bigr) \, dx \, dt
  \end{align*}
  and 
  \begin{align*}
   I_2 &= - \frac{2d}{1-d^2} \int_{0}^{\frac{1}{4}} \int_{\R^4} \biggl( \sum_{j=1}^4 F_{n,0j} F_{n,1j} + \Re \, \bigl( (\partial_t + i A_{n,0}) \phi_n \overline{ (\partial_1 + i A_{n,1}) \phi_n } \bigr) \biggr)\bigl(L(t,x)\bigr) \, dx \, dt.
  \end{align*}
  We recall that the above integrands are evaluated at 
  \[
   L(t,x) = \biggl( \frac{t - d x_1}{\sqrt{1 - d^2}}, \frac{x_1 - d t}{\sqrt{1-d^2}}, x_2, x_3, x_4 \biggr).
  \]
  Next we compute the derivative of $I_1 + I_2$ with respect to $d$. To this end we note that for a regular function $f$ of compact support, it holds that (see page 173 in \cite{KM})
  \[
   \frac{\partial}{\partial d} \int_{\R^4} f \bigl( L(t,x) \bigr) \, dx = - \frac{1}{1-d^2} \frac{\partial}{\partial t} \int_{\R^4} x_1 f\bigl( L(t,x) \bigr) \, dx.
  \]
  Using our assumption that the spatial components of $A_n$ as well as $\phi_n$ are compactly supported on fixed time slices, we thus obtain
  \begin{align*}
   \frac{\partial}{\partial d} I_1(d) &= - \frac{1}{1-d^2} \int_0^{\frac{1}{4}} \frac{\partial}{\partial t} \int_{\R^4} x_1 \biggl( \frac{1}{4} \sum_{\alpha, \beta} F_{n, \alpha \beta}^2 + \frac{1}{2} \sum_\alpha \bigl|(\partial_\alpha + i A_{n,\alpha}) \phi_n \bigr|^2 \biggr)\bigl(L(t,x)\bigr) \, dx \, dt \\
   &\quad + \frac{2 d}{(1-d^2)^2} \int_0^{\frac{1}{4}} \int_{\R^4} \biggl( \sum_{j=2}^4 \bigl( F_{n,0j}^2 + F_{n,1j}^2 \bigr) + \sum_{\alpha = 0}^1 \bigl| (\partial_\alpha + i A_{n, \alpha}) \phi_n \bigr|^2 \biggr)\bigl(L(t,x)\bigr) \, dx \, dt \\
   &\quad - \frac{d^2}{(1-d^2)^2} \int_0^{\frac{1}{4}} \frac{\partial}{\partial t} \int_{\R^4} x_1 \biggl( \sum_{j=2}^4 \bigl( F_{n,0j}^2 + F_{n,1j}^2 \bigr) + \sum_{\alpha = 0}^1 \bigl| (\partial_\alpha + i A_{n, \alpha}) \phi_n \bigr|^2 \biggr)\bigl(L(t,x)\bigr) \, dx \, dt 
  \end{align*}
  and 
  \begin{align*}
   \frac{\partial}{\partial d} I_2(d) &= - \frac{2 + 2 d^2}{(1-d^2)^2} \int_0^{\frac{1}{4}} \int_{\R^4} \biggl( \sum_{j=1}^4 F_{n,0j} F_{n,1j} + \Re \, \bigl( (\partial_t + i A_{n,0}) \phi_n \overline{ (\partial_1 + i A_{n,1}) \phi_n } \bigr) \biggr)\bigl(L(t,x)\bigr) \, dx \, dt \\
   &\quad + \frac{2d}{(1-d^2)^2} \int_0^{\frac{1}{4}} \frac{\partial}{\partial t} \int_{\R^4} x_1 \biggl( \sum_{j=1}^4 F_{n,0j} F_{n,1j} + \Re \, \bigl( (\partial_t + i A_{n,0}) \phi_n \overline{ (\partial_1 + i A_{n,1}) \phi_n } \bigr) \biggr)\bigl(L(t,x)\bigr) \, dx \, dt.
  \end{align*}
  But then
  \begin{align*}
   \frac{\partial}{\partial d} (I_1 + I_2) \bigg|_{d=0} &= - \int_{\R^4} x_1 \biggl( \frac{1}{4} \sum_{\alpha, \beta} F_{n, \alpha \beta}^2 + \frac{1}{2} \sum_\alpha \bigl|(\partial_\alpha + i A_{n,\alpha}) \phi_n \bigr|^2 \biggr)(t,x) \, dx \bigg|_{t = 0}^{t = \frac{1}{4}} \\
   &\quad - 2 \int_0^{\frac{1}{4}} \int_{\R^4} \biggl( \sum_{j=1}^4 F_{n,0j} F_{n,1j} + \Re \, \bigl( (\partial_t + i A_{n,0}) \phi_n \overline{ (\partial_1 + i A_{n,1}) \phi_n } \bigr) \biggr)(t,x) \, dx \, dt.
  \end{align*}
  Using the weighted energy identity \eqref{equ:weighted_energy} (for $R \to \infty$) and \eqref{equ:momentum_lower_and_upper_bound}, we conclude that
  \begin{align*}
   \frac{\partial}{\partial d} (I_1 + I_2) \bigg|_{d=0} &= - \int_0^{\frac{1}{4}} \int_{\R^4} \biggl( \sum_{j=1}^4 F_{n,0j} F_{n,1j} + \Re \, \bigl( (\partial_t + i A_{n,0}) \phi_n \overline{ (\partial_1 + i A_{n,1}) \phi_n } \bigr) \biggr)(t,x) \, dx \, dt \leq - \frac{1}{4} \gamma
  \end{align*}
  uniformly for all sufficiently large $n$. Also, by energy conservation for $(A_n, \phi_n)$, we have for all $n$ that
  \[
   (I_1 + I_2)(d=0) = \frac{1}{4} E(A_n, \phi_n) < \frac{1}{4} E_{crit}.
  \]
  Hence, we can find a small $d > 0$ so that
  \[
   \frac{1}{4} E\bigl( A_n^L, \phi_n^L \bigr)(0) \leq \frac{1}{4} E_{crit} - \kappa
  \]
  uniformly for all sufficiently large $n$ for some $\kappa \equiv \kappa(\gamma, \cA^\infty, \Phi^\infty) > 0$, which yields \eqref{equ:vanishing_momentum_uniform_upper_bound_energies} and thus finishes the proof of Proposition~\ref{prop:vanishing_momentum_finite_time}.
 \end{proof}

 We next state the analogous result to Proposition~\ref{prop:vanishing_momentum_finite_time} when $I^+$ is infinite. Its proof essentially follows the argument of the proof of Proposition~4.11 in \cite{KM} using the same modifications as in the preceding proof of Proposition~\ref{prop:vanishing_momentum_finite_time}.
 \begin{prop} \label{prop:vanishing_momentum_infinite_time}
  Let $({\mathcal A}^\infty, \Phi^\infty)$ be as above. Assume that $I^+ = [0,\infty)$. Suppose in addition that $\lambda(t) \geq \lambda_0 > 0$ for all $t \geq 0$. Then we have for $k = 1, \ldots, 4$ and all $t \geq 0$ that
  \begin{equation}
   \int_{\R^4} \biggl( \sum_{j=1}^4 \cF_{0j}^\infty \cF_{kj}^\infty + \Re \, \bigl( \mathcal{D}_0 \Phi^\infty \overline{ \mathcal{D}_k \Phi^\infty } \bigr) \biggr)(t,x) \, dx = 0.
  \end{equation}
 \end{prop}

 It remains to give the proof of Proposition~\ref{prop:S_norm_Lorentz}.
 \begin{proof}[Proof of Proposition~\ref{prop:S_norm_Lorentz}]
  We are given an admissible global solution $(A, \phi)$ to MKG-CG and a Lorentz transformation $L: \R^{1+4} \to \R^{1+4}$ of the form \eqref{equ:specific_Lorentz_transform} for small $d \in \R$. Applying the Lorentz transformation $L$ to $(A, \phi)$, we obtain a global solution $(A^L, \phi^L)$ to the Maxwell-Klein-Gordon system. Next we define the gauge transform 
  \begin{equation*}
   \gamma = \sum_{l=1}^4 \Delta^{-1} \bigl( \partial_l A_l^L \bigr) = \Delta^{-1} \bigl( \partial^l A_l^L \bigr)
  \end{equation*}
  and set 
  \[
   \tilde{\phi}^L = e^{i \gamma} \phi^L, \qquad \tilde{A}^L_\alpha = A^L_\alpha - \partial_\alpha \gamma, \quad \alpha = 0, 1, \ldots, 4.
  \]
  Then $(\tilde{A}^L, \tilde{\phi}^L)$ is in Coulomb gauge and a global solution to MKG-CG. By assumption we have that
  \[
   \bigl\| \bigl( \tilde{A}^L, \tilde{\phi}^L \bigr) \bigr\|_{S^1} < \infty.
  \]
  Now the difficulty in controlling the $S^1$ norm of $(A, \phi)$ is that this norm is far from invariant under the operation of Lorentz transformations. Nonetheless, one can establish control over a certain set of norms of $(A, \phi)$ that are essentially invariant under Lorentz transformations, and which in turn imply control over the full $S^1$ norm of $(A, \phi)$. We do this in the following observations.

  \medskip

  \noindent {\bf Observation 1:} {\it For $C = C(L)$ with $C(L) \to \infty$ as $L \to Id$, i.e. as $d \to 0$, we have the bounds 
  \begin{align}
   \biggl( \sum_{k\in \Z} \big\| \nabla_{x} P_k Q_{[k+\frac{1}{2} C, k+C]^c} \phi^L \big\|_{X^{0, \frac{1}{2}}_{\infty}}^2 \biggr)^{\frac{1}{2}} &\lesssim \big\| \big(\tilde{A}^L, \tilde{\phi}^{L}\big) \big\|_{S^1}, \label{equ:Lorentz_observation1_phi1} \\
   \biggl( \sum_{k\in \Z} 2^{-\nu k} \big\| \nabla_{x} P_k Q_{[k+\frac{1}{2} C, k+C]} \phi^L \big\|_{L_t^2 L_x^{\frac{8}{3}+}}^2 \biggr)^{\frac{1}{2}} &\lesssim \bigl\| \bigl( \tilde{A}^L, \tilde{\phi}^{L} \bigr) \bigr\|_{S^1} \label{equ:Lorentz_observation1_phi2}
  \end{align}
  for some $\nu > 0$. Similarly for $A^L$, we have the bounds 
  \begin{align}
   \biggl( \sum_{k\in \Z} \big\| \nabla_{x} P_k Q_{[k+\frac{1}{2}C, k + C]^c} A^L \big\|_{X^{0, \frac{1}{2}}_{\infty}}^2 \biggr)^{\frac{1}{2}} &\lesssim \big\|\big(\tilde{A}^L, \tilde{\phi}^{L}\big)\big\|_{S^1}, \label{equ:Lorentz_observation1_A1} \\
   \biggl( \sum_{k\in \Z} 2^{-(1+)k} \big\|\nabla_{x} P_k Q_{[k+\frac{1}{2}C, k+C]} A^L \big\|_{L_t^2 L_x^{4+}}^2 \biggr)^{\frac{1}{2}} &\lesssim \big\|\big(\tilde{A}^L, \tilde{\phi}^{L}\big)\big\|_{S^1}. \label{equ:Lorentz_observation1_A2}
  \end{align}
  Moreover, we have the bounds 
  \begin{align}
   \biggl( \sum_{k\in \Z} \big\|\nabla_{x} P_k Q_{\leq k + \frac{1}{4}C} \phi^L \big\|_{L_t^\infty L_x^2}^2 \biggr)^{\frac{1}{2}} \lesssim  \big\|\bigl(\tilde{A}^L, \tilde{\phi}^{L}\bigr)\big\|_{S^1}, \label{equ:Lorentz_observation1_phi3} \\
   \biggl( \sum_{k\in \Z} 2^{-k} \big\|P_k Q_{\leq k + \frac{1}{4}C} \phi^L \big\|_{L^2_t L^\infty_x}^2 \biggr)^{\frac{1}{2}} \lesssim  \big\|\bigl(\tilde{A}^L, \tilde{\phi}^{L}\bigr)\big\|_{S^1}, \label{equ:Lorentz_observation1_phi4}   
  \end{align}
  Here the implicit constants may also depend on $\big\|\big(\tilde{A}^L, \tilde{\phi}^L \big)\big\|_{S^1}$. }
  \begin{proof}[Proof of Observation 1] We first derive suitable estimates on the gauge transform $\gamma$, which will then allow us to obtain the claimed bounds in the statement of Observation 1. To this end we compute $\gamma$ in terms of $\tilde{A}^L$, for which we already have good bounds by assumption. Note that in view of \eqref{equ:Lorentz_transformation_of_A}, we have 
  \[
   \bigl( \tilde{A}^L \bigr)^{L^{-1}} = A - ( \nabla_{t,x} \gamma )^{L^{-1}} = A - \nabla_{t,x} \bigl( \gamma(L^{-1}\cdot) \bigr)
  \]
  and so we get 
  \[
   \gamma(L^{-1} \cdot) = - \Delta^{-1} \partial^l \bigl( \bigl( \tilde{A}^L \bigr)^{L^{-1}} \bigr)_l.
  \]
  Thus, we find 
  \[
   \gamma(t,x) = - \Bigl( \Delta^{-1} \partial^l \bigl( \bigl( \tilde{A}^L \bigr)^{L^{-1}} \bigr)_l \Bigr)^L(t,x) = - \Bigl( \Delta^{-1} \partial^l \bigl( \bigl( \tilde{A}^L \bigr)^{L^{-1}} \bigr)_l \Bigr)(L(t,x)).
  \]
  Now for any fixed dyadic frequency $k \in \Z$, we can write 
  \[
   \Bigl( P_k \Delta^{-1} \partial^l \bigl( \bigl( \tilde{A}^L \bigr)^{L^{-1}} \bigr)_l \Bigr)(t,x) = \int_{\R^4} m_k^l(a) \, \bigl( \bigl( \tilde{A}^L \bigr)^{L^{-1}} \bigr)_l(t, x-a) \, da 
  \]
  for suitable $L^1_x$-functions $m_k^l(a)$ with $L^1_x$-mass $\sim 2^{-k}$, and further
  \[
   \Bigl( P_k \Delta^{-1} \partial^l \bigl( \bigl( \tilde{A}^L \bigr)^{L^{-1}} \bigr)_l \Bigr)^L(t,x) = \int_{\R^4} m_k^l(a) \, \Bigl( \bigl( L^{-1} \bigr)^t \tilde{A}^L \Bigr)_l \bigl( (t,x) - L^{-1}(0,a) \bigr) \, da.
  \]
  Also, if $j \leq k + \frac{1}{2} C$ for suitable $C = C(L)$, then Fourier localization to dyadic modulation $2^j$ and spatial frequency $2^k$ essentially commute with the Lorentz transformation, provided $C$ is not too large depending on $d$, and so we have 
  \[
   \Bigl( P_k Q_j \Delta^{-1} \partial^l \bigl( \bigl( \tilde{A}^L \bigr)^{L^{-1}} \bigr)_l \Bigr)^L(t,x) = P_k Q_j \biggl( \int_{\R^4} m_k^l(a) \, \Bigl( P_{k+O(1)} Q_{j+O(1)} \Bigl( \bigl( L^{-1} \bigr)^t \tilde{A}^L \Bigr)_l \bigl( (t,x) - L^{-1}(0,a) \bigr) \, da \biggr),
  \]
  where we note that the right hand side is a linear combination of all the components $\bigl( \tilde{A}^L \bigr)_\alpha$.
  This immediately implies for $j \leq k + \frac{1}{2} C$ that
  \begin{align*}
   2^{\frac{1}{2}j} \big\| P_k Q_j \nabla_x \gamma \big\|_{L^2_t L^2_x} = 2^{\frac{1}{2} j} \big\|P_k  Q_j \nabla_x \Bigl( \Delta^{-1} \partial^l \bigl( \bigl( \tilde{A}^L \bigr)^{L^{-1}} \bigr)_l \Bigr)^L \big\|_{L^2_t L^2_x} \lesssim \big\|P_k \tilde{A}^L\big\|_{X^{0,\frac{1}{2}}_{\infty}}.
  \end{align*}
  Similarly, one shows that for $j > k+C$ we have 
  \begin{align*}
   2^{\frac{1}{2}j} \big\|P_k Q_j \nabla_x^2 \gamma \big\|_{L^2_t L^2_x} \lesssim 2^{k-j} \big\| P_j \nabla_{t,x} \tilde{A}^L \big\|_{X^{0,\frac{1}{2}}_{\infty}}.
  \end{align*}
  In fact, here, the very large modulation $j$ then gets transferred to the frequency after Lorentz transform. Finally, for the expression $P_{k} Q_{[k+\frac{1}{2}C, k+C]} \gamma$, the Lorentz transformation may lead to small frequencies~$\lesssim 2^k$, which is why we can only place the expression into $L_t^2 L_x^{4+}$ then via Bernstein, i.e. 
  \begin{equation*}
   \biggl( \sum_{k\in \Z} 2^{-(1+)k} \big\| P_k Q_{[k+\frac{1}{2}C, k+C]} \nabla_{t,x} \gamma\big\|_{L_t^2 L_x^{4+}}^2 \biggr)^{\frac{1}{2}} \lesssim \big\|\tilde{A}^L\big\|_{S^1}.
  \end{equation*}
  One also infers by similar reasoning that 
  \begin{equation*}
   \biggl( \sum_{k\in \Z} \big\| P_k \nabla_{t,x} \gamma \big\|^2_{L_t^2 L_x^{8} \cap L_t^2 \dot{W}^{\frac{1}{6}, 6}_x} \biggr)^{\frac{1}{2}} \lesssim \big\|\tilde{A}^L\big\|_{S^1}
  \end{equation*}
  as well as 
  \begin{equation*}
  \big\| P_k Q_{\leq k+\frac{1}{2}C} \nabla_{t,x} \gamma \big\|_{L_t^\infty L_x^2} \lesssim \big\| P_{k+O(1)} \tilde{A}^L \big\|_{L_t^\infty L_x^2}.
  \end{equation*}

  \medskip

  With these bounds on $\gamma$ in hand, we can now start the derivation of the bounds for $\phi^L = e^{-i\gamma}\tilde{\phi}^L$. For $j \leq k+\frac{1}{2} C$ we write
  \begin{align} \label{equ:Lorentz_obs1_phi1}
   \begin{aligned}   
    P_k Q_j \big( e^{-i\gamma}\tilde{\phi}^L\big) &= P_k Q_j \big( P_{\leq j-20} Q_{\leq j-10}(e^{-i\gamma}) \tilde{\phi}^L \big) + P_k Q_j \big( P_{\leq j-20} Q_{> j-10} (e^{-i\gamma}) \tilde{\phi}^L \big) \\
    &\qquad + P_k Q_j \big( P_{> j-20} (e^{-i\gamma}) \tilde{\phi}^L \big).
   \end{aligned}
  \end{align}
  For the first term on the right, we have 
  \[
   P_k Q_j \big( P_{\leq j-20} Q_{\leq j-10}(e^{-i\gamma}) \tilde{\phi}^L \big) = P_k Q_j \big( P_{\leq j-20} Q_{\leq j-10}(e^{-i\gamma}) P_{k+O(1)} Q_{j+O(1)} \tilde{\phi}^L\big),
  \]
  and so we infer 
  \begin{equation} \label{eq:philly1}
   2^{\frac{1}{2} j} \big\| \nabla_{t,x} P_k Q_j \big( P_{\leq j-20} Q_{\leq j-10}(e^{-i\gamma}) \tilde{\phi}^L \big) \big\|_{L^2_t L^2_x} \lesssim \big\| P_k \nabla_{t,x} \tilde{\phi}^L\big\|_{X^{0,\frac{1}{2}}_{\infty}}.
  \end{equation}
  For the second term on the right hand side of \eqref{equ:Lorentz_obs1_phi1}, we write schematically 
  \begin{align*}
   P_k Q_j \big( P_{\leq j-20} Q_{> j-10}(e^{-i\gamma}) \tilde{\phi}^L \big) = 2^{-j} P_k Q_j \big( P_{\leq j-20} Q_{> j-10}(\partial_t \gamma e^{-i\gamma}) \tilde{\phi}^L \big)
  \end{align*}
  and so we get from the preceding 
  \begin{equation} \label{eq:philly2}
   \begin{split}
    &2^{\frac{1}{2}j} \big\| \nabla_{t,x} P_k Q_j \big( P_{\leq j-20} Q_{> j-10}(e^{-i\gamma}) \tilde{\phi}^L \big) \big\|_{L^2_t L^2_x} \\
    &\lesssim 2^{\frac{1}{2}j} \cdot 2^{-\frac{1}{2}j} \big\| P_{\leq j-20} Q_{> j-10}(\partial_t \gamma e^{-i\gamma}) \big\|_{L_t^2 L_x^8} \big\| P_k \nabla_{t,x} \tilde{\phi}^L\big\|_{L_t^\infty L_x^2} \\
    &\lesssim \big\| \tilde{A}^L \big\|_{S^1} \big\|P_k \nabla_{t,x} \tilde{\phi}^L\big\|_{L_t^\infty L_x^2}.
   \end{split}
  \end{equation}
  For the last term on the right hand side of \eqref{equ:Lorentz_obs1_phi1}, write it as 
  \begin{align} \label{equ:Lorentz_obs1_phi1_last_term}
   \begin{aligned}
    P_k Q_j \big( P_{> j-20}(e^{-i\gamma})\tilde{\phi}^L \big) &= P_k Q_j \big( P_{[j-20, k-10]}(e^{-i\gamma}) \tilde{\phi}^L \big) + P_k Q_j \big( P_{[k-10, k+10]}(e^{-i\gamma}) \tilde{\phi}^L \big) \\
    &\quad \quad + P_k Q_j \big( P_{>k+10}(e^{-i\gamma}) \tilde{\phi}^L\big).
   \end{aligned}
  \end{align}
  The first term on the right is bounded by 
  \begin{equation} \label{eq:philly3}
   \begin{aligned}
    &2^{\frac{1}{2} j} \big\|\nabla_x P_k Q_j \big( P_{[j-20, k-10]}(e^{-i\gamma}) \tilde{\phi}^L \big) \big\|_{L^2_t L^2_x} \\
    &\lesssim 2^{\frac{1}{2}j} 2^{\frac{1}{2} k} 2^{-j} \big\| P_{[j-20, k-10]}(\nabla_x \gamma e^{-i\gamma})\big\|_{L_t^2 L_x^8} \big\|\nabla_x P_k \tilde{\phi}^L\big\|_{L_t^\infty L_x^2} \\
    &\lesssim 2^{\frac{1}{2}(k-j)} \big\|\tilde{A}^L\big\|_{S^1} \big\|P_k \nabla_x \tilde{\phi}^L\big\|_{L_t^\infty L_x^2}.
   \end{aligned}
  \end{equation}
  For the second term on the right hand side of \eqref{equ:Lorentz_obs1_phi1_last_term}, we write it schematically as
  \begin{align} \label{equ:Lorentz_obs1_phi_1_last_term_second_term}
   \begin{aligned}
    &P_k Q_j \big( P_{[k-10, k+10]}(e^{-i\gamma}) \tilde{\phi}^L \big) \\ 
    &= 2^{-k} P_k Q_j \big( P_{[k-10, k+10]}(\nabla_x \gamma P_{\leq k-20} Q_{\leq k-20}(e^{-i\gamma})) P_{\leq k+20} \tilde{\phi}^L \big) \\
    &\quad + 2^{-k} P_k Q_j \big( P_{[k-10, k+10]}(\nabla_x \gamma P_{\leq k-20} Q_{> k-20}(e^{-i\gamma})) P_{\leq k+20} \tilde{\phi}^L \big) \\
    &\quad + 2^{-k} P_k Q_j \big( P_{[k-10, k+10]}(\nabla_x \gamma P_{>k-20}(e^{-i\gamma})) P_{\leq k+20} \tilde{\phi}^L \big).
   \end{aligned}
  \end{align}
  Then we estimate the first term on the right of \eqref{equ:Lorentz_obs1_phi_1_last_term_second_term} by
  \begin{equation} \label{eq:philly4}
   \begin{aligned}
    &2^{\frac{1}{2} j} \big\|\nabla_x 2^{-k} P_k Q_j \big( P_{[k-10, k+10]}(\nabla_x \gamma P_{\leq k-20} Q_{\leq k-20}(e^{-i\gamma})) P_{\leq k+20} \tilde{\phi}^L \big) \big\|_{L^2_t L^2_x} \\
    &\lesssim 2^{\frac{1}{2} j} \big\|P_{[k-10, k+10]} Q_{\leq k + \frac{1}{2} C} \nabla_x \gamma \big\|_{L_t^\infty L_x^2} \big\|P_{\leq k+20}Q_{\leq k} \tilde{\phi}^L \big\|_{L_t^2 L_x^\infty} \\
    &\quad + 2^{\frac{1}{2} j} \big\| P_{[k-10, k+10]} \nabla_x \gamma \big\|_{L_t^2 L_x^{4+}} \big\|P_{\leq k+20} Q_{> k} \tilde{\phi}^L \big\|_{L_t^\infty L_x^{4-}} \\
    &\lesssim 2^{\frac{1}{2} (j-k)} \big\|P_{k+O(1)} \nabla_x \tilde{A}^L \big\|_{L_t^\infty L_x^2} \big\| \tilde{\phi}^L \big\|_{S^1}.
   \end{aligned}
  \end{equation}
  Further, we get for the second term on the right of \eqref{equ:Lorentz_obs1_phi_1_last_term_second_term} that 
  \begin{equation} \label{eq:philly5}
   \begin{split}
    &2^{\frac{1}{2} j} \big\|\nabla_x 2^{-k} P_k Q_j \big( P_{[k-10, k+10]}(\nabla_x \gamma P_{\leq k-20} Q_{> k-20}(e^{-i\gamma})) P_{\leq k+20} \tilde{\phi}^L \big) \big\|_{L^2_t L^2_x} \\
    &\lesssim 2^{j-k} \big\| P_{k+O(1)} \nabla_x \gamma \big\|_{L_t^2 L_x^6} \big\|\partial_t \gamma \big\|_{L_t^2 L_x^8} \big\|P_{\leq k+20} \tilde{\phi}^L \big\|_{L_t^\infty L_x^{\frac{24}{5}}}\\
    &\lesssim 2^{j-k} 2^{\frac{1}{6} k} \big\|P_{k+O(1)} \nabla_x \gamma\big\|_{L_t^2 L_x^6} \big\|(\tilde{A}^L, \tilde{\phi}^L)\big\|_{S^1} \big\|\tilde{\phi}^L\big\|_{S^1}.
   \end{split}
  \end{equation}
  The third term on the right hand side of \eqref{equ:Lorentz_obs1_phi_1_last_term_second_term}
  \[
   2^{-k}P_k Q_j \big( P_{[k-10, k+10]} (\nabla_x \gamma P_{> k-20} (e^{-i\gamma})) P_{\leq k+20} \tilde{\phi}^L \big)
  \]
  is handled similarly, which concludes the treatment of the contribution of the second term on the right hand side of \eqref{equ:Lorentz_obs1_phi1_last_term}, namely
  \[
   P_k Q_j \big( P_{[k-10, k+10]}(e^{-i\gamma}) \tilde{\phi}^L \big). 
  \]
  To treat the third term on the right hand side of \eqref{equ:Lorentz_obs1_phi1_last_term}, i.e. the high-high interaction term 
  \[
   P_k Q_j \big( P_{>k+10}(e^{-i\gamma}) \tilde{\phi}^L \big),
  \]
  we write it schematically as
  \begin{align*}
   P_k Q_j \big( P_{>k+10}(e^{-i\gamma}) \tilde{\phi}^L \big) &= \sum_{\substack{k_1 > k+10 \\ k_1 = k_2+O(1)}} P_k Q_j \big( P_{k_1}(e^{-i\gamma}) P_{k_2} \tilde{\phi}^L \big) \\
   &= \sum_{\substack{k_1 > k+10 \\ k_1 = k_2 + O(1)}} 2^{-k_1} P_k Q_j \big( P_{k_1}(\nabla_x \gamma e^{-i\gamma}) P_{k_2} \tilde{\phi}^L \big)
  \end{align*}
  and so we can estimate this by 
  \begin{equation} \label{eq:philly6}
   \begin{split}
    2^{\frac{1}{2} j} \big\| \nabla_x P_k Q_j \big( P_{>k+10}(e^{-i\gamma}) \tilde{\phi}^L \big) \big\|_{L^2_t L^2_x} \lesssim \sum_{\substack{k_1>k+10 \\ k_1 = k_2+O(1)}} 2^{\frac{1}{2}(j+k)} 2^{k-k_1} 2^{k-k_2} \big\| \nabla_x \gamma \big\|_{L_t^2 L_x^8} \big\|P_{k_2} \nabla_x \tilde{\phi}^L\big\|_{L_t^\infty L_x^2}.
   \end{split}
  \end{equation}
  Combining the bounds \eqref{eq:philly1} -- \eqref{eq:philly6} and square-summing over $k$, the estimate  
  \[
   \biggl( \sum_{k\in \Z} \big\|\nabla_{t,x} P_k Q_{\leq k + \frac{1}{2} C} \phi^L \big\|_{X^{0, \frac{1}{2}}_{\infty}}^2 \biggr)^{\frac{1}{2}} \lesssim \big\| (\tilde{A}^L, \tilde{\phi}^{L}) \big\|_{S^1}
  \]
  with implied constant also depending on $\big\|(\tilde{A}^L, \tilde{\phi}^{L})\big\|_{S^1}$ easily follows. We omit the estimate for 
  \[
   \biggl( \sum_{k\in \Z} \big\|\nabla_{t,x} P_k Q_{>k+C} \phi^L \big\|_{X^{0, \frac{1}{2}}_{\infty}}^2 \biggr)^{\frac{1}{2}}
  \]
  as it is similar. This proves the first bound \eqref{equ:Lorentz_observation1_phi1} in the statement of Observation~1.

  \medskip

  Next, we turn to the proof of the second bound \eqref{equ:Lorentz_observation1_phi2} and consider 
  \begin{align} \label{equ:Lorentz_obs1_phi2}
   \begin{aligned}
    P_kQ_{[k + \frac{1}{2} C, k+C]}(e^{-i\gamma} \tilde{\phi}^L) &= P_k Q_{[k+\frac{1}{2}C, k+C]} \big( P_{\leq k-20} Q_{\leq k-20}(e^{-i\gamma}) \tilde{\phi}^L \big) \\ 
    &\quad + P_k Q_{[k+\frac{1}{2} C, k+C]} \big( P_{\leq k-20}Q_{> k-20}(e^{-i\gamma}) \tilde{\phi}^L \big) \\
    &\quad + P_k Q_{[k+\frac{1}{2} C, k+C]} \big( P_{[k-20, k+20]}(e^{-i\gamma})\tilde{\phi}^L \big) \\ 
    &\quad + P_k Q_{[k+\frac{1}{2} C, k+C]} \big( P_{>k+20}(e^{-i\gamma})\tilde{\phi}^L \big).
   \end{aligned}
  \end{align}
  Each of these terms is straightforward to estimate. For the first term on the right, we obtain 
  \begin{align*}
   &\big\|P_kQ_{[k+\frac{1}{2}C, k+C]} \big(P_{\leq k-20} Q_{\leq k-20}(e^{-i\gamma}) \tilde{\phi}^L \big) \big\|_{L_t^2 L_x^{\frac{8}{3}+}} \lesssim \big\|P_{k+O(1)}Q_{k+O(1)}\tilde{\phi}^L\big\|_{L_t^2 L_x^{\frac{8}{3}+}} \lesssim 2^{(\nu-1)k} \big\|P_k\tilde{\phi}^L\big\|_{S^1}.
  \end{align*}
  Also, we get 
  \begin{align*}
   \big\|P_k Q_{[k+\frac{1}{2}C, k+C]} (P_{\leq k-20} Q_{> k-20}(e^{-i\gamma}) \tilde{\phi}^L) \big\|_{L_{t}^2 L_x^{\frac{8}{3}+}} &\lesssim 2^{-k} \big\|\partial_t\gamma\big\|_{L_t^2 L_x^8} \big\|P_{k+O(1)} \tilde{\phi}^L \big\|_{L_t^\infty L_x^{4+}} \\
   &\lesssim  2^{(\nu-1)k} \big\|\partial_t\gamma\big\|_{L_t^2 L_x^8} \big\|P_k\tilde{\phi}^L\big\|_{S^1},
  \end{align*}
  and
  \begin{align*}
   \big\|P_k Q_{[k+\frac{1}{2}C, k+C]} \big( P_{[k-20, k+20]}(e^{-i\gamma})\tilde{\phi}^L \big) \big\|_{L_{t}^2 L_x^{\frac{8}{3}+}} &\lesssim 2^{-k} \big\|\nabla_x \gamma \big\|_{L_t^2 L_x^8} \big\|P_{\leq k+O(1)} \tilde{\phi}^L \big\|_{L_t^\infty L_x^{4+}} \\
   &\lesssim 2^{(\nu-1)k} \big\|\nabla_x \gamma \big\|_{L_t^2 L_x^8} \sum_{l \leq k + O(1)} 2^{\nu (l-k)} \big\|P_l \tilde{\phi}^L\big\|_{S^1}.
  \end{align*}
  The last term on the right hand side of \eqref{equ:Lorentz_obs1_phi2} can be handled similarly. These estimates then yield the second inequality \eqref{equ:Lorentz_observation1_phi2} of Observation~1. 

  \medskip 
 
  We also observe that the estimates on $\gamma$ established earlier yield the required bounds \eqref{equ:Lorentz_observation1_A1} and \eqref{equ:Lorentz_observation1_A2} for $A^L  = \tilde{A}^L + \nabla_{t,x} \gamma$. 

  \medskip

  Now we turn to the last bounds \eqref{equ:Lorentz_observation1_phi3} and \eqref{equ:Lorentz_observation1_phi4} in the statement of Observation~1. We only prove \eqref{equ:Lorentz_observation1_phi3}, the proof of \eqref{equ:Lorentz_observation1_phi4} being similar. We write 
  \begin{align} \label{equ:Lorentz_obs1_phi3}
   \begin{aligned}
    P_kQ_{\leq k+\frac{1}{4} C} \phi^L &= P_k Q_{\leq k+\frac{1}{4}C} \big(P_{\leq k-10}(e^{-i\gamma})\tilde{\phi}^L \big) + P_k Q_{\leq k+\frac{1}{4}C} \big(P_{[k-10, k+10]}(e^{-i\gamma}) \tilde{\phi}^L\big) \\
    &\quad \quad + P_k Q_{\leq k+\frac{1}{4}C} \big(P_{>k+10}(e^{-i\gamma}) \tilde{\phi}^L \big).
   \end{aligned}
  \end{align}
  The first term is directly bounded by 
  \begin{equation} \label{eq:muehsam1}
   \big\|\nabla_x P_k Q_{\leq k+\frac{1}{4}C} \big(P_{\leq k-10}(e^{-i\gamma}) \tilde{\phi}^L \big)\big\|_{L_t^\infty L_x^2} \lesssim \big\|\nabla_x P_{k+O(1)}\tilde{\phi}^L\big\|_{L_t^\infty L_x^2}.
  \end{equation}
  The second term on the right of \eqref{equ:Lorentz_obs1_phi3} is a bit more complicated. We write it schematically as
  \begin{align} \label{equ:Lorentz_obs1_phi3_second_term}
   \begin{aligned}
    &P_k Q_{\leq k+ \frac{1}{4} C} \big(P_{[k-10, k+10]}(e^{-i\gamma})\tilde{\phi}^L\big) \\
    &= P_k Q_{\leq k + \frac{1}{4} C} \big(2^{-k}P_{[k-10, k+10]}(\nabla_x \gamma e^{-i\gamma})P_{\leq k+O(1)}\tilde{\phi}^L\big)\\
    &= P_k Q_{\leq k + \frac{1}{4} C} \big(2^{-k} P_{[k-10, k+10]}(\nabla_x \gamma P_{\leq k-30}Q_{\leq k-30}(e^{-i\gamma}))P_{\leq k-30} \tilde{\phi}^L \big) \\
    &\quad + P_k Q_{\leq k+\frac{1}{4}C} \big( 2^{-k} P_{[k-10, k+10]}(\nabla_x \gamma P_{\leq k-30} Q_{> k-30}(e^{-i\gamma})) P_{\leq k-30} \tilde{\phi}^L \big)\\
    &\quad + P_kQ_{\leq k + \frac{1}{4}C} \big(2^{-k} P_{[k-10,k+10]}(\nabla_x \gamma P_{> k-30}(e^{-i\gamma})) P_{\leq k-30} \tilde{\phi}^L \big)\\
    &\quad + P_k Q_{\leq k + \frac{1}{4}C} \big( 2^{-k}P_{[k-10, k+10]}(\nabla_x \gamma P_{\leq k-30} (e^{-i\gamma})) P_{> k-30} \tilde{\phi}^L \big).
   \end{aligned}
  \end{align}
  Then we get for the first term of the last list of four terms 
  \begin{equation}\label{eq:muehsam2}
   \begin{split}
    &\big\|\nabla_x P_k Q_{\leq k+\frac{1}{4}C} \big(2^{-k} P_{[k-10, k+10]}(\nabla_x \gamma P_{\leq k-30} Q_{\leq k-30}(e^{-i\gamma})) P_{\leq k-30} \tilde{\phi}^L \big) \big\|_{L_t^\infty L_x^2} \\
    &\lesssim \big\| P_{[k-15, k+15]} Q_{\leq k+\frac{1}{4} C + 5} \nabla_x \gamma \big\|_{L_t^\infty L_x^2} \big\|P_{\leq k-30} \tilde{\phi}^L \big\|_{L^\infty_t L^\infty_x} \\
    &\lesssim \big\| P_{k+O(1)} \tilde{A}^L \big\|_{L^\infty_t L^2_x} \sum_{l \leq k-30} 2^l \big\|P_l \tilde{\phi}^L \big\|_{L^\infty_t L^4_x} \\
    &\lesssim \big\| P_{k+O(1)} \nabla_x \tilde{A}^L \big\|_{S^1} \big\| \tilde{\phi}^L \big\|_{S^1},
   \end{split}
  \end{equation}
  where we have taken advantage of our previous considerations on the structure of $\gamma$. For the second term on the above list \eqref{equ:Lorentz_obs1_phi3_second_term}, we get 
  \begin{equation} \label{eq:muehsam3} 
   \begin{aligned}
    &\big\|\nabla_x P_k Q_{\leq k + \frac{1}{4} C} \big(2^{-k} P_{[k-10, k+10]}(\nabla_x \gamma P_{\leq k-30} Q_{> k-30}(e^{-i\gamma})) P_{\leq k-30} \tilde{\phi}^L \big) \big\|_{L_t^\infty L_x^2} \\
    &\lesssim \big\|P_{\leq k+O(1)}\nabla_x \gamma \big\|_{L_t^\infty L_x^{4+}} \big\|P_{\leq k-30} Q_{> k-30}(e^{-i\gamma})\big\|_{L_t^\infty  L_x^{4+}} \big\|P_{\leq k-30} \tilde{\phi}^L\big\|_{L_t^\infty L_x^{\infty-}} \\
    &\lesssim \bigl\|\tilde{A}^L\big\|_{S^1}^2 \sum_{l \leq k-30} 2^{-\sigma (k-l) } \big\|P_l \tilde{\phi}^L \big\|_{S^1}. 
   \end{aligned}
  \end{equation}
  The term 
  \[
   P_k Q_{\leq k + \frac{1}{4} C} \big( 2^{-k}P_{[k-10,k+10]}(\nabla_x \gamma P_{> k-30}(e^{-i\gamma}))P_{\leq k-30} \tilde{\phi}^L \big)
  \]
  is handled similarly. Finally, we have 
  \begin{equation} \label{eq:muehsam4}
   \begin{split}
    &\big\|\nabla_x P_k Q_{\leq k + \frac{1}{4}C} \big( 2^{-k} P_{[k-10, k+10]}(\nabla_x \gamma P_{\leq k-30}(e^{-i\gamma})) P_{> k-30} \tilde{\phi}^L \big)\big\|_{L_t^\infty L_x^2} \\
    &\lesssim \big\|P_{k+O(1)} \nabla_x \gamma \big\|_{L_t^\infty L_x^4} \big\|P_{k+O(1)} \tilde{\phi}^L \big\|_{L_t^\infty L_x^4}.
   \end{split}
  \end{equation}
  The bounds \eqref{eq:muehsam1} -- \eqref{eq:muehsam4} suffice to perform the square summation over $k$ in the last inequality of Observation~1. The last term on the right hand side of \eqref{equ:Lorentz_obs1_phi3}
  \[
   P_kQ_{\leq k+\frac{1}{4}C} \big( P_{>k+10}(e^{-i\gamma}) \tilde{\phi}^L\big)
  \]
  is treated similarly and hence omitted here. 
 \end{proof}

 \medskip

 \noindent {\bf Observation 2:} {\it We have the bound 
 \[
  \sum_{k_1 \leq k_2} 2^{-k_1} \big\| P_{k_1} Q^{\pm}_{\leq k_1 + \frac{1}{4}C} \phi^L \, P_{k_2} Q^{\pm}_{\leq k_1 +\frac{1}{4}C} \nabla_{t,x} \phi^L \big\|_{L^2_t L^2_x}^2 \lesssim \big\|(\tilde{A}^L, \tilde{\phi}^L) \big\|_{S^1}^4.
 \]
 Moreover, for any $L^1$-space-time integrable weight function $m(a)$, $a\in \R^{1+4}$, we have 
 \[
  \sum_{k_1 \leq k_2} 2^{-k_1} \biggl\| \int_{\R^{1+4}} m(a) P_{k_1} Q^{\pm}_{\leq k_1 + \frac{1}{4}C} \phi^L(\cdot - a) \, P_{k_2} Q^{\pm}_{\leq k_1 +\frac{1}{4}C} \nabla_{t,x} \phi^L(\cdot - a) \, da \biggr\|_{L^2_t L^2_x}^2 \lesssim \big\|(\tilde{A}^L, \tilde{\phi}^L)\big\|_{S^1}^4.
 \]
 Similar bounds can be obtained upon replacing one or more factors by $A^L$. We make the crucial observation that these bounds are essentially invariant under mild Lorentz transformations. Thus, we infer similar bounds for $A$ and $\phi$. }
 \begin{proof}[Proof of Observation 2] Here one places the low frequency input
 \[
  P_{k_1} Q_{\leq k_1+\frac{1}{4}C}^{\pm}\phi^L
 \]
 into $L_t^2 L_x^\infty$ and the high frequency input 
 \[
  P_{k_2} Q_{\leq k_1+\frac{1}{4}C}^{\pm} \nabla_{t,x} \phi^L
 \]
 into $L_t^\infty L_x^2$ by using Observation 1. 
 \end{proof}

 Using Observation~1 and Observation~2, we can now move toward controlling the norm $\|(A, \phi)\|_{S^1}$. From above, we know a priori that we control 
 \[
  \biggl( \sum_{k \in \Z} \big\| P_k Q_{[k+\frac{1}{2}C, k+C]^c} \nabla_{t,x} \phi^L \big\|_{X^{0, \frac{1}{2}}_{\infty}}^2 \biggr)^{\frac{1}{2}}
 \]
 as well as norms of the form 
 \[
  \Biggl( \sum_{k_1 \leq k_2} 2^{-k_1} \biggl\| \int_{\R^{1+4}} m(a) P_{k_1} Q^{\pm}_{\leq k_1 + \frac{1}{4}C} \phi^L(\cdot - a) \, P_{k_2} Q^{\pm}_{\leq k_1 +\frac{1}{4}C} \nabla_{t,x} \phi^L(\cdot - a) \, da \biggr\|_{L^2_t L^2_x}^2 \Biggr)^{1/2}.
 \]
 The latter has the crucial divisibility property, i.e. for any $\delta>0$ we can partition the time axis $\R$ into intervals $I_1, I_2,\ldots, I_N$, with $N$ depending on the size of the norm as well as $\delta$ such that we have for each $n = 1, \ldots, N$ that
 \begin{align*}
  \Biggl( \sum_{k_1 \leq k_2} 2^{-k_1} \biggl\| \chi_{I_n}(t) \int_{\R^{1+4}} m(a)  P_{k_1} Q^{\pm}_{\leq k_1 + \frac{1}{4}C} \phi^L(\cdot - a) \, P_{k_2} Q^{\pm}_{\leq k_1 +\frac{1}{4}C} \nabla_{t,x} \phi^L(\cdot - a) \, da \biggr\|_{L^2_t L^2_x}^2 \Biggr)^{1/2} \leq \delta.
 \end{align*}
 Of course, we get a similar statement for weakened versions of the former norm such as
 \[
  \biggl( \sum_{k\in \Z} \big\| P_k Q_{[k-\frac{1}{2}C, k+\frac{1}{2}C]} \nabla_{t,x} \phi^L \big\|_{X^{0, \frac{1}{2}}_{\infty}}^2 \biggr)^{\frac{1}{2}}.
 \]
 In order to infer the desired $S^1$ norm bound on $(A, \phi)$, we shall partition the time axis $\R$ into finitely many intervals $I_1, \ldots, I_N$, whose number depends on 
 \[
  \big\| \big( \tilde{A}^L, \tilde{\phi}^L \big) \big\|_{S^1}
 \]
 and such that on each of these $I_n$, we can infer via a direct bootstrap argument a bound on 
 \[
  \|(A, \phi)\|_{S^1(I_n\times \R^4)}.
 \]
 This will then suffice to obtain the desired bound on $\|( A, \phi)\|_{S^1(\R \times \R^4)}$. We do this in two steps, which we outline below.

 \medskip

 \noindent {\bf Step 1:} {\it Given $\delta_1>0$ and $\delta_2>0$, using the known a priori bound on $\bigl\| \bigl(\tilde{A}^L, \tilde{\phi}^L\bigr) \bigr\|_{S^1}$ and choosing the intervals $I_n$ suitably as above (whose number will depend on $\delta_1, \delta_2$, as  well as the assumed bound on $\big\|\bigl(\tilde{A}^L, \tilde{\phi}^L \bigr)\big\|_{S^1})$, we infer from the equation for $A$, upon writing
 \[
  A|_{I_n} = A^{free, (I_n)} + A^{nonlin, (I_n)},
 \]
 that there exists a decomposition 
 \[
  A^{nonlin, (I_n)} =  A^{1nonlin, (I_n)} +  A^{2nonlin, (I_n)},
 \]
 where we have schematically
 \[
  A^{2nonlin} = \sum_k \Box^{-1} P_k Q_{k+O(1)} (P_{k+O(1)} Q_{<k-C_1}\phi\nabla_xP_{k+O(1)}Q_{<k-C_1}\phi),
 \]
 while we also have the bound 
 \[
  \big\| A^{1nonlin} \big\|_{\ell^1 S^1(I_n \times \R^4)} \leq \delta_1 + \delta_2 \Bigl( \|(A, \phi)\|_{S^1(I_n\times \R^4)} + \|(A, \phi)\|_{S^1(I_n\times \R^4)}^3 \Bigr)
 \]
 for all $n = 1, 2, \ldots, N$. Moreover, it holds that
 \[
  \big\| A^{2nonlin} \big\|_{S^1(I_n \times \R^4)} \leq \delta_1
 \]
 for all $n = 1, 2, \ldots, N$.
 }
 
 \medskip

 The idea behind this bound is to insert it into the equation for $\phi$ and pick $\delta_1$, $\delta_2$ small depending on $E_{crit}$. 

 \begin{proof}[Proof of Step 1]
 We proceed as in the proof of Lemma~\ref{lem:nonlinear_structure_small_short_interval} and write the source term in the equation for $A$ in the schematic form 
 \[
  \Box A_i = \Delta^{-1} \nabla^r \mathcal{N}_{ir}\bigl(\phi, \overline{\phi} \bigr) + A |\phi|^2.
 \]
 Localizing this to frequency $k = 0$, we write the right hand side in the form 
 \begin{align*}
  P_0 \Big( \Delta^{-1} \nabla^r \mathcal{N}_{ir} \bigl( \phi, \overline{\phi} \bigr) + A |\phi|^2 \Big) = P_0 \bigg( \sum_{k_1, k_2} \Delta^{-1} \nabla^r \mathcal{N}_{ir} \bigl( P_{k_1} \phi, P_{k_2} \overline{\phi} \bigr) + \sum_{k_1, k_2, k_3} P_{k_1}A P_{k_2}\phi P_{k_3} \overline{\phi} \bigg).
 \end{align*}
 We first deal with the quadratic null form term and reduce this to moderate frequencies by observing that for $C = C(\delta_2)$ large enough, we obtain (for a suitable absolute constant $\sigma$ independent of all other constants)
 \[
  \biggl\| \sum_{|k_1|>C, k_2} P_0 \Big( \Delta^{-1}\nabla^r \mathcal{N}_{ir}(P_{k_1}\phi, P_{k_2}\overline{\phi}) \Big) \biggr\|_{N_0} \leq \delta_2 \sum_{k_1>C} 2^{-\sigma |k_1|} \big\|P_{k_1} \phi\big\|_{S^1}^2.
 \]
 Generalizing to arbitrary output frequencies, one easily gets from here the bound 
 \begin{align*}
  \sum_{k} \bigg\| \sum_{|k_1-k|>C, k_2} P_k \Big( \Delta^{-1} \nabla^r \mathcal{N}_{ir} \big( P_{k_1} \phi, P_{k_2} \overline{\phi} \big) \Big) \biggr\|_{N_k} \lesssim \delta_2 \big\|\phi\big\|_{S^1}^2.
 \end{align*}
 Next, we pick $C_1 = C_1(E_{crit})$ such that 
 \[
  \bigg\|P_0 \Big( \sum_{|k_1|, |k_2| < C} \Delta^{-1} \nabla^r \mathcal{N}_{ir}\bigl( P_{k_1} Q_{>C_1} \phi, P_{k_2} \overline{\phi}) \Big) \biggr\|_{N_0} \leq \delta_2 \sum_{|k_1|, |k_2|< C} \big\|P_{k_1}\phi\big\|_{S_{k_1}} \big\|P_{k_2} \phi \big\|_{S_{k_2}},
 \]
 and generalizing to general output frequencies, we then reduce to 
 \[
  \sum_k P_k \bigg( \sum_{|k_{1,2}-k|<C} \Delta^{-1} \nabla^r \mathcal{N}_{ir} \big( P_{k_1} Q_{\leq k_1 + C_1} \phi, P_{k_2}Q_{\leq k_2+C_1} \overline{\phi} \big) \bigg).
 \]
 Depending on our choice of $C_1$, we may assume the Lorentz transform $L$ to be chosen sufficiently close to the identity, i.e. $|d|$ sufficiently small, such that according to Observation~1 we have 
 \[
  \bigg( \sum_{k_1} \big\|P_{k_1} Q_{[k_1-C_1, k_1+C_1]} \nabla_{t,x} \phi \big\|_{X^{0,\frac{1}{2}}_{\infty}}^2 \biggr)^{\frac{1}{2}} \lesssim \big\|(\tilde{A}^L, \tilde{\phi}^L)\big\|_{S^1}.
 \]
 As observed before, this norm has the divisibility property, so that restricting to suitable time intervals $I_n$, $n = 1, \ldots, N$, which form a partition of the time axis $\R$, we may assume 
 \[
  \biggl( \sum_{k_1} \big\|P_{k_1} Q_{[k_1-C_1, k_1+C_1]} \nabla_{t,x} \phi \big\|_{X^{0,\frac{1}{2}}_{\infty}(I_n\times \R^4)}^2 \biggr)^{\frac{1}{2}} \leq \delta_2 
 \]
 for all $n = 1, \ldots, N$. But then we easily infer the bound 
 \begin{align*}
  &\bigg\| P_0 \bigg( \sum_{|k_1|, |k_2|<C} \Delta^{-1} \nabla^r \mathcal{N}_{ir}\bigl( P_{k_1} Q_{[k_1-C_1, k_1+C_1]} \phi, P_{k_2} \overline{\phi}\big) \bigg) \bigg\|_{N_0(I_n \times \R^4)} \\
  &\lesssim \big\|P_{k_1}Q_{[k_1-C_1, k_1+C_1]}\phi\big\|_{L^2_t L^2_x(I_n \times \R^4)} \big\|P_{k_2} \phi \big\|_{L_t^2 L_x^\infty} \\
  &\leq \delta_2 \big\|P_{k_2}\phi\big\|_{S_{k_2}},
 \end{align*}
 and this suffices again, after generalizing this to arbitrary output frequency. In fact, we get 
 \begin{align*}
  &\sum_k \, \biggl\| P_k \bigg( \sum_{|k_{1,2}-k|<C} \Delta^{-1} \nabla^r \mathcal{N}_{ir}\big( P_{k_1} Q_{[k_1-C_1, k_1+C_1]} \phi, P_{k_2} \overline{\phi} \big) \bigg) \biggr\|_{N_k(I_n \times \R^4)} \\
  &\lesssim \bigg( \sum_{k_1} \big\|P_{k_1} Q_{[k_1-C_1, k_1+C_1]} \nabla_{t,x} \phi \big\|_{X^{0,\frac{1}{2}}_{\infty}(I_n\times \R^4)}^2 \bigg)^{\frac{1}{2}} \big\|\phi\big\|_{S^1(I_n \times \R^4)} \\
  &\leq \delta_2 \big\|\phi\big\|_{S^1(I_n \times \R^4)}.
 \end{align*}
 We have now reduced to the expression 
 \[
  \sum_k P_k \bigg( \sum_{|k_{1,2}-k|<C} \Delta^{-1} \nabla^r \mathcal{N}_{ir}\big( P_{k_1}Q_{\leq k_1-C_1} \phi, P_{k_2} Q_{\leq k_2-C_1} \overline{\phi}\big) \bigg).
 \]
 The last reduction here consists in removing extremely small angular separation between the inputs 
 \[
  P_{k_1} Q_{\leq k_1-C_1} \phi, \quad P_{k_2} Q_{\leq k_2-C_1} \phi.
 \]
 Thus, there exists $C_2 = C_2(\delta_2)$ such that we have 
 \begin{align*}
  \bigg\| P_0 \bigg( \sum_{|k_{1,2}|<C} \Delta^{-1} \nabla^r \mathcal{N}_{ir} \big( P_{k_1} Q_{\leq k_1-C_1} \phi, P_{k_2} Q_{\leq k_2 - C_1} \overline{\phi} \big) \bigg)' \bigg\|_{N_0} \leq \delta_2 \sum_{|k_{1,2}|<C} \big\|P_{k_1}\phi\big\|_{S_{k_1}} \big\|P_{k_2} \phi\big\|_{S_{k_2}},
 \end{align*}
 where the prime indicates that the inputs are reduced to have closely aligned Fourier supports of angular separation $C_2^{-1}$. Finally, we write
 \begin{align*}
  &\sum_{|k_{1,2}|<C} P_0 \bigg( \Delta^{-1}\nabla^r \mathcal{N}_{ir} \big( P_{k_1} Q_{\leq k_1-C_1} \phi, P_{k_2} Q_{\leq k_2-C_1} \overline{\phi} \big) \bigg) \\
  &= \sum_{|k_{1,2}|<C} P_0 \bigg( \Delta^{-1} \nabla^r \mathcal{N}_{ir}\big( P_{k_1}Q_{\leq k_1-C_1} \phi, P_{k_2} Q_{\leq k_2-C_1} \overline{\phi} \big) \bigg)'\\
  &\quad + \sum_{|k_{1,2}|<C} P_0 \bigg( \Delta^{-1}\nabla^r \mathcal{N}_{ir} \big( P_{k_1} Q_{\leq k_1-C_1} \phi, P_{k_2} Q_{\leq k_2-C_1} \overline{\phi} \big) \biggr)'',
 \end{align*}
 where the second term is of the form $A^{2nonlin}$ as required for Step 1. In fact, the angular separation of the inputs and small modulation forces the output to have modulation $\sim 1$. Moreover, replacing the output frequency by $k$ and square-summing over $k$ results in a small norm due to the fact that 
 \begin{align*}
  &\bigg\| \sum_{|k_{1,2}|<C} P_0 \bigg(\Delta^{-1} \nabla^r \mathcal{N}_{ir}\big( P_{k_1} Q_{\leq k_1-C_1} \phi, P_{k_2} Q_{\leq k_2-C_1} \overline{\phi}\big) \bigg)''\bigg\|_{N_0(I_n\times \R^4)} \\
  &\lesssim \bigg\| \sum_{|k_{1,2}|<C} P_0 \bigg( \Delta^{-1} \nabla^r \mathcal{N}_{ir}\big( P_{k_1} Q_{\leq k_1-C_1} \phi, P_{k_2} Q_{\leq k_2-C_1} \overline{\phi}) \bigg)'' \bigg\|_{L^2_t L^2_x(I_n \times \R^4)}
 \end{align*}
 and we can then take advantage of Observation~2 to obtain 
 \begin{align*}
  &\bigg( \sum_k \bigg\| \sum_{|k_{1,2}-k|<C} P_k \bigg( \Delta^{-1}\nabla^r \mathcal{N}_{ir}\big( P_{k_1} Q_{\leq k_1-C_1} \phi, P_{k_2} Q_{\leq k_2-C_1} \overline{\phi} \big) \bigg)'' \bigg\|_{N_k(I_n \times \R^4)}^2 \bigg)^{\frac{1}{2}} \leq \delta_1
 \end{align*}
 by choosing the intervals $I_n$ suitably. The cubic term $ \sum_{k_{1,2,3}} P_{k_1}A P_{k_2}\phi P_{k_3}\overline{\phi}$ is handled similarly. 
 \end{proof}

 \medskip

 \noindent {\bf Step 2:} {\it Choosing the time intervals $I_n$ suitably as in Step 1, we obtain the equation 
 \[
  \sum_{k\in \Z} \Box_{A^{free}_{<k}}P_k \phi = F,
 \]
 where we have 
 \[
  \|F\|_{N(I_n\times \R^4)} \leq \delta_1 + \delta_2 \big( \|(A, \phi)\|_{S^1(I_n \times \R^4)} + \big\|(A, \phi)\big\|_{S^1(I_n \times \R^4)}^4 \big),
 \]
 where $A^{free}$ is the free wave evolution of the data for $A$ at the beginning endpoint of $I_n$.}
 \begin{proof}[Proof of Step 2]
  This to a large extent mimics the argument for the proof of Proposition~\ref{prop:decomposition}. In fact, we recall from there that we can write $F = \sum_{k\in\Z}F_k$ with 
  \begin{align*}
   F_k &= -2i P_k \big(A^{free}_{\geq k, \nu}\partial^{\nu}\phi\big) - [P_k, \Box_{A^{free}_{<k}}]\phi - P_k \big( (\Box_A - \Box_{A^{free}}) \phi \big) \\
   &\quad \quad + P_k \big( \Box_{A^{free}_{<k}}\phi + 2i\big(A^{free}_{\geq k, \nu}\partial^{\nu}\phi\big) - \Box_{A^{free}}\phi\big).
  \end{align*}
  As usual, we treat each term separately.

  \medskip

  \noindent {\it First term}. Similar to the proof of Proposition~\ref{prop:decomposition}, we reduce it to 
  \[
   -2i P_k \big(P_{k+O(1)} A^{free}_{, \nu} \partial^{\nu} P_{k+O(1)} \phi \big)
  \]
  up to terms satisfying the conclusion of Step 2. Then using divisibility for the norm 
  \[
   \sum_{k\in \Z}2^{-k}\big\|P_kA^{free}_{, \nu}\big\|_{L_t^2 L_x^\infty}^2
  \]
  as well as the inequality 
  \begin{align*}
   &\bigg\| \sum_{k\in \Z} -2i P_k \Big( P_{k+O(1)}A^{free}_{, \nu}\partial^{\nu}P_{k+O(1)}\phi \Big) \bigg\|_{N(I_n\times \R^4)} \lesssim \bigg( \sum_{k\in \Z} 2^{-k} \big\|\chi_{I_n} P_k A^{free}_{, \nu} \big\|_{L_t^2 L_x^\infty}^2 \bigg)^{\frac{1}{2}} \|\phi\|_{S^1(I_n\times \R^4)},
  \end{align*}
  we get the conclusion of Step 2 by choosing the intervals $I_n$ suitably and by subdividing the intervals obtained from Step 1, if necessary. 

  \medskip

  \noindent {\it{Second term}}. This is handled like the first term, since it can be written in the form 
  \[
   \sum_{k}\tilde{P}_k\big( 2i\nabla_xA^{free}_{<k, \nu}\partial^{\nu}\phi\big) +  \tilde{P}_k\big(\nabla_x((A^{free}_{<k})^2)\phi\big).
  \]

  \medskip

  \noindent {\it{Third term}}. As usual this term is the  most difficult one, since it contains 
  \[
   \sum_{k\in \Z} 2i P_{<k}A_{\nu}^{nonlin} \partial^{\nu} P_k \phi. 
  \]
  We essentially follow the reductions performed in the proof of Proposition~\ref{prop:decomposition}, whence we shall be correspondingly brief.

  \medskip

  \noindent {\bf{Reduction to $\mathcal{H}^\ast {\mathcal N}^{lowhi}$}}. 
  Using the same notation as in that proof and restricting to frequency $k = 0$, and also keeping in mind Step 1, we get 
  \begin{align*}
   &\big\|{\mathcal N}^{low hi} \big( P_{<0} A^{1nonlin, (I_n)}, P_0 \phi \big) - \mathcal{H}^\ast {\mathcal N}^{low hi} \big( P_{<0} A^{1nonlin, (I_n)}, P_0 \phi \big) \big\|_{N(I_n\times \R^4)} \\
   &\qquad \qquad \qquad \qquad \qquad \qquad \qquad \qquad \qquad \qquad \lesssim \big\| A^{1nonlin, (I_n)} \big\|_{S^1(I_n \times \R^4)} \big\| P_0 \phi \big\|_{S^1(I_n\times \R^4)}.
  \end{align*}
  Hence, replacing the output frequency by general $k \in \Z$ and square-summing gives the bound 
  \[
   \lesssim \|(A, \phi)\|_{S^1(I_n \times \R^4)} \Big( \delta_1 + \delta_2 \big( \|(A, \phi)\|_{S^1(I_n\times \R^4)} + \|(A, \phi)\|_{S^1(I_n\times \R^4)}^3 \big) \Big),
  \]
  which is of the desired form. This then reduces the estimate of
  \[
   {\mathcal N}^{low hi}\big( P_{<0} A^{nonlin, (I_n)}, P_0 \phi \big) -  \mathcal{H}^\ast {\mathcal N}^{low hi} \big( P_{<0} A^{nonlin, (I_n)}, P_0 \phi \big) 
  \]
  to the contribution of $P_{<0} A^{2nonlin, (I_n)}$, whose explicit form we recall from Step 1. This means we have to estimate the expression 
  \[
   \big(1-\mathcal{H}^*\big) \bigg( \sum_{k<0} \Box^{-1} P_k Q_{k+O(1)} \big( P_{k+O(1)} Q_{\leq k-C_1} \phi \, \nabla_x P_{k+O(1)} Q_{\leq k-C_1} \phi \big) \nabla_{t,x} P_0 \phi \bigg).
  \]
  The idea here is to use the a priori bounds from Observation~1 and Observation~2 to arrive at the required estimate. For this, we split the above expression into the following 
  \begin{align*}
   &\big(1-\mathcal{H}^*\big) \bigg( \sum_{k<0} \Box^{-1} P_k Q_{k+O(1)} \big( P_{k+O(1)} Q_{\leq k-C_1} \phi \, \nabla_x P_{k+O(1)} Q_{\leq k-C_1}\phi \big) \nabla_{t,x} P_0 \phi \bigg) \\
   &= \big(1-\mathcal{H}^*\big) \bigg( \sum_{k<0} \Box^{-1} P_k Q_{k+O(1)} \big( P_{k+O(1)}Q_{\leq k-C_1} \phi \, \nabla_x P_{k+O(1)} Q_{\leq k-C_1} \phi \big) Q_{[\frac{1}{2}C, C]} \nabla_{t,x} P_0 \phi \bigg) \\
   &\quad + \big(1-\mathcal{H}^*\big) \bigg( \sum_{k<0} \Box^{-1} P_k Q_{k+O(1)} \big( P_{k+O(1)} Q_{\leq k-C_1} \phi \, \nabla_x P_{k+O(1)} Q_{\leq k-C_1} \phi \big) Q_{[k-C_1, \frac{1}{2}C]} \nabla_{t,x} P_0 \phi \bigg) \\
   &\quad + \big(1-\mathcal{H}^*\big) \bigg( \sum_{k<0} \Box^{-1} P_k Q_{k+O(1)} \big( P_{k+O(1)} Q_{\leq k-C_1} \phi \, \nabla_x P_{k+O(1)} Q_{\leq k-C_1} \phi \big) Q_{\leq k-C_1} \nabla_{t,x} P_0 \phi \bigg) \\
   &\quad + \big(1-\mathcal{H}^*\big) \bigg( \sum_{k<0} \Box^{-1}P_k Q_{k+O(1)} \big( P_{k+O(1)}Q_{\leq k-C_1}\phi \, \nabla_x P_{k+O(1)}Q_{\leq k-C_1} \phi \big) Q_{>C}\nabla_{t,x} P_0 \phi \bigg) \\
   &\equiv I + II + III + IV.
  \end{align*}
  We now estimate each of the terms on the right in turn. 

  \medskip

  \noindent {\it Estimate of term $I$}. We distinguish between very small $k$ and $k= O(1)$. In the latter case, we schematically estimate the term in the following fashion. We shall suppress the distinction between space-time translates of $\phi$ and $\phi$, as our norms are invariant under these, and also keep in mind that the operator 
  \[
   \Box^{-1} P_k Q_{k+O(1)}
  \]
  is given by (space-time) convolution with a kernel of $L^1$-mass $\sim 2^{-2k}$. Then we get in case $k = O(1)$,
  \begin{align*}
   &\bigg\| \big(1-\mathcal{H}^*\big) \bigg( \sum_{k<0} \Box^{-1} P_k Q_{k+O(1)} \big( P_{k+O(1)}Q_{\leq k-C_1} \phi \, \nabla_x P_{k+O(1)} Q_{\leq k-C_1}\phi \big) Q_{[\frac{1}{2}C, C]} \nabla_{t,x}P_0 \phi \bigg) \bigg\|_{N(I_n\times \R^4)} \\
   &\leq \bigg\| \big(1-\mathcal{H}^*\big) \bigg( \sum_{k<0} \Box^{-1} P_k Q_{k+O(1)} \big( P_{k+O(1)} Q_{\leq k-C_1} \phi \, \nabla_x P_{k+O(1)} Q_{\leq k-C_1} \phi \big) Q_{[\frac{1}{2}C, C]} \nabla_{t,x} P_0 \phi \bigg) \bigg\|_{L_t^1 L_x^2(I_n \times \R^4)} \\
   &\lesssim 2^{-2k} \big\|P_{k+O(1)} Q_{\leq k-C_1} \phi \big\|_{L_t^2 L_x^\infty} \big\|\nabla_x P_{k+O(1)} Q_{\leq k-C_1} \phi \, Q_{[\frac{1}{2} C, C]} \nabla_{t,x} P_0 \phi \big\|_{L^2_t L^2_x}.
  \end{align*}
  Here the second factor is essentially invariant under mild Lorentz transformations, and so we get (up to changing the meaning of the constants slightly)
  \begin{align*}
   &\big\| \nabla_x P_{k+O(1)} Q_{\leq k-C_1} \phi \, Q_{[\frac{1}{2} C, C]} \nabla_{t,x} P_0 \phi \big\|_{L^2_t L^2_x} \lesssim \big\|\nabla_x P_{k+O(1)} Q_{\leq k-C_1} \phi^L \, Q_{[\frac{1}{2}C, C]} \nabla_{t,x} P_{\leq O(1)} \phi^L \big\|_{L^2_t L^2_x}.
  \end{align*}
  We estimate the last norm using Observation 1, resulting in the bound 
  \begin{align*}
   \big\|\nabla_x P_{k+O(1)} Q_{\leq k-C_1} \phi \, Q_{[\frac{1}{2} C, C]} \nabla_{t,x} P_0 \phi \big\|_{L^2_t L^2_x} &\lesssim \big\|\nabla_x P_{k+O(1)} Q_{\leq k-C_1} \phi^L \big\|_{L_t^\infty L_x^{8-}} \big\| Q_{[\frac{1}{2} C, C]} \nabla_{t,x} P_{\leq O(1)} \phi^L \big\|_{L_t^2 L_x^{\frac{8}{3}+}} \\
   &\leq C \Big( \big\| \big( \tilde{A}^L, \tilde{\phi}^L \big) \big\|_{S^1} \Big).
  \end{align*}
  Then by divisibility of the norm 
  \begin{align*}
   \bigg( \sum_{l \in \Z} 2^{-l} \Big( \sum_{k-l = O(1)} \big\|\nabla_x P_{k+O(1)} Q_{\leq k-C_1} \phi \, Q_{[l-C,l+ C]} \nabla_{t,x} P_l \phi \big\|_{L^2_t L^2_x} \Big)^2 \bigg)^{\frac{1}{2}},
  \end{align*}
  we arrive upon suitable choice of the intervals $I_n$ at the conclusion that 
  \begin{align*}
   &\bigg\| \sum_{l \in \Z} \big(1-\mathcal{H}^\ast \big) \bigg( \sum_{|k-l| = O(1)} \Box^{-1} P_k Q_{k+O(1)} \big( P_{k+O(1)}Q_{\leq k-C_1} \phi \\&\hspace{4cm} \cdot \nabla_x P_{k+O(1)} Q_{\leq k-C_1} \phi \big) Q_{[l+\frac{1}{2}C, l+C]} \nabla_{t,x} P_l \phi \bigg) \bigg\|_{N(I_n\times \R^4)} \\
   &\leq \delta_2 \|\phi\|_{S^1(I_n\times \R^4)}.
  \end{align*}
  This completes the contribution of the term $I$ when $k = O(1)$. On the other hand, when $k \ll -1$, the smallness gain comes directly from $k$. Indeed, we can then estimate 
  \begin{align*}
   &\bigg\| \big(1-\mathcal{H}^\ast\big) \bigg( \sum_{k \ll -1} \Box^{-1} P_k Q_{k+O(1)} \big( P_{k+O(1)}Q_{\leq k-C_1} \phi \, \nabla_x P_{k+O(1)} Q_{\leq k-C_1} \phi \big) Q_{[\frac{1}{2}C, C]} \nabla_{t,x} P_0 \phi \bigg) \bigg\|_{N(I_n\times \R^4)} \\
   &\lesssim \sum_{k \ll -1} 2^{-2k} \big\|P_{k+O(1)} Q_{\leq k-C_1} \phi \big\|_{L^\infty_t L^\infty_x(I_n \times \R^4)} \big\|\nabla_x P_{k+O(1)} Q_{\leq k-C_1} \phi \big\|_{L_t^2 L_x^\infty(I_n \times \R^4)} \big\|Q_{[\frac{1}{2}C, C]} \nabla_{t,x} P_0 \phi \big\|_{L^2_t L^2_x(I_n\times \R^4)} \\
   &\lesssim \sum_{k\ll -1} 2^{\frac{1}{2} k} \big\|P_0 \phi \big\|_{S^1(I_n\times \R^4)} \|\phi\|_{S^1(I_n\times \R^4)}^2 \\
   &\leq \delta_2 \big\|P_0 \phi\big\|_{S^1(I_n\times \R^4)}.
  \end{align*}
  Replacing $P_0 \phi$ by $P_l \phi$, $l \in \Z$, and square-summing over $l$ results in the desired bound. This completes the estimate for term $I$.

  \medskip

  \noindent {\it{Estimate of term $II$}}. Here we use the bound 
  \begin{align*}
   &\bigg\| \big(1-\mathcal{H}^*\big) \bigg( \sum_{k<0} \Box^{-1}P_k Q_{k+O(1)} \big( P_{k+O(1)} Q_{\leq k-C_1} \phi \, \nabla_x P_{k+O(1)} Q_{\leq k-C_1} \phi \big) Q_{[k-2C, \frac{1}{2}C]} \nabla_{t,x} P_0 \phi \bigg) \bigg\|_{N(I_n \times \R^4)} \\
   &\lesssim \sum_{\substack{k<0 \\ l \in [k-2C,\frac{1}{2}C]}} 2^{-2k} \big\|P_{k+O(1)} Q_{\leq k-C_1} \phi \big\|_{L_t^2 L_x^\infty(I_n\times \R^4)} \big\|\nabla_x P_{k+O(1)} Q_{\leq k-C_1} \phi \big\|_{L^\infty_t L^\infty_x(I_n \times \R^4)} \big\| Q_l \nabla_{t,x} P_0 \phi \big\|_{L^2_t L^2_x(I_n\times \R^4)}.
  \end{align*}
  Now if we further restrict the above term to $|k-l| \gg 1$, we easily bound it by 
  \begin{align*}
   &\leq \delta_2 \big\|\phi\big\|_{S^1(I_n\times \R^4)}^2 \big\|P_0 \phi\big\|_{X^{0,\frac{1}{2}}_{\infty}(I_n\times \R^4)} \leq \delta_2 \|\phi \|_{S^1(I_n \times \R^4)}^3,
  \end{align*}
  which is as desired. On the other hand, when restricting the modulation of $Q_{[k-2C, \frac{1}{2} C]} \nabla_{t,x} P_0 \phi$ to $l = k+O(1)$, we use the fact that for $k \ll -1$,
  \begin{align*}
   \big\| \nabla_x P_{k+O(1)} Q_{\leq k-C_1} \phi Q_{k+O(1)} \nabla_{t,x} P_0 \phi \big\|_{L^2_t L^2_x} &\lesssim \big\|\nabla_x P_{k+O(1)} Q_{\leq k-C_1} \phi^L Q_{k+O(1)} \nabla_{t,x} P_0 \phi^L \big\|_{L^2_t L^2_x} \\
   &\lesssim \big\| P_{k+O(1)}Q_{\leq k-C_1} \nabla_x \phi^L \big\|_{L^\infty_t L^\infty_x} \big\|Q_{k+O(1)}\nabla_{t,x}P_0 \phi^L \big\|_{L^2_t L^2_x}.
  \end{align*}
  Then changing the frequency $0$ to general $m \in \Z$ and using Observation 1, we infer 
  \begin{align*}
   \sum_{k<m-C} 2^{-3k} \big\|\nabla_x P_{k+O(1)} Q_{\leq k-C_1} \phi \, Q_{k+O(1)} \nabla_{t,x} P_m \phi \big\|_{L^2_t L^2_x}^2 \leq C \Big(\big\|\big(\tilde{A}^L, \tilde{\phi}^L \big)\big\|_{S^1} \Big). 
  \end{align*}
  Also, the square-sum norm on the left has the divisibility property, whence by restricting to suitable time intervals $I_n$, we may arrange it to be $\ll \delta_2$. Finally, we infer the bound
  \begin{align*}
   &\sum_m \bigg\| \big(1-\mathcal{H}^*\big) \bigg( \sum_{k<m-C} \Box^{-1}P_k Q_{k+O(1)} \big( P_{k+O(1)}Q_{\leq k-C_1} \phi \nabla_x P_{k+O(1)} Q_{\leq k-C_1} \phi \big) Q_{k+O(1)} \nabla_{t,x} P_m \phi \bigg) \bigg\|_{N(I_n\times \R^4)}^2 \\
   &\lesssim \bigg( \sum_k 2^{-k} \big\|P_{k+O(1)} Q_{\leq k-C_1} \phi \big\|_{L_t^2 L_x^\infty}^2 \bigg)^{\frac{1}{2}} \bigg( \sum_{k < m-C} 2^{-3k} \big\|\nabla_x P_{k+O(1)} Q_{\leq k-C_1} \phi \, Q_{k+O(1)} \nabla_{t,x} P_m \phi \big\|_{L^2_t L^2_x}^2 \bigg)^{\frac{1}{2}} \\
   &\leq \delta_2 \|\phi\|_{S^1(I_n\times \R^4)}.
  \end{align*}

  \medskip

  \noindent {\it{Estimate of term $III$}}. This follows the same pattern as for term $I$, by placing the product 
  \[
   \nabla_x P_{k+O(1)} Q_{\leq k-C_1} \phi \, Q_{\leq k-2C} \nabla_{t,x} P_0 \phi
  \]
  into $L^2_t L^2_x$ and using Observation 2. 

  \medskip

  \noindent {\it{Estimate of term $IV$}}. Here one places 
  \[
   \Box^{-1} P_k Q_{k+O(1)} \big( P_{k+O(1)} Q_{\leq k-C_1} \phi \, \nabla_x P_{k+O(1)} Q_{\leq k-C_1} \phi \big)
  \]
  into $L_t^2 L_x^\infty$ and 
  \[
   Q_{>C} \nabla_{t,x} P_0 \phi
  \]
  into $L^2_t L^2_x$, keeping in mind that $C \gg C_1 = C_1(E_{crit})$ is very large. 

  \medskip

  \noindent {\bf Reduction to $\mathcal{H}^\ast {\mathcal N}^{lowhi} \big( \mathcal{H} P_{<0} A^{nonlin, (I_n)} , P_0 \phi \big)$}. To begin with, recall the notation from the proof of Proposition \ref{prop:decomposition} for the definition of the symbol $\mathcal{H}$ applied to bilinear expressions. To reduce to this term, we need to estimate the difference 
  \begin{align*}
   \big\|\mathcal{H}^* {\mathcal N}^{lowhi} \big( P_{<0} A^{nonlin, (I_n)}, P_0 \phi \big) - \mathcal{H}^* {\mathcal N}^{lowhi} \big( \mathcal{H} P_{<0} A^{nonlin, (I_n)}, P_0 \phi \big) \big\|_{N(I_n\times \R^4)}.
  \end{align*}
  Here we recall that 
  \[
   \mathcal{H}_k \mathcal{M}(\phi, \psi) = \sum_{j \leq k+C} Q_{j} \mathcal{M} \big( Q_{\leq j-C} \phi, Q_{\leq j-C} \psi \big)
  \]
  as well as 
  \[
   \mathcal{H} \mathcal{M} \big( \phi, \psi \big) = \sum_{ \substack{ k \leq k_{1,2}-C, \\ k \leq \min \{k_1, k_2\} - C } } \mathcal{H}_k \mathcal{M}\big( P_{k_1} \phi, P_{k_2} \psi \big). 
  \]
  Then write for the spatial components of $\big( I-\mathcal{H} \big) P_{<0} A^{nonlin, (I_n)}$,
  \begin{align*}
   \big( I-\mathcal{H} \big) P_{<0} A^{nonlin, (I_n)} &= \sum_{ \substack{ k<0 \\ k > \max\{k_1, k_2\}-C}}\Box^{-1}P_k\mathcal{P}_x\big(P_{k_1}\phi\nabla_x P_{k_2}\overline{\phi}\big) \\
   &\quad + \sum_{\substack{k \leq k_{1,2}-C \\ j > k+C}}\Box^{-1}P_kQ_j\mathcal{P}_x\big(P_{k_1}\phi\nabla_x P_{k_2} \overline{\phi} \big) \\
   &\quad + \sum_{\substack{k \leq k_{1,2}-C \\ j \leq k+C}} \Box^{-1} P_k Q_j \mathcal{P}_x\big( P_{k_1} Q_{> j-C} \phi \nabla_x P_{k_2} \overline{\phi} \big) \\
   &\quad + \sum_{\substack{k \leq k_{1,2} -C \\ j \leq k+C}} \Box^{-1}P_kQ_j\mathcal{P}_x\big(P_{k_1}Q_{\leq j-C} \phi\nabla_x P_{k_2}Q_{> j-C} \overline{\phi} \big).
  \end{align*}
  For the first term on the right, employing notation introduced in \cite{KST} and also used in the proof of Proposition~\ref{prop:decomposition}, we get upon further restricting to 
  \[
   \big| \max\{k_{1,2}\} - \min\{k_{1,2}\} \big| \gg 1,
  \]
  the smallness gain 
  \begin{align*} 
   \bigg\| \sum_{ \substack{ k<0, \\ k \geq \max\{k_1, k_2\}-C}} \Box^{-1} P_k \mathcal{P}_x\big(P_{k_1}\phi\nabla_x P_{k_2} \overline{\phi} \big) \bigg\|_{Z} \ll \delta_2 \big\|\phi\big\|_{S^1}^2
  \end{align*}
  and the corresponding contribution to $\mathcal{H}^* {\mathcal N}^{lowhi} \big( (I-\mathcal{H}) P_{<0} A^{nonlin, (I_n)}, P_0 \phi \big)$ can then be bounded with respect to $\big\|\cdot\big\|_{N(I_n \times \R^4)}$ by 
  \begin{align*}
   \leq \delta_2\big\|\phi\big\|_{S^1}^2\big\|P_0 \phi\big\|_{S^1},
  \end{align*}
  which upon replacing $0$ by general frequencies and square summing gives the desired bound. Similarly, for the remaining terms on the right above, one may reduce to $k_{1,2} = k+O(1)$, see estimate (134) in \cite{KST}. Finally, in each of these terms, we may reduce the output to modulation $\sim 2^k$, since else one gains smallness due to the null form structure for 
  \[
   \bigg\|\mathcal{H}^* {\mathcal N}^{lowhi} \big( P_{<0} A^{nonlin, (I_n)}, P_0 \phi \big) - \mathcal{H}^* {\mathcal N}^{lowhi} \big( \mathcal{H} P_{<0} A^{nonlin, (I_n)}, P_0 \phi \big) \bigg\|_{N(I_n\times \R^4)}.
  \]
  Thus we have now reduced to estimating (and gaining a smallness factor) for the schematic expression 
  \begin{align*}
   \sum_{\substack{k<0, \\ k_{1,2} = k+O(1)}} \Box^{-1} P_k Q_{k+O(1)} \mathcal{P}_x \big( P_{k_1} \phi \nabla_x P_{k_2} \overline{\phi} \big) \partial^{\nu} Q_{\leq k-C} P_0 \phi. 
  \end{align*}
  Here we can suppress the operator $\Box^{-1}P_kQ_{k+O(1)}$, which is given by convolution with a space-time kernel of $L^1$-norm $\sim 2^{-2k}$, and then schematically estimate the preceding via 
  \begin{align*}
   &\bigg\| \sum_{\substack{k<0\\ k_{1,2} = k+O(1)}} \Box^{-1}P_kQ_{k+O(1)}\mathcal{P}_x\big(P_{k_1}\phi\nabla_x P_{k_2} \overline{\phi} \big)\partial^{\nu}Q_{\leq k-C} P_0 \phi \bigg\|_{N(I_n \times \R^4)} \\
   &\lesssim \sum_{k_1 = k_2+O(1)<O(1)} 2^{-2k_1} \big\|P_{k_1}\phi\big\|_{L_t^2 L_x^\infty(I_n\times \R^4)} \big\| \nabla_x P_{k_2} \phi Q_{\leq k-C} P_0 \phi \big\|_{L^2_t L^2_x(I_n \times \R^4)}.
  \end{align*}
  Here we exploit Lorentz invariance of the norm of the right factor to obtain 
  \begin{align*}
   \sum_{k_2<0}2^{-3k_2}\big\|\nabla_x P_{k_2}\phi Q_{\leq k-C} P_0 \phi \big\|_{L^2_t L^2_x}^2 \lesssim \sum_{k_2<0} 2^{-3k_2} \big\|\nabla_{t,x} (P_{k_2}\phi)^L Q_{\leq k-C}(P_0 \phi)^L \big\|_{L^2_t L^2_x}^2 \lesssim \big\|(\tilde{A}^L, \tilde{\phi}^L)\big\|_{S^1}^4.
  \end{align*}
  In fact, distinguishing as usual between different frequency/modulation configurations for either of the factors, one estimates the $L^2_t L^2_x$-norm of the input by placing the first input into $L_t^2 L_x^\infty$ and the second into $L_t^\infty L_x^2$, both of which are controlled by Observation 1. Using divisibility of the $L^2_t L^2_x$ norm, it now follows that upon proper choice of the intervals $I_n$, whose number of course only depends on $\big\|\big( \tilde{A}^L, \tilde{\phi}^L \big)\big\|_{S^1}$, we get the estimate 
  \begin{align*}
   &\biggl\| \sum_{\substack{k<0 \\ k_{1,2} = k+O(1)}} \Box^{-1} P_k Q_{k+O(1)} \mathcal{P}_x\big(P_{k_1}\phi\nabla_x P_{k_2}\overline{\phi}\big)\partial^{\nu}Q_{\leq k-C} P_0 \phi \bigg\|_{N(I_n\times \R^4)} \\
   &\lesssim \bigg( \sum_{k_1<0} 2^{-k_1} \big\|P_{k_1}\phi\big\|_{L_t^2 L_x^\infty(I_n\times \R^4)}^2 \bigg)^{\frac{1}{2}} \bigg( \sum_{k_2<0} 2^{-3k_2} \big\|\nabla_x P_{k_2}\phi Q_{\leq k-C} P_0 \phi \big\|_{L^2_t L^2_x(I_n \times \R^4)}^2 \bigg)^{\frac{1}{2}} \\
   &\ll \delta_2 \big\|\phi\big\|_{S^1(I_n \times \R^4)}.
  \end{align*}
  Of course, one gets the same bound upon replacing the frequency $0$ by general $l \in \Z$ and square summing. 

  \medskip

  As usual, similar reductions can be applied to the elliptic interaction term $P_{<0} A_{0} \partial_t P_0 \phi$. 

  \medskip

  \noindent {\bf{Dealing with $\mathcal{H}^* {\mathcal N}^{lowhi} \big( \mathcal{H} P_{<0} A^{nonlin, (I_n)}, P_0 \phi \big)$}}. Here we exploit the null structure arising from combining the elliptic as well as hyperbolic terms, just as in the proof of Proposition~\ref{prop:decomposition}, or as in \cite{KST}. Correspondingly, we have to analyze three null forms, each in turn.

  \medskip

  \noindent {\it{The first null form}}. We can write it as 
  \begin{align*}
   \sum_{ \substack{ j \leq k < 0, \\ k_{1,2}>k+O(1)} } \Box^{-1}P_kQ_j \big(Q_{\leq j-C} P_{k_1} \phi \partial_{\alpha} Q_{\leq j-C} P_{k_2} \phi \big) \partial^{\alpha} Q_{\leq j-C} P_0 \phi.
  \end{align*}
  From (148) in \cite{KST}, it follows that we may restrict to $j = k+O(1)$, as otherwise the desired smallness just follows from the off-diagonal decay of the estimate (even without restriction to smaller time intervals). Furthermore, if $k_{1}, k_2 < 0$, then we gain exponentially in the difference $k-k_1$, while if $k_{1}, k_{2} \geq 0$, we gain exponentially in $k$. 
  So we may further restrict to 
  \begin{align*}
   &\sum_{ \substack{ k<0, \\ k_{1,2}=k+O(1) } } \Box^{-1} P_k Q_{k+O(1)} \big( Q_{\leq k-C} P_{k_1} \phi \partial_{\alpha} Q_{\leq k-C} P_{k_2} \phi \big) \partial^{\alpha} Q_{\leq k-C} P_0 \phi
  \end{align*}
  and from here the argument proceeds exactly as before by suppressing the operator $\Box^{-1} P_k Q_{k+O(1)}$ and schematically estimating 
  \[
   \big\|\partial_{\alpha}Q_{\leq k-C} P_{k_2} \phi \partial^{\alpha} Q_{\leq k-C} P_0 \phi \big\|_{L^2_t L^2_x}
  \]
  using Observation 2, while placing $Q_{\leq k-C} P_{k_1} \phi$ into $L_t^2 L_x^\infty$. 

  \medskip

  \noindent {\it{The second and third null forms}}. These are treated identically and hence omitted here. 
 \end{proof}

 \medskip

 This completes Step 2. Together with Step 1, the linear theory for the operator $\sum_k \Box_{A^{free}_{<k}}P_k$ and a standard bootstrap argument, this yields the bounds claimed in Proposition~\ref{prop:S_norm_Lorentz} for the localized norms $\|(A, \phi)\|_{S^1(I_n\times\R^4)}$. From there one can glue the localized components together to get the global bounds. 
 \end{proof}

\subsection{Rigidity I: Infinite time interval and reduction to the self-similar case for finite time intervals} \label{subsec:rigidity1}

 As in \cite[Theorem 5.1]{KM}, our goal is now to establish the following rigidity result. 
 \begin{prop} \label{prop:rigidity} 
  Let $(\mathcal{A}^\infty, \Phi^\infty)$ be as before with lifespan $I = (-T_0, T_1)$. Then it is not possible to have $T_0 < \infty$ or $T_1 < \infty$. Moreover, if $\lambda(t) \geq \lambda_0 > 0$ for all $t\in \R$, then we necessarily have $(\mathcal{A}^\infty, \Phi^\infty) = (0,0)$, whence there cannot be a minimal energy blowup solution under the given assumptions. 
 \end{prop}
 
 We begin the proof of Proposition~\ref{prop:rigidity} in the case when $T_1 = \infty$ and $\lambda(t) \geq \lambda_0 > 0$ on $[0,\infty)$. To this end we follow the method of proof in \cite{KS}, which in turn follows the strategy in \cite{KM}, but also adds a crucial Vitali type covering argument that is inspired by the covering argument in \cite{Struwe}. Using the assumption $\lambda(t) \geq \lambda_0 > 0$ on $[0,\infty)$ and the compactness property expressed in Theorem~\ref{thm:compact_orbit}, we obtain that for any $\varepsilon > 0$, there exists $R_0(\varepsilon) > 0$ such that for all $t \geq 0$,
 \begin{equation} \label{equ:small_outside}
  \int_{\big\{ |x + \frac{\bar{x}(t)}{\lambda(t)}| \geq R_0(\varepsilon) \big\}} \biggl( \sum_{\alpha} \bigl|\nabla_{t,x} \mathcal{A}^\infty_{\alpha}(t, x) \bigr|^2 + \bigl| \nabla_{t,x} \Phi^\infty(t,x) \bigr|^2 + \frac{|\Phi^\infty(t,x)|^2}{|x|^2} \biggr) \, dx \leq \varepsilon. 
 \end{equation}
 Then we have in perfect analogy with \cite[Lemma 10.9]{KS} and \cite[Lemma 5.4]{KM} the following 
 \begin{lem} \label{lem:analogue_WM_Lemma_10.9} 
  There exist $\varepsilon_1 > 0$ and $C > 0$ such that if $\varepsilon \in (0, \varepsilon_1)$, there exists $R_0(\varepsilon)$ so that if $R > 2 R_0(\varepsilon)$, then there exists $t_0 = t_0(R, \varepsilon)$ with $0 \leq t_0 \leq C R$ and the property that for all $0 < t <t_0$ we have
  \[
   \Bigl| \frac{\bar{x}(t)}{\lambda(t)} \Bigr| < R - R_0(\varepsilon), \quad \Big|\frac{\bar{x}(t_0)}{\lambda(t_0)} \Big| = R - R_0(\varepsilon).
  \]
 \end{lem}
 \begin{proof} 
  We adjust the proof of \cite[Lemma 10.9]{KS} by using the weighted momentum monotonicity identity \eqref{equ:weighted_momentum_monotonicity} from Proposition~\ref{prop:virial_identities}. To begin with, we show that there exists $\alpha > 0$ with 
  \begin{equation} \label{equ:cruxlowerbound}
   \int_{I} \int_{\R^4} \biggl( \sum_k \cF_{0k}^\infty(t,x)^2 + \big| {\mathcal D}_0 \Phi^\infty(t,x) \big|^2 \biggr) \, dx \, dt \geq \alpha > 0
  \end{equation}
  for all intervals $I \subset [0,\infty)$ of unit length. We argue by contradiction. Suppose not, then there exists a sequence of intervals $J_n = [t_n, t_n+1]$ with $t_n \rightarrow \infty$ such that 
  \begin{equation} \label{equ:getting_small_1}
   \int_{J_n} \int_{\R^4} \biggl( \sum_k \cF_{0k}^\infty(t,x)^2 + \big|{\mathcal D}_0 \Phi^\infty(t,x) \big|^2 \biggr) \, dx \, dt \leq \frac{1}{n}. 
  \end{equation}
  For a sequence of times $\{s_n\}_n$ with $s_n \in J_n$, the set
  \[
   \Bigg\{ \bigg( \frac{1}{\lambda(s_n)^2} \nabla_{t,x} {\mathcal A}^\infty_x \bigg(s_n, \frac{\cdot-\bar{x}(s_n)}{\lambda(s_n)} \bigg), \frac{1}{\lambda(s_n)^2} \nabla_{t,x} \Phi^\infty \bigg(s_n, \frac{\cdot-\bar{x}(s_n)}{\lambda(s_n)} \bigg) \bigg) \Bigg\}_n
  \]
  is pre-compact in $\big(L^2_x(\R^4)\big)^5$ by Theorem~\ref{thm:compact_orbit}. Then by Corollary~\ref{cor:lifespancompact}, there exists a non-empty open interval $I^\ast$ around $t = 0$ such that 
  \[
   \frac{1}{\lambda(s_n)^2} \bigg( \nabla_{t,x} \mathcal{A}^\infty_x, \nabla_{t,x} \Phi^\infty \bigg) \bigg( s_n + t \lambda(s_n)^{-1}, \frac{\cdot - \bar{x}(s_n)}{\lambda(s_n)} \bigg)
  \]
  converges to a limiting function 
  \[
   \big( \nabla_{t,x} \mathcal{A}^\ast_x, \nabla_{t,x} \Phi^\ast \big) \in C^0 \big(I^\ast, (L^2_x(\R^4))^5\big)
  \]  
  as $n \to \infty$ in the given topology. $(\mathcal{A}^\ast, \Phi^\ast)$ is a weak solution to MKG-CG on $I^\ast \times \R^4$ in the $L^2_t \dot{H}^1_x$-sense and satisfies the Coulomb condition. 

  \medskip
 
  We now distinguish two cases: Either there exists a sequence of times $\{s_n\}_n$ with $s_n \in J_n$ such that $\{\lambda(s_n)\}_n$ remains bounded or $\{ \lambda(s_n) \}_n$ does not remain bounded for any sequence of times $\{s_n\}_n$ with $s_n \in J_n$. In the first case, noting that $\lambda(t) \geq \lambda_0 > 0$, we may replace $I^\ast$ by a smaller non-empty time interval $I^\dagger$ and assume that
  \[
   s_n + \lambda(s_n)^{-1} I^\dagger \subset J_n 
  \]
  for all $n \geq 1$. From \eqref{equ:getting_small_1} we infer that 
  \[
   \int_{I^\dagger} \int_{\R^4} \big| \big(\partial_t \Phi^\ast + i \mathcal{A}_0^\ast \Phi^\ast \big)(t,x) \big|^2 \, dx \, dt = 0,
  \]
  whence $\partial_t \Phi^\ast + i \mathcal{A}_0^\ast \Phi^\ast \equiv 0$ on $I^\dagger \times \R^4$. But then we have in the weak sense that
  \[
   \sum_{k = 1}^4 \big( \partial_k + i \mathcal{A}^\ast_k \big)^2 \Phi^\ast \equiv 0 \text{ on } I^\dagger \times \R^4.
  \]
  This implies $\Phi^\ast|_{I^\dagger \times \R^4} \equiv 0$ and hence also $\partial_t \Phi^\ast|_{I^\dagger \times \R^4} \equiv 0$. We conclude that $(\mathcal{A}^\ast, \Phi^\ast)$ must be a ``trivial'' solution in that the spatial components of $\mathcal{A}^\ast$ are finite energy free waves, while the temporal component vanishes, and we have $\Phi^\ast \equiv 0$. But this solution has finite $S^1$-bounds, which is a contradiction upon applying Proposition~\ref{prop:scatdata}. 

  \medskip

  Next, we consider the case that $\{ \lambda(s_n) \}_n$ does not remain bounded for any sequence of times $\{s_n\}_n$ with $s_n \in J_n$. Then we essentially replicate the preceding argument, but need to also add a Vitali type covering trick. We write for each $n \geq 1$,
  \[
   J_n = \bigcup_{s\in J_n} \big[s - \lambda(s)^{-1}, s + \lambda(s)^{-1} \big] \cap J_n.
  \]
  Applying Vitali's covering lemma, we may pick a disjoint subcollection of intervals $\{ I_s \}_{s \in A_n}$ with $I_s := [s - \lambda(s)^{-1}, s + \lambda(s)^{-1}] \cap J_n$ for some subset $A_n \subset J_n$ such that 
  \[
   \sum_{s \in A_n} 5 |I_s| \geq 1.
  \]
  It follows that we may pick a sequence of times $\{s_n\}_n$ with $s_n \in J_n$ such that we have 
  \begin{align*}
   \int_{I_{s_n}} \int_{\R^4} \biggl( \sum_k \cF_{0k}^\infty(t,x)^2 + \big|{\mathcal D}_0 \Phi^\infty (t,x) \big|^2 \biggr) \, dx \, dt = o(\lambda(s_n)^{-1}). 
  \end{align*}
  In particular, we obtain
  \begin{align*}
   \int_{-1}^1 \int_{\R^4} \Biggl( \chi_{J_n} \biggl( \sum_k \bigl( \cF_{0k}^\infty \bigr)^2 + \big| \mathcal{D}_0 \Phi^\infty \big|^2 \biggr) \Biggr)\bigl( s_n + t \lambda(s_n)^{-1}, x \bigr) \, dx \, dt = o(1).  
  \end{align*}
  But then, passing to a subsequence, we may again extract a limiting function from 
  \begin{align*}
   \frac{1}{\lambda(s_n)^2} \bigg( \nabla_{t,x} \mathcal{A}^\infty_x, \nabla_{t,x} \Phi^\infty \bigg) \bigg( s_n + t \lambda(s_n)^{-1}, \frac{\cdot - \bar{x}(s_n)}{\lambda(s_n)} \bigg),
  \end{align*}
  which yields a time independent solution and leads to a contradiction as before. 

 \medskip

 This shows that \eqref{equ:cruxlowerbound} is indeed valid for suitable $\alpha > 0$. We note that $\lambda(t) \geq \lambda_0 > 0$ and that we may assume $\bar{x}(0) = 0$. If the assertion of the lemma was false, then we would have
 \[
  \Bigl| \frac{\bar{x}(t)}{\lambda(t)} \Bigr| < R - R_0(\varepsilon)
 \]
 for all $0 \leq t < C R$ with $C > 0$ to be picked sufficiently large later on. We now use the weighted momentum monotonicity identity \eqref{equ:weighted_momentum_monotonicity} to obtain a contradiction. In view of \eqref{equ:small_outside}, we conclude that the corresponding remainder term \eqref{equ:virial_identities_remainder} satisfies
 \[
  r(R) \lesssim \varepsilon.
 \]
 Now choose $\varepsilon > 0$ so small that we have for any time interval $I$ of unit length,
 \[
  \int_I \Biggl( \int_{\R^4} \sum_k \cF_{0k}^\infty(t,x)^2 + \big| {\mathcal D}_0 \Phi^\infty(t,x) \big|^2 \, dx + O(r(R)) \Biggr) \, dt \geq \frac{\alpha}{2},
 \]
 provided $C R$ is sufficiently large. In particular, this implies 
 \[
  \int_0^{CR} \Biggl( \int_{\R^4} \sum_k \cF_{0k}^\infty(t,x)^2 + \big| {\mathcal D}_0 \Phi^\infty(t,x) \big|^2 \, dx + O(r(R)) \Biggr) \, dt \geq \frac{\alpha}{2} (C R - 1).
 \]
 On the other hand, integrating \eqref{equ:weighted_momentum_monotonicity} in time from $0$ to $C R$, we find 
 \[
  \int_0^{CR} \Biggl( \int_{\R^4} \sum_k \cF_{0k}^\infty(t,x)^2 + \big| {\mathcal D}_0 \Phi^\infty(t,x) \big|^2 \, dx + O(r(R)) \Biggr) \, dt \lesssim R E_{crit}
 \]
 with a universal implied constant. The two preceding bounds contradict each other for $C$ large.
\end{proof}

To finish off the proof of Proposition~\ref{prop:rigidity} in the case when $T_1 = \infty$ and $\lambda(t) \geq \lambda_0 > 0$ on $[0,\infty)$, we now use Proposition~\ref{prop:vanishing_momentum_infinite_time} to conclude a contradiction to the preceding Lemma~\ref{lem:analogue_WM_Lemma_10.9}.
\begin{lem} \label{lem:analogue_WM_Lemma_10.10}
 There exist $\varepsilon_2 > 0$, $R_1(\varepsilon) > 0$, $C_0 > 0$ such that if $R > R_1(\varepsilon)$, $t_0 = t_0(R,\varepsilon)$ are as in Lemma~\ref{lem:analogue_WM_Lemma_10.9}, then for $0 < \varepsilon < \varepsilon_2$,
 \[
  t_0(R,\varepsilon) \geq \frac{C_0 R}{\varepsilon}.
 \]
\end{lem}
\begin{proof}
 The proof proceeds exactly as in \cite[Lemma 5.5]{KM} using the weighted energy identity \eqref{equ:weighted_energy} and that the minimal blowup solution $(\mathcal{A}^\infty, \Phi^\infty)$ has vanishing momentum by Proposition~\ref{prop:vanishing_momentum_infinite_time}.
\end{proof}

It remains to prove Proposition~\ref{prop:rigidity} when $T_1 < \infty$. As in \cite{KM} and \cite{KS}, we first reduce this case to a self-similar blowup scenario. After rescaling we may assume that $T_1 = 1$. We recall from Lemma~\ref{lem:finite_lifespan_lower_bound_lambda} that there exists a constant $C_0(K) > 0$ such that
\[
 0 < \frac{C_0(K)}{1-t} \leq \lambda(t)
\]
for all $0 \leq t < 1$. Moreover, we know from Lemma~\ref{lem:finite_lifespan_support_in_ball} that after spatial translation 
\begin{equation} \label{equ:min_blowup_solution_support_in_ball}
 \text{supp} \, \Big( \cF_{\alpha \beta}^\infty(t, \cdot), \Phi^\infty(t, \cdot) \Big) \subset \overline{B}(0, 1-t) 
\end{equation}
for all $0 \leq t < 1$ and all $\alpha, \beta \in \{0, 1, \ldots,4\}$. Next, we prove an upper bound on $\lambda(t)$.
\begin{lem} \label{lem:finite_lifespan_upper_bound_lambda}
 Let $(\mathcal{A}^\infty, \Phi^\infty)$ be as above with $T_1 = 1$. Then there exists $C_1(K) > 0$ such that
 \[
  \lambda(t) \leq \frac{C_1(K)}{1-t}
 \]
 for all $0 \leq t < 1$.
\end{lem}
\begin{proof}
 We follow the argument in Lemma~10.11 in \cite{KS}. Suppose the claim was false. Then consider for $0 \leq t < 1$ the functional
 \begin{align*}
  z(t) &= \int_{\R^4} \sum_k x_k \, \biggl( \sum_j \cF_{0j}^\infty \cF_{kj}^\infty + \Re \, \bigl( \mathcal{D}_0 \Phi^\infty \overline{\mathcal{D}_k \Phi^\infty} \bigr) \biggr) \, dx + \int_{\R^4} \Re \, \bigl( \Phi^\infty \overline{\mathcal{D}_0 \Phi^\infty} \bigr) \, dx. 
 \end{align*}
 From the weighted momentum monotonicity identity \eqref{equ:weighted_momentum_monotonicity} in Proposition~\ref{prop:virial_identities} we obtain that
 \[
  z'(t) = - \int_{\R^4} \biggl( \sum_k \bigl( \cF_{0k}^\infty \bigr)^2 + \big| {\mathcal D}_0 \Phi^\infty \big|^2 \biggr) \, dx.
 \]
 Since we have by \eqref{equ:min_blowup_solution_support_in_ball} and Hardy's inequality that $z(t) \to 0$ as $t \to 1$, we can write
 \[
  z(t) = \int_t^1 \int_{\R^4} \biggl( \sum_k \cF_{0k}^\infty(s,x)^2 + \big| {\mathcal D}_0 \Phi^\infty(s,x)\big|^2 \biggr) \, dx \, ds.
 \]
 Now we distinguish between two possibilities: Either there exists $\alpha > 0$ such that for all $0 \leq t < 1$ we have
 \[
  \int_t^1 \int_{\R^4} \biggl( \sum_k \cF_{0k}^\infty(s,x)^2 + \big| {\mathcal D}_0 \Phi^\infty(s,x)\big|^2 \biggr) \, dx \, ds \geq \alpha (1-t),
 \]
 or else there exists a sequence $\{t_n\}_{n} \subset [0,1)$ with $t_n \to 1$ such that, denoting $J_n := [t_n, 1)$, it holds that
 \[
  \frac{1}{|J_n|} \int_{J_n} \int_{\R^4} \biggl( \sum_k \cF_{0k}^\infty(s,x)^2 + \big| {\mathcal D}_0 \Phi^\infty(s,x)\big|^2 \biggr) \, dx \, ds \to 0.
 \]
 In the former case, we argue exactly as in \cite[Lemma 5.6]{KM} to obtain the conclusion of the lemma. In particular, here we invoke Proposition~\ref{prop:vanishing_momentum_finite_time}. In the latter case, a contradiction ensues as follows. Using the same Vitali covering argument as in the proof of Lemma~\ref{lem:analogue_WM_Lemma_10.9}, we can select a sequence of intervals $J_n' = [s_n - \lambda(s_n)^{-1}, s_n + \lambda(s_n)^{-1}]$ with $s_n \in J_n$ such that 
 \[
  \frac{1}{|J_n'|} \int_{J_n'} \int_{\R^4} \biggl( \sum_k \cF_{0k}^\infty(s,x)^2 + \big| {\mathcal D}_0 \Phi^\infty(s,x)\big|^2 \biggr) \, dx \, ds \to 0.
 \]
But then, using compactness, we again extract a trivial limiting solution, and obtain a contradiction as in the proof of Lemma~\ref{lem:analogue_WM_Lemma_10.9}.
\end{proof}

We are now in a position to reduce to the exactly self-similar case. 

\begin{cor} \label{cor:reduction_to_exactly_self_similar_case}
 Let $(\mathcal{A}^\infty, \Phi^\infty)$ be as above with $T_1 = 1$. Then the set
 \[
  \biggl\{ \Bigl( (1-t)^2 \bigl( \nabla_{t,x} \mathcal{A}^\infty_x \bigr)(t, (1-t) x), (1-t)^2 \bigl( \nabla_{t,x} \Phi^\infty \bigr)(t, (1-t) x) \Bigr) : 0 \leq t < 1 \biggr\}
 \]
 is pre-compact in $\bigl( L^2_x(\R^4) \bigr)^5$.
\end{cor}
\begin{proof}
 Here we can proceed similarly to the proof of Proposition~5.7 in \cite{KM}. Our point of departure is Theorem~\ref{thm:compact_orbit}. From Lemma~\ref{lem:finite_lifespan_lower_bound_lambda} and Lemma~\ref{lem:finite_lifespan_upper_bound_lambda} we know that 
 \[
  C_0(K) \leq (1-t) \lambda(t) \leq C_1(K)
 \]
 for all $0 \leq t < 1$. Using the sharp support properties \eqref{equ:min_blowup_solution_support_in_ball} and that $E_{crit} > 0$, we also conclude that $|\overline{x}(t)| \leq C$ for all $0 \leq t < 1$ for some constant $C > 0$. Then the claim follows from the compactness assertion in Theorem~\ref{thm:compact_orbit}.
\end{proof}

\subsection{Rigidity II: The self-similar case} \label{subsec:rigidity2}

 In this subsection we rule out the existence of a minimal blowup solution $\bigl( \cA^\infty, \Phi^\infty \bigr)$ as in Corollary~\ref{cor:reduction_to_exactly_self_similar_case}. To this end we use self-similar variables and derive a suitable Lyapunov functional. For ease of notation we drop the superscript $\infty$ and denote the minimal blowup solution from Corollary~\ref{cor:reduction_to_exactly_self_similar_case} just by $(\mathcal{A}, \Phi)$. Following \cite{KM} and \cite{KS}, we introduce the self-similar variables
 \[
  y = \frac{x}{1-t}, \quad s = - \log(1-t), \quad 0 \leq t < 1
 \]
 and set
 \begin{align*}
  \widetilde{\Phi}(s,y, 0) &= e^{-s} \Phi(1 - e^{-s}, e^{-s} y), \\
  \widetilde{{\mathcal A}}_\alpha(s,y,0) &= e^{-s} {\mathcal A}_\alpha(1 - e^{-s}, e^{-s} y), \quad 0 \leq \alpha \leq 4.
 \end{align*}
 We also define the associated covariant derivatives of $\widetilde{\Phi}$ and the curvature 2-form associated with $\widetilde{\cA}$ in self-similar variables by
 \begin{alignat}{2} \label{equ:definition_cov_derivative_and_curvature_self_sim}
  \begin{aligned}
   \widetilde{ {\mathcal D}_\alpha \Phi }(s,y,0) &= e^{-2s} {\mathcal D}_\alpha \Phi(1 - e^{-s}, e^{-s} y), \quad &  &0 \leq \alpha \leq 4, \\
   \widetilde{ {\mathcal F}_{\alpha \beta} }(s,y,0) &= e^{-2s} {\mathcal F}_{\alpha \beta}(1 - e^{-s}, e^{-s} y), \quad &  &0 \leq \alpha, \beta \leq 4.
  \end{aligned}
 \end{alignat}
 Observe that $\bigl(\widetilde{{\mathcal A}}, \widetilde{\Phi}\bigr)(s,y,0)$ are defined for $0 \leq s < \infty$. Moreover, in view of \eqref{equ:min_blowup_solution_support_in_ball}, $\widetilde{\Phi}(s,\cdot, 0)$ and the curvature components $\widetilde{F_{\alpha \beta}}(s,\cdot,0)$ have support in $\{ y \in \R^4 : |y| \leq 1 \}$. For small $\delta > 0$, we also define
 \[
  y = \frac{x}{1+\delta-t}, \quad s = - \log(1+\delta-t), \quad 0 \leq t < 1
 \]
 and set
 \begin{align*}
  \widetilde{\Phi}(s,y,\delta) &= e^{-s} \Phi(1 + \delta - e^{-s}, e^{-s} y), \\
  \widetilde{{\mathcal A}}_\alpha(s,y,\delta) &= e^{-s} {\mathcal A}_\alpha(1 + \delta - e^{-s}, e^{-s} y), \quad 0 \leq \alpha \leq 4.
 \end{align*}
 Analogously to \eqref{equ:definition_cov_derivative_and_curvature_self_sim}, we introduce $\widetilde{ {\mathcal D}_\alpha \Phi }(s, y, \delta)$ for $0 \leq \alpha \leq 4$ and $\widetilde{F_{\alpha \beta}}(s,y,\delta)$ for $0 \leq \alpha,\beta \leq 4$. We note that $\bigl(\widetilde{{\mathcal A}}, \widetilde{\Phi}\bigr)(s,y, \delta)$ is defined for $-\log(1+\delta) \leq s < \log(\frac{1}{\delta})$.

 \medskip

 In self-similar variables the Maxwell-Klein-Gordon system is given by
 \begin{equation} \label{equ:mkg_system_self_sim}
  \left\{ \begin{aligned}
   \partial^k \widetilde{\cF_{0k}} &= \Im \bigl( \widetilde{\Phi} \overline{\widetilde{\cD_0 \Phi}} \bigr), \\
   - \bigl(\partial_s + 2 + y \cdot \nabla_y \bigr) \widetilde{\cF_{j0}} + \partial^k \widetilde{\cF_{jk}} &= \Im \bigl( \widetilde{\Phi} \overline{ \widetilde{\cD_j \Phi} } \bigr), \\ 
   \bigl( \partial_s + i \widetilde{\cA}_0 + 2 + y \cdot \nabla_y \bigr) \widetilde{\cD_0 \Phi} &= \bigl( \partial^k + i \widetilde{\cA}^k \bigr) \widetilde{\cD_k \Phi},
  \end{aligned} \right.  
 \end{equation}
 where $\partial_k$ denotes partial differentiation with respect to the $y$ variable. We begin by stating the following properties of $\bigl( \widetilde{\cA}, \widetilde{\Phi} \bigr)$.

 \begin{lem}
  \begin{enumerate}[wide,itemindent={0.4em}]
   \item For fixed $\delta > 0$, we have for all $0 \leq s < \log(\frac{1}{\delta})$ that
    \begin{align} \label{equ:support_in_self_similar_variables_delta}
     \begin{aligned}
      \supp \, \bigl( \widetilde{\Phi}(s, \cdot, \delta) \bigr) &\subset \bigl\{ y \in \R^4 : |y| \leq 1 - \delta \bigr\}, \\
      \supp \, \bigl( \widetilde{F_{\alpha \beta}}(s, \cdot, \delta) \bigr) &\subset \bigl\{ y \in \R^4 : |y| \leq 1 - \delta \bigr\}.
     \end{aligned}
    \end{align}
   \item Uniformly for all $\delta > 0$ and all $0 \leq s < \log (\frac{1}{\delta})$, it holds that
    \begin{equation} \label{equ:auxiliary_estimate_kinetic_energy_self_sim}
     \int_{B_1} \sum_\alpha \bigl| \widetilde{\cD_\alpha \Phi}(s,y,\delta) \bigr|^2 \, dy \lesssim E_{crit}
    \end{equation}
   and
    \begin{equation} \label{equ:hardy_type_inequality_self_sim}
     \int_{B_1} \frac{ | \widetilde{\Phi}(s,y,\delta) |^2 }{ (1-|y|^2)^2 } \, dy \lesssim E_{crit}.
    \end{equation}
  \end{enumerate}
 \end{lem}
 \begin{proof}
  (i) For $0 \leq s < \log(\frac{1}{\delta})$ we infer from the support properties \eqref{equ:min_blowup_solution_support_in_ball} that
  \[
   \supp \, \bigl( \widetilde{\Phi}(s, \cdot, \delta) \bigr) \subset \biggl\{ y \in \R^4 : |y| \leq \frac{1-t}{1+\delta - t} = \frac{e^{-s} - \delta}{e^{-s}} \leq 1 - \delta \biggr\}
  \]
  and similarly for the support of $\widetilde{\cF_{\alpha \beta}}(s, \cdot, \delta)$.  \\
  (ii) The estimate \eqref{equ:auxiliary_estimate_kinetic_energy_self_sim} follows immediately from a change of variables. Noting that $\widetilde{\Phi}(s, \cdot, \delta) \in H_0^1(B_1)$ for all $\delta > 0$ and $0 \leq s < \log(\frac{1}{\delta})$, we then use the Hardy-type inequality (0.5) from \cite{Brezis_Marcus} together with the diamagnetic inequality to conclude that
  \begin{align*}
   \int_{B_1} \frac{|\widetilde{\Phi}|^2}{(1-|y|^2)^2} \, dy &\leq \int_{B_1} \frac{|\widetilde{\Phi}|^2}{(1-|y|)^2} \, dy \lesssim \int_{B_1} \bigl|\nabla_y |\widetilde{\Phi}| \bigr|^2 \, dy \lesssim \int_{B_1} \sum_k \bigl| \widetilde{\cD_k \Phi} \bigr|^2 \, dy \lesssim E_{crit}.
  \end{align*}
 \end{proof}

 For small $\delta > 0$ and $\bigl( \widetilde{\cA}, \widetilde{\Phi} \bigr)(s,y,\delta)$ with associated covariant derivatives $\widetilde{\cD_\alpha \Phi}(s,y,\delta)$ and curvature components $\widetilde{\cF_{\alpha \beta}}(s,y,\delta)$ as above, we now introduce a Lyapunov functional
 \begin{equation*}
  \begin{aligned}
   \widetilde{E}(s) &= \int_{B_1} \biggl( \frac{1}{2} \sum_j \widetilde{\cF_{j0}}^2 + \frac{1}{4} \sum_{j,k} \widetilde{\cF_{jk}}^2 - \sum_{j,k} y_k \widetilde{\cF_{j0}} \widetilde{\cF_{jk}}  \biggr) \frac{dy}{(1-|y|^2)^{\frac{1}{2}}} \\
   &\quad + \int_{B_1} \biggl( \frac{1}{2} \sum_\alpha \bigl| \widetilde{{\cD}_\alpha \Phi} \bigr|^2 - \sum_k y_k \Re \big( \widetilde{\cD_k \Phi} \overline{ \widetilde{\cD_0 \Phi} } \big) - \Re \big( \widetilde{\Phi} \overline{ \widetilde{\cD_0 \Phi} } \big) - \frac{1}{2} \frac{|\widetilde{\Phi}|^2}{1-|y|^2} \biggr) \frac{dy}{(1-|y|^2)^{\frac{1}{2}}}
  \end{aligned}
 \end{equation*}
 and define the non-negative quantity
 \begin{align*}
  \widetilde{\Xi}(s) &= \int_{B_1} \sum_{k} \bigg( \widetilde{\cF_{k0}} - \Big( \sum_{j} \frac{y_j}{|y|} \widetilde{\cF_{j0}} \Big) \frac{y_k}{|y|} + \sum_{j} y_j \widetilde{\cF_{jk}} \bigg)^2 \frac{dy}{(1-|y|^2)^{\frac{3}{2}}} \\
  &\quad + \int_{B_1} \frac{1}{|y|^2} \bigg( \sum_{j} y_j \widetilde{\cF_{j0}} \bigg)^2 \frac{dy}{(1-|y|^2)^{\frac{1}{2}}} + \int_{B_1} \Bigl| \widetilde{\cD_0 \Phi} - \sum_k y_k \widetilde{\cD_k \Phi} - \widetilde{\Phi} \Bigr|^2 \frac{dy}{(1-|y|^2)^{\frac{3}{2}}}.
 \end{align*}
 We emphasize that both $\widetilde{E}$ and $\widetilde{\Xi}$ are gauge invariant quantities. They are well-defined for all $\delta > 0$ in view of the support properties \eqref{equ:support_in_self_similar_variables_delta}. In the next proposition we establish a key monotonicity property of the functional $\widetilde{E}$.
 \begin{prop} \label{prop:monotonicity_lyapunov}
  Let $\bigl( \widetilde{\cA}, \widetilde{\Phi} \bigr)(s,y,\delta)$ for $\delta > 0$ be as above. Then we have for 
  \[
   0 \leq s_1 < s_2 < \log \Bigl( \frac{1}{\delta} \Bigr)
  \]
  that
  \begin{equation} \label{equ:monotonicity_lyapunov}
   \widetilde{E}(s_2) - \widetilde{E}(s_1) = \int_{s_1}^{s_2} \widetilde{\Xi}(s) \, ds.
  \end{equation}
  Moreover, it holds that
  \begin{equation} \label{equ:limit_lyapunov_energy_at_time_infty}
   \lim_{s \to \log(\frac{1}{\delta})} \widetilde{E}(s) \leq E_{crit}.
  \end{equation}
 \end{prop}
 The crucial monotonicity identity \eqref{equ:monotonicity_lyapunov} can be derived in a gauge invariant manner. However, the computations simplify significantly by imposing the Cronstrom-type gauge condition
 \begin{equation} \label{equ:cronstrom_type_gauge}
  \sum_{k=1}^4 x_k \cA_k(t,x) = 0
 \end{equation}
 for all $0 \leq t < 1$ and $x \in \R^4$. This does not change the energy regularity of $\bigl( \cA, \Phi \bigr)$. In self-similar variables the gauge condition \eqref{equ:cronstrom_type_gauge} reads
 \begin{equation} \label{equ:cronstrom_type_gauge_self_similar}
  \sum_{k=1}^4 y_k \widetilde{\cA}_k(s,y,\delta) = 0
 \end{equation}
 for all $\delta > 0$, $0 \leq s < \log (\frac{1}{\delta})$ and $y \in \R^4$. Under the gauge condition \eqref{equ:cronstrom_type_gauge_self_similar} the functional $\widetilde{E}$ can be written as 
 \begin{align*}
  \widetilde{E}(s) &= \int_{B_1} \biggl( \frac{1}{2} \sum_j \bigl( \partial_s \widetilde{\cA}_j - \partial_j \widetilde{\cA}_0 \bigr)^2 + \frac{1}{4} \sum_{j,k} \bigl( \partial_j \widetilde{\cA}_k - \partial_k \widetilde{\cA}_j \bigr)^2 - \frac{1}{2} \sum_j \bigl( (1 + y \cdot \nabla_y) \widetilde{\cA}_j \bigr)^2 \biggr) \frac{dy}{(1-|y|^2)^{\frac{1}{2}}} \\
  &\quad + \int_{B_1} \biggl( \frac{1}{2} \bigl| (\partial_s + i \widetilde{\cA}_0) \widetilde{\Phi} \bigr|^2 + \sum_k \bigl|(\partial_k + i \widetilde{\cA}_k) \widetilde{\Phi} \bigr|^2 - \bigl| (1 + y \cdot \nabla_y) \widetilde{\Phi} \bigr|^2 - \frac{|\widetilde{\Phi}|^2}{1-|y|^2} \biggr) \frac{dy}{(1-|y|^2)^{\frac{1}{2}}}
 \end{align*}
 and the quantity $\widetilde{\Xi}$ reads
 \begin{align*}
  \widetilde{\Xi}(s) &= \int_{B_1} \sum_k \biggl( \partial_k \widetilde{\cA}_0 - \Bigl( \frac{y}{|y|} \cdot \nabla_y \widetilde{\cA}_0 \Big) \frac{y_k}{|y|} - \partial_s \widetilde{\cA}_k \biggr)^2 \frac{dy}{(1-|y|^2)^{\frac{3}{2}}} \\
  &\quad + \int_{B_1} \frac{1}{|y|^2} \Bigl( y \cdot \nabla_y \widetilde{\cA}_0 \Bigr)^2 \frac{dy}{(1-|y|^2)^{\frac{1}{2}}} + \int_{B_1} \bigl|(\partial_s + i \widetilde{\cA}_0) \widetilde{\Phi}\bigr|^2 \frac{dy}{(1-|y|^2)^{\frac{3}{2}}},
 \end{align*}
 where $\partial_k$ denotes partial differentiation with respect to the $y$ variable. It is not obvious that the above expressions for $\widetilde{E}$ and $\widetilde{\Xi}$ in the Cronstrom-type gauge \eqref{equ:cronstrom_type_gauge_self_similar} are even well-defined for all $\delta > 0$. However, this follows from the gauge invariant support properties of $\widetilde{\Phi}$ and $\widetilde{\cF_{\alpha \beta}}$, and the following easily verified identities (assuming the gauge condition \eqref{equ:cronstrom_type_gauge_self_similar})
 \begin{align} \label{equ:identities_self_sim}
  \begin{aligned}   
   \partial_s \widetilde{\cA}_j - \partial_j \widetilde{\cA}_0 &= \widetilde{\cF_{0j}} - \sum_k y_k \widetilde{\cF_{kj}}, \\
   \partial_j \widetilde{\cA}_k - \partial_k \widetilde{\cA}_j &= \widetilde{\cF_{jk}}, \\
   y \cdot \nabla_y \widetilde{\cA}_0 &= \sum_j y_j \widetilde{\cF_{j0}}, \\
   \bigl( 1 + y \cdot \nabla_y \bigr) \widetilde{\cA}_j &= \sum_k y_k \widetilde{\cF_{kj}}, \\
   \bigl( \partial_s + i \widetilde{\cA}_0 \bigr) \widetilde{\Phi} &= \widetilde{\cD_0 \Phi} - \sum_k y_k \widetilde{\cD_k \Phi} - \widetilde{\Phi}, \\
   \bigl( \partial_k + i \widetilde{\cA}_k \bigr) \widetilde{\Phi} &= \widetilde{\cD_k \Phi}, \\
   y \cdot \nabla_y \widetilde{\Phi} &= \sum_k y_k \widetilde{\cD_k \Phi}.
  \end{aligned}
 \end{align}

 \begin{proof}[Proof of Proposition~\ref{prop:monotonicity_lyapunov}]
  In order to justify the computations in the derivation of the monotonicity identity \eqref{equ:monotonicity_lyapunov}, we shall have to assume smoothness of $\bigl( \cA, \Phi \bigr)$. However, smoothing the components as in Definition~\ref{defn:energy_class_solution} destroys the crucial support properties of $\Phi$ and the curvature components $\F_{\alpha \beta}$, and thus of $\widetilde{\Phi}$ and $\widetilde{\cF_{\alpha \beta}}$. In that sense certain expressions below involving weights of the form $(1-|y|^2)^{-\frac{1}{2}}$ or $(1-|y|^2)^{-\frac{3}{2}}$ become singular at $|y| = 1$. To deal with this, we need to introduce an additional smooth cutoff $\chi \bigl( \frac{1 - |y|}{\varepsilon} \bigr)$ that smoothly localizes away from the boundary, but such that $\lim_{\varepsilon \rightarrow 0} \chi \bigl(\frac{y}{\varepsilon}\bigr) = \chi_{(0,1]}(|y|)$, where $\chi_{(0,1]}$ is the sharp characteristic cutoff to the the interval $(0,1]$. Thus, for the calculations below to be rigorous, we really need to consider the weight
  \[
   \chi \biggl( \frac{1 - |y|}{\varepsilon} \biggr) \frac{1}{(1-|y|^2)^{-\frac{1}{2}}},
  \]
  which will lead to additional error terms localized near the boundary. But then letting the frequency cutoff converge toward $|\xi| = +\infty$ in the regularization for fixed but sufficiently small $\varepsilon>0$, it will be easy to convince oneself that the additional errors vanish in the limit due to the support properties \eqref{equ:min_blowup_solution_support_in_ball} of the underlying $(\cA, \Phi)$. We shall formally omit this additional cutoff. 

  \medskip 

  Further, in order to simplify the computations below, we impose the Cronstrom-type gauge condition \eqref{equ:cronstrom_type_gauge} on $(\cA, \Phi)$. This leads to another technical complication in that the $C^\infty$ smoothness of the regularized $(\cA, \Phi)$ in Coulomb gauge will be lost. This can again be dealt with via smooth truncation of the functional, this time away from the origin by including the cutoff $\chi \bigl( \frac{|y|}{\varepsilon} \bigr)$. Since all the integrations by parts to be performed below involve an operator $y \cdot \nabla_y$, the error terms are seen to be controllable in terms of the energy on smaller and smaller balls, and hence negligible in the limit as $\varepsilon \to 0$. Again, we shall gloss over this technicality in the formulas below. 

  \medskip

  We now begin with the derivation of the monotonicity identity \eqref{equ:monotonicity_lyapunov}, where we assume that $\bigl( \widetilde{\cA}, \widetilde{\Phi} \bigr)(s,y,\delta)$ satisfy the Cronstrom-type gauge condition
  \begin{equation} \label{equ:cronstrom_type_gauge_in_derivation}
   \sum_{k=1}^4 y_k \widetilde{\cA}_k(s,y,\delta) = 0
  \end{equation}
  for all $0 \leq s < \log(\frac{1}{\delta})$, $y \in \R^4$ and are smooth solutions to the Maxwell-Klein-Gordon system \eqref{equ:mkg_system_self_sim} in self-similar variables. In order to make the notation less heavy in this derivation, we write $(\tilde{A}, \tilde{\phi})$ instead of $\bigl( \widetilde{\cA}, \widetilde{\Phi} \bigr)$, and $\widetilde{D_\alpha \phi}, \widetilde{F_{\alpha \beta}}$ instead of $\widetilde{\cD_\alpha \Phi}$, $\widetilde{\cF_{\alpha \beta}}$. We will repeatedly apply the following easily verified identities without further referencing,
  \begin{align*}
   ( \partial_s + i \tA_0  + 1 + y \cdot \nabla_y ) \tp &= \widetilde{D_0 \phi}, \\
   (\partial_k + i \tA_k) (\partial_s + i \tA_0) &= (\partial_s + i \tA_0) (\partial_k + i \tA_k) + i (\partial_k \tA_0 - \partial_s \tA_k), 
  \end{align*}
  where $\partial_k$ denotes partial differentiation with respect to the $y$ variable. We also recall the identities \eqref{equ:identities_self_sim}. Moreover, we use the notation
  \[
   \rho(y) = \frac{1}{(1-|y|^2)^{\frac{1}{2}}}
  \]
  and observe that
  \[
   \partial_k \rho(y) = y_k \rho(y)^3, \quad (1 + y \cdot \nabla_y) \rho(y) = \rho(y)^3.
  \]
  The equation for $\tp$ can be written in expanded form as
  \begin{equation*}
   (\partial_s + i \tA_0 + 2 + y \cdot \nabla_y) (\partial_s + i \tA_0 + 1 + y\cdot\nabla_y) \tp = \sum_k (\partial_k + i \tA_k)^2 \tp,
  \end{equation*}
  or alternatively as
  \begin{equation*}
   (\partial_s + i \tA_0)^2 \tp + (3 + 2 y \cdot \nabla_y) (\partial_s + i \tA_0) \tp + (2 + y \cdot \nabla_y) (1 + y \cdot \nabla_y) \tp - i (y \cdot \nabla_y \tA_0) \tp = \sum_k (\partial_k + i \tA_k)^2 \tp.
  \end{equation*}
  We start analyzing the derivative with respect to $s$ of the following energy functional
  \begin{align*}
   \frac{d}{ds} \int \frac{1}{2} |(\partial_s + i\tA_0)\tp|^2 \rho(y) \, dy &= \int \Re \, \Bigl( \partial_s (\partial_s + i \tA_0)\tp \, \overline{(\partial_s + i\tA_0)\tp} \Bigr) \rho(y) \, dy \\
   &= \int \Re \, \Bigl( (\partial_s + i \tA_0)^2 \tp \, \overline{(\partial_s + i\tA_0)\tp} \Bigr) \rho(y) \, dy.
  \end{align*}
  Inserting the equation for $\tp$, we obtain 
  \begin{align*}
   \int \Re \, \Bigl( (\partial_s + i \tA_0)^2 \tp \, \overline{(\partial_s + i\tA_0)\tp} \Bigr) \rho(y) \, dy &= \sum_k \int \Re \, \Bigl( (\partial_k + i \tA_k)^2 \tp \, \overline{ (\partial_s + i \tA_0) \tp } \Bigr) \rho(y) \, dy \\
   &\quad - \int \Re \, \Bigl( (3 + 2 y \cdot \nabla_y) (\partial_s + i \tA_0) \tp \, \overline{(\partial_s + i \tA_0) \tp} \Bigr) \rho(y) \, dy \\
   &\quad - \int \Re \, \Bigl( ( 2 + y \cdot \nabla_y) (1 + y \cdot \nabla_y ) \tp \, \overline{ (\partial_s + i \tA_0) \tp } \Bigr) \rho(y) \, dy \\
   &\quad + \int \Re \, \Bigl( i (y \cdot \nabla_y \tA_0) \tp \, \overline{ (\partial_s + i \tA_0) \tp } \Bigr) \rho(y) \, dy \\
   &\equiv I + II + III + IV.   
  \end{align*}
  Integrating by parts in the term $I$, we find
  \begin{align*}
   I &= - \sum_k \int \Re \, \Bigl( (\partial_k + i \tA_k) \tp \, \overline{(\partial_k + i \tA_k) (\partial_s + i \tA_0) \tp} \Bigr) \rho(y) \, dy \\
   &\quad - \sum_k \int \Re \, \Bigl( (\partial_k + i \tA_k) \tp \, \overline{(\partial_s + i \tA_0) \tp} \Bigr) \partial_k \rho(y) \, dy \\
   &= - \sum_k \int \Re \, \Bigl( (\partial_k + i \tA_k) \tp \, \overline{ (\partial_s + i \tA_0) (\partial_k + i \tA_k) \tp } \Bigr) \rho(y) \, dy \\
   &\quad - \sum_k \int \Re \, \Bigl( (\partial_k + i \tA_k) \tp \, \overline{ i (\partial_k \tA_0 - \partial_s \tA_k) \tp } \Bigr) \rho(y) \, dy \\
   &\quad - \sum_k \int \Re \, \Bigl( (\partial_k + i \tA_k) \tp \, \overline{ (\partial_s + i \tA_0) \tp } \Bigr) y_k \rho(y)^3 \, dy \\
   &= - \frac{d}{ds} \int \biggl( \sum_k \frac{1}{2} |(\partial_k + i \tA_k)\tp|^2 \biggr) \rho(y) \, dy \\
   &\quad - \sum_k \int \Im \, \Bigl( \tp \overline{\widetilde{D_k \phi}} \Bigr) (\partial_s \tA_k - \partial_k \tA_0) \rho(y) \, dy \\
   &\quad - \sum_k \int \Re \, \Bigl( \partial_k \tp \, \overline{(\partial_s + i \tA_0)\tp} \Bigr) y_k \rho(y)^3 \, dy,
  \end{align*}
  where in the last step we took advantage of the gauge condition \eqref{equ:cronstrom_type_gauge_in_derivation}. We expect the second to last term to cancel against a corresponding term from a suitable energy functional for $\tA$. On the other hand, the last term is expected to cancel against other terms from the equation for $\tp$. Next, integrating by parts in the term $II$ yields
  \[
   II = - \int \Re \, \Bigl( (3 + 2 y \cdot \nabla_y) (\partial_s + i \tA_0) \tp \, \overline{(\partial_s + i \tA_0) \tp} \Bigr) \rho(y) \, dy = \int |(\partial_s + i \tA_0) \tp|^2 \rho(y)^3 \, dy.
  \]
  Performing another round of integration by parts, now in the term $III$, we find that
  \begin{align*}
   III &= - \int \Re \, \Bigl( ( 2 + y \cdot \nabla_y) (1 + y \cdot \nabla_y ) \tp \, \overline{ (\partial_s + i \tA_0) \tp } \Bigr) \rho(y) \, dy \\
   &= \int \Re \, \Bigl( (1 + y \cdot \nabla_y) \tp \overline{(\partial_s + i \tA_0) (1+y\cdot\nabla_y) \tp} \Bigr) \rho(y) \, dy \\
   &\quad + \int \Re \, \Bigl( (1 + y \cdot \nabla_y) \tp \, \overline{ i (y \cdot \nabla_y \tA_0) \tp } \Bigr) \rho(y) \, dy \\
   &\quad + \int \Re \, \Bigl( y \cdot \nabla_y \tp \, \overline{(\partial_s + i \tA_0) \tp} \Bigr) \rho(y)^3 \, dy \\
   &\quad + \int \Re \, \Bigl( \tp \, \overline{(\partial_s + i \tA_0) \tp} \Bigr) \rho(y)^3 \, dy,
  \end{align*}
  and thus, 
  \begin{align*}
   III &= \frac{d}{ds} \int \frac{1}{2} |(1 + y \cdot \nabla_y) \tp|^2 \rho(y) \, dy \\
   &\quad + \int \Re \, \Bigl( (1 + y \cdot \nabla_y) \tp \, \overline{i (y \cdot \nabla_y \tA_0) \tp} \Bigr) \rho(y) \, dy \\
   &\quad + \sum_k \int \Re \, \Bigl( \partial_k \tp \, \overline{ (\partial_s + i \tA_0) \tp } \Bigr) y_k \rho(y)^3 \, dy \\
   &\quad + \frac{d}{ds} \int \frac{1}{2} |\tp|^2 \rho(y)^3 \, dy.
  \end{align*}
  Now we note that the second term on the right hand side of $III$ and the term $IV$ nicely combine to give
  \begin{align*}
   &\int \Re \, \Bigl( (1 + y \cdot \nabla_y) \tp \, \overline{ i (y \cdot \nabla_y \tA_0) \tp } \Bigr) \rho(y) \, dy + \int \Re \, \Bigl( i (y \cdot \nabla_y \tA_0) \tp \, \overline{ (\partial_s + i \tA_0) \tp } \Bigr) \rho(y) \, dy \\
   &= \int \Re \, \Bigl( i (y \cdot \nabla_y \tA_0) \tp \, \overline{ (\partial_s + i \tA_0 + 1 + y \cdot \nabla_y) \tp } \Bigr) \rho(y) \, dy \\
   &= - \int (y \cdot \nabla_y \tA_0 ) \Im \, \Bigl( \tp \overline{ \widetilde{D_0 \phi} } \Bigr) \rho(y) \, dy.
  \end{align*}
  Moreover, we observe that the last term on the right hand side of $I$ cancels out with the third term on the right hand side of $III$. We summarize the preceding computations as follows
  \begin{align} \label{equ:derivative_phi_part_lyapunov_self_sim}
   \begin{aligned}
    &\frac{d}{ds} \int \biggl( \frac{1}{2} |(\partial_s + i \tA_0) \tp|^2 + \sum_k \frac{1}{2} |(\partial_k + i \tA_k) \tp|^2 - \frac{1}{2} |(1 + y \cdot \nabla_y) \tp|^2 - \frac{1}{2} |\tp|^2 \rho(y)^2 \biggr) \rho(y) \, dy \\
    &= \int |(\partial_s + i \tA_0) \tp|^2 \rho(y)^3 \, dy - \sum_k \int \Im \, \Bigl( \tp \overline{\widetilde{D_k \phi}} \Bigr) (\partial_s \tA_k - \partial_k \tA_0) \rho(y) \, dy \\
    &\quad \quad  - \int (y \cdot \nabla_y \tA_0 ) \Im \, \Bigl( \tp \overline{ \widetilde{D_0 \phi} } \Bigr) \rho(y) \, dy.
   \end{aligned}
  \end{align}
  We expect the last two terms on the right hand side to cancel against corresponding terms generated by differentiating a suitable energy functional for $\tA$, while the first term furnishes the key monotonicity.

  \medskip

  At this point we have to pass to the corresponding equation for $\tilde{A}$. It is given in expanded form for $j = 1, \ldots, 4$ by
  \begin{align*}
   &\partial_s (\partial_s \tA_j - \partial_j \tA_0) + (3 + 2 y \cdot \nabla_y) (\partial_s \tA_j) + (2 + y \cdot \nabla_y) (1 + y \cdot \nabla_y) \tA_j \\
   &\qquad \qquad \qquad \qquad - (2 + y \cdot \nabla_y) (\partial_j \tA_0) - \sum_k \partial_k (\partial_k \tA_j - \partial_j \tA_k) = \Im \, \Bigl( \tp \overline{ \widetilde{D_j \phi} } \Bigr).
  \end{align*}
  We begin with a tentative ansatz for the correct energy functional for $\tA$ to leading order, which we differentiate with respect to $s$,
  \begin{align*}
   &\frac{d}{ds} \int \biggl( \frac{1}{2} \sum_j (\partial_s \tA_j - \partial_j \tA_0)^2 + \frac{1}{4} \sum_{j,k} (\partial_j \tA_k - \partial_k \tA_j)^2 \biggr) \rho(y) \, dy \\
   &= \int \sum_j \partial_s (\partial_s \tA_j - \partial_j \tA_0) \, (\partial_s \tA_j - \partial_j \tA_0) \rho(y) \, dy + \frac{1}{2} \int \sum_{j,k} \partial_s (\partial_j \tA_k - \partial_k \tA_j) \, (\partial_j \tA_k - \partial_k \tA_j) \rho(y) \, dy \\
   &\equiv \text{\ding{172}} + \text{\ding{173}}.
  \end{align*}
  Inserting the equation for $\tA$ in the term \ding{172}, we obtain
  \begin{align*}
   \text{\ding{172}} &= - \sum_j \int (3 + 2 y \cdot \nabla_y) (\partial_s \tA_j) \, (\partial_s \tA_j - \partial_j \tA_0) \rho(y) \, dy \\
   &\quad - \sum_j \int (2 + y \cdot \nabla_y) (1 + y \cdot \nabla_y) \tA_j \, (\partial_s \tA_j - \partial_j \tA_0) \rho(y) \, dy \\
   &\quad + \sum_j \int (2 + y \cdot \nabla_y) (\partial_j \tA_0) \, (\partial_s \tA_j - \partial_j \tA_0) \rho(y) \, dy \\
   &\quad + \sum_{j,k} \int \partial_k (\partial_k \tA_j - \partial_j \tA_k) \, (\partial_s \tA_j - \partial_j \tA_0) \rho(y) \, dy \\
   &\quad + \sum_j \int \Im \, \Bigl( \tp \overline{ \widetilde{D_j \phi} } \Bigr) \, (\partial_s \tA_j - \partial_j \tA_0) \rho(y) \, dy \\
   &= \widetilde{I} + \widetilde{II} + \widetilde{III} + \widetilde{IV} + \widetilde{V},
  \end{align*}
  where we already see that the term $\widetilde{V}$ cancels against the second term on the right hand side of \eqref{equ:derivative_phi_part_lyapunov_self_sim}. The term \ding{173} can be rewritten as 
  \begin{align*}
   \text{\ding{173}} &= \frac{1}{2} \sum_{j,k} \int \Bigl( \partial_j (\partial_s \tA_k - \partial_k \tA_0) - \partial_k (\partial_s \tA_j - \partial_j \tA_0) \Bigr) \, (\partial_j \tA_k - \partial_k \tA_j) \rho(y) \, dy \\
   &= \sum_{j,k} \int \partial_j (\partial_s \tA_k - \partial_k \tA_0) \, (\partial_j \tA_k - \partial_k \tA_j) \rho(y) \, dy \\
   &= - \sum_{j,k} \int (\partial_s \tA_k - \partial_k \tA_0) \, \partial_j (\partial_j \tA_k - \partial_k \tA_j) \rho(y) \, dy \\
   &\quad - \sum_{j,k} \int (\partial_s \tA_k - \partial_k \tA_0) \, (\partial_j \tA_k - \partial_k \tA_j) y_j \rho(y)^3 \, dy \\
   &= - \sum_{j,k} \int (\partial_s \tA_k - \partial_k \tA_0) \, \partial_j (\partial_j \tA_k - \partial_k \tA_j) \rho(y) \, dy \\
   &\quad - \sum_k \int (\partial_s \tA_k - \partial_k \tA_0) \, (1 + y \cdot \nabla_y) \tA_k \rho(y)^3 \, dy,
  \end{align*}
  where in the second to last step we integrated by parts and in the last step we used that 
  \[
   \sum_j y_j (\partial_j \tA_k - \partial_k \tA_j) = (1 + y \cdot \nabla_y) \tA_k
  \]
  due to the gauge condition \eqref{equ:cronstrom_type_gauge_in_derivation}. We see that the term $\widetilde{IV}$ on the right hand side of \ding{172} cancels against the first term on the right hand side of \ding{173}. Next, we integrate by parts in the term $\widetilde{I}$ to find
  \begin{align*}
   \widetilde{I} &= - \sum_j \int (3 + 2 y \cdot \nabla_y) (\partial_s \tA_j) \, (\partial_s \tA_j) \rho(y) \, dy + \sum_j \int (3 + 2 y \cdot \nabla_y) (\partial_s \tA_j) \, (\partial_j \tA_0) \rho(y) \, dy \\
   &= \sum_j \int (\partial_s \tA_j)^2 \rho(y)^3 \, dy - \sum_j \int (\partial_s \tA_j) \, (5 + 2 y \cdot \nabla_y) (\partial_j \tA_0) \rho(y) \, dy \\
   &\quad - \sum_j 2 \int (\partial_s \tA_j) \, (\partial_j \tA_0) y \cdot \nabla_y \rho(y) \, dy.
  \end{align*}
  Integrating by parts also in the term $\widetilde{II}$ yields
  \begin{align*}
   \widetilde{II} &= \sum_j \int (1 + y \cdot \nabla_y) \tA_j \, (1 + y \cdot \nabla_y) (\partial_s \tA_j - \partial_j \tA_0) \rho(y) \, dy \\
   &\quad + \sum_j \int (1 + y \cdot \nabla_y) \tA_j \, (\partial_s \tA_j - \partial_j \tA_0) \rho(y)^3 \, dy \\
   &= \frac{d}{ds} \int \biggl( \sum_j \frac{1}{2} |(1 + y \cdot \nabla_y) \tA_j|^2 \biggr) \rho(y) \, dy - \sum_j \int (1 + y \cdot \nabla_y) \tA_j \, (1 + y \cdot \nabla_y) (\partial_j \tA_0) \rho(y) \, dy \\
   &\quad + \sum_j \int (1 + y \cdot \nabla_y) \tA_j \, (\partial_s \tA_j - \partial_j \tA_0) \rho(y)^3 \, dy
  \end{align*}
  and we observe that the third term on the right hand side of $\widetilde{II}$ cancels against the second term on the right hand side of \ding{173}. Another round of integration by parts, now in the term $\widetilde{III}$, leads to
  \begin{align*}
   \widetilde{III} &= \sum_j \int (2 + y \cdot \nabla_y) (\partial_j \tA_0) \, (\partial_s \tA_j) \rho(y) \, dy + \sum_j \int \frac{1}{2} (\partial_j \tA_0)^2 \, y \cdot \nabla_y \rho(y) \, dy.
  \end{align*}
  Combining the above expressions, we are thus reduced to 
  \begin{align*} 
    \text{\ding{172}} + \text{\ding{173}} &= \sum_j \int (\partial_s \tA_j)^2 \rho(y)^3 \, dy + \frac{d}{ds} \int \biggl( \sum_j \frac{1}{2} |(1 + y \cdot \nabla_y) \tA_j|^2 \biggr) \rho(y) \, dy \\
    &\quad + \sum_j \int \frac{1}{2} (\partial_j \tA_0)^2 \, y \cdot \nabla_y \rho(y) \, dy - \sum_j 2 \int (\partial_s \tA_j) \, (\partial_j \tA_0) y \cdot \nabla_y \rho(y) \, dy \\
    &\quad - \sum_j \int (\partial_s \tA_j) \, (3 + y \cdot \nabla_y) (\partial_j \tA_0) \rho(y) \, dy \\
    &\quad - \sum_j \int (1 + y \cdot \nabla_y) \tA_j \, (1 + y \cdot \nabla_y) (\partial_j \tA_0) \, \rho(y) \, dy \\
    &\quad + \sum_j \int \Im \, \Bigl( \tp \overline{ \widetilde{D_j \phi} } \Bigr) \, (\partial_s \tA_j - \partial_j \tA_0) \rho(y) \, dy.
  \end{align*}
  We reformulate this as
  \begin{align} \label{equ:derivation_monotonicity_sum_1_2}
   \begin{aligned} 
    \text{\ding{172}} + \text{\ding{173}} &= \sum_j \int (\partial_s \tA_j)^2 \rho(y)^3 \, dy + \frac{d}{ds} \int \biggl( \sum_j \frac{1}{2} |(1 + y \cdot \nabla_y) \tA_j|^2 \biggr) \rho(y) \, dy \\
    &\quad + \sum_j \int \frac{1}{2} (\partial_j \tA_0)^2 \, y \cdot \nabla_y \rho(y) \, dy - \sum_j 2 \int (\partial_s \tA_j) \, (\partial_j \tA_0) \rho(y)^3 \, dy \\
    &\quad - \sum_j \int (\partial_s + 1 + y \cdot \nabla_y) \tA_j \, (1 + y \cdot \nabla_y) (\partial_j \tA_0) \rho(y) \, dy \\
    &\quad + \sum_j \int \Im \, \Bigl( \tp \overline{ \widetilde{D_j \phi} } \Bigr) \, (\partial_s \tA_j - \partial_j \tA_0) \rho(y) \, dy.
   \end{aligned}
  \end{align}
  Next, we further analyze the second to last term on the right hand side of the above identity. Integration by parts gives
  \begin{align} \label{equ:derivation_monotonicity_difficult_term_time_derivative}
   \begin{aligned}
    &- \sum_j \int (\partial_s + 1 + y \cdot \nabla_y) \tA_j \, (1 + y \cdot \nabla_y) (\partial_j \tA_0) \rho(y) \, dy \\
    &= - \sum_j \int (\partial_s + 1 + y \cdot \nabla_y) \tA_j \, \partial_j  (y \cdot \nabla_y \tA_0) \rho(y) \, dy \\
    &= + \sum_j \int \partial_j (\partial_s + 1 + y \cdot \nabla_y) \tA_j \, (y \cdot \nabla_y \tA_0) \rho(y) \, dy \\
    &\quad + \sum_j \int (\partial_s + y + y \cdot \nabla_y) \tA_j \, (y \cdot \nabla_y \tA_0) y_j \rho(y)^3 \, dy \\
    &= + \sum_j \int \partial_j (\partial_s + 1 + y \cdot \nabla_y) \tA_j \, (y \cdot \nabla_y \tA_0) \rho(y) \, dy,
   \end{aligned}
  \end{align}
  where in the last step we used that due to the gauge condition \eqref{equ:cronstrom_type_gauge_in_derivation},
  \[
   \sum_j y_j (\partial_s + 1 + y \cdot \nabla_y) \tA_j = 0.
  \]
  Moreover, one easily verifies that
  \begin{equation} \label{equ:derivation_monotonicity_difficult_term_manipulation}
   \sum_j \partial_j (\partial_s + 1 + y \cdot \nabla_y) \tA_j = \sum_j \partial_j^2 \tA_0 + \sum_j \partial_j \widetilde{F_{0j}} = \sum_j \partial_j^2 \tA_0 + \Im \, \Bigl( \phi \overline{\widetilde{D_0 \phi}} \Bigr),
  \end{equation}
  where in the last equality we linked with the equation for $\tilde{A}$. Inserting \eqref{equ:derivation_monotonicity_difficult_term_manipulation} back into \eqref{equ:derivation_monotonicity_difficult_term_time_derivative} and integrating by parts several times more, we conclude that
  \begin{align} \label{equ:derivation_monotonicity_difficult_term_time_derivative_better}
   \begin{aligned}
    &- \sum_j \int (\partial_s + 1 + y \cdot \nabla_y) \tA_j \, (1 + y \cdot \nabla_y) (\partial_j \tA_0) \rho(y) \, dy \\
    &= \sum_j \int (\partial_j \tA_0)^2 \rho(y) \, dy + \sum_j \int \frac{1}{2} (\partial_j \tA_0)^2 y \cdot \nabla_y \rho(y) \, dy \\
    &\quad - \int (y \cdot \nabla_y \tA_0)^2 \rho(y)^3 \, dy + \int \Im \, \Bigl( \phi \overline{\widetilde{D_0 \phi}} \Bigr) \, (y \cdot \nabla_y \tA_0) \rho(y) \, dy.
   \end{aligned}
  \end{align}
  Finally, inserting \eqref{equ:derivation_monotonicity_difficult_term_time_derivative_better} back into \eqref{equ:derivation_monotonicity_sum_1_2} and combining terms, we may summarize the preceding computations as follows
  \begin{align} \label{equ:derivative_A_part_lyapunov_self_sim}
   \begin{aligned}
    &\frac{d}{ds} \int \biggl( \frac{1}{2} \sum_j (\partial_s \tA_j - \partial_j \tA_0)^2 + \frac{1}{4} \sum_{j,k} (\partial_j \tA_k - \partial_k \tA_j)^2 - \sum_j \frac{1}{2} |(1 + y \cdot \nabla_y) \tA_j|^2  \biggr) \rho(y) \, dy \\
    &= \int \sum_j (\partial_s \tA_j - \partial_j \tA_0)^2 \rho(y)^3 \, dy - \int (y \cdot \nabla_y \tA_0)^2 \rho(y)^3 \, dy \\
    &\quad + \int \Im \, \Bigl( \phi \overline{\widetilde{D_0 \phi}} \Bigr) \, (y \cdot \nabla_y \tA_0) \rho(y) \, dy + \sum_j \int \Im \, \Bigl( \tp \overline{ \widetilde{D_j \phi} } \Bigr) \, (\partial_s \tA_j - \partial_j \tA_0) \rho(y) \, dy.
   \end{aligned}
  \end{align}
  We observe that the last two terms on the right hand side cancel against the last two terms on the right hand side of \eqref{equ:derivative_phi_part_lyapunov_self_sim}. However, it is not yet obvious that the first two terms on the right hand side of the above identity \eqref{equ:derivative_A_part_lyapunov_self_sim} yield the desired monotonicity. To this end we decompose the $4$-vector $(\partial_j \tA_0)_{j=1}^4$ into its radial and angular part. The gauge condition \eqref{equ:cronstrom_type_gauge_in_derivation} then allows to rewrite this as
  \begin{align} \label{equ:derivation_monotonicity_rewrite_rhs}
   \begin{aligned}
    &\int \sum_j (\partial_s \tA_j - \partial_j \tA_0)^2 \rho(y)^3 \, dy - \int (y \cdot \nabla_y \tA_0)^2 \rho(y)^3 \, dy \\
    &= \int \sum_j \biggl( \partial_s \tA_j - \partial_j \tA_0 + \Bigl( \frac{y}{|y|} \cdot \nabla_y \tA_0 \Bigr) \frac{y_j}{|y|} \biggr)^2 \rho(y)^3 \, dy \\
    &\quad + \int \Bigl( \frac{y}{|y|} \cdot \nabla_y \tA_0 \Bigr)^2 \rho(y)^3 \, dy - \int \bigl( y \cdot \nabla_y \tA_0 \bigr)^2 \rho(y)^3 \, dy \\
    &= \int \sum_j \biggl( \partial_s \tA_j - \partial_j \tA_0 + \Bigl( \frac{y}{|y|} \cdot \nabla_y \tA_0 \Bigr) \frac{y_j}{|y|} \biggr)^2 \rho(y)^3 \, dy + \int \frac{1}{|y|^2} \bigl( y \cdot \nabla_y \tA_0 \bigr)^2 \rho(y) \, dy. 
   \end{aligned}
  \end{align}
  Combining \eqref{equ:derivative_phi_part_lyapunov_self_sim}, \eqref{equ:derivative_A_part_lyapunov_self_sim}, and \eqref{equ:derivation_monotonicity_rewrite_rhs} finishes the proof of the monotonicity identity \eqref{equ:monotonicity_lyapunov}.

  \medskip

  It remains to prove \eqref{equ:limit_lyapunov_energy_at_time_infty}. Using the gauge invariant formulation of the Lyapunov functional~$\widetilde{E}$, we proceed exactly as in the proof of Proposition~6.2 (iii) in \cite{KM} to show that for all $\delta > 0$,
  \[
   \lim_{s \to \log(\frac{1}{\delta})} \int_{B_1} \biggl( \frac{1}{2} \sum_j \widetilde{\cF_{j0}}^2 + \frac{1}{4} \widetilde{\cF_{jk}}^2 + \frac{1}{2} \sum_\alpha \bigl|\widetilde{\cD_\alpha \Phi}\bigr|^2 \biggr) \frac{dy}{(1-|y|^2)^{\frac{1}{2}}} \leq E_{crit},
  \]
  while 
  \[
   \lim_{s \to \log(\frac{1}{\delta})} \int_{B_1} \biggl( \sum_{j,k} y_k \widetilde{\cF_{j0}} \widetilde{\cF_{jk}} + \sum_k y_k \Re \big( \widetilde{\cD_k \Phi} \overline{ \widetilde{\cD_0 \Phi} } \big) + \Re \big( \widetilde{\Phi} \overline{ \widetilde{\cD_0 \Phi} } \big) + \frac{1}{2} \frac{|\widetilde{\Phi}|^2}{1-|y|^2} \biggr) \frac{dy}{(1-|y|^2)^{\frac{1}{2}}} = 0.
  \]
 \end{proof}

 Next, we prove upper and lower bounds for the Lyapunov functional $\widetilde{E}(s)$ uniformly in $\delta > 0$ and $0 \leq s < \log(\frac{1}{\delta})$.

 \begin{lem} \label{lem:upper_lower_bound_lyapunov}
  For all $\delta > 0$ and all $0 \leq s < \log(\frac{1}{\delta})$, we have 
  \begin{equation}
   -C E_{crit} \leq \widetilde{E}(s) \leq E_{crit}
  \end{equation}
  for some absolute constant $C > 0$.
 \end{lem}
 \begin{proof}
  The upper bound is immediate from \eqref{equ:limit_lyapunov_energy_at_time_infty} and the monotonicity property \eqref{equ:monotonicity_lyapunov} of the functional $\widetilde{E}$. In order to prove the lower bound, we work with the gauge invariant formulation of the Lyapunov functional $\widetilde{E}$ and first observe that for $|y| \leq 1$, the quantities
  \[
   \frac{1}{2} \sum_\alpha \bigl|\widetilde{\cD_\alpha \Phi}\bigr|^2 - \sum_k y_k \Re \bigl( \widetilde{\cD_k \Phi} \overline{\widetilde{\cD_0 \Phi}} \bigr)
  \]
  and
  \[
   \frac{1}{2} \sum_j \widetilde{\cF_{j0}}^2 + \frac{1}{4} \sum_{j,k} \widetilde{\cF_{jk}}^2 - \sum_{j,k} y_k \widetilde{\cF_{j0}} \widetilde{\cF_{jk}}
  \]
  are non-negative. This is straightforward to see for the first expression, while for the second one we use that
  \[
   \sum_{j,k} y_k \widetilde{\cF_{j0}} \widetilde{\cF_{jk}} = \frac{1}{2} |y| \sum_{j,k} \Bigl( \frac{y_k}{|y|} \widetilde{\cF_{j0}} - \frac{y_j}{|y|} \widetilde{F_{k0}} \Bigr) \widetilde{\cF_{jk}} \leq \frac{1}{4} |y| \sum_{j,k} \Bigl( \frac{y_k}{|y|} \widetilde{\cF_{j0}} - \frac{y_j}{|y|} \widetilde{F_{k0}} \Bigr)^2 + \frac{1}{4} |y| \sum_{j,k} \widetilde{\cF_{j,k}}^2.
  \]
  From the general identity
  \[
   \sum_{j,k} \bigl(\omega_k r_j - \omega_j r_k\bigr)^2 = 2 \bigl(r^2 - (r \cdot \omega)^2 \bigr) \leq 2 r^2
  \]
  for $r, \omega \in \R^4$ with $|\omega|=1$, we then conclude that
  \[
   \sum_{j,k} y_k \widetilde{\cF_{j0}} \widetilde{\cF_{jk}} \leq \frac{1}{2} |y| \sum_j \widetilde{\cF_{j0}}^2 + \frac{1}{4} |y| \sum_{j,k} \widetilde{\cF_{jk}}^2.
  \]
  It therefore suffices to obtain an upper bound on 
  \[
   \int_{B_1} \biggl( \Re \bigl( \widetilde{\Phi} \overline{\widetilde{\cD_0 \Phi}} \bigr) + \frac{1}{2} \frac{|\widetilde{\Phi}|^2}{1-|y|^2} \biggr) \frac{dy}{(1-|y|^2)^{\frac{1}{2}}} 
  \]
  uniformly for all $\delta > 0$ and $0 \leq s < \log(\frac{1}{\delta})$. From H\"older's inequality, \eqref{equ:auxiliary_estimate_kinetic_energy_self_sim} and \eqref{equ:hardy_type_inequality_self_sim} we easily infer that 
  \[
   \int_{B_1} \biggl( \Re \bigl( \widetilde{\Phi} \overline{\widetilde{\cD_0 \Phi}} \bigr) + \frac{1}{2} \frac{|\widetilde{\Phi}|^2}{1-|y|^2} \biggr) \frac{dy}{(1-|y|^2)^{\frac{1}{2}}} \lesssim \int_{B_1} |\widetilde{\cD_0 \Phi}|^2 \, dy + \int_{B_1} \frac{|\widetilde{\Phi}|^2}{(1-|y|^2)^2} \, dy \lesssim E_{crit}.
  \]
 \end{proof}

 As a corollary of Proposition~\ref{prop:monotonicity_lyapunov} and Lemma~\ref{lem:upper_lower_bound_lyapunov}, we obtain the following decay property as $\delta \to 0$.

 \begin{cor} \label{cor:decay_self_sim}
  For each $\delta > 0$, there exists $\bar{s}_\delta \in \bigl( \frac{1}{2} \log( \frac{1}{\delta} ), \log( \frac{1}{\delta}) \bigr)$ such that
  \begin{equation}
   \int_{\bar{s}_\delta}^{\bar{s}_\delta + \log( \frac{1}{\delta})^{\frac{1}{2}}} \widetilde{\Xi}(s) \, ds \leq \frac{C E_{crit}}{\log(\frac{1}{\delta})^{\frac{1}{2}}}.
  \end{equation}
 \end{cor}
 \begin{proof}
  From \eqref{equ:monotonicity_lyapunov} and Lemma~\ref{lem:upper_lower_bound_lyapunov} we have that
  \[
   \int_0^{\log(\frac{1}{\delta})} \widetilde{\Xi}(s) \, ds \leq C E_{crit}.
  \]
  Then the claim is immediate.
 \end{proof}

 Our goal is now to extract a limiting solution $\bigl( \cA^\ast, \Phi^\ast \bigr)$ and to eventually show that $\Phi^\ast$ must vanish. This will yield a contradiction to the minimal blowup solution $\bigl( \cA, \Phi \bigr)$ having infinite $S^1$ norm.

 \medskip
 
 Let $\overline{t}_\delta = 1 + \delta - e^{- \overline{s}_\delta}$, where $\overline{s}_\delta$ is as in Corollary~\ref{cor:decay_self_sim}. By Corollary~\ref{cor:reduction_to_exactly_self_similar_case} we can pick a sequence $\delta_l \to 0$ as $l \to \infty$ such that
 \[
  \Bigl( (1-\overline{t}_{\delta_l})^2 \nabla_{t,x} \cA_x \bigl(\overline{t}_{\delta_l}, (1 - \overline{t}_{\delta_l}) x \bigr), (1-\overline{t}_{\delta_l})^2 \nabla_{t,x} \Phi \bigl(\overline{t}_{\delta_l}, (1-\overline{t}_{\delta_l}) x\bigr) \Bigr) \rightarrow \Bigl( \nabla_{t,x} \cA^\ast_x(x), \nabla_{t,x} \Phi^\ast(x) \Bigr)
 \]
 strongly in $\bigl( L^2_x(\R^4) \bigr)^5$ as $\delta_l \to 0$. We may also arrange that 
 \begin{equation} \label{equ:extracting_limiting_solution}
  \begin{aligned}
   &\Bigl( (1 + \delta_l - \overline{t}_{\delta_l})^2 \nabla_{t,x} \cA_x \bigl(\overline{t}_{\delta_l}, (1 + \delta_l - \overline{t}_{\delta_l}) x \bigr), (1 + \delta_l -\overline{t}_{\delta_l})^2 \nabla_{t,x} \Phi \bigl(\overline{t}_{\delta_l}, (1 + \delta_l - \overline{t}_{\delta_l}) x\bigr) \Bigr) \\
   &\qquad \qquad \qquad \qquad \qquad \qquad \qquad \qquad \qquad \qquad \qquad \qquad \rightarrow \Bigl( \nabla_{t,x} \cA^\ast_x(x), \nabla_{t,x} \Phi^\ast(x) \Bigr)
  \end{aligned}
 \end{equation}
 in $\bigl( L^2_x(\R^4) \bigr)^5$ as $\delta_l \to 0$. We now consider the MKG-CG evolutions in the sense of Definition~\ref{defn:energy_class_solution} of the energy class Coulomb data given by the left hand side of \eqref{equ:extracting_limiting_solution}. Denote these evolutions by $\bigl( \cA^{l\ast}, \Phi^{l\ast} \bigr)$. By the perturbative results from Corollary~\ref{cor:lifespancompact}, these evolutions exist on some fixed time interval $[0,T^\ast]$, where we may assume that $0 < T^\ast < 1$. Moreover, we have on $[0,T^\ast]$ that
 \begin{align*}
  \cA^{l\ast} &= (1 + \delta_l - \overline{t}_{\delta_l}) \cA \bigl(\overline{t}_{\delta_l} + (1+\delta_l-\overline{t}_{\delta_l}) t, (1+\delta_l-\overline{t}_{\delta_l})x \bigr), \\
  \Phi^{l\ast}(t,x) &= (1 + \delta_l - \overline{t}_{\delta_l}) \Phi \bigl(\overline{t}_{\delta_l} + (1 + \delta_l - \overline{t}_{\delta_l}) t, (1+\delta_l - \overline{t}_{\delta_l}) x \bigr),
 \end{align*}
 and 
 \[
  \bigl( \nabla_{t,x} \cA^{l\ast}_x(t,\cdot), \nabla_{t,x} \Phi^{l\ast}(t,\cdot) \bigr) \to \bigl( \nabla_{t,x} \cA^\ast_x(t,\cdot), \nabla_{t,x} \Phi^\ast(t,\cdot) \bigr)
 \]
 in $\bigl( L^2_x(\R^4) \bigr)^5$ as $l \to \infty$ uniformly for all $0 \leq t \leq T^\ast$, where $\bigl( \cA^\ast, \Phi^\ast \bigr)$ is a weak solution to MKG-CG on $[0,T^\ast] \times \R^4$. Note that on account of these identities we have 
 \[
  \supp \, \bigl( \Phi^{l\ast}(t,\cdot) \bigr) \subset \biggl\{ x \in \R^4 : |x| \leq \frac{1 - \overline{t}_{\delta_l}}{1+\delta_l-\overline{t}_{\delta_l}} - t < 1 - t \biggr\}
 \]
 and similarly
 \[
  \supp \, \bigl( (\partial_\alpha \cA^{l\ast}_\beta - \partial_\beta \cA_\alpha^{l\ast})(t,\cdot) \bigr) \subset \bigl\{ x \in \R^4 : |x| < 1-t \bigr\}.
 \]
 We now switch to the self-similar variables
 \[
  s = - \log(1-t), \quad y = \frac{x}{1-t}, \quad 0 \leq t \leq T^\ast,
 \]
 and define
 \begin{align*}
  \widetilde{\cA^{l\ast}_\alpha}(s,y) &= e^{-s} \cA^{l\ast}_\alpha(1-e^{-s}, e^{-s}y), \\
  \widetilde{\Phi^{l\ast}}(s,y) &= e^{-s} \Phi^{l\ast}(1-e^{-s}, e^{-s}y),
 \end{align*}
 and similarly for $\bigl( \widetilde{\cA^\ast}, \widetilde{\Phi^\ast} \bigr)$. We conclude exactly as in \cite{KM} after Remark 6.8 there that 
 \begin{align} \label{equ:extracting_limiting_solution_relation_convergent_sequence}
  \begin{aligned}
   \widetilde{\cA^{l\ast}_\alpha}(s,y) &= \widetilde{\cA}_\alpha (\overline{s}_{\delta_l} + s, y, \delta_l), \\
   \widetilde{\Phi^{l\ast}}(s,y) &= \widetilde{\Phi}(\overline{s}_{\delta_l} + s, y, \delta_l),
  \end{aligned}
 \end{align}
 and
 \begin{equation} \label{equ:extracting_limiting_solution_convergence_self_sim}
  \bigl( \nabla_{s,y} \widetilde{\cA^{l\ast}_x}, \nabla_{s,y} \widetilde{\Phi^{l\ast}} \bigr)(s, \cdot) \rightarrow \bigl( \nabla_{s,y} \widetilde{\cA^\ast_x}, \nabla_{s,y} \widetilde{\Phi^\ast} \bigr)(s,\cdot)
 \end{equation}
 in $\bigl( L^2_y(\R^4) \bigr)^5$ as $l \to \infty$ uniformly for all $0 \leq s \leq - \log (1-\frac{1}{2} T^\ast) =: S$. Then $\bigl( \widetilde{\cA^\ast}, \widetilde{\Phi^\ast} \bigr)$ is a weak solution to the Maxwell-Klein-Gordon system in self-similar variables \eqref{equ:mkg_system_self_sim}. Denoting by $\widetilde{\cD_\alpha \Phi^\ast}$ and $\widetilde{\cF_{\alpha \beta}^\ast}$ the covariant derivatives and curvature components in self-similar variables associated with $\bigl( \widetilde{\cA^\ast}, \widetilde{\Phi^\ast} \bigr)$, we conclude that
 \begin{align} \label{equ:limiting_solution_support}
  \begin{aligned}
   \supp \, \bigl\{ \widetilde{\Phi^\ast}(s, \cdot) \bigr\} &\subset \{ y \in \R^4 : |y| \leq 1 \}, \\
   \supp \, \bigl\{ \widetilde{\cF_{\alpha \beta}^\ast}(s,\cdot) \bigr\} &\subset \{ y \in \R^4 : |y| \leq 1 \}.
  \end{aligned}
 \end{align}

 \begin{lem}
  Let $\bigl( \widetilde{\cA^\ast}, \widetilde{\Phi^\ast} \bigr)$ be as above. Then it holds that
  \begin{equation} \label{equ:property_limiting_solution}
   \sum_{j} y_j \widetilde{\cF_{j0}^\ast} \equiv 0, \qquad \widetilde{\cD_0 \Phi^\ast} - \sum_k y_k \widetilde{\cD_k \Phi^\ast} - \widetilde{\Phi^\ast} \equiv 0.
  \end{equation}
 \end{lem}
 \begin{proof}
  For large $l$ we obtain from \eqref{equ:extracting_limiting_solution_relation_convergent_sequence}, \eqref{equ:extracting_limiting_solution_convergence_self_sim}, and Corollary~\ref{cor:decay_self_sim} that 
  \begin{align*}
   &\int_0^S \int_{B_1} \frac{1}{|y|^2} \bigg( \sum_{j} y_j \widetilde{\cF^\ast_{j0}}(s,y) \bigg)^2 \frac{dy}{(1-|y|^2)^{\frac{1}{2}}} \, ds \\
   &\quad + \int_0^S \int_{B_1} \Bigl| \Bigl( \widetilde{\cD_0 \Phi^\ast} - \sum_k y_k \widetilde{\cD_k \Phi^\ast} - \widetilde{\Phi^\ast} \Bigr)(s,y) \Bigr|^2 \frac{dy}{(1-|y|^2)^{\frac{3}{2}}} \, ds \\
   &\leq \liminf_{l\to\infty} \Bigg\{ \int_0^S \int_{B_1} \frac{1}{|y|^2} \bigg( \sum_{j} y_j \widetilde{\cF_{j0}}(\overline{s}_{\delta_l} + s, y, \delta_l) \bigg)^2 \frac{dy}{(1-|y|^2)^{\frac{1}{2}}} \, ds \\
   &\quad \quad \quad \quad \quad + \int_0^S \int_{B_1} \Bigl| \Bigl( \widetilde{\cD_0 \Phi} - \sum_k y_k \widetilde{\cD_k \Phi} - \widetilde{\Phi} \Bigr)(\overline{s}_{\delta_l} + s, y, \delta_l) \Bigr|^2 \frac{dy}{(1-|y|^2)^{\frac{3}{2}}} \, ds \Bigg\}\\
   &\leq \liminf_{l\to\infty} \Bigg\{ \int_{\overline{s}_{\delta_l}}^{\overline{s}_{\delta_l} + S} \int_{B_1} \frac{1}{|y|^2} \bigg( \sum_{j} y_j \widetilde{\cF_{j0}}(s, y, \delta_l) \bigg)^2 \frac{dy}{(1-|y|^2)^{\frac{1}{2}}} \, ds \\
   &\quad \quad \quad \quad \quad + \int_{\overline{s}_{\delta_l}}^{\overline{s}_{\delta_l} + S} \int_{B_1} \Bigl| \Bigl( \widetilde{\cD_0 \Phi} - \sum_k y_k \widetilde{\cD_k \Phi} - \widetilde{\Phi} \Bigr)(s, y, \delta_l) \Bigr|^2 \frac{dy}{(1-|y|^2)^{\frac{3}{2}}} \, ds \Bigg\} \\  
   &\leq \liminf_{l \to \infty} \frac{ C E_{crit} }{\log(\frac{1}{\delta_l})^{\frac{1}{2}}} \\
   &= 0. 
  \end{align*}
 \end{proof}

 \begin{prop} \label{prop:trivial_limiting_solution}
  Let $\bigl( \widetilde{\cA^\ast}, \widetilde{\Phi^\ast} \bigr)$ be as above. Then we have $\widetilde{\Phi^\ast} \equiv 0$.
 \end{prop}
 
 Going back to the $(t,x)$ coordinates, the preceding proposition implies that $\cA^\ast_k$ is a free wave for $k = 1, \ldots, 4$, while $\cA^\ast_0 \equiv 0$. This contradicts Proposition~\ref{prop:minenblowup} and hence completes the proof of Proposition~\ref{prop:rigidity}.

 \begin{proof}[Proof of Proposition~\ref{prop:trivial_limiting_solution}]
  In order to simplify the computations below, we assume that $\bigl( \widetilde{\cA^\ast}, \widetilde{\Phi}^\ast \bigr)$ satisfy the Cronstrom-type gauge condition (in self-similar variables)
  \begin{equation} \label{equ:limiting_solution_in_cronstrom_type_gauge}
   \sum_{k=1}^4 y_k \widetilde{\cA^\ast_k}(s,y) = 0
  \end{equation}
  for all $0 \leq s \leq S$ and $y \in \R^4$. Then the properties \eqref{equ:property_limiting_solution} of the limiting solution $\bigl( \widetilde{\cA^\ast}, \widetilde{\Phi^\ast} \bigr)$ can be written as
  \[
   y \cdot \nabla_y \widetilde{\cA^\ast_0} \equiv 0, \quad \quad \bigl( \partial_s + i \widetilde{\cA^\ast_0} \bigr) \widetilde{\Phi}^\ast \equiv 0
  \]
  and the equation for $\widetilde{\Phi^\ast}$ simplifies to 
  \begin{align*}
   (2 + y \cdot \nabla_y) (1 + y \cdot \nabla_y) \widetilde{\Phi^\ast} = \sum_k \bigl( \partial_k + i \widetilde{\cA^\ast_k} \bigr)^2 \widetilde{\Phi^\ast}.
  \end{align*}
  Integrating this equation against $\widetilde{\Phi^\ast}$, we find
  \begin{equation} \label{equ:limiting_solution_integral}
   \int_{\R^4} \Bigl( (2 + y \cdot \nabla_y ) ( 1 + y \cdot \nabla_y) \widetilde{\Phi^\ast} \Bigr) \overline{\widetilde{\Phi^\ast}} \, dy = - \sum_k \int_{\R^4} \bigl| \bigl( \partial_k + i \widetilde{\cA^\ast_k} \bigr) \widetilde{\Phi^\ast} \bigr|^2 \, dy.
  \end{equation}
  A simple integration by parts shows that the left hand side of \eqref{equ:limiting_solution_integral} is given by
  \[
   \int_{\R^4} \Bigl( (2 + y \cdot \nabla_y ) ( 1 + y \cdot \nabla_y) \widetilde{\Phi^\ast} \Bigr) \overline{\widetilde{\Phi^\ast}} \, dy = 4 \int_{\R^4} \bigl| \widetilde{\Phi^\ast} \bigr|^2 \, dy - \int_{\R^4} \bigl| y \cdot \nabla_y \widetilde{\Phi^\ast} \bigr|^2 \, dy.
  \]
  Decomposing the $4$-vector $\bigl( \partial_k \widetilde{\Phi^\ast} \bigr)_{k=1}^4$ into its radial and angular part, we observe that the gauge condition \eqref{equ:limiting_solution_in_cronstrom_type_gauge} allows to rewrite the right hand side of \eqref{equ:limiting_solution_integral} as 
  \begin{align*}
   &- \sum_k \int \bigl| (\partial_k + i \widetilde{\cA^\ast_k}) \widetilde{\Phi^\ast} \bigr|^2 \, dy \\
   &= - \int_{\R^4} \biggl( \Bigl| \frac{y}{|y|} \cdot \nabla_y \widetilde{\Phi^\ast} \Bigr|^2 + \sum_k \Bigl| \partial_k \widetilde{\Phi^\ast} -  \Bigl( \frac{y}{|y|} \cdot \nabla_y \widetilde{\Phi^\ast} \Bigr) \frac{y_k}{|y|} + i \widetilde{\cA^\ast_k} \widetilde{\Phi^\ast} \Bigr|^2 \biggr) \, dy \\
   &\leq - \int_{\R^4} \Bigl| \frac{y}{|y|} \cdot \nabla_y \widetilde{\Phi^\ast} \Bigr|^2 \, dy.
  \end{align*}
  Thus, we find that
  \[
   4 \int_{\R^4} \bigl| \widetilde{\Phi^\ast} \bigr|^2 \, dy \leq - \, \Biggl( \int_{\R^4} \Bigl| \frac{y}{|y|} \cdot \nabla_y \widetilde{\Phi^\ast} \Bigr|^2 \, dy - \int_{\R^4} \bigl| y \cdot \nabla_y \widetilde{\Phi^\ast} \bigr|^2 \, dy \Biggr),
  \]
  and in view of the support properties \eqref{equ:limiting_solution_support} of $\widetilde{\Phi^\ast}$, we must have $\widetilde{\Phi^\ast} \equiv 0$.
 \end{proof}

 To conclude the rigidity argument, we need to reduce to the additional assumption $\lambda(t) \geq \lambda_0 > 0$ for all $t \in \R$ made in the statement of Proposition~\ref{prop:rigidity}. However, this follows as in Lemma~10.18 in \cite{KS}. 

 \medskip

 Finally, we summarize the proof of the global existence assertion in Theorem~\ref{thm:TheMainTheorem} and address the proof of the scattering assertion.
 \begin{proof}[Proof of Theorem~\ref{thm:TheMainTheorem}]
  From the concentration compactness step in Section~\ref{sec:concentration_compactness_step} and the rigidity argument in this section, we infer the existence of a non-decreasing function $K\colon (0,\infty) \to (0, \infty)$ with the following property: Let $(A_x, \phi)[0]$ be admissible Coulomb class data of energy $E$. Then there exists a unique global admissible solution $(A, \phi)$ to MKG-CG with initial data $(A_x, \phi)[0]$ satisfying the a priori bound
  \[
   \bigl\| (A_x, \phi) \bigr\|_{S^1(\R\times\R^4)} \leq K(E).
  \]
  It remains to prove that the dynamical variables $(A_x, \phi)$ of the global solution $(A,\phi)$ to MKG-CG scatter to finite energy free waves. To this end it suffices to show that
  \[
   \| \Box A_j \|_{N(\R\times\R^4)} < \infty 
  \]
  for $j = 1, \ldots, 4$ and 
  \[
   \| \Box \phi \|_{N(\R\times\R^4)} < \infty.
  \]
  Here the only concern is to bound the low-high interactions in the magnetic interaction term $-2i A_j^{free} \partial^j \phi$ in the equation for $\phi$, where $A_j^{free}$ is the free wave evolution of the initial data $A_j[0]$. In this case, the bound $\|(A_x, \phi)\|_{S^1(\R\times\R^4)} < \infty$ does not suffice and we have to invest our strong assumptions about the spatial decay of the initial data. More precisely, from \cite{KST} we have the following estimate for dyadic frequencies $k_1 \leq k_2 - C$,
  \[  
   \bigl\| P_{k_1} A_j^{free} P_{k_2} \partial^j \phi \bigr\|_N \lesssim \bigl\| P_{k_1} A_x[0] \bigr\|_{\dot{H}^1_x \times L^2_x} \bigl\| P_{k_2} \phi \bigr\|_{S^1}.
  \]
  Thus, we may bound the low-high interactions in the magnetic interaction term $-2i A_j^{free} \partial^j \phi$ by
  \[
   \Bigl\| \sum_{k \in \Z} P_{\leq k - C} A_j^{free} P_k \partial^j \phi \Bigr\|_{N} \lesssim \bigl\| A_x[0] \bigr\|_{\ell^1 (\dot{H}^1_x \times L^2_x)} \| \phi \|_{S^1}.
  \]
  To see that $\bigl\| A_x[0] \bigr\|_{\ell^1 (\dot{H}^1_x \times L^2_x)}$ is finite, we observe that in the Coulomb gauge we have for $j = 1, \ldots, 4$ that
  \[
   A_j = - \Delta^{-1} \partial^l F_{jl}.
  \]
  Hence, we obtain for $j = 1, \ldots, 4$ that
  \[
   \sum_{k\in\Z} \|P_k A_j(0)\|_{\dot{H}^1_x} \lesssim \sum_{k\in\Z} \sum_{l=1}^4 \|P_k \Delta^{-1} \partial^l F_{jl}(0)\|_{\dot{H}^1_x} \lesssim \sum_{k\in\Z} \sum_{l=1}^4 \|P_k F_{jl}(0)\|_{L^2_x} < \infty,
  \]
  since the spatial curvature components $F_{jl}(0)$ are of Schwartz class by assumption. Similarly, we conclude that $\|\partial_t A_x(0)\|_{\ell^1 L^2_x} < \infty$, which finishes the proof.
 \end{proof}

\bibliographystyle{amsplain}
\bibliography{references}

\end{document}